\documentclass[a4paper]{amsart}
\usepackage{amsmath,amsthm,amssymb,latexsym,epic,bbm}
\usepackage{graphicx,enumerate,stmaryrd}
\usepackage[all]{xy}
\xyoption{knot}
\xyoption{poly}
\usepackage[mathscr]{eucal}

\newtheorem{theorem}{Theorem}[section]
\newtheorem{lemma}[theorem]{Lemma}
\newtheorem{remark}[theorem]{Remark}
\newtheorem{corollary}[theorem]{Corollary}
\newtheorem{proposition}[theorem]{Proposition}
\newtheorem{example}[theorem]{Example}
\newtheorem{exercise}[theorem]{Exercise}
\newtheorem{problem}[theorem]{Problem}
\newtheorem{conjecture}[theorem]{Conjecture}
\newtheorem{definition}[theorem]{Definition}
\numberwithin{equation}{section}

\newcommand{\tto}{\twoheadrightarrow}

\font\sc=rsfs10
\newcommand{\cC}{\sc\mbox{C}\hspace{1.0pt}}
\newcommand{\cS}{\sc\mbox{S}\hspace{1.0pt}}

\newcommand{\cA}{\sc\mbox{A}\hspace{1.0pt}}
\newcommand{\cF}{\sc\mbox{F}\hspace{1.0pt}}

\font\scc=rsfs7
\newcommand{\ccC}{\scc\mbox{C}\hspace{1.0pt}}

\makeindex

\begin{document}

\title{Lectures on algebraic categorification}
\author{Volodymyr Mazorchuk}
\date{\today}

\maketitle

\tableofcontents

\section*{Introduction}\label{s0}

The most traditional approach to solving mathematical problems is:
start with a difficult problem and simplify it until it becomes easy 
enough to be solved. However, developing mathematical theories quite
often goes in the opposite direction: starting with an ``easy'' theory 
one tries to ``generalize'' it to something more complicated which could
hopefully describe a much bigger class of phenomena.

An example of the latter is the development of what is now known as
{\em Khovanov homology}. The {\em Jones polynomial} is a very elementary
classical combinatorial invariant appearing in the low dimensional topology. 
However, as most of known topological invariants, it is not an absolute one
(it does not distinguish all knots). Some twelve years ago Mikhail Khovanov
has developed a very advanced ``refinement'' of Jones polynomial which 
seriously increased the level of theoretical sophistication necessary to 
be able to define and work with it. Instead of elementary combinatorics 
and basic algebra, Khovanov's definition was based on category theory and
homological algebra and very soon led to the study of higher categorical 
structures. This ``categorification'' of the Jones polynomial created a
new direction in topology and attracted a lot of attention from some other
parts of mathematics, notably algebra and category theory. Within a few 
years {\em categorification} became an intensively studied subject in
several mathematical areas. It completely changed the viewpoint on many
long standing problems and led to several spectacular results and applications.

This text is a write-up of the lectures given by the author during the Master 
Class ``Categorification'' at {\AA}rhus University, Denmark in October 2010.
It mostly concentrates on algebraical aspects of the theory, presented in the
historical perspective, but also contains several topological applications, 
in particular, an algebraic (or, more precisely, representation theoretical) 
approach to categorification of the Jones polynomial mentioned above. The text 
consists of fifteen sections corresponding to fifteen one hour lectures given 
during the Master Class and fairly describes the content of these lectures. 
There are some exercises (which were proposed to participants of the Master Class) 
collected at the end of the text and a rather extensive list of references.
Unfortunately, the time constrains on the Master Class resulted in the fact that
several recent developments related to categorification did not make it into the text.

The text is aimed to be and introductory overview of the subject rather than 
a fully detailed monograph. The emphasis is made on definitions, examples
and formulations of the results. Most proofs are either shortly outlined or omitted, 
however, complete proofs could be found by tracking references. It is assumed
that the reader is familiar with basics of category theory, representation 
theory, topology and Lie algebras.

\section{Basics: decategorification and categorification}\label{s1}

\subsection{The idea of categorification}\label{s1.1}

The term ``categorification'' was introduced by Louis Crane in \cite{Cr}
and the idea originates from the earlier joint work \cite{CF} with 
Igor Frenkel. The term refers to the process of replacing set-theoretic 
notions by the corresponding category-theoretic analogues as shown in the 
following table:

\begin{center}
\begin{tabular}{|c||c|}
\hline
Set Theory & Category Theory\\
\hline\hline
set & category\\
\hline
element & object\\
\hline
relation between elements& morphism of objects\\
\hline
function & functor\\
\hline
relation between functions& natural transformation of functors \\
\hline
\end{tabular}
\end{center}

The general idea (or hope) is that, replacing a ``simpler'' object by 
something ``more complicated'', one gets a bonus in the form of some extra 
structure which may be used to study the original object. A priori
there are no explicit rules how to categorify some object and the
answer might depend on what kind of extra structure and properties
one expects.

\begin{example}\label{exm1}
{\rm  
The category $\cF\cS$ of finite sets may be considered as
a categorification of the semi-ring $(\mathbb{N}_0,+,\cdot)$
of non-negative integers. In this picture addition is categorified 
via the disjoint union and multiplication via the Cartesian 
product. Note that the categorified operations satisfy 
commutativity, associativity and distributivity laws only up
to a natural isomorphism.
}
\end{example}

In these lectures we will deal with some rather special 
categorifications of algebraic objects (which are quite
different from the above example). There exist many others,
even for the same objects. Our categorifications are usually
motivated by the naturality of their constructions and 
various applications.

It is always easier to ``forget'' information than to ``make it up''.
Therefore it is much more natural to start the study of 
categorification with the study of the opposite process
of forgetting information, called {\em decategorification}.
One of the most natural classical ways to ``forget'' the 
categorical information encoded in a category is  to
consider the corresponding {\em Grothendieck group}.

\subsection{Grothendieck group}\label{s1.2}

Originally, the Grothendieck group is defined for a commutative monoid 
and provides the universal way of making that monoid into an abelian 
group. Let $M=(M,+,0)$ be a  commutative  monoid. 
The {\em Grothendieck group} of $M$ is a pair
\index{Grothendieck group}
$(G,\varphi)$, where $G$ is a commutative group and
$\varphi:M\to G$ is a homomorphism of monoids, such that for 
every monoid homomorphism $\psi:M\to A$, where $A$ is a commutative 
group, there is a unique group homomorphism $\Psi:G\to A$ making the
following diagram commutative:
\begin{displaymath}
\xymatrix{ 
M\ar[rr]^{\varphi}\ar@{.>}[dr]_{\psi}&&G\ar@{=>}[dl]^{\Psi}\\ &A& 
} 
\end{displaymath}
In the language of category theory, the functor that sends a commutative monoid $M$ to its Grothendieck group $G$ is left adjoint to the 
forgetful functor from the category of abelian groups to the category 
of commutative monoids. As usual, uniqueness of the 
Grothendieck group (up to isomorphism) follows directly from the
universal property. Existence is guaranteed by the following construction:

Consider the set $G=M\times M/\sim$, where $(m,n)\sim(x,y)$
if and only if $m+y+s=n+x+s$ for some $s\in M$.

\begin{lemma}\label{lem1}
\begin{enumerate}[$($a$)$]
\item\label{lem1.1} The relation $\sim$ is a congruence on the
monoid $M\times M$ (i.e. $a\sim b$ implies $ac\sim bc$ and
$ca\sim cb$ for all $a,b,c\in M\times M$)
and the quotient $G$ is a commutative group
(the identity element of $G$ is $\overline{(0,0)}$; and the
inverse of $\overline{(m,n)}$ is $\overline{(n,m)}$). 
\item\label{lem1.2} The map $\varphi:M\to G$ defined via
$\varphi(m)=(m,0)$ is a homomorphism of monoids.
\item\label{lem1.3} The pair $(G,\varphi)$ is a Grothendieck group 
of $M$.
\end{enumerate}
\end{lemma}

This idea of the Grothendieck group can be easily generalized
to the situation of essentially small categories 
with some additional structures. Recall that a category is called
{\em essentially small} if its skeleton is small. Categories which
\index{essentially small category}
will normally appear in our context are module categories,
additive subcategories of modules categories, and
derived categories of module categories. All such categories are
easily seen to be essentially small.

The most classical example is the Grothendieck group of an
abelian category. Let $\cA$ be an essentially small 
abelian category with a fixed skeleton $\underline{\cA}$. Then the
{\em Grothendieck group} $[\cA]=\mathrm{K}_0(\cA)$ of $\cA$ is 
\index{Grothendieck group}
defined as the quotient of the free abelian group generated by
$[X]$, where $X\in \underline{\cA}$, modulo the relation $[Y]=[X]+[Z]$
for every exact sequence
\begin{equation}\label{eq01}
0\to X\to Y\to Z\to 0 
\end{equation}
in $\underline{\cA}$. This comes together with the natural map 
$[\cdot]:\cA\to[\cA]$ which maps $M\in\cA$ to the class $[M']$ 
in $[\cA]$, where $M'\in \underline{\cA}$ is the unique object 
satisfying $M\cong M'$. The group $[\cA]$ has the following natural universal 
property: for every abelian group $A$ and for every
{\em additive} function $\chi:\cA\to A$ (i.e. a function such that
\index{additive function}
$\chi(Y)=\chi(X)+\chi(Z)$ for any exact sequence \eqref{eq01})
there is a unique group homomorphism $\overline{\chi}:[\cA]\to A$ 
making the following diagram commutative:
\begin{displaymath}
\xymatrix{ 
\cA\ar[rr]^{[\cdot]}\ar@{.>}[dr]_{\chi}&&[\cA]
\ar@{=>}[dl]^{\overline{\chi}}\\ &A& 
} 
\end{displaymath}
The Grothendieck group of $\cA$ is the ``easiest'' way to make
$\cA$ into just an abelian group. In some classical cases the
group $[\cA]$ admits a very natural description:

\begin{example}\label{exm2}
{\rm
Let $\Bbbk$ be a  field and $\cA=A\text{-}\mathrm{mod}$ the category 
of finite-dimensional (left) modules over some finite dimensional 
$\Bbbk$-algebra $A$. As every $A$-module has a composition series, 
the group $[\cA]$ is isomorphic to the free abelian group with the 
basis given by classes of simple $A$-modules.
}
\end{example}

Similarly one defines the notion of a  Grothendieck group for
additive and triangulated categories. Let $\cA$ be an essentially
small additive category with biproduct $\oplus$ and 
a fixed skeleton $\underline{\cA}$. Then the
{\em split Grothendieck group} $[\cA]_{\oplus}$ of 
\index{split Grothendieck group}
$\cA$  is defined as the quotient of the free abelian group generated 
by $[X]$, where $X\in \underline{\cA}$, modulo the relations $[Y]=[X]+[Z]$
whenever $Y\cong X\oplus Z$. Note that any abelian category is
additive, however, if $\cA$ is abelian, then the group $[\cA]_{\oplus}$ 
can be bigger than $[\cA]$ if there are exact sequences of the
form \eqref{eq01} which do not split.

Let $\cC$ be an essentially small triangulated category with a 
fixed skeleton $\underline{\cC}$. Then the
{\em Grothendieck group} $[\cC]$ of  $\cC$  is defined as the quotient 
\index{Grothendieck group}
of the free abelian group generated by $[X]$, where $X\in \underline{\cC}$, modulo 
the relations $[Y]=[X]+[Z]$ for every distinguished triangle
\begin{displaymath}
X\to Y\to Z\to X[1]. 
\end{displaymath}
Again, a triangulated category is always additive, but $[\cC]_{\oplus}$ 
is usually bigger than $[\cC]$ by the same arguments as for
abelian categories. 

Let $\Bbbk$ be a field and $A$ a finite dimensional $\Bbbk$-algebra. 
Then we have two naturally defined triangulated categories associated 
with $A\text{-}\mathrm{mod}$: the bounded derived category 
$\mathcal{D}^b(A)$ and its subcategory $\mathcal{P}(A)$ of perfect 
complexes (i.e. complexes, quasi-isomorphic to finite complexes
of $A$-projectives). For projective $P$ the map
\begin{displaymath}
\begin{array}{ccc}
[\mathcal{P}(A)] &\overset{\varphi}{\longrightarrow}
&[A\text{-}\mathrm{mod}] \\
\left[P\right]&\mapsto &\left[P\right]
\end{array}
\end{displaymath}
is a group homomorphism. 

\begin{example}\label{exm3}
{\rm 
If $A$ has finite global dimension, then
$\mathcal{D}^b(A)\cong \mathcal{P}(A)$ and the map
$\varphi$ is an isomorphism. This means that in this case the group
$[A\text{-}\mathrm{mod}]$ has another distinguished basis, namely
the one corresponding to isomorphism classes of indecomposable
projective $A$-modules.
}
\end{example}

\begin{example}\label{exm5}
{\rm 
Consider the algebra $D=\mathbb{C}[x]/(x^2)$ of {\em dual numbers}.
\index{dual numbers}
This algebra has a unique simple module $L:=\mathbb{C}$
(which is annihilated by $x$). The projective cover of
$L$ is isomorphic to the left regular module $P:={}_DD$.
The module $P$ has length $2$. Therefore the group
$[D\text{-}\mathrm{mod}]$ is the free abelian group with
basis $[L]$. The group $[\mathcal{P}(D)]$ is the free abelian 
group with basis $[P]$ and $\varphi([P])=2[L]$. This means that
$[\mathcal{P}(D)]$ is a proper subgroup of $[D\text{-}\mathrm{mod}]$.
}
\end{example} 

\subsection{Decategorification}\label{s1.3}

From now on all abelian, triangulated and additive categories are assumed
to be essentially small.

\begin{definition}\label{def6}
{\rm
Let $\cC$ be an abelian or triangulated, respectively additive, category.
Then the  {\em decategorification} of $\cC$ is the abelian
\index{decategorification}
group $[\cC]$, resp. $[\cC]_{\oplus}$.
}
\end{definition}

In what follows, objects which we would like to
categorify will usually be algebras over some base ring (field).
Hence we now have to extend the notion of decategorification to
allow base rings. This is done in the usual way
(see \cite[Section~2]{MS2}) as follows: 
Let $\mathbb{F}$ be a commutative ring with $1$.

\begin{definition}\label{def7}
{\rm
Let $\cC$ be an abelian or triangulated, respectively additive, category.
Then the  {\em $\mathbb{F}$-decategorification} of $\cC$ is the
$\mathbb{F}$-module $[\cC]^{\mathbb{F}}:= \mathbb{F}\otimes_{\mathbb{Z}}
[\cC]$ (resp. $[\cC]^{\mathbb{F}}_{\oplus}:=
\mathbb{F}\otimes_{\mathbb{Z}}  [\cC]_{\oplus}$). 
}
\end{definition}

The element $1\otimes [M]$ of some $\mathbb{F}$-decategorification 
will be denoted by $[M]$ for simplicity. We have 
$[\cC]=[\cC]^{\mathbb{Z}}$ 
and $[\cC]_{\oplus}= [\cC]^{\mathbb{Z}}_{\oplus}$.

\subsection{(Pre)categorification of an $\mathbb{F}$-module}\label{s1.4}

\begin{definition}\label{def8}
{\rm
Let $V$ be an $\mathbb{F}$-module. An {\em $\mathbb{F}$-precategorification}
\index{precategorification}
$(\cC,\varphi)$ of $V$ is an abelian (resp. triangulated or additive)
category $\cC$ with a fixed monomorphism $\varphi$ from $V$ to the $\mathbb{F}$-decategorification of $\cC$.  If $\varphi$ is an isomorphism,
then $(\cC,\varphi)$ is called an {\em $\mathbb{F}$-categorification}
\index{categorification}
of $V$.
}
\end{definition}

Whereas the decategorification of a category is uniquely
defined, there are usually many different categorifications
of an $\mathbb{F}$-module $V$. For example, in case 
$\mathbb{F}=\mathbb{Z}$ we can consider the category 
$A\text{-}\mathrm{mod}$ for any $\Bbbk$-algebra $A$ having exactly
$n$ simple modules and realize $A\text{-}\mathrm{mod}$ as a
categorification of the free module $V=\mathbb{Z}^n$. In particular,
$V$ has the {\em trivial} categorification given by a semisimple category 
\index{trivial categorification}
of the appropriate size, for example by
$\mathbb{C}^{n}\text{-}\mathrm{mod}$.

\begin{definition}\label{def9}
{\rm
Let $V$ be an $\mathbb{F}$-module. Let further $(\cC,\varphi)$
and $(\cA,\psi)$ be two $\mathbb{F}$-(pre)categorifications of $V$ via
abelian (resp. triangulated or additive) categories. An 
exact (resp. triangular or additive) functor $\Phi:\cC\to \cA$
is called  a {\em morphism of categorifications} provided that 
\index{morphism of categorifications}
the following diagram commutes:
\begin{displaymath}
\xymatrix{ 
[\cC]^{\mathbb{F}}\ar[rr]^{[\Phi]}&&[\cA]^{\mathbb{F}}\\
&V\ar[ul]^{\varphi}\ar[ur]_{\psi}&
}
\end{displaymath}
where $[\Phi]$ denotes the $\mathbb{F}$-linear transformation
induced by $\Phi$.
}
\end{definition}

Definition~\ref{def9} turns all $\mathbb{F}$-(pre)categorifications 
of $V$ into a category. In what follows we will usually categorify
$\mathbb{F}$-modules using module categories for finite
dimensional $\Bbbk$-algebras. We note that extending scalars
without changing the category $\cC$ may turn 
an $\mathbb{F}$-precategorification of $V$ into an
$\mathbb{F}'$-categorification of $\mathbb{F}'\otimes_{\mathbb{F}}V$.

\begin{example}\label{exm10}
{\rm  
Consider the algebra $D$ of dual numbers (see Example~\ref{exm5}).
Then there is a unique monomorphism 
$\varphi:\mathbb{Z}\to[D\text{-}\mathrm{mod}]$ such that 
$\varphi(1)=[D]$. The homomorphism $\varphi$ is not surjective,
however, it induces an isomorphism after tensoring over
$\mathbb{Q}$, that is $\overline{\varphi}:
\mathbb{Q}\overset{\sim}{\to}[D\text{-}\mathrm{mod}]^{\mathbb{Q}}$.
Hence $(D\text{-}\mathrm{mod},\varphi)$ is a pre-categorification of
$\mathbb{Z}$ while $(D\text{-}\mathrm{mod},\overline{\varphi})$
is a $\mathbb{Q}$-categorification of $\mathbb{Q}=
\mathbb{Q}\otimes_{\mathbb{Z}}\mathbb{Z}$.
}
\end{example}

The last example to some extend explains the necessity of the notion 
of precategorification. We will usually study categorifications
of various modules. Module structures will be categorified using
functorial actions, say by exact functors. Such functors
are completely determined by their action on the additive
category of projective modules. This means that in most cases the
natural ``basis'' for categorification in the one given by 
indecomposable projectives. As we saw in Example~\ref{exm5},
isomorphism classes of indecomposable projectives do not have to 
form a basis of the decategorification. 

\subsection{Graded setup}\label{s1.5}

By {\em graded} we will always mean $\mathbb{Z}$-graded. Let
\index{graded}
$R$ be a graded ring. Consider the category $R\text{-}\mathrm{gMod}$
of all graded $R$-modules and denote by $\langle 1\rangle$ the
shift of grading autoequivalence of $R\text{-}\mathrm{gMod}$ 
normalized as follows: for a graded module $M=\oplus_{i\in\mathbb{Z}}
M_i$ we have $(M\langle 1\rangle)_j=M_{j+1}$. Assume that 
$\cC$ is a category of graded $R$-modules closed under 
$\langle \pm 1\rangle$ (for example the abelian category 
$R\text{-}\mathrm{gMod}$ or the additive category of graded 
projective modules or  the triangulated derived category of 
graded modules). Then the group $[\cC]$ (resp. $[\cC]_{\oplus}$) 
becomes a $\mathbb{Z}[v,v^{-1}]$-module via $v^{i}[M]=[M\langle -i\rangle]$ 
for any $M\in\cC$, $i\in\mathbb{Z}$.

To extend the notion of decategorification to a category of
graded modules (or complexes of graded modules), let $\mathbb{F}$ 
be a unitary commutative ring and $\iota:\mathbb{Z}[v,v^{-1}]
\to\mathbb{F}$ be a fixed homomorphism of unitary rings. 
Then $\iota$ defines on $\mathbb{F}$ the structure of a (right)
$\mathbb{Z}[v,v^{-1}]$-module.

\begin{definition}\label{def11}
{\rm
The {\em $\iota$-decategorification} of $\cC$ is
\index{decategorification}
the $\mathbb{F}$-module
\begin{displaymath}
[\cC]^{(\mathbb{F},\iota)}:=
\mathbb{F}\otimes_{\mathbb{Z}[v,v^{-1}]} [\cC]\quad
(\text{resp. }[\cC]^{(\mathbb{F},\iota)}_{\oplus}:=
\mathbb{F}\otimes_{\mathbb{Z}[v,v^{-1}]} [\cC]_{\oplus}).
\end{displaymath}
}
\end{definition}

In most of our examples the homomorphism $\iota:\mathbb{Z}[v,v^{-1}]\to
\mathbb{F}$ will be the obvious canonical inclusion. In such cases
we will omit $\iota$ in the notation. We have
\begin{displaymath}
[\cC]=[\cC]^{(\mathbb{Z}[v,v^{-1}],\mathrm{id})},\quad
[\cC]_{\oplus}=[\cC]^{(\mathbb{Z}[v,v^{-1}],\mathrm{id})}_{\oplus}.
\end{displaymath}

\begin{definition}\label{def12}
{\rm
Let $V$ be an $\mathbb{F}$-module. A {\em $\iota$-precategorification}
\index{precategorification}
$(\cC,\varphi)$ of $V$ is an abelian or triangulated, respectively additive,
category $\cC$ with a fixed free action of $\mathbb{Z}$ and a fixed
monomorphism $\varphi$ from $V$ to the $(\mathbb{F},\iota)$-decategorification
of $\cC$. If $\varphi$ is an isomorphism, $(\cC,\varphi)$ is called a
{\em $\iota$-categorification} of $V$.
\index{categorification}
}
\end{definition}

\begin{example}\label{exm14}
{\rm
The algebra $D$ of dual numbers is naturally graded with
$x$ being of degree $2$ (this is motivated by the realization of
$D$ as the cohomology ring of a flag variety). Let
$\cC=D\text{-}\mathrm{gmod}$. Then $[\cC]\cong\mathbb{Z}[v,v^{-1}]$ 
as a $\mathbb{Z}[v,v^{-1}]$-module, hence the graded category $\cC$ 
is a $(\mathbb{Z}[v,v^{-1}],\mathrm{id})$-categorification  
of $\mathbb{Z}[v,v^{-1}]$. 
}
\end{example}

\subsection{Some constructions}\label{s1.6}

As already mentioned before, for any $\mathbb{F}$ we have the
{\em trivial} categorification of the free module $\mathbb{F}^n$
\index{trivial categorification}
given by $\Bbbk^n\text{-}\mathrm{mod}$ and the isomorphism 
$\varphi:\mathbb{F}^n \to [\Bbbk^n\text{-}\mathrm{mod}]^{\mathbb{F}}$ 
which maps the usual basis of $\mathbb{F}^n$ to the basis of 
$[\Bbbk^n\text{-}\mathrm{mod}]^{\mathbb{F}}$ given by 
isomorphism classes of simple modules.

If $A\text{-}\mathrm{mod}$ categorifies some $\mathbb{F}^k$
and $B\text{-}\mathrm{mod}$ categorifies $\mathbb{F}^n$
for some finite dimensional $\Bbbk$-algebras $A$ and $B$,
then $A\oplus B\text{-}\mathrm{mod}$
categorifies $\mathbb{F}^k\oplus\mathbb{F}^n$.
This follows from the fact that every simple $A\oplus B$-modules
is either a simple $A$-module or a simple $B$-module.

If $A\text{-}\mathrm{mod}$ categorifies $\mathbb{F}^k$
and $B\text{-}\mathrm{mod}$ categorifies $\mathbb{F}^n$
for some finite dimensional $\Bbbk$-algebras $A$ and $B$,
then $A\otimes_{\Bbbk}B\text{-}\mathrm{mod}$
categorifies $\mathbb{F}^k\otimes_{\mathbb{F}}\mathbb{F}^n$.
This follows from the fact that simple $A\otimes_{\Bbbk}B$-modules
are of the form $L\otimes_{\Bbbk}N$, where $L$ is a simple
$A$-module and $N$ is a simple $B$-module.

Let $\cA=A\text{-}\mathrm{mod}$ categorify $\mathbb{F}^k$
such that the natural basis of $\mathbb{F}^k$ is given by
the isomorphism classes of simple $A$-modules.
Let $\cC$ be a Serre subcategory of $\cA$ (i.e. for any 
exact sequence \eqref{eq01} in $\cA$ we have
$Y\in\cC$ if and only if $X,Z\in\cC$). Then there is
an idempotent $e\in A$ such that $\cC$ is the category of all
modules annihilated by $e$. Thus $\cC$ is equivalent to
$B\text{-}\mathrm{mod}$, where $B=A/AeA$. The group
$[\cC]$ is a subgroup of $[\cA]$ spanned by all simple
$A$-modules belonging to $\cC$. Hence $\cC$ categorifies
the corresponding direct summand $V$ of $\mathbb{F}^k$. 
We also have the associated abelian quotient category
$\cA/\cC$ which has the same objects as $\cA$ and
morphisms given by
\begin{displaymath}
\cA/\cC(X,Y):= \lim_{\rightarrow}
\cA(X',Y/Y'), 
\end{displaymath}
where the limit is taken over all $X'\subset X$ and $Y'\subset Y$
such that $X/X',Y'\in\cC$.
The category $\cA/\cC$ is equivalent to
$C\text{-}\mathrm{mod}$, where $C=eAe$ (see for example
\cite[Section~9]{AM} for details). It follows that
$\cA/\cC$ categorifies the quotient $\mathbb{F}^k/V$,
which is also isomorphic to the direct complement of
$V$ in $\mathbb{F}^k$.

\section{Basics: from categorification of linear maps 
to $2$-categories}\label{s2}

The aim of the section is to discuss various approaches to 
categorification of algebras and modules. An important
common feature  is that any such approach categorifies
linear maps as a special case.

\subsection{Categorification of linear maps}\label{s2.1}

Let $V,W$ be $\mathbb{F}$-modules and $f:V\to W$ be
a homomorphism. Assume that $(\cA,\varphi)$ and
$(\cC,\psi)$ are abelian (resp. additive, triangulated) 
$\mathbb{F}$-categorifications of $V$ and $W$, respectively.

\begin{definition}\label{def201}
{\rm
An {\em $\mathbb{F}$-categorification} of $f$ is an exact
(resp. additive, triangular) functor $\mathrm{F}:\cA\to\cC$
such that $[\mathrm{F}]\circ\varphi=\psi\circ f$, where
$[\mathrm{F}]:[\cA]^{\mathbb{F}}_{(\oplus)}\to [\cC]^{\mathbb{F}}_{(\oplus)}$
denotes the induced homomorphism. In other words, the
following diagram commutes:
\begin{displaymath}
\xymatrix{ 
V\ar[rr]^f\ar[d]_{\varphi}&&W\ar[d]^{\psi}\\
[\cA]^{\mathbb{F}}_{(\oplus)}\ar[rr]^{[\mathrm{F}]}
&&[\cC]^{\mathbb{F}}_{(\oplus)}
}
\end{displaymath}
}
\end{definition}

For example, the identity functor is a categorification of the
identity morphism; the zero functor is the (unique) categorification
of the zero morphism.

\subsection{Na{\"\i}ve categorification}\label{s2.2}

Let $A$ be an (associative) $\mathbb{F}$-algebra with a fixed 
generating system $\mathtt{A}=\{a_i:i\in I\}$. Given an
$A$-module $M$, every $a_i$ defines a linear transformation
$a_i^M$ of $M$. 

\begin{definition}\label{def202}
{\rm
A {\em na{\"\i}ve $\mathbb{F}$-categorification} of $M$ is a tuple
\index{na{\"\i}ve categorification}
$(\cA,\varphi,\{\mathrm{F}_i:i\in I\})$, where $(\cA,\varphi)$
is an $\mathbb{F}$-categorification of $M$ and for every $i\in I$
the functor $\mathrm{F}_i$ is an $\mathbb{F}$-categorification of $a_i^M$.
}
\end{definition}

There are several natural ways to define morphisms of 
na{\"\i}ve categorifications. Let $A$ and $\mathtt{A}$ be as above,
$M$ and $N$ be two $A$-modules and $(\cA,\varphi,\{\mathrm{F}_i:i\in I\})$,
$(\cC,\psi,\{\mathrm{G}_i:i\in I\})$ be na{\"\i}ve categorifications of
$M$ and $N$, respectively. In what follows dealing with different 
categorifications we always assume that they have the same type
(i.e. either they all are abelian or additive or triangulated).
By a {\em structural} functor we will mean an exact functor
\index{structural functor}
between abelian categories, an additive functor between additive
categories and a triangular functor between triangulated categories.

\begin{definition}\label{def203}
{\rm
A {\em na{\"\i}ve morphism} of categorifications from 
\index{na{\"\i}ve morphism}
$(\cA,\varphi,\{\mathrm{F}_i:i\in I\})$ to
$(\cC,\psi,\{\mathrm{G}_i:i\in I\})$ is a structural functor 
$\Phi:\cA\to\cC$ such that for every $i\in I$ 
the following digram commutes:
\begin{displaymath}
\xymatrix{ 
[\cA]^{\mathbb{F}}\ar[rr]^{[\mathrm{F}_i]}\ar[d]_{[\Phi]}
&& [\cA]^{\mathbb{F}}\ar[d]^{[\Phi]}\\
[\cC]^{\mathbb{F}}\ar[rr]^{[\mathrm{G}_i]}&& [\cC]^{\mathbb{F}},\\
}
\end{displaymath}
where $[\Phi]$ denotes the morphism, induced by $\Phi$.
}
\end{definition}

\begin{definition}\label{def204}
{\rm
A {\em weak morphism} from 
\index{weak morphism}
$(\cA,\varphi,\{\mathrm{F}_i:i\in I\})$ to
$(\cC,\psi,\{\mathrm{G}_i:i\in I\})$ is a structural functor 
$\Phi:\cA\to\cC$ such that for every $i\in I$ 
the following digram commutes (up to isomorphism of functors):
\begin{displaymath}
\xymatrix{ 
\cA\ar[rr]^{\mathrm{F}_i}\ar[d]_{\Phi}
&& \cA\ar[d]^{\Phi}\\
\cC\ar[rr]^{\mathrm{G}_i}&& \cC\\
}
\end{displaymath}
}
\end{definition}

\begin{definition}\label{def205}
{\rm
A {\em strict morphism} from 
\index{strict morphism}
$(\cA,\varphi,\{\mathrm{F}_i:i\in I\})$ to
$(\cC,\psi,\{\mathrm{G}_i:i\in I\})$ is a structural functor 
$\Phi:\cA\to\cC$ such that for every $i\in I$ 
the following digram commutes strictly:
\begin{displaymath}
\xymatrix{ 
\cA\ar[rr]^{\mathrm{F}_i}\ar[d]_{\Phi}
&& \cA\ar[d]^{\Phi}\\
\cC\ar[rr]^{\mathrm{G}_i}&& \cC\\
}
\end{displaymath}
}
\end{definition}

Definitions~\ref{def203}--\ref{def205} give rise to the {\em na{\"\i}ve}
(resp. {\em weak} or {\em strict}) category of na{\"\i}ve categorifications
of $A$-modules (with respect to the basis $\mathtt{A}$). These categories
are in the natural way (not full) subcategories of each other.
\index{na{\"\i}ve categorification}
\index{weak categorification}
\index{strict categorification}

\begin{example}\label{exm206}
{\rm  
Consider the complex group algebra $\mathbb{C}[\mathbb{S}_n]$ of the
symmetric group $\mathbb{S}_n$. For $\lambda\vdash n$ let $S_{\lambda}$
be the corresponding Specht module and $d_{\lambda}$ be its
dimension (the number of standard Young tableaux of shape $\lambda$). 
Choose in $\mathbb{C}[\mathbb{S}_n]$
the standard generating system $\mathtt{A}$ consisting of transpositions
$s_i=(i,i+1)$, $i=1,2,\dots,n-1$. In $S_{\lambda}$ choose the basis
consisting of standard polytabloids (see e.g. \cite[Chapter~2]{Sa}). 
Then the action of every $s_i$ in this basis is given by some
matrix $M_i=(m^i_{st})_{s,t=1}^{d_{\lambda}}$ with integral coefficients.
Categorify $S_{\lambda}$ via $\mathcal{D}^b(\mathbb{C}^{d_{\lambda}})$, 
such that the basis of standard  polytabloids corresponds to the usual
basis of $[\mathcal{D}^b(\mathbb{C}^{d_{\lambda}})]$ given by simple
modules. Categorify the action of every $s_i$ using the appropriate 
direct sum, given by the corresponding coefficient in $M_i$, of 
the identity functors (shifted by $1$ in homological position in 
the case of negative coefficients). We obtain a (trivial) 
na{\"\i}ve categorification of the Specht module $S_{\lambda}$.
}
\end{example}

Similarly to Example~\ref{exm206} one can construct trivial 
na{\"\i}ve categorifications for any module with a fixed basis in which
the action of generators has integral coefficients. Instead of 
$\mathbb{C}\text{-}\mathrm{mod}$ one can also use the category 
of modules over any local algebra. We refer the reader to 
\cite[Chapter~7]{Ma} for more details.

\subsection{Weak categorification}\label{s2.3}

To define the na{\"\i}ve categorification of an $A$-module $M$ we
simply required that the functor $\mathrm{F}_i$ categorifies the
action of $a_i^M$ only {\em numerically}, that is only on the level
of the Grothendieck group. This is an extremely weak requirement
so it is natural to expect that there should exist lots of different
categorifications of $M$ and that it should be almost impossible
to classify, study and compare them in the general case. To make the
classification problem more realistic we should impose some extra
conditions. To see what kinds of conditions we may consider, we have
to analyze what kind of structure we (usually) have.

The most important piece of information which we (intentionally)
neglected up to this point is that $A$ is an algebra and hence
elements of $A$ can be multiplied. As we categorify the
action of the elements of $A$ via functors, it is natural
to expect that the multiplication in $A$ should be categorified
as the composition of functors. As we have already fixed a
generating system $\mathtt{A}$ in $A$, we can consider some 
presentation of $A$ (or the corresponding image with respect to the
action on $M$) relative to this generating system. In other words, 
the generators $a_i$'s could satisfy some relations. So, we can try
to look for functorial interpretations of such relations. Here is
a list of some natural ways to do this:
\begin{itemize}
\item equalities can be interpreted as isomorphisms of functors;
\item addition in $A$ can be interpreted as direct sum of functors;
\item for triangulated categorifications one could interpret the
negative coefficient $-1$ as the shift by $1$ in homological position
(see Example~\ref{exm206}), in particular, subtraction in $A$
can be sometimes interpreted via taking cone in the derived category;
\item negative coefficients could be made positive by moving the
corresponding terms to the other side of an equality.
\end{itemize}
Another quite common feature is that the algebra $A$ we are working
with usually comes equipped with an anti-involution $*$. The most natural
way for the functorial interpretation of an anti-involution is via
(bi)adjoint functors. Of course one has to emphasize that none of
the above interpretations is absolutely canonical. Still, we might
give the following loose definition from \cite{MS2} 
(from now on, if $\mathbb{F}$
is fixed, we will omit it in our notation for simplicity):

\begin{definition}\label{def207}
{\rm
A na{\"\i}ve categorification $(\cA,\varphi,\{\mathrm{F}_i:i\in I\})$
of an $A$-module $M$ is called a {\em weak categorification} if 
\index{weak categorification}
it satisfies the conditions given by some chosen interpretations 
of defining relations and eventual anti-involution for $A$.
}
\end{definition}

In what follows for a functor $\mathrm{F}$ we will denote by
$\mathrm{F}^*$ the biadjoint of $\mathrm{F}$ (if it exists).

\begin{example}\label{exm208}
{\rm  
Let $A=\mathbb{C}[a]/(a^2-2a)$ and $\mathtt{A}=\{a\}$. Let further
$M=\mathbb{C}$ be the $A$-module with the action $a\cdot 1=0$ and
$N=\mathbb{C}$ the $A$-module with the action $a\cdot 1=2$.
Consider $\cC=\mathbb{C}\text{-}\mathrm{mod}$,
$\mathrm{F}=0$ and $\mathrm{G}=\mathrm{Id}_{\cC}\oplus \mathrm{Id}_{\cC}$.
Define $\varphi:M\to [\cC]$ and $\psi:N\to [\cC]$ by sending 
$1$ to $[\mathbb{C}]$ (the class of the simple $\mathbb{C}$-module).
The algebra $A$ has the $\mathbb{C}$-linear involution $*$ defined 
via $a^*=a$. We have both $\mathrm{F}^*=\mathrm{F}$ and
$\mathrm{G}^*=\mathrm{G}$. We interpret $a^2-2a=0$ as
$a\cdot a=a+a$. We have $\mathrm{F}\circ \mathrm{F}\cong
\mathrm{F}\oplus \mathrm{F}$ and $\mathrm{G}\circ \mathrm{G}\cong
\mathrm{G}\oplus \mathrm{G}$. Hence $(\cC,\varphi,\mathrm{F})$
and $(\cC,\psi,\mathrm{G})$ are weak categorifications of 
$M$ and $N$, respectively. Note that $A\cong \mathbb{C}[\mathbb{S}_2]$
and that under this identification the modules $M$ and $N$ become
the sign and the trivial $\mathbb{C}[\mathbb{S}_2]$-modules, respectively.
}
\end{example}

\subsection{$2$-categories}\label{s2.4}

Following our description above we now can summarize that to
categorify the action of some algebra $A$ on some module $M$ we
would like to ``lift'' this action to a functorial ``action''
of $A$ on some category $\cC$. So, the image of our ``lift''
should be a ``nice'' subcategory of the category of endofunctors
on $\cC$. This latter category has an extra structure, which we
already tried to take into account in the previous subsections,
namely, we can compose endofunctors. This is a special case
of the structure known as a {\em $2$-category}.
\index{$2$-category}

\begin{definition}\label{def208}
{\rm
A $2$-category is a category enriched over the category of categories.
}
\end{definition}

This means that if $\cC$ is a $2$-category, then for any 
$\mathtt{i},\mathtt{j}\in\cC$ the morphisms $\cC(\mathtt{i},\mathtt{j})$
form a category, its objects are called {\em $1$-morphisms} and its
morphisms are called {\em $2$-morphisms}. 
\index{$1$-morphism}\index{$2$-morphism}
Composition $\circ_0:
\cC(\mathtt{j},\mathtt{k})\times \cC(\mathtt{i},\mathtt{j})\to
\cC(\mathtt{i},\mathtt{k})$ 
is called {\em horizontal} composition and is strictly associative and
\index{horizontal composition}
unital. Composition $\circ_1$ of $2$-morphisms inside
$\cC(\mathtt{i},\mathtt{j})$ is called 
{\em vertical} composition and is strictly associative and unital. 
\index{vertical composition}
We have the following {\em interchange law} for any composable
\index{interchange law}
$2$-morphisms $\alpha,\beta,\gamma,\delta$:
\begin{displaymath}
(\alpha\circ_0\beta)\circ_1(\gamma\circ_0\delta)=
(\alpha\circ_1\gamma)\circ_0(\beta\circ_1\delta) 
\end{displaymath}
This is usually depicted as follows:
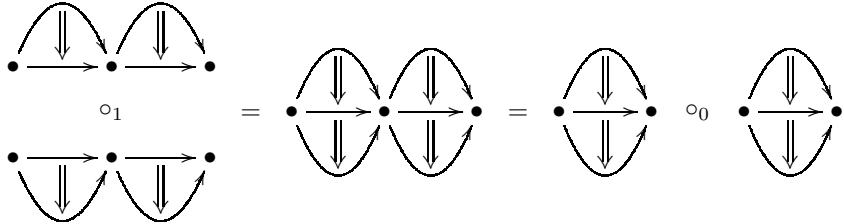
\footnote{Thanks to Hnaef on Wikimedia for the \LaTeX-source file!}
\begin{displaymath}
\vcenter{
\xymatrix@R=0.5pc{
  {\bullet} \ar[r]^{}="mt1" \ar@/^2pc/[r]_{}="t1" \ar@{=>}"t1";"mt1" &
  {\bullet} \ar[r]^{}="mt2" \ar@/^2pc/[r]_{}="t2" \ar@{=>}"t2";"mt2" &
  {\bullet} \\
  & \circ_1 & \\
  {\bullet} \ar[r]_{}="mb1" \ar@/_2pc/[r]_{}="b1" \ar@{=>}"mb1";"b1" &
  {\bullet} \ar[r]_{}="mb2" \ar@/_2pc/[r]_{}="b2" \ar@{=>}"mb2";"b2" &
  {\bullet}
} }\,\,=\,\,
\xymatrix{
  {\bullet} \ar[r]^{}="ml" \ar@/^2pc/[r]_{}="tl" \ar@/_2pc/[r]^{}="bl"
  \ar@{=>}"tl";"ml" \ar@{=>}"ml";"bl" & 
  {\bullet} \ar[r]^{}="ml" \ar@/^2pc/[r]_{}="tl" \ar@/_2pc/[r]^{}="bl"
  \ar@{=>}"tl";"ml" \ar@{=>}"ml";"bl" & {\bullet}
} \,\,=\,\,
\xymatrix{
  {\bullet} \ar[r]^{}="ml" \ar@/^2pc/[r]_{}="tl" \ar@/_2pc/[r]^{}="bl"
  \ar@{=>}"tl";"ml" \ar@{=>}"ml";"bl" & {\bullet} \ar@{}[r]|{\textstyle\circ_0}
  &
  {\bullet} \ar[r]^{}="ml" \ar@/^2pc/[r]_{}="tl" \ar@/_2pc/[r]^{}="bl"
  \ar@{=>}"tl";"ml" \ar@{=>}"ml";"bl" & {\bullet}
}
\end{displaymath}

There is a weaker notion of a {\em bicategory}, in which, in particular,
\index{bicategory}
composition of $1$-morphisms is only required to be associative up to a 
$2$-isomorphism. There is a natural extension of the notion of 
categorical equivalence to bicategories, called {\em biequivalence}.
\index{biequivalence}
We will, however, always work with $2$-categories,
which is possible thanks to the following statement:

\begin{theorem}[\cite{MP,Le}]\label{thm209}
Every bicategory is biequivalent to a $2$-category.
\end{theorem}

A typical example of a $2$-category is the category of functors
on some category. It has one object, its $1$-morphisms are functors,
and its $2$-morphisms are natural transformations of functors.
One can alternatively describe this using the notion of a (strict)
{\em tensor category}, which is equivalent to the notion of a
$2$-category with one object.
\index{tensor category}

A $2$-functor $\mathrm{F}:\cA\to\cC$ between two $2$-categories 
is a triple of functions sending objects, $1$-morphisms and
$2$-morphisms of $\cA$ to items of the same type in $\cC$ such that
it preserves (strictly) all the categorical structures. If
$\mathrm{G}:\cA\to\cC$ is another $2$-functor, then a $2$-natural
transformation $\zeta$ from $\mathrm{F}$ to $\mathrm{G}$ is a 
function sending $\mathtt{i}\in\cA$ to a $1$-morphism
$\zeta_{\mathtt{i}}\in\cC$ such that for every $2$-morphism
$\alpha:f\to g$, where $f,g\in \cA(\mathtt{i},\mathtt{j})$, we have
\begin{displaymath}
\xymatrix{
\mathrm{F}(\mathtt{i})
\ar@/_2pc/[rr]_{\mathrm{F}(g)}="mt1" \ar@/^2pc/[rr]^{\mathrm{F}(f)}="t1" 
\ar@{=>}"t1";"mt1"|-{\mathrm{F}(\alpha)}
&&
\mathrm{F}(\mathtt{j})\ar[r]^{\zeta_{\mathtt{j}}}&\mathrm{G}(\mathtt{j})
}=\xymatrix{
\mathrm{F}(\mathtt{i})\ar[r]^{\zeta_{\mathtt{i}}}&
\mathrm{G}(\mathtt{i})
\ar@/_2pc/[rr]_{\mathrm{G}(g)}="mt1" \ar@/^2pc/[rr]^{\mathrm{G}(f)}="t1" 
\ar@{=>}"t1";"mt1"|-{\mathrm{G}(\alpha)}
&&
\mathrm{G}(\mathtt{j})
}
\end{displaymath}
In particular, applied to the identity $2$-morphisms we get that 
$\zeta$ is an ordinary natural transformation between the associated
ordinary functors $\mathrm{F}$ and $\mathrm{G}$. Note that $2$-categories 
with $2$-functors and $2$-natural transformations form a $2$-category.

A $2$-category is called {\em additive} if it is enriched over
the category of additive categories. If $\Bbbk$ is a fixed field,
a $2$-category is called {\em $\Bbbk$-linear} if it is enriched 
over the category of $\Bbbk$-linear categories.
\index{additive $2$-category}\index{$\Bbbk$-linear $2$-category}

Now some notation: for a 
$2$-category $\cC$, objects of $\cC$ will be denoted by 
$\mathtt{i},\mathtt{j}$ and so on, objects of 
$\cC(\mathtt{i},\mathtt{j})$ (that is $1$-morphisms) will 
be called $f,g$ and so on, and $2$-morphisms
from $f$ to $g$ will be written $\alpha,\beta$
and so on. The identity $1$-morphism in $\cC(\mathtt{i},\mathtt{i})$
will be denoted $\mathbbm{1}_{\mathtt{i}}$ and the identity 
$2$-morphism from $f$ to $f$ will be denoted
$\mathrm{id}_{f}$. Composition of $1$-morphisms will be
denoted by $\circ$, horizontal composition of $2$-morphisms will be
denoted by $\circ_0$ and vertical composition of $2$-morphisms will be
denoted by $\circ_1$.

\subsection{(Genuine) categorification}\label{s2.5}

Let $\cC$ be an additive $2$-category.

\begin{definition}\label{def210}
{\rm
The {\em Grothendieck category} $[\cC]$ of $\cC$ is the category
\index{Grothendieck category}
defined as follows: $[\cC]$ 
has the same objects as $\cC$, for $\mathtt{i},\mathtt{j}\in [\cC]$ 
we have $[\cC](\mathtt{i},\mathtt{j})=[\cC(\mathtt{i},\mathtt{j})]$
and the multiplication of morphisms in $[\cC]$ is given by 
$[M]\circ[N]:=[M\circ_0 N]$.
}
\end{definition}

Note that $[\cC]$ is a preadditive category (i.e. it is enriched
over the category of abelian groups). As before, let $\mathbb{F}$ 
be a commutative ring with $1$.

\begin{definition}\label{def250}
{\rm
The {\em $\mathbb{F}$-decategorification} $[\cC]^{\mathbb{F}}$ of 
$\cC$ is the category $\mathbb{F}\otimes_{\mathbb{Z}}[\cC]$.
\index{decategorification}
}
\end{definition}

The category $[\cC]^{\mathbb{F}}$ is $\mathbb{F}$-linear (i.e.
enriched over $\mathbb{F}\text{-}\mathrm{Mod}$) by definition. 
Now we are ready to define our central notion of (genuine) 
categorification.

\begin{definition}\label{def211}
{\rm
Let $\mathcal{A}$ be an $\mathbb{F}$-linear category
with at most countably many objects. A 
{\em categorification} of $\mathcal{A}$ is a pair
\index{categorification}
$(\cA,\varphi)$, where $\cA$ is an additive $2$-category and
$\varphi:\mathcal{A}\to[\cA]^{\mathbb{F}}$ is an isomorphism.
}
\end{definition}

In the special case when $\mathcal{A}$ has only one object, say
$\mathtt{i}$, the morphism set $\mathcal{A}(\mathtt{i},\mathtt{i})$
is an $\mathbb{F}$-algebra. Therefore Definition~\ref{def211}
contains, as a special case, the definition of categorification
for arbitrary $\mathbb{F}$-algebras.

\begin{example}\label{exm212}
{\rm
Let $A=\mathbb{C}[a]/(a^2-2a)$. For the algebra $D$ of dual numbers
consider the bimodule $X=D\otimes_{\mathbb{C}}D$ and denote by
$\cC$ the $2$-category with one object 
$\mathtt{i}=D\text{-}\mathrm{mod}$ such that $\cC(\mathtt{i},\mathtt{i})$ 
is the full additive subcategory of the category of endofunctors 
of $\mathtt{i}$, consisting of all functors isomorphic to direct sums of
copies of $\mathrm{Id}=\mathrm{Id}_{D\text{-}\mathrm{mod}}$
and  $\mathrm{F}=X\otimes_{D}{}_-$. It is easy to check that
$X\otimes_{D}X\cong X\oplus X$, which implies that 
$\cC(\mathtt{i},\mathtt{i})$ is closed under  composition of 
functors. The classes  $[\mathrm{Id}]$ and 
$[\mathrm{F}]$ form a basis of $[\cC(\mathtt{i},\mathtt{i})]$.
Since $\mathrm{F}\circ\mathrm{F}=\mathrm{F}\oplus \mathrm{F}$, it follows
that the map $\varphi:A\to [\cC(\mathtt{i},\mathtt{i})]^{\mathbb{C}}$ 
such that $1\mapsto [\mathrm{Id}]$ and $a\mapsto [\mathrm{F}]$ is an
isomorphism. Hence $(\cC,\varphi)$ is a $\mathbb{C}$-categorification 
of $A$.
}
\end{example}

The above leads to the following major problem:

\begin{problem}\label{exm214}
{\rm 
Given an $\mathbb{F}$-linear category $\mathcal{A}$ 
(satisfying some reasonable integrality conditions),
construct a categorification of $\mathcal{A}$.
}
\end{problem}

Among various solutions to this problem in some special cases
one could mention \cite{Ro0,Ro,La,KL}. Whereas \cite{Ro0,Ro}
propagates algebraic approach (using generators and relations),
the approach of \cite{La,KL} uses diagrammatic calculus and
is motivated and influenced by topological methods. The idea 
to use  $2$-categories for a proper definition of algebraic 
categorification seems to go back at least to \cite{Ro0}
and is based on the results of \cite{CR} which will be mentioned
later on. One of the main advantages of this approach when compared
with weak categorification is that now all extra properties
(e.g. relations or involutions) for the generators of our
algebra can be encoded into the internal structure of the
$2$-category. Thus relations between generators now can be
interpreted as invertability of some $2$-morphisms and
biadjointness of elements connected by an anti-involution
can be interpreted in terms of existence of adjunction
morphisms, etc.

\section{Basics: $2$-representations of finitary $2$-categories}\label{s3}

\subsection{$2$-representations of $2$-categories}\label{s3.1}

Let $\cC$ be a $2$-category and $\Bbbk$ a field.  As usual, a $2$-representation of a
$\cC$ is a $2$-functor to some other $2$-category. We will deal
with $\Bbbk$-linear representations. Denote by $\mathfrak{A}_{\Bbbk}$,
$\mathfrak{R}_{\Bbbk}$ and $\mathfrak{D}_{\Bbbk}$ the $2$-categories whose objects 
are fully additive $\Bbbk$-linear categories with finitely many isomorphism classes
of indecomposable objects, categories equivalent to 
module categories of finite-dimensional $\Bbbk$-algebras,
and their (bounded) derived categories, respectively; 
$1$-morphisms are functors between objects; and $2$-morphisms are 
natural transformations of functors.  Define the $2$-categories 
$\cC\text{-}\mathfrak{amod}$, $\cC\text{-}\mathfrak{mod}$ and 
$\cC\text{-}\mathfrak{dmod}$  of {\em $\Bbbk$-linear additive $2$-representations},
{\em $\Bbbk$-linear $2$-representations} and
\index{$2$-representation}
{\em $\Bbbk$-linear triangulated $2$-representations} of $\cC$ as follows:
\begin{itemize}
\item Objects of $\cC\text{-}\mathfrak{amod}$ 
(resp. $\cC\text{-}\mathfrak{mod}$ and $\cC\text{-}\mathfrak{dmod}$) are $2$-functors
from $\cC$ to $\mathfrak{A}_{\Bbbk}$
(resp. $\mathfrak{R}_{\Bbbk}$ and $\mathfrak{D}_{\Bbbk}$);
\item $1$-morphisms are $2$-natural transformations
(these are given by a collection of structural functors);
\item $2$-morphisms are the so-called {\em modifications}, defined
\index{modification}
as follows: Let $\mathbf{M},\mathbf{N}\in \cC\text{-}\mathfrak{mod}$
(or $\cC\text{-}\mathfrak{amod}$ or $\cC\text{-}\mathfrak{dmod}$)
and $\zeta,\xi:\mathbf{M}\to\mathbf{N}$ be $2$-natural 
transformations. A modification $\theta:\zeta\to\xi$ is 
a function, which assigns to every $\mathtt{i}\in \cC$
a $2$-morphism $\theta_{\mathtt{i}}:\zeta_{\mathtt{i}}\to
\xi_{\mathtt{i}}$ such that for every $1$-morphisms
$f,g\in \cC(\mathtt{i},\mathtt{j})$ and any $2$-morphism
$\alpha:f\to g$ we have
\begin{displaymath}
\xymatrix@!=0.6pc{
\mathrm{F}(\mathtt{i})
\ar@/_2pc/[rr]_{\mathrm{F}(g)}="mt1" \ar@/^2pc/[rr]^{\mathrm{F}(f)}="t1" 
\ar@{=>}"t1";"mt1"|-{\mathrm{F}(\alpha)}
&&
\mathrm{F}(\mathtt{j})
\ar@/_2pc/[rr]_{\xi_{\mathtt{j}}}="mt1" 
\ar@/^2pc/[rr]^{\zeta_{\mathtt{j}}}="t1" 
\ar@{=>}"t1";"mt1"|-{\theta_{\mathtt{j}}}
&&\mathrm{G}(\mathtt{j})
}=\xymatrix@!=0.6pc{
\mathrm{F}(\mathtt{i})\ar@/_2pc/[rr]_{\xi_{\mathtt{i}}}="mt1" 
\ar@/^2pc/[rr]^{\zeta_{\mathtt{i}}}="t1" 
\ar@{=>}"t1";"mt1"|-{\theta_{\mathtt{i}}}
&&
\mathrm{G}(\mathtt{i})
\ar@/_2pc/[rr]_{\mathrm{G}(g)}="mt1" \ar@/^2pc/[rr]^{\mathrm{G}(f)}="t1" 
\ar@{=>}"t1";"mt1"|-{\mathrm{G}(\alpha)}
&&
\mathrm{G}(\mathtt{j})
} 
\end{displaymath}
\end{itemize}
For simplicity we will identify objects in $\cC(\mathtt{i},\mathtt{k})$  
with their images under a $2$-representation (i.e. we will use the
module notation). We will also use {\em $2$-action} and
\index{$2$-action}\index{$2$-module}
{\em $2$-module} as a synonym for {\em $2$-representation}.
Define an equivalence relation on $2$-representations of $\cC$ as the minimal equivalence
relations such that two $2$-representations of $\cC$ are {\em equivalent} if there is a 
morphism between these $2$-representations, such that the restriction of it to every object
of $\cC$ is an equivalence of categories. Now we 
can define genuine categorifications for $A$-modules.
\index{equivalent $2$-representation}

\begin{definition}\label{def301}
{\rm
Let $\Bbbk$ be a field, $A$ a $\Bbbk$-linear category 
with at most countably many objects and $M$ an
$A$-module. A {\em (pre)categorification} of $M$ is a tuple
$(\cA,\mathbf{M},\varphi,\psi)$, where 
\begin{itemize}
\item $(\cA,\varphi)$ is a categorification of $A$;
\item $\mathbf{M}\in \cA\text{-}\mathfrak{amod}$ or $\mathbf{M}\in \cA\text{-}\mathfrak{mod}$
or $\mathbf{M}\in \cA\text{-}\mathfrak{dmod}$ is such that
for every $\mathtt{i},\mathtt{j}\in \cA$  and 
$a\in\cA(\mathtt{i},\mathtt{j})$ the functor
$\mathbf{M}(a)$ is additive, exact or triangulated, respectively;
\item $\psi=(\psi_{\mathtt{i}})_{\mathtt{i}\in\cA}$, where every
$\psi_{\mathtt{i}}:M(\mathtt{i})\to [\mathbf{M}(\mathtt{i})]$ is
a monomorphism such that for every $\mathtt{i},\mathtt{j}\in \cA$ 
and $a\in\cA(\mathtt{i},\mathtt{j})$ the following diagram commutes:
\begin{displaymath}
\xymatrix{ 
M(\mathtt{i})\ar[rr]^{M(\varphi^{-1}([a]))}\ar[d]_{\psi_{\mathtt{i}}} && 
M(\mathtt{j})\ar[d]^{\psi_{\mathtt{j}}}\\ 
[\mathbf{M}(\mathtt{i})]_{(\oplus)}^{\Bbbk}\ar[rr]^{[\mathbf{M}(a)]}
&&[\mathbf{M}(\mathtt{j})]_{(\oplus)}^{\Bbbk}.  
} 
\end{displaymath}
\end{itemize}
If every $\psi_{\mathtt{i}}$ is an isomorphism, 
$(\cA,\mathbf{M},\varphi,\psi)$ is called a 
{\em categorification} of $M$.
}
\end{definition}

It is worth to note the following: if  $(\cA,\mathbf{M},\varphi,\psi)$
and $(\cA,\mathbf{M}',\varphi,\psi')$ are categorifications of $M$
and $M'$, respectively, and, moreover, the $2$-representations $\mathbf{M}$
and $\mathbf{M}'$ are equivalent, then the modules $M$ and $M'$ are isomorphic.

An important feature of the above definition is that categorification
of all relations and other structural properties of $A$ is encoded
into the internal structure of the $2$-category $\cA$. In particular,
the requirement for $\mathbf{M}(a)$ to be structural is often
automatically satisfied because of the existence of adjunctions in 
$\cA$ (see the next subsection). Later on we will see many examples
of categorifications of modules over various algebras. For the
moment we would like to look at the other direction: to be able
to categorify modules one should develop some abstract
$2$-representation theory of $2$-categories. Some overview of
this (based on \cite{MM}) is the aim of this section.

\subsection{Fiat-categories}\label{s3.2}

An additive $2$-category $\cC$ with a weak involution $*$
is called a {\em fiat-category} (over $\Bbbk$) provided that
\index{fiat-category}
\begin{enumerate}[(I)]
\item\label{fin.1} $\cC$ has finitely many objects;
\item\label{fin.2} for every $\mathtt{i},\mathtt{j}\in \cC$
the category $\cC(\mathtt{i},\mathtt{j})$ is fully additive with
finitely many isomorphism classes of indecomposable objects;
\item\label{fin.3} for every $\mathtt{i},\mathtt{j}\in \cC$
the category $\cC(\mathtt{i},\mathtt{j})$ is enriched over
$\Bbbk\text{-}\mathrm{mod}$ (in particular, all spaces of $2$-morphisms 
are  finite dimensional) and all compositions are $\Bbbk$-bilinear;
\item\label{fin.4} for every $\mathtt{i}\in \cC$ the identity
object in $\cC(\mathtt{i},\mathtt{i})$ is indecomposable;
\item\label{fin.5} for any $\mathtt{i},
\mathtt{j}\in\cC$ and any $1$-morphism
$f\in\cC(\mathtt{i},\mathtt{j})$ there exist 
$2$-morphisms $\alpha:f\circ f^*\to
\mathbbm{1}_{\mathtt{j}}$ and $\beta:\mathbbm{1}_{\mathtt{i}}\to
f^*\circ f$ such that 
$(\alpha\circ_0\mathrm{id}_{f})\circ_1
(\mathrm{id}_{f}\circ_0\beta)=\mathrm{id}_{f}$ and
$(\mathrm{id}_{f^*}\circ_0\alpha)\circ_1
(\beta\circ_0\mathrm{id}_{f^*})=\mathrm{id}_{f^*}$.
\end{enumerate}

If $\cC$ is a fiat category and $\mathbf{M}$ a $2$-representation
of $\cC$, then for every $1$-morphism $f$ the functor 
$\mathbf{M}(f)$ is both left and right adjoint to
$\mathbf{M}(f^*)$, in particular, both are exact if
$\mathbf{M}\in\cC\text{-}\mathfrak{mod}$.

\begin{example}\label{exm302}
{\rm 
The category $\cC$ from Example~\ref{exm212} is a fiat-category.
}
\end{example}

Let  $\mathbf{M}\in \cC\text{-}\mathfrak{mod}$, $\mathtt{i}\in \cC$
and $M\in \mathbf{M}(\mathtt{i})$. For $\mathtt{j}\in \cC$ denote by 
$\langle M\rangle(\mathtt{j})$ the additive closure in $\mathbf{M}(\mathtt{j})$
of all objects of the form $f\, M$, where $f\in \cC(\mathtt{i},\mathtt{j})$.
Then $\langle M\rangle$ inherits from $\mathbf{M}$ the structure of an additive
$2$-representation of $\cC$, moreover, $\langle M\rangle\in \cC\text{-}\mathfrak{amod}$
if $\cC$ is fiat.

\subsection{Principal $2$-representations of fiat-categories}\label{s3.3}

Let $\cC$ be a fiat category and $\mathtt{i}\in\cC$. Consider the
$2$-functor $\mathbf{P}_{\mathtt{i}}:\cC\to\mathfrak{A}_{\Bbbk}$ defined
as follows: 
\begin{itemize}
\item for all $\mathtt{j}\in \cC$ the additive category 
$\mathbf{P}_{\mathtt{i}}(\mathtt{j})$ is defined to be 
$\cC(\mathtt{i},\mathtt{j})$;
\item for all $\mathtt{j},\mathtt{k}\in \cC$ and any $1$-morphism
$f\in \cC(\mathtt{j},\mathtt{k})$ the functor $\mathbf{P}_{\mathtt{i}}(f)$
is defined to be the functor $f\circ_0{}_-:\cC(\mathtt{i},\mathtt{j})\to
\cC(\mathtt{i},\mathtt{k})$;
\item for all $\mathtt{j},\mathtt{k}\in \cC$, all $1$-morphisms
$f,g\in \cC(\mathtt{j},\mathtt{k})$ and any $2$-morphism $\alpha:f\to g$,
the natural transformation $\mathbf{P}_{\mathtt{i}}(\alpha)$
is defined to be $\alpha\circ_0{}_-:f\circ_0{}_-\to g\circ_0{}_-$.
\end{itemize}
The $2$-representation $\mathbf{P}_{\mathtt{i}}$ is called the
{\em $\mathtt{i}$-th principal additive} $2$-representation of $\cC$.
\index{principal additive $2$-representation}

Define also the {\em abelianization $2$-functor} 
$\overline{\cdot}:\cC\text{-}\mathfrak{amod}\to\cC\text{-}\mathfrak{mod}$
as follows: Let $\mathbf{M}\in \cC\text{-}\mathfrak{amod}$.
\index{abelianization $2$-functor}
For $\mathtt{i}\in \cC$ define $\overline{\mathbf{M}}(\mathtt{i})$ as
the category with objects $X\overset{a}{\to} Y$, where $X,Y\in \mathbf{M}(\mathtt{i})$
and $a\in \mathrm{Hom}_{\mathbf{M}(\mathtt{i})}(X,Y)$. Morphisms in $\overline{\mathbf{M}}(\mathtt{i})$
are equivalence  classes of diagrams as given by the solid part of the following picture:
\begin{displaymath}
\xymatrix{
X\ar[rr]^{a}\ar[d]_{b}&&Y
\ar[d]^{b'}\ar@{.>}[dll]_{c}\\
X'\ar[rr]^{a'}&&Y'
},
\end{displaymath}
where $X,X',Y,Y'\in \mathbf{M}(\mathtt{i})$, $a\in \mathbf{M}(\mathtt{i})(X,Y)$
$a'\in \mathbf{M}(\mathtt{i})(X',Y')$,
$b\in \mathbf{M}(\mathtt{i})(X,X')$ and $b'\in \mathbf{M}(\mathtt{i})(Y,Y')$, 
modulo the ideal generated by all  morphisms for which there exists $c$ as 
shown by the dotted arrow above  such that $a'c=b'$. In particular, the category
$\overline{\mathbf{M}}(\mathtt{i})$ is equivalent to the category of modules over
the opposite category of a skeleton of $\mathbf{M}(\mathtt{i})$. The $2$-action of $\cC$
on $\overline{\mathbf{M}}(\mathtt{i})$ is induced from that
on ${\mathbf{M}}(\mathtt{i})$ by applying elements of $\cC$ component-wise. 
The rest of the $2$-structure of 
$\overline{\cdot}$ is also defined via component-wise application. 
The $2$-representation $\overline{\mathbf{P}}_{\mathtt{i}}$ is called the 
{\em $\mathtt{i}$-th principal} $2$-representation of $\cC$.
\index{principal $2$-representation}

\begin{proposition}[The universal property of $\mathbf{P}_{\mathtt{i}}$]
\label{prop303}
Let $\mathbf{M}\in \cC\text{-}\mathfrak{amod}$, 
$M\in \mathbf{M}(\mathtt{i})$. 
\begin{enumerate}[$($a$)$]
\item \label{prop3.1}
For $\mathtt{j}\in \cC$ define $\Phi^M_{\mathtt{j}}:
\mathbf{P}_{\mathtt{i}}(\mathtt{j})\to \mathbf{M}(\mathtt{j})$ 
as follows: $\Phi^M_{\mathtt{j}}(f):=\mathbf{M}(f)\, M$ 
for every object $f\in \mathbf{P}_{\mathtt{i}}(\mathtt{j})$; 
$\Phi^M_{\mathtt{j}}(\alpha):=\mathbf{M}(\alpha)_M$
for every objects $f,g\in \mathbf{P}_{\mathtt{i}}(\mathtt{j})$
and every morphism $\alpha:f\to g$.
Then $\Phi^M=(\Phi^M_{\mathtt{j}})_{\mathtt{j}\in\ccC}$ is the
unique morphism from $\mathbf{P}_{\mathtt{i}}$ to $\mathbf{M}$
sending $\mathbbm{1}_{\mathtt{i}}$ to $M$.
\item \label{prop3.2} The correspondence 
$M\mapsto \Phi^M$ is functorial.
\end{enumerate}
\end{proposition}

\begin{proof}[Idea of the proof.]
This follows from the $2$-functoriality of $\mathbf{M}$.
\end{proof}

\subsection{Cells}\label{s3.4}

Let $\cC$ be a fiat-category. For $\mathtt{i},\mathtt{j}\in\cC$ denote
by $\mathcal{C}_{\mathtt{i},\mathtt{j}}$ the set of isomorphism classes of
indecomposable $1$-morphisms in $\cC(\mathtt{i},\mathtt{j})$.
Set $\mathcal{C}=\cup_{\mathtt{i},\mathtt{j}}
\mathcal{C}_{\mathtt{i},\mathtt{j}}$.
Let $\mathtt{i},\mathtt{j},\mathtt{k},\mathtt{l}\in\cC$,
$f\in\mathcal{C}_{\mathtt{i},\mathtt{j}}$ and
$g\in\mathcal{C}_{\mathtt{k},\mathtt{l}}$. 
We will write $f\leq_R g$ provided that there 
exists $h\in\cC(\mathtt{j},\mathtt{l})$ such that $g$
occurs as a direct summand of $h\circ f$
(note that this is possible only if $\mathtt{i}=\mathtt{k}$).
Similarly, we will write $f\leq_L g$ provided that there 
exists $h\in\cC(\mathtt{k},\mathtt{i})$ such that $g$
occurs as a direct summand of $f\circ h$
(note that this is possible only if $\mathtt{j}=\mathtt{l}$).
Finally, we will write $f\leq_{LR} g$ provided that there 
exists $h_1\in\cC(\mathtt{k},\mathtt{i})$ 
and $h_2\in\cC(\mathtt{j},\mathtt{l})$ such that $g$
occurs as a direct summand of $h_2\circ f\circ h_1$.
The relations $\leq_{L}$, $\leq_{R}$ and $\leq_{LR}$ are partial
preorders on $\mathcal{C}$. The map $f\mapsto f^*$
preserves $\leq_{LR}$ and swaps $\leq_{L}$ and $\leq_{R}$.

For $f\in \mathcal{C}$ the set of all 
$g\in \mathcal{C}$ such that $f\leq_R g$
and $g\leq_R f$ will be called the {\em right
cell} of $f$ and denoted by $\mathcal{R}_{f}$.
The {\em left cell}  $\mathcal{L}_{f}$ 
and the {\em two-sided cell}  $\mathcal{LR}_{f}$ 
are defined analogously. 
\index{right cell}\index{left cell}\index{two-sided cell}

\begin{example}\label{exm304}
{\rm  
The $2$-category $\cC$ from Example~\ref{exm212} has two right 
cells, namely $\{\mathrm{Id}\}$ and $\{\mathrm{F}\}$, 
which are also left cells and thus two-sided cells as well.
}
\end{example}

\subsection{Cells modules}\label{s3.5}

Let $\cC$ be a fiat-category.
For $\mathtt{i},\mathtt{j}\in\cC$ indecomposable projective 
modules in  $\overline{\mathbf{P}}_{\mathtt{i}}(\mathtt{j})$ are indexed
by objects of $\mathcal{C}_{\mathtt{i},\mathtt{j}}$. 
For $f\in \mathcal{C}_{\mathtt{i},\mathtt{j}}$ we denote by
$L_f$ the unique simple quotient of the indecomposable projective object
$P_f:=(0\to f)\in \overline{\mathbf{P}}_{\mathtt{i}}(\mathtt{j})$.

\begin{proposition}\label{prop306}
\begin{enumerate}[$($a$)$]
\item\label{prop306.01} For $f,g\in \mathcal{C}$ the inequality
$f\,L_{g}\neq 0$ is equivalent to 
$f^*\leq_L g$.
\item\label{prop306.02}  For $f,g,h\in \mathcal{C}$ 
the inequality $[f\,L_{g}:L_{h}]\neq 0$ 
implies $h\leq_R g$.
\item\label{prop306.03} For $g,h\in \mathcal{C}$ 
such that $h\leq_R g$ there exists 
$f\in \mathcal{C}$ such that 
$[f\,L_{g}:L_{h}]\neq 0$.
\item\label{prop306.04} 
Let $f,g,h\in \mathcal{C}$.
If $L_{f}$ occurs in the top or in the socle of
$h\,L_{g}$, then $f\in \mathcal{R}_{g}$.
\item\label{prop306.05}
For any $f\in \mathcal{C}_{\mathtt{i},\mathtt{j}}$ there is a unique 
(up to scalar) nontrivial homomorphism from $P_{\mathbbm{1}_{\mathtt{i}}}$
to $f^*L_{f}$. In particular, $f^*L_{f}\neq 0$.
\end{enumerate}
\end{proposition}

\begin{proof}[Idea of the proof.]
To prove \eqref{prop306.01}, without loss of generality we may assume 
$g\in \mathcal{C}_{\mathtt{i},\mathtt{j}}$
and $f\in \mathcal{C}_{\mathtt{j},\mathtt{k}}$.
Then $f\,L_{g}\neq 0$ if and only if there is 
$h\in\mathcal{C}_{\mathtt{i},\mathtt{k}}$ such that 
$\mathrm{Hom}_{\overline{\mathbf{P}}_{\mathtt{i}}(\mathtt{k})}
(P_{h},f\,L_{g})\neq 0$. Using
$P_{h}=h\, P_{\mathbbm{1}_{\mathtt{i}}}$
and adjunction we obtain
\begin{displaymath}
0\neq \mathrm{Hom}_{\overline{\mathbf{P}}_{\mathtt{i}}(\mathtt{k})}
(P_{h},f\,L_{g})=
\mathrm{Hom}_{\overline{\mathbf{P}}_{\mathtt{i}}(\mathtt{j})}
(f^*\circ h\, P_{\mathbbm{1}_{\mathtt{i}}},L_{g}).
\end{displaymath}
This inequality is equivalent to the claim that
$P_{g}=g\, P_{\mathbbm{1}_{\mathtt{i}}}$ 
is a direct summand of $f^*\circ h\, 
P_{\mathbbm{1}_{\mathtt{i}}}$, that is $g$ is a direct summand
of $f^*\circ h$. Claim \eqref{prop306.01} follows.
Other claims are proved similarly
\end{proof}

Fix $\mathtt{i}\in\cC$.
Let $\mathcal{R}$ be a right cell in 
$\mathcal{C}$ such that $\mathcal{R}\cap 
\mathcal{C}_{\mathtt{i},\mathtt{j}}\neq \varnothing$ for
some $\mathtt{j}\in\cC$.

\begin{proposition}\label{prop305}
\begin{enumerate}[$($a$)$]
\item\label{prop305.01}
There is a unique submodule $K=K_{\mathcal{R}}$ of 
$P_{\mathbbm{1}_{\mathtt{i}}}$ which has the following
properties:
\begin{enumerate}[$($i$)$]
\item\label{prop305.1} Every simple subquotient of 
$P_{\mathbbm{1}_{\mathtt{i}}}/K$ is annihilated by any
$f\in\mathcal{R}$.
\item\label{prop305.2} The module $K$ has simple top,
which we denote by $L_{g_{\mathcal{R}}}$, and 
$f\, L_{g_{\mathcal{R}}}\neq 0$
for any $f\in\mathcal{R}$.
\end{enumerate}
\item\label{prop305.02} For any $f\in\mathcal{R}$
the module $f\, L_{g_{\mathcal{R}}}$
has simple top $L_{f}$.
\item\label{prop305.03} We have $g_{\mathcal{R}}\in\mathcal{R}$.
\item\label{prop305.04} For any $f\in\mathcal{R}$
we have $f^*\leq_L g_{\mathcal{R}}$
and $f\leq_R g_{\mathcal{R}}^*$.
\item\label{prop305.05} We have $g^*_{\mathcal{R}}\in 
\mathcal{R}$.
\end{enumerate}
\end{proposition}

\begin{proof}[Idea of the proof.]
The module $K$ is defined as the minimal module with property
\eqref{prop305.1}. Then for every $f\in\mathcal{R}$ we have
$f\,K=f\, P_{\mathbbm{1}_{\mathtt{i}}}=P_f$. The latter module
has simple top, which implies that $K$ has simple top.
The rest follows from Proposition~\ref{prop306}.
\end{proof}

Set $L=L_{g_{\mathcal{R}}}$. For $\mathtt{j}\in\cC$ denote by 
$\mathcal{D}_{\mathcal{R},\mathtt{j}}$ the full subcategory of
$\mathbf{P}_{\mathtt{i}}(\mathtt{j})$ with objects
$g\, L$, $g\in \mathcal{R}\cap \mathcal{C}_{\mathtt{i},\mathtt{j}}$.  
Note that $2$-morphisms in $\cC$ surject onto homomorphisms between 
these $g\, L$. 

\begin{theorem}[Construction of cell modules.]
\label{thm307}
\begin{enumerate}[$($a$)$]
\item\label{thm307.01}
For every $f\in\mathcal{C}$ and $g\in \mathcal{R}$,
the module $f\circ g\,L$ is isomorphic to a direct
sum of modules of the form $h\,L$, $h\in \mathcal{R}$.
\item\label{thm307.02}
For every $f,h\in \mathcal{R}\cap
\mathcal{C}_{\mathtt{i},\mathtt{j}}$ we have
\begin{displaymath}
\dim\mathrm{Hom}_{\overline{\mathbf{P}}_{\mathtt{i}}(\mathtt{j})}
(f\,L,h\,L)=[h\,L:L_{f}].
\end{displaymath}
\item\label{thm307.03}
For $f\in\mathcal{R}$ let $\mathrm{Ker}_{f}$ be the kernel of
$P_f\tto f\, L$. Then the module
$\oplus_{f\in \mathcal{R}}\mathrm{Ker}_{f}$ is stable under any 
endomorphism of $\oplus_{f\in \mathcal{R}}P_{f}$.
\item\label{thm307.04}
The full subcategory $\mathbf{C}_{\mathcal{R}}(\mathtt{j})$ of
$\overline{\mathbf{P}}_{\mathtt{i}}(\mathtt{j})$ consisting of all objects
$M$ which admit a two step resolution $X_1\to X_0\tto M$,
$X_1,X_0\in\mathrm{add}(\oplus_{f\in \mathcal{R}\cap
\mathcal{C}_{\mathtt{i},\mathtt{j}}}f\, L)$, 
is equivalent to 
$\mathcal{D}_{\mathcal{R},\mathtt{j}}^{\mathrm{op}}\text{-}\mathrm{mod}$.
\item\label{thm307.05}
Restriction from  $\overline{\mathbf{P}}_{\mathtt{i}}$ defines the structure of a 
$2$-representation of $\cC$ on  $\mathbf{C}_{\mathcal{R}}$, which
is called the {\em cell module} corresponding to
$\mathcal{R}$.
\index{cell module}
\end{enumerate}
\end{theorem}

\begin{example}\label{exm312}
{\rm Consider the category $\cC$ from Example~\ref{exm212}. For the
cell representation $\mathbf{C}_{\{\mathbbm{1}_{\mathtt{i}}\}}$
we have $G_{\{\mathbbm{1}_{\mathtt{i}}\}}=\mathbbm{1}_{\mathtt{i}}$,
which implies that $\mathbf{C}_{\{\mathbbm{1}_{\mathtt{i}}\}}(\mathtt{i})=
\mathbb{C}\text{-}\mathrm{mod}$; 
$\mathbf{C}_{\{\mathbbm{1}_{\mathtt{i}}\}}(\mathrm{F})=0$
and $\mathbf{C}_{\{\mathbbm{1}_{\mathtt{i}}\}}(\alpha)=0$ for 
all radical $2$-morphisms $\alpha$. For the
cell representation $\mathbf{C}_{\{\mathrm{F}\}}$
we have $G_{\{\mathrm{F}\}}=\mathrm{F}$,
which implies that $\mathbf{C}_{\{\mathrm{F}\}}(\mathtt{i})\cong
D\text{-}\mathrm{mod}$, 
$\mathbf{C}_{\{\mathrm{F}\}}(\mathrm{F})=\mathrm{F}$
and $\mathbf{C}_{\{\mathrm{F}\}}(\alpha)=\alpha$ for 
all radical $2$-morphisms $\alpha$.
}
\end{example}

\subsection{Homomorphisms from a cell module}\label{s3.6}

Let $\cC$ be a fiat category, $\mathcal{R}$ a right cell in 
$\cC$ and $\mathtt{i}\in\cC$ be such 
that $g_{\mathcal{R}}\in\mathcal{C}_{\mathtt{i},\mathtt{i}}$.
Let further $f\in\cC(\mathtt{i},\mathtt{i})$
and $\alpha:f\to g_{\mathcal{R}}$ be such that 
$\overline{\mathbf{P}}_{\mathtt{i}}(\alpha):f\,P_{\mathbbm{1}_{\mathtt{i}}}
\to g_{\mathcal{R}}\,P_{\mathbbm{1}_{\mathtt{i}}}$
gives a projective presentation for $L_{g_{\mathcal{R}}}$.

\begin{theorem}\label{thm309}
Let $\mathbf{M}$ be a $2$-representation of $\cC$. Denote by
$\Theta=\Theta_{\mathcal{R}}^{\mathbf{M}}$ the full subcategory of 
$\mathbf{M}(\mathtt{i})$ consisting of all objects isomorphic to
those which appear as cokernel of $\mathbf{M}(\alpha)$. 
\begin{enumerate}[$($a$)$] 
\item\label{thm309.2} For every morphism $\Psi$ 
from $\mathbf{C}_{\mathcal{R}}$ to $\mathbf{M}$ we have
$\Psi(L_{g_{\mathcal{R}}})\in\Theta$.
\item\label{thm309.3} For every $M\in\Theta$
there is a morphism $\Psi^M$  from $\mathbf{C}_{\mathcal{R}}$ 
to $\overline{\langle M\rangle}$ sending $L_{g_{\mathcal{R}}}$ to $M$.
\end{enumerate}
\end{theorem}

\begin{proof}[Idea of the proof.]
Using the universal property of $\mathbf{P}_{\mathtt{i}}$,
this follows from the $2$-functoriality of $\mathbf{M}$.
\end{proof}

\subsection{Serre subcategories and quotients}\label{s3.7}

Let $\cC$ be a fiat category and $\mathbf{M}\in\cC\text{-}\mathfrak{mod}$.
In every $\mathbf{M}(\mathtt{i})$ choose a set
of simple modules and denote by $\mathbf{N}(\mathtt{i})$ the Serre
subcategory of $\mathbf{M}(\mathtt{i})$ which these modules generate.
If $\mathbf{N}$ turns out to be stable with respect to the
action of $\cC$ (restricted from $\mathbf{M}$), then 
$\mathbf{N}$ become a $2$-representation of $\cC$
(a {\em Serre submodule}). Moreover,
\index{Serre submodule}
the quotient $\mathbf{Q}$, defined via 
$\mathbf{Q}(\mathtt{i}):=\mathbf{M}(\mathtt{i})/\mathbf{N}(\mathtt{i})$,
also carries the natural structure of a $2$-representation of $\cC$.

\begin{example}\label{exm315}
{\rm 
The cell module $\mathbf{C}_{\{\mathbbm{1}_{\mathtt{i}}\}}$ in
Example~\ref{exm312} is a Serre submodule of
$\mathbf{P}_{\mathtt{i}}$ and the cell module
$\mathbf{C}_{\{\mathrm{F}\}}$ is equivalent to the corresponding quotient.
}
\end{example}

Serre submodules of $\mathbf{M}$ can be organized into a
{\em Serre filtration} of $\mathbf{M}$.
\index{Serre filtration}

\subsection{Naturally commuting functors}\label{s3.8}

A $2$-morphism between two $2$-representations of some
$2$-category $\cC$ can be understood via functors {\em naturally
commuting} with the functors defining the action of $\cC$
\index{naturally commuting functors}
in the terminology of \cite{Kh}. Note that this notion is not
symmetric, that is if some functor $\mathrm{F}$ naturally commutes 
with the action of $\cC$ it does not follow that an element
of this action naturally commutes with $\mathrm{F}$ (in fact, the
latter does not really make sense as $\mathrm{F}$ is not specified
as an object of some $2$-action of a $2$-category).

Another interesting notion from \cite{Kh} is that of a category
with full projective functors, which just means that we have a
$2$-representation of a $2$-category $\cC$ such that 
$2$-morphisms of $\cC$ surject onto homomorphisms between
projective modules of this representation.

\section{Category $\mathcal{O}$: definitions}\label{s4}

\subsection{Definition of category $\mathcal{O}$}\label{s4.1}

One of the main sources for categorification models is 
the Bernstein-Gelfand-Gelfand (BGG) category $\mathcal{O}$ associated 
to a fixed triangular decomposition of a semi-simple complex 
finite dimensional Lie algebra  $\mathfrak{g}$. 
This category appears in \cite{BGG} as a non-semi-simple extension of 
the semi-simple category of finite-dimensional $\mathfrak{g}$-modules,
which, on the one hand, contains a lot of new objects, notably all
simple highest weight modules, but, on the other hand, has several very
nice properties, notably the celebrated BGG reciprocity, see Theorem~\ref{thm405}.
The study of category $\mathcal{O}$ that followed revealed a number
of different kinds of symmetries (e.g. Ringel self-duality and Koszul self-duality)
and spectacular connections to, in particular, combinatorics and geometry.
Most importantly for us, the Lie-theoretic nature of $\mathcal{O}$ leads to
a variety of naturally defined functorial actions on this category which, as we
will see, are very useful and play an important role in various categorifications.
Let us start with a short recap of basic  properties of $\mathcal{O}$
(see also \cite{Di,Hu}).

For $n\in\mathbb{N}$ consider the reductive complex Lie algebra 
$\mathfrak{g}=\mathfrak{g}_n=\mathfrak{gl}_n=\mathfrak{gl}_n(\mathbb{C})$.
Let $\mathfrak{h}$ denote the commutative subalgebra of all
diagonal matrices (the {\em Cartan subalgebra} of $\mathfrak{g}$).
\index{Cartan subalgebra}
Denote by $\mathfrak{n}_+$ and $\mathfrak{n}_-$ the Lie subalgebras
of upper and lower triangular matrices, respectively.
Then we have the {\em standard triangular decomposition}
\index{triangular decomposition}
$\mathfrak{g}=\mathfrak{n}_-\oplus\mathfrak{h}\oplus\mathfrak{n}_+$.
The subalgebra $\mathfrak{b}=\mathfrak{h}\oplus\mathfrak{n}_+$
is the {\em Borel subalgebra} of $\mathfrak{g}$.
\index{Borel subalgebra}
For a Lie algebra $\mathfrak{a}$ we denote  by $U(\mathfrak{a})$
the universal enveloping algebra of $\mathfrak{a}$.

\begin{definition}\label{def401}
{\rm 
The category $\mathcal{O}=\mathcal{O}(\mathfrak{g})$ is the full
subcategory of the category of $\mathfrak{g}$-modules which consists
of all modules $M$ satisfying the following conditions:
\begin{itemize}
\item $M$ is finitely generated;
\item the action of $\mathfrak{h}$ on $M$ is diagonalizable;
\item the action of $U(\mathfrak{n}_+)$ on $M$ is locally finite, 
that is $\dim U(\mathfrak{n}_+)v<\infty$ for all $v\in M$.  
\end{itemize}
}
\end{definition}
\index{category $\mathcal{O}$}

For example, all semi-simple finite-dimensional $\mathfrak{g}$-modules
are objects of $\mathcal{O}$. Elements of $\mathfrak{h}^*$ are called
{\em weights}. For a $\mathfrak{g}$-module $M$ and
\index{weight}
$\lambda\in\mathfrak{h}^*$ define the corresponding weight space
\begin{displaymath}
M_{\lambda}:=\{v\in M:hv=\lambda(h)v\text{ for all }
h\in\mathfrak{h}\}. 
\end{displaymath}
Then the condition of $\mathfrak{h}$-diagonalizability can be
written in the following form:
\begin{displaymath}
M=\bigoplus_{\lambda\in\mathfrak{h}^*} M_{\lambda}.
\end{displaymath}

For $i,j\in\{1,2,\dots,n\}$ denote by $e_{ij}$ the corresponding
matrix unit. Then $\{e_{ii}\}$ form a standard basis of $\mathfrak{h}$. 
For $i<j$ set $\alpha_{ij}=e_{ii}^*-e_{jj}^*$. Then 
$\mathbf{R}:=\{\pm\alpha_{ij}\}$
is a root system of type $A_n$ (in its linear hull). Let 
$W\cong\mathbb{S}_n$ be the corresponding Weyl group. It acts on
$\mathfrak{h}^*$ (and $\mathfrak{h}$) by permuting indexes 
of elements of the standard basis. For a root $\alpha$ we denote
by $s_{\alpha}$ the corresponding reflection in $W$.
For $i=1,2,\dots,n-1$ we also denote by $s_i$ the simple
reflection $s_{\alpha_{ii+1}}$.  We denote by $S$ the set of all 
simple reflections. The corresponding positive roots form a basis 
of $\mathbf{R}$. In terms of the generators $s_1,s_2,\dots,s_{n-1}$
the set of defining relations for $W$ consists of the relations
$s_i^2=e$ for all $i$ (meaning that every $s_i$ is an involution) 
together with the following relations, called {\em braid relations}:
\index{braid relations}
\begin{displaymath}
\begin{array}{ccc}
s_is_j=s_js_i&\text{ for all } i,j \text{ such that } &|i-j|>1;\\
s_is_{j}s_i=s_{j}s_is_{j}&\text{ for all } i,j \text{ such that } & |i-j|=1.
\end{array}
\end{displaymath}
A presentation of the {\em braid group} $\mathbb{B}_n$ is given by the 
\index{braid group}
same generators using only braid relations (so, the generators of $\mathbb{B}_n$
are no longer involutive). In this way $\mathbb{B}_n$ appears with the
natural epimorphism onto $W$ given by the identity map on the generators.

Let $\rho$ be the half of the sum of all positive 
roots. Define the {\em dot-action} of $W$ on $\mathfrak{h}^*$ as follows:
\index{dot-action}
$w\cdot\lambda=w(\lambda+\rho)-\rho$. 

For any $M$, $\lambda$ and $i<j$ we have
$e_{ij} M_{\lambda}\subset M_{\lambda+\alpha_{ij}}$ and
$e_{ji} M_{\lambda}\subset M_{\lambda-\alpha_{ij}}$.
For $i<j$ we have positive roots
$\{\alpha_{ij}\}$ and negative roots
$\{-\alpha_{ij}\}$, moreover,  the matrix
units $e_{ij}$ and $e_{ji}$ are root element for roots $\alpha_{ij}$
and $-\alpha_{ij}$, respectively. In particular, 
$\mathfrak{n}_+$ and $\mathfrak{n}_-$ are the linear spans of all 
positive and negative root spaces, respectively. With respect to
the system
$S$ we have the length function $\mathfrak{l}:W\to\{0,1,2,\dots\}$.
We denote by $w_o$ the longest element of $W$ and by 
$\leq$ the Bruhat order on $W$.

Define the {\em standard} partial order $\leq$ on $\mathfrak{h}^*$ as 
\index{standard partial order}
follows: $\lambda\leq \mu$ if and only if $\mu-\lambda$ is a 
linear combination of positive roots with non-negative integral 
coefficients. Let $\mathfrak{h}^*_{\mathrm{dom}}$ denote the set of all 
elements in $\mathfrak{h}^*$ dominant with respect to the dot-action.

\subsection{Verma modules}\label{s4.2}

For $\lambda\in \mathfrak{h}^*$ let $\mathbb{C}_{\lambda}$ be the
one-dimensional $\mathfrak{h}$-module on which elements of 
$\mathfrak{h}$ act via $\lambda$. Setting 
$\mathfrak{n}_+\mathbb{C}_{\lambda}=0$ defines on $\mathbb{C}_{\lambda}$
the structure of a $\mathfrak{b}$-module. The induced module
\begin{displaymath}
M(\lambda):=U(\mathfrak{g})\bigotimes_{U(\mathfrak{b})}
\mathbb{C}_{\lambda}  
\end{displaymath}
is called the {\em Verma module} with highest weight $\lambda$,
\index{Verma module}
see \cite{Ve,BGG0,Hu} and \cite[Chapter~7]{Di}. By adjunction,
$M(\lambda)$ has the following universal property: for any
$\mathfrak{g}$-module $N$ we have
\begin{displaymath}
\begin{array}{ccc}
\mathrm{Hom}_{\mathfrak{g}}(M(\lambda),N)
&\cong&\{v\in N_{\lambda}:\mathfrak{n}_+v=0\}\\
\varphi&\mapsto& \varphi(1\otimes 1).
\end{array} 
\end{displaymath}
A weight vector $v$ satisfying $\mathfrak{n}_+v=0$ is called a
{\em highest weight} vector. A module generated by a highest weight 
\index{highest weight vector}\index{highest weight module}
vector is called a {\em highest weight} module.

It is easy to check that $M(\lambda)\in\mathcal{O}$ and that 
$M(\lambda)_{\lambda}=\mathbb{C}\langle 1\otimes 1\rangle$.
Further, $M(\lambda)$ has the unique maximal submodule, namely
the sum of all submodules which do not intersect $M(\lambda)_{\lambda}$.
Hence $M(\lambda)$ has the unique simple quotient denoted by
$L(\lambda)$. Obviously $L(\lambda)\in \mathcal{O}$.
The set $\{L(\lambda):\lambda\in \mathfrak{h}^*\}$ is
a complete and irredundant set of simple highest weight modules
(which are exactly the simple objects in $\mathcal{O}$).

\begin{theorem}\label{thm402}
\begin{enumerate}[$($a$)$]
\item\label{thm402.1} $M(\lambda)$ has finite length. 
\item\label{thm402.2} $[M(\lambda):L(\lambda)]=1$ and
$[M(\lambda):L(\mu)]\neq 0$ implies $\mu\leq \lambda$. 
\item\label{thm402.3} $M(\lambda)$ has simple socle. 
\item\label{thm402.8} $\dim\mathrm{Hom}_{\mathfrak{g}}
(M(\lambda):M(\mu))\leq 1$ and any nonzero element of this space is 
injective. 
\item\label{thm402.4} The following conditions are equivalent:
\begin{enumerate}[$($i$)$]
\item\label{thm402.4.1} $\dim\mathrm{Hom}_{\mathfrak{g}}
(M(\mu):M(\lambda))\neq 0$. 
\item\label{thm402.4.2} $[M(\lambda):L(\mu)]\neq 0$. 
\item\label{thm402.4.3} There is a sequence $\beta_1,\dots,\beta_k$
of positive roots such that 
\begin{displaymath}
\lambda\geq s_{\beta_1}\cdot \lambda\geq 
s_{\beta_2}s_{\beta_1}\cdot \lambda\geq\dots\geq
s_{\beta_k}\cdots s_{\beta_2}s_{\beta_1}\cdot \lambda=\mu.
\end{displaymath}
\end{enumerate}
\item\label{thm402.5} Every endomorphism of $M(\lambda)$ is
scalar.
\end{enumerate} 
\end{theorem}

\subsection{Block decomposition}\label{s4.3}

Denote by $Z(\mathfrak{g})$ the center of $U(\mathfrak{g})$.
By Theorem~\ref{thm402}\eqref{thm402.5}, the action of $Z(\mathfrak{g})$
on $M(\lambda)$ is scalar. Let 
$\chi_{\lambda}:Z(\mathfrak{g})\to\mathbb{C}$ be the corresponding
{\em central character}. The map $\lambda\mapsto\chi_{\lambda}$ sets
\index{central character}
up a bijection between $\mathfrak{h}^*_{\mathrm{dom}}$ and the
set of central characters (see \cite[Section~7.4]{Di}).
For a dominant weight $\lambda$ denote by $\mathcal{O}_{\lambda}$
the full subcategory of $\mathcal{O}$ consisting of all modules
on which the kernel of $\chi_{\lambda}$ acts locally nilpotently.

\begin{theorem}\label{thm403}
Every simple module in $\mathcal{O}_{\lambda}$ has the form
$L(w\cdot\lambda)$ for some $w\in W$ and we have
\begin{displaymath}
\mathcal{O}=\bigoplus_{\lambda\in
\mathfrak{h}^*_{\mathrm{dom}}}\mathcal{O}_{\lambda}. 
\end{displaymath}
\end{theorem}

In particular, it follows that every object in $\mathcal{O}$ has
finite length. Categories $\mathcal{O}_{\lambda}$ are called 
{\em blocks} of $\mathcal{O}$. For $\lambda=0$ the block
\index{block}\index{principal block}
$\mathcal{O}_0$ contains the (unique up to isomorphism) one-dimensional 
$\mathfrak{g}$-module and is called the {\em principal} block 
of $\mathcal{O}$. Note that blocks $\mathcal{O}_{\lambda}$ might
be further decomposable. To remedy this, in what follows we will
restrict our attention to {\em integral} blocks, that is blocks
\index{integral block}
corresponding to $\lambda$, the coordinates of which in the
standard basis have integral differences. By \cite{So0}, any block 
is equivalent to an integral block, possibly for a smaller $n$. 
Note that, for integral $\mu\in\mathfrak{h}^*$, the Verma module 
$M(\mu)$ is simple if and only if $\mu$ is dot-antidominant, further,
$M(\mu)$ is projective if and only if $\mu$ is dot-dominant.

An integral block $\mathcal{O}_{\lambda}$ is called {\em regular} 
\index{regular block}
if the orbit $W\cdot \lambda$ is regular. In this case simple
objects in the block are bijectively indexed by elements of $W$.
An integral block $\mathcal{O}_{\lambda}$ is called {\em singular}
\index{singular block}
if it is not regular. Let $W_{\lambda}$ denote the dot-stabilizer 
of $\lambda$ in $W$. Then simple objects in the singular block
$\mathcal{O}_{\lambda}$ are bijectively indexed by 
cosets $W/W_{\lambda}$, or , alternatively, by longest coset
representatives in these cosets.

\begin{theorem}\label{thm404}
Every $\mathcal{O}_{\lambda}$ has enough projectives, in particular,
$\mathcal{O}_{\lambda}$ is equivalent to the category 
$B_{\lambda}\text{-}\mathrm{mod}$ of finite dimensional modules over 
some finite dimensional associative algebra $B_{\lambda}$.
\end{theorem}

For $\mu\in\mathfrak{h}^*$ we denote by $P(\mu)$ and
$I(\mu)$ the indecomposable projective cover and injective hull of
$L(\mu)$ in $\mathcal{O}$, respectively. For $\lambda\in
\mathfrak{h}^*_{\mathrm{dom}}$ set $P_{\lambda}:=\oplus_{\mu\in
W\cdot\lambda}P(\mu)$. Then $P_{\lambda}$ is a projective generator
of $\mathcal{O}_{\lambda}$ and hence we can take 
$B_{\lambda}=\mathrm{End}_{\mathfrak{g}}(P_{\lambda})^{\mathrm{op}}$.

\subsection{BGG reciprocity and quasi-hereditary structure}\label{s4.4}

A module $N\in \mathcal{O}$ is said to have a {\em standard filtration}
\index{standard filtration}\index{Verma flag}
or {\em Verma flag} if there is a filtration of $N$ whose subquotients
are Verma modules. The category of all modules in $\mathcal{O}$ which
have a standard filtration is denoted by $\mathcal{F}(\Delta)$.

\begin{theorem}[BGG reciprocity]\label{thm405}
\begin{enumerate}[$($a$)$]
\item\label{thm405.1} Every projective module in 
$\mathcal{O}$ has a standard filtration.
\item\label{thm405.2} If some $N\in \mathcal{O}$ has a standard filtration,
then for any $\mu\in\mathfrak{h}^*$ the multiplicity 
$[N:M(\mu)]$ of $M(\mu)$ as a subquotient of a standard
filtration of $N$ does not depend on the choice of such filtration.
\item\label{thm405.3} For every $\mu,\nu\in\mathfrak{h}^*$ we have
$[P(\mu):M(\nu)]=[M(\nu):L(\mu)]$.
\item\label{thm405.4} We have $[P(\mu):M(\mu)]=1$ 
and $[P(\mu):M(\nu)]\neq 0$ implies $\nu\geq \mu$.
\end{enumerate}
\end{theorem}

The claims of Theorem~\ref{thm405}\eqref{thm405.4} and
Theorem~\ref{thm402}\eqref{thm402.2} literally mean
that for $\lambda\in \mathfrak{h}^*_{\mathrm{dom}}$ 
the category $\mathcal{O}_{\lambda}$ is a {\em highest weight}
\index{highest weight category}
category in the sense of \cite{CPS}, and the algebra $B_{\lambda}$ is 
{\em quasi-hereditary}  in the sense of \cite{DR} with Verma modules being
\index{quasi-hereditary algebra}
{\em standard modules} for this quasi-hereditary structure. 
\index{standard module}
Traditionally, standard modules for quasi-hereditary algebras
are denoted $\Delta(\mu)=M(\mu)$.

The category $\mathcal{O}$ has a {\em simple preserving duality}
\index{simple preserving duality}
$\star$, that is an involutive contravariant equivalence which
preserves isoclasses of simple modules. This implies that 
$B_{\lambda}\cong B_{\lambda}^{\mathrm{op}}$ such that the
isomorphism preserves the equivalence classes of primitive 
idempotents. We have $L(\mu)^{\star}=L(\mu)$ and 
$P(\mu)^{\star}\cong I(\mu)$. The modules 
$\nabla(\mu)=\Delta(\mu)^{\star}$ are called {\em dual Verma}
\index{dual Verma module}
modules. Applying $\star$ to Theorem~\ref{thm405} we, in particular, 
have that every injective in $\mathcal{O}$ has a filtration whose 
subquotients are dual Verma modules (the so-called 
{\em costandard filtration}). The category of all modules in 
\index{costandard filtration}
$\mathcal{O}$ which have a costandard filtration is denoted by 
$\mathcal{F}(\nabla)$.

\begin{corollary}\label{cor406}
For every $\lambda\in \mathfrak{h}^*_{\mathrm{dom}}$ the block 
$\mathcal{O}_{\lambda}$ has finite global dimension. 
\end{corollary}

\begin{proof}[Idea of the proof.]
Using induction with respect to $\leq$ and 
Theorem~\ref{thm405}\eqref{thm405.1}, show that every Verma 
module has finite projective dimension. Then, using the opposite
induction with respect to $\leq$ and 
Theorem~\ref{thm402}\eqref{thm402.2} show that every simple module
has finite projective dimension.
\end{proof}

One can show that the global dimension of $\mathcal{O}_0$ equals
$2\mathfrak{l}(w_o)$, see e.g. \cite{Ma2}. In fact, there are 
explicit formulae for the projective dimension of all structural
modules in $\mathcal{O}_0$, see \cite{Ma2,Ma3}.
Another important homological feature of category $\mathcal{O}$,
as a highest weight category, is:

\begin{proposition}\label{cor436}
$\mathcal{F}(\Delta)=\{N\in\mathcal{O}:\mathrm{Ext}_{\mathcal{O}}^i
(N,\nabla(\mu))=0\text{ for all }\mu\in\mathfrak{h}^*, i>0\}$. 
\end{proposition}

\subsection{Tilting modules and Ringel self-duality}\label{s4.5}

One of the nicest properties of quasi-hereditary algebras is
existence of the so-called {\em tilting modules}, established in
\index{tilting module}
\cite{Ri} (see also \cite{CI} for the special case of the category
$\mathcal{O}$).

\begin{theorem}\label{thm406}
For every $\mu\in\mathfrak{h}^*$ there is a
unique (up to isomorphism) indecomposable module $T(\mu)\in
\mathcal{F}(\Delta)\cap\mathcal{F}(\nabla)$ such that 
$\Delta(\mu)\subset T(\mu)$ and the cokernel of this
inclusion admits a standard filtration.
\end{theorem}

For $\lambda\in\mathfrak{h}^*_{\mathrm{dom}}$ the 
module $T_{\lambda}:=\oplus_{\mu\in W\cdot\lambda}T(\mu)$ is
called the {\em characteristic tilting module}. The word
\index{characteristic tilting module}
{\em tilting} is justified by the following property:

\begin{theorem}\label{thm407}
The module $T_{\lambda}$ is ext-selforthogonal, 
has finite projective dimension and there is an exact sequence
\begin{displaymath}
0\to P_{\lambda}\to Q_0\to Q_1\to\dots\to Q_k\to 0 
\end{displaymath}
such that $Q_i\in\mathrm{add}(T_{\lambda})$ for all $i$.
\end{theorem}

The (opposite of the) endomorphism algebra of $T_{\lambda}$ is called 
the  {\em Ringel dual} of $B_{\lambda}$. The Ringel dual is defined
\index{Ringel dual}
for any quasi-hereditary algebra and is again a quasi-hereditary algebra.
However, for category $\mathcal{O}$ we have the following
{\em Ringel self-duality}:
\index{Ringel self-duality}

\begin{theorem}[\cite{So}]\label{thm408}
$\mathrm{End}_{\mathfrak{g}}(T_{\lambda})=B_{\lambda}$.
\end{theorem}

For every $\mu\in\mathfrak{h}^*$ we have $T(\mu)=T(\mu)^*$
and hence tilting modules in $\mathcal{O}$ are also {\em cotilting}.
\index{cotilting module}

\subsection{Parabolic category $\mathcal{O}$}\label{s4.6}

Let $\mathfrak{p}\subset \mathfrak{g}$ be a parabolic subalgebra
containing $\mathfrak{b}$. In what follows we only consider such
parabolic subalgebras. Define $\mathcal{O}^{\mathfrak{p}}$
as the full subcategory of $\mathcal{O}$ consisting of all 
modules, on which the action of $U(\mathfrak{p})$ is locally finite.
Alternatively, one could consider modules which decompose into a
direct sum of finite dimensional modules, when restricted to 
the Levi factor of $\mathfrak{p}$. For example, taking
$\mathfrak{g}=\mathfrak{p}$ the category $\mathcal{O}^{\mathfrak{p}}$
becomes the subcategory  of finite dimensional modules in
$\mathcal{O}$.

From the definition it follows
that $\mathcal{O}^{\mathfrak{p}}$ is a Serre subcategory of 
$\mathcal{O}$ and thus is uniquely determined by simple modules
which it contains. Let 
$\mathfrak{i}_{\mathfrak{p}}:\mathcal{O}^{\mathfrak{p}}\to
\mathcal{O}$ be the natural inclusion, which is obviously exact. 
Further, denote by $\mathrm{Z}_{\mathfrak{p}}:\mathcal{O}\to
\mathcal{O}^{\mathfrak{p}}$ and 
$\hat{\mathrm{Z}}_{\mathfrak{p}}:\mathcal{O}\to
\mathcal{O}^{\mathfrak{p}}$ the left and the right adjoints to
$\mathfrak{i}_{\mathfrak{p}}$, respectively. The functor
$\mathrm{Z}_{\mathfrak{p}}$ is called the {\em Zuckerman} functor
\index{Zuckerman functor}\index{dual Zuckerman functor}
and $\hat{\mathrm{Z}}_{\mathfrak{p}}$ is called the {\em dual Zuckerman} 
functor (see \cite{Zu}). These functors can be described as the functors of 
taking the maximal quotient and submodule in $\mathcal{O}^{\mathfrak{p}}$,
respectively. Denote by $W_{\mathfrak{p}}$ the parabolic subgroup
of $W$ corresponding to the Levi factor of $\mathfrak{p}$.

\begin{proposition}\label{prop411}
For $\mu\in\mathfrak{h}^*$ we have 
$L(\mu)\in \mathcal{O}^{\mathfrak{p}}$ if and only if
$w\cdot\mu<\mu$ for all $w\in W_{\mathfrak{p}}$, $w\neq e$.
\end{proposition}

In other words, $L(\mu)\in \mathcal{O}^{\mathfrak{p}}$ if and
only if $\mu$ is $W_{\mathfrak{p}}$-dominant. In particular, if
$\lambda$ is dominant, integral and regular, then simple
modules in $\mathcal{O}^{\mathfrak{p}}_{\lambda}$ have the form
$L(w\cdot \lambda)$, where $w$ is the shortest coset
representative for some coset from $W_{\mathfrak{p}}\backslash W$.
For $\mu\in\mathfrak{h}^*$ we set
\begin{displaymath}
L^{\mathfrak{p}}(\mu)= \mathrm{Z}_{\mathfrak{p}}L(\mu),\quad
P^{\mathfrak{p}}(\mu)= \mathrm{Z}_{\mathfrak{p}}P(\mu),\quad
\Delta^{\mathfrak{p}}(\mu)= \mathrm{Z}_{\mathfrak{p}}\Delta(\mu).
\end{displaymath}
Modules $\Delta^{\mathfrak{p}}(\mu)$ are called {\em parabolic
Verma modules} and can be alternatively described using
\index{parabolic Verma module}
parabolic induction from a simple
finite-dimensional $\mathfrak{p}$-module.

For $\lambda\in \mathfrak{h}^*_{\mathrm{dom}}$ the category
$\mathcal{O}^{\mathfrak{p}}_{\lambda}$ is equivalent to 
the category of modules over the quotient algebra 
$B^{\mathfrak{p}}_{\lambda}=B_{\lambda}/B_{\lambda}eB_{\lambda}$,
where $e$ is a maximal idempotent annihilating 
$\mathcal{O}^{\mathfrak{p}}_{\lambda}$. The duality $\star$
preserves $\mathcal{O}^{\mathfrak{p}}$ and hence
$B^{\mathfrak{p}}_{\lambda}=(B^{\mathfrak{p}}_{\lambda})^{\mathrm{op}}$.
We set
\begin{displaymath}
\nabla^{\mathfrak{p}}(\mu)=\Delta^{\mathfrak{p}}(\mu)^{\star},\quad
I^{\mathfrak{p}}(\mu)=P^{\mathfrak{p}}(\mu)^{\star}.
\end{displaymath}

Finally, the algebra $B^{\mathfrak{p}}_{\lambda}$ is quasi-hereditary
with parabolic Verma modules being standard modules with respect to
this structure, see \cite{RC}. We denote by $T^{\mathfrak{p}}(\mu)$
the corresponding indecomposable tilting modules and by
$T^{\mathfrak{p}}_{\lambda}$ the corresponding characteristic tilting 
module. Note that these cannot be obtained from the corresponding
tilting modules in $\mathcal{O}$ using Zuckerman functors na{\"\i}vely.
Similarly to $B_{\lambda}$, the algebra $B^{\mathfrak{p}}_{\lambda}$
is Ringel self-dual. One important difference with the usual 
category $\mathcal{O}$ is that parabolic Verma modules might have
non-simple socle and that the dimension of the homomorphism space
between different parabolic Verma modules might be bigger than two.

\subsection{$\mathfrak{gl}_2$-example}\label{s4.7}

For $n=2$ every singular block is semi-simple (as $W$ contains
only two elements). Note that we have $\rho=\frac{1}{2}(1,-1)$. 
The regular block $\mathcal{O}_0$ contains two simple modules,
$L(0,0)$ and $L(-1,1)=\Delta(-1,1)=\Delta(-1,1)^*=T(-1,1)$. The module
$L(0,0)$ is the trivial module. The action of $\mathfrak{g}$ on the 
module $L(-1,1)$, considered as a Verma module, is given by
the following picture (here left arrows represent the action of 
$e_{21}$, right arrows represent the action of 
$e_{12}$, and numbers are coefficients):
\begin{displaymath}
\xymatrix{ 
\dots\ar@/^1pc/[rr]^{-12} && 
e_{21}^2\otimes 1\ar@/^1pc/[ll]^1\ar@/^1pc/[rr]^{-6}  
&& e_{21}\otimes 1\ar@/^1pc/[ll]^1\ar@/^1pc/[rr]^{-2} && 
1\otimes 1\ar@/^1pc/[ll]^1
}
\end{displaymath}
The Verma module $\Delta(0,0)=P(0,0)$ is given by
the following picture (no right arrow means that the corresponding
coefficient is zero):
\begin{displaymath}
\xymatrix{ 
\dots\ar@/^1pc/[rr]^{-6} && 
e_{21}^2\otimes 1\ar@/^1pc/[ll]^1\ar@/^1pc/[rr]^{-2}  
&& e_{21}\otimes 1\ar@/^1pc/[ll]^1&& 
1\otimes 1\ar@/^1pc/[ll]^1
}
\end{displaymath}
The dual Verma module $\Delta(0,0)^{\star}$ is given by
the following picture:
\begin{displaymath}
\xymatrix{ 
\dots\ar@/^1pc/[rr]^{1} && 
(e_{21}^2\otimes 1)^*\ar@/^1pc/[ll]^{-6}\ar@/^1pc/[rr]^{1}  
&& (e_{21}\otimes 1)^*\ar@/^1pc/[ll]^{-2}\ar@/^1pc/[rr]^{1}&& 
(1\otimes 1)^*
}
\end{displaymath}
In this small example there is only one indecomposable module
left, namely the projective module $P(-1,1)$. One can choose
a basis in $P(-1,1)$ (given by bullets below) such that the action of
$e_{12}$ and $e_{21}$ in this basis can be given by the following picture
(see \cite[Chapter~3]{Ma} for explicit coefficients):
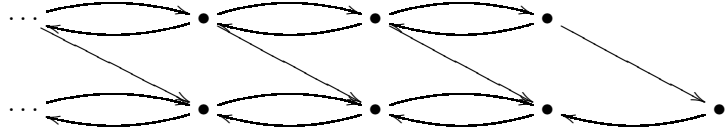
\begin{displaymath}
\xymatrix{ 
\dots\ar@/^/[rr]\ar[rrd] 
&&\bullet\ar@/^/[ll]\ar@/^/[rr]\ar[rrd]
&&\bullet\ar@/^/[ll]\ar@/^/[rr]\ar[rrd] 
&&\bullet\ar@/^/[ll]\ar[rrd] && \\
\dots\ar@/^/[rr]&&\bullet\ar@/^/[ll]\ar@/^/[rr]
&&\bullet\ar@/^/[ll]\ar@/^/[rr] 
&&\bullet\ar@/^/[ll] && \bullet\ar@/^/[ll]\\
}
\end{displaymath}
In particular, we see that $\Delta(0)\subset P(-1,1)$ and the cokernel
of this map is isomorphic to  $\Delta(-1,1)$. This means that 
$P(-1,1)=T(0)=P(-1,1)^*$. The algebra $B_0$ is given by the
following quiver and relations:
\begin{equation}\label{eq457}
\xymatrix{ 
s\ar@/^/[rr]^a &&e\ar@/^/[ll]^b
},\quad\quad ab=0.
\end{equation}
Here the vertex $e$ corresponds to the simple $L(0,0)$ and
the vertex $s$ corresponds to the simple $L(-1,1)$. Note that all
indecomposable $B_0$-modules are uniserial. Abbreviating
$B_0$-simples by the notation for the corresponding vertexes,
we obtain the following unique Loewy filtrations of 
indecomposable $B_0$-modules:
\begin{displaymath}
\begin{array}{|c|c|c|c|c|}
\hline
L(0,0)&L(-1,1)&\Delta(0,0)&\nabla(0,0)&P(-1,1)\\
\hline\hline
\begin{array}{c}e\\\\\end{array}&\begin{array}{c}s\\\\\end{array}&
\begin{array}{c}e\\s\\\end{array}&\begin{array}{c}s\\e\\\end{array}&
\begin{array}{c}s\\e\\s\end{array}\\
\hline
\end{array}
\end{displaymath}

The only non-Borel parabolic subalgebra of $\mathfrak{g}$ is
$\mathfrak{g}$ itself. Hence $\mathcal{O}_0^{\mathfrak{g}}$
is the semi-simple category with simple module $L(0,0)$.

\section{Category $\mathcal{O}$: projective and shuffling 
functors}\label{s5}

\subsection{Projective functors}\label{s5.1}

Recall that for any two $\mathfrak{g}$-modules $X$ and $Y$ the vector
space $X\otimes Y$ has the structure of a $\mathfrak{g}$-module given
by $g(x\otimes y)=g(x)\otimes y+ x\otimes g(y)$ for all 
$g\in \mathfrak{g}$, $x\in X$ and $y\in Y$. Hence for every 
$\mathfrak{g}$-module $V$ we have the endofunctor 
$V\otimes_{\mathbb{C}}{}_-$ on the category of 
$\mathfrak{g}$-modules. If $V$ is finite dimensional, the functor
$V\otimes_{\mathbb{C}}{}_-$ preserves $\mathcal{O}$.

\begin{definition}\label{def501}
{\rm 
A functor $\theta:\mathcal{O}\to\mathcal{O}$ 
(or $\theta:\mathcal{O}_{\lambda}\to \mathcal{O}_{\lambda'}$)
is called {\em projective}
\index{projective functor}
if it is isomorphic to a direct summand of some 
$V\otimes_{\mathbb{C}}{}_-$, where $V$ is finite dimensional.
} 
\end{definition}

For example, the identity functor is a projective functor
(it is isomorphic to the tensoring with the one-dimensional 
$\mathfrak{g}$-module). Projective functors appeared already 
in \cite{BGG0,Ja0}. Indecomposable projective functors are classified by:

\begin{theorem}[\cite{BG}]\label{thm502}
Let $\lambda$ and $\lambda'$ be dominant and integral. 
\begin{enumerate}[$($a$)$]
\item \label{thm502.1} For every $W_{\lambda}$-antidominant 
$\mu\in W\cdot \lambda'$ there is a unique indecomposable projective 
functor $\theta_{\lambda,\mu}$ such that 
$\theta_{\lambda,\mu}\,\Delta(\lambda)=P(\mu)$.
\item\label{thm502.2} Every indecomposable projective functor 
from $\mathcal{O}_{\lambda}$ to $\mathcal{O}_{\lambda'}$ is isomorphic 
to $\theta_{\lambda,\mu}$ for some $W_{\lambda}$-antidominant 
$\mu\in W\cdot \lambda'$.
\end{enumerate}
\end{theorem}

Theorem~\ref{thm502} implies that an indecomposable projective functor 
$\theta$ is completely determined by its value $\theta\,\Delta(\lambda)$
on the corresponding dominant Verma module $\Delta(\lambda)$.
Moreover, as $\theta\,\Delta(\lambda)$ is projective and projective
modules form a basis of $[\mathcal{O}_{\lambda'}]$ (as 
$\mathcal{O}_{\lambda'}$,
being quasi-hereditary, has finite global dimension), the functor
$\theta$ is already uniquely determined by $[\theta\,\Delta(\lambda)]$.

Another consequence is that in case  $\lambda$ is regular,  
indecomposable projective
endofunctors of $\mathcal{O}_{\lambda}$ are in a natural bijection 
with $W$. For $w\in W$ we will denote by $\theta_w$ the indecomposable
projective functor such that $\theta_w\,\Delta(\lambda)=P(w\cdot \lambda)$.
Here are some basic general properties of projective functors
(see e.g. \cite{BG}):

\begin{proposition}\label{prop503}
\begin{enumerate}[$($a$)$]
\item\label{prop503.1} Any direct summand of a projective functor is
a projective functor.
\item\label{prop503.2} Any direct sum of projective functors is
a projective functor.
\item\label{prop503.3} Any composition of projective functors is
a projective functor.
\item\label{prop503.5} Every projective functor is both left and
right adjoint to some other projective functor.
\item\label{prop503.4} Projective functors are exact.
\item\label{prop503.6} Projective functors preserve the additive
category of projective modules and the additive
category of injective modules.
\item\label{prop503.7} Projective functors preserve both
$\mathcal{F}(\Delta)$ and $\mathcal{F}(\nabla)$ and hence also the 
additive category of tilting modules.
\item\label{prop503.8} Projective functors commute with $\star$. 
\end{enumerate}
\end{proposition}

For projective endofunctors of $\mathcal{O}_{\lambda}$ 
the claim of Proposition~\ref{prop503}\eqref{prop503.5} can be
made more explicit:

\begin{proposition}\label{prop504}
Let $\lambda$ be regular, integral and dominant. Then for $w\in W$ 
the functor $\theta_w$ is both left and right adjoint to 
$\theta_{w^{-1}}$.
\end{proposition}

\subsection{Translations through walls}\label{s5.2}

Let $\lambda$ be regular, integral and dominant.
For a simple reflection $s$ the projective endofunctor 
$\theta_s$ of $\mathcal{O}_{\lambda}$ is called {\em 
translation through the $s$-wall}. 
\index{translation through the wall}
Let $\mu$ be integral and 
dominant such that $W_{\mu}=\{e,s\}$. Then
\begin{equation}\label{eq591}
\theta_s\cong\theta_{\mu,s\cdot\lambda}\circ
\theta_{\lambda,\mu}. 
\end{equation}
The functor $\theta_{\mu,s\cdot\lambda}$ is called {\em translation
out of the $s$-wall} and is both left and right adjoint 
to the functor $\theta_{\lambda,\mu}$, the latter being called
{\em translation to the $s$-wall}. These ``smallest'' translations 
$\theta_s$ have the following properties:
\index{translation out of the wall}
\index{translation to the wall}

\begin{proposition}\label{prop505}
Let $\lambda$ be regular, integral and dominant.
\begin{enumerate}[$($a$)$]
\item\label{prop505.1} Let $w\in W$ and $w=s_1s_2\cdots s_k$ be
a reduced decomposition. Then there is a unique direct summand
of $\theta_{s_k}\theta_{s_{k-1}}\cdots\theta_{s_1}$ 
isomorphic to $\theta_w$. All other direct summands of
$\theta_{s_k}\theta_{s_{k-1}}\cdots\theta_{s_1}$ are isomorphic to
$\theta_x$ for $x<w$.
\item\label{prop505.2} For any simple reflection $s\in S$ and
any $w\in W$ there are exact sequences
\begin{displaymath}
\begin{array}{rc}
0\to \Delta(w\cdot\lambda)\to 
\theta_s\Delta(w\cdot\lambda) \to 
\Delta(ws\cdot\lambda)\to 0, & ws>w;\\
0\to \Delta(ws\cdot\lambda)\to 
\theta_s\Delta(w\cdot\lambda) \to 
\Delta(w\cdot\lambda)\to 0, & ws<w.
\end{array}
\end{displaymath}
\item\label{prop505.3} $\theta_s\circ \theta_s\cong\theta_s\oplus\theta_s$.
\item\label{prop505.4} If $s$ and $t$ are commuting simple reflections, then $\theta_s\circ \theta_t\cong\theta_t\circ\theta_s$.
\item\label{prop505.5} If $s$ and $t$ are simple reflections such that
$sts=tst$, then 
\begin{displaymath}
(\theta_s\circ \theta_t\circ \theta_s)\oplus\theta_t \cong(\theta_t\circ\theta_s\circ\theta_t)\oplus\theta_s. 
\end{displaymath}
\end{enumerate} 
\end{proposition}

\begin{proof}[Idea of the proof.]
Apply functors to the dominant Verma and compute the images 
of the results in the Grothendieck group.
\end{proof}

For a regular, integral and dominant $\lambda$ denote by 
$\cS_{\lambda}$ the category of projective endofunctors 
of $\mathcal{O}_{\lambda}$. The category $\cS_{\lambda}$
is additive and monoidal (the monoidal structure is given
by composition of functors). Hence $\cS_{\lambda}$ can be
viewed as the endomorphism category of the object in a 
$2$-category with one object. Recall that $W\cong\mathbb{S}_n$.
The following results provides a categorification of the
integral group ring $\mathbb{Z}\left[\mathbb{S}_n\right]$:

\begin{corollary}[Categorification of the integral group ring
of $\mathbb{S}_n$]\label{cor506}
The map
\begin{displaymath}
\begin{array}{ccc}
\mathbb{Z}[W]&\to& [\cS_{\lambda}]\\
e+s&\mapsto& [\theta_s]
\end{array} 
\end{displaymath}
induces an anti-isomorphism of unital rings.
\end{corollary}

\begin{proof}[Idea of the proof.]
The ring $\mathbb{Z}[W]$ is generated, as a unital ring, by 
elements $e+s$, where $s$ is a simple reflection. The defining
relations of $\mathbb{Z}[W]$ in this basis are:
\begin{gather*}
(e+s)^2=e+s;\quad 
(e+s)(e+t)=(e+t)(e+s) \text{ if } st=ts;\\
(e+s)(e+t)(e+s)+(e+t)=(e+t)(e+s)(e+t)+(e+s) \text{ if } sts=tst.
\end{gather*}
By Proposition~\ref{prop505}\eqref{prop505.3}-\eqref{prop505.5}, 
the elements $[\theta_s]$ satisfy
these relations and hence the map $e+s\mapsto [\theta_s]$ induces
a ring homomorphism. It is surjective as translations through walls 
generate $\cS_{\lambda}$ by 
Proposition~\ref{prop505}\eqref{prop505.1}. Bijectivity now
follows by comparing ranks.
\end{proof}

From Corollary~\ref{cor506} it follows that the set 
$\{[\theta_w]:w\in W\}$ is a new basis of $\mathbb{Z}[W]$.
Later we will identify this basis as the {\em Kazhdan-Lusztig} basis.
\index{Kazhdan-Lusztig basis}
Note that existence of this basis is a bonus of our categorification
result. From Proposition~\ref{prop505}\eqref{prop505.2} it follows
that the action of projective functors on $\mathcal{O}_{\lambda}$
categorifies the right regular representation of $\mathbb{S}_n$:

\begin{proposition}[Categorification of the right regular
$\mathbb{S}_n$-module]\label{prop521}
\begin{enumerate}[$($a$)$]
\item\label{prop521.1} There is a unique isomorphism of abelian groups
$\varphi:\mathbb{Z}[\mathbb{S}_n]\to [\mathcal{O}_{\lambda}]$ such that
$\varphi(w)=[\Delta(w\cdot\lambda)]$.
\item\label{prop521.2} For any $s\in S$ the following 
diagram commutes:
\begin{displaymath}
\xymatrix{ 
\mathbb{Z}[\mathbb{S}_n]\ar[rr]^{\cdot (e+s)}\ar[d]_{\varphi}
&&\mathbb{Z}[\mathbb{S}_n]\ar[d]_{\varphi} \\
[\mathcal{O}_{\lambda}]\ar[rr]^{[\theta_s]\cdot}&&[\mathcal{O}_{\lambda}]
}
\end{displaymath}
\end{enumerate}
\end{proposition}

The decategorification of the action of projective functors 
on $\mathcal{O}_{\lambda}$ can thus be depicted as follows:
\begin{displaymath}
\xymatrix{\mathcal{O}_{\lambda}\ar@(d,r)[]_{\cS_{\lambda}}}
\qquad\mapsto\qquad
\xymatrix{[\mathcal{O}_{\lambda}]\ar@(d,r)[]_{[\cS_{\lambda}]}}
\qquad\cong\qquad
\xymatrix{\mathbb{Z}[\mathbb{S}_n]\ar@(d,r)[]_{\mathbb{Z}[\mathbb{S}_n]}}
\end{displaymath}

\subsection{Description via Harish-Chandra bimodules}\label{s5.3}

Projective functors, being right exact, can be described, by the
general nonsense, as tensoring with certain bimodules. 
The identity functor is of course isomorphic to the functor 
$U(\mathfrak{g})\otimes_{U(\mathfrak{g})}{}_-$. Hence, for a
finite dimensional module $V$, the functor $V\otimes_{\mathbb{C}}{}_-$
is isomorphic  to the functor $V\otimes_{\mathbb{C}}U(\mathfrak{g})
\otimes_{U(\mathfrak{g})}{}_-$. However, the bimodule 
$V\otimes_{\mathbb{C}}U(\mathfrak{g})$ is ``too big''. On the other
hand, if we fix dominant $\lambda$, then any module in 
$\mathcal{O}_{\lambda}$ is annihilated by 
$(\mathrm{Ker}\chi_{\lambda})^k$ for some fixed big enough $k$ 
(since every module is a quotient of a projective module, we have only 
finitely many indecomposable projectives in  $\mathcal{O}_{\lambda}$, every 
projective  has finite length, and any simple is annihilated by
$\mathrm{Ker}\chi_{\lambda}$). Therefore, restricted to 
$\mathcal{O}_{\lambda}$, the functors
$U(\mathfrak{g})\otimes_{U(\mathfrak{g})}{}_-$ and
$U(\mathfrak{g})/(\mathrm{Ker}\chi_{\lambda})^k\otimes_{U(\mathfrak{g})}{}_-$
are isomorphic. The bimodule 
$M=U(\mathfrak{g})/(\mathrm{Ker}\chi_{\lambda})^k$ is now reasonably 
``small'' in the following sense: Under the adjoint action
of $\mathfrak{g}$ the bimodule $M$ decomposes into a direct sum 
of simple finite-dimensional $\mathfrak{g}$-modules, each occurring with 
finite multiplicity. A finitely generated $\mathfrak{g}$-bimodule
satisfying this condition is called a {\em Harish-Chandra} bimodule.
\index{Harish-Chandra bimodule}
The category of Harish-Chandra bimodules is denoted by 
$\mathcal{H}$ (see \cite[Kapitel~6]{Ja} for details). 

We have the following decomposition of $\mathcal{H}$ with respect to the
action of $Z(\mathfrak{g})$:
\begin{displaymath}
\mathcal{H}=\bigoplus_{\lambda,\mu\in\mathfrak{h}^*_{\mathrm{dom}}} 
{}_{\lambda}\mathcal{H}_{\mu},
\end{displaymath}
where ${}_{\lambda}\mathcal{H}_{\mu}$ denotes the full subcategory of
$\mathcal{H}$, which consists of all bimodules $M$ such that for any
$v\in M$ we have 
$(\mathrm{Ker}\chi_{\lambda})^kv=0$ and
$v(\mathrm{Ker}\chi_{\mu})^k=0$ for $k\gg 0$. For 
$k,l\in\{\infty,1,2,\dots\}$ denote by 
${}_{\lambda}^{k}\mathcal{H}_{\mu}^l$ the full subcategory of
${}_{\lambda}\mathcal{H}_{\mu}$ consisting of all $M$ satisfying 
$(\mathrm{Ker}\chi_{\lambda})^kM=0$ and
$M(\mathrm{Ker}\chi_{\mu})^l=0$. Then we have:

\begin{theorem}[\cite{BG}]\label{thm507}
Let $\lambda'$ be integral and dominant and $\lambda$ be integral,
dominant and regular. Then tensoring with 
$\Delta(\lambda)$ induces an equivalence 
${}_{\,\lambda'}^{\infty}\mathcal{H}_{\lambda}^1
\cong\mathcal{O}_{\lambda'}$.
\end{theorem}

Under the equivalence from Theorem~\ref{thm507}, indecomposable
projective functors from $\mathcal{O}_{\lambda}$ to $\mathcal{O}_{\lambda'}$
correspond to indecomposable Harish-Chandra bimodules. It follows
that every $\theta_{\lambda,\mu}:\mathcal{O}_{\lambda}\to
\mathcal{O}_{\lambda'}$ is isomorphic to tensoring with the
corresponding indecomposable projective object from 
${}_{\,\lambda'}^{\infty}\mathcal{H}_{\lambda}^l$ for appropriate $l$.
The same description is true in the case of singular $\lambda$,
though in this case ${}_{\,\lambda'}^{\infty}\mathcal{H}_{\lambda}^1$
is isomorphic only to a certain quotient of $\mathcal{O}_{\lambda'}$.

Let $\lambda$ be integral, dominant and regular. Then the category
${}_{\,\,\lambda}^{\infty}\mathcal{H}_{\lambda}^{\infty}$ has the
natural monoidal structure given by tensor product of bimodules. 
The disadvantage of ${}_{\,\,\lambda}^{\infty}\mathcal{H}_{\lambda}^{\infty}$
is that this category does not have projective modules. The category
${}_{\,\,\lambda}^{\infty}\mathcal{H}_{\lambda}^{\infty}$ can
be described as a {\em deformation} of the category
\index{deformation}
${}_{\,\,\lambda}^{\infty}\mathcal{H}_{\lambda}^{1}$ and can be
realized via the so-called {\em thick category $\tilde{\mathcal{O}}$},
\index{thick category $\tilde{\mathcal{O}}$}
defined by weakening the condition of $\mathfrak{h}$-diagonalizability
to the condition of local finiteness  of the $U(\mathfrak{h})$-action.

\subsection{Shuffling functors}\label{s5.4}

As we have seen in Corollary~\ref{cor506}, the projective functors
$\theta_s$ are categorical analogues of the elements $e+s\in
\mathbb{Z}[\mathbb{S}_n]$. A natural question is: what
are analogues of simple reflections? As $s=(e+s)-e$, and the
categorical analogue of $e$ is the identity functor $\theta_e$, to
answer this question we have to ``subtract'' $\theta_e$ from
$\theta_s$ in some way. Na{\"\i}vely this is not possible as
$\theta_s$ is an indecomposable functor. But, recall from 
Subsection~\ref{s2.3} that we have another natural way to interpret
negative coefficients: as shifts in homological position. This
suggests to try to realize $s$ as the cone of some morphism between
$\theta_e$ and $\theta_s$. And it turns out that there is a very
natural candidate for this morphism.

By \eqref{eq591}, the functor $\theta_s$ is the composition of 
two functors, both left and right adjoint to each other. Let 
$\mathrm{adj}_s:\theta_e\to \theta_s$ and 
$\overline{\mathrm{adj}}_s:\theta_s\to \theta_e$ 
be the corresponding adjunction
morphisms. Denote by $\mathrm{C}_s$ the cokernel of $\mathrm{adj}_s$
and by $\mathrm{K}_s$ the kernel of 
$\overline{\mathrm{adj}}_s$. The functor
$\mathrm{C}_s$ is called the {\em shuffling functor}  
\index{shuffling functor}\index{coshuffling functor}
and the functor $\mathrm{K}_s$ is called the {\em coshuffling functor}. 
These functors appeared first in \cite{Ca} and then were studied
in more details in \cite{Ir} and further in \cite{MS0}.

Being the cokernel of a natural transformation between two exact functors,
the functor $\mathrm{C}_s$ is only right exact (and, similarly, the 
functor $\mathrm{K}_s$ is only left exact). Hence, for the  induced action of 
these functors on the level of the Grothendieck group to make sense, 
we will be forced to consider their derived  versions acting 
on some derived version of  category $\mathcal{O}$.
Here are some basic properties of (co)shuffling functors and their derived
versions (see e.g. \cite{MS0}):

\begin{proposition}\label{prop508}
Let $s\in W$ be a simple reflection and $\lambda$ be dominant integral
and regular.
\begin{enumerate}[$($a$)$]
\item\label{prop508.1}  The pair $(\mathrm{C}_s,\mathrm{K}_s)$ is adjoint.
\item\label{prop508.2}  There is an isomorphism 
$\mathrm{K}_s\cong \star\circ \mathrm{C}_s\circ \star$.
\item\label{prop508.25} For $w\in W$ such that $ws>w$ we have
$\mathrm{C}_s\Delta(w\cdot\lambda)\cong\Delta(ws\cdot\lambda)$.
\item\label{prop508.26} For $w\in W$ such that $ws<w$ there is 
an exact sequence
\begin{displaymath}
0\to \mathrm{Coker}_w\to \mathrm{C}_s\Delta(w\cdot\lambda)
\to \Delta(w\cdot\lambda)\to 0,
\end{displaymath}
where $\mathrm{Coker}_w$ denotes the cokernel of the natural inclusion
$\Delta(w\cdot\lambda)\hookrightarrow \Delta(ws\cdot\lambda)$.
\item\label{prop508.3} We have 
$\mathcal{L}_i\mathrm{C}_s\Delta(w\cdot\lambda)=0$ for all $w\in W$ and 
all $i>0$, in particular, the restriction of $\mathrm{C}_s$ to 
$\mathcal{F}(\Delta)$ is exact.
\item\label{prop508.4}
$\mathcal{L}_i\mathrm{C}_s=
\begin{cases}\mathrm{Ker}(\mathrm{adj}_s), & i=1;\\0,& i>1.\end{cases}$
\item\label{prop508.5} $\mathcal{L}\mathrm{C}_s$ is an autoequivalence of
$\mathcal{D}^b(\mathcal{O}_{\lambda})$ with inverse 
$\mathcal{R}\mathrm{K}_s$.
\end{enumerate}
\end{proposition}

Combining Proposition~\ref{prop508}\eqref{prop508.25}-\eqref{prop508.26}
we have the following corollary, which says that $\mathrm{C}_s$ is
a ``good'' candidate for the role of categorification of $s$.

\begin{corollary}\label{cor509}
For any $w\in W$ and any simple reflection $s\in W$ we have the following
equality: $[\mathrm{C}_s\Delta(w\cdot\lambda)]=[\Delta(ws\cdot\lambda)]$.
\end{corollary}

Since $\mathrm{C}_s$ is only right exact, to be able to categorify anything
using it we have to consider the endofunctor $\mathcal{L}\mathrm{C}_s$
of $\mathcal{D}^b(\mathcal{O}_{\lambda})$. The latter is invertible by 
Proposition~\ref{prop508}\eqref{prop508.5}. Unfortunately, from
Proposition~\ref{prop508}\eqref{prop508.4} it follows that 
$\mathcal{L}\mathrm{C}_s\circ \mathcal{L}\mathrm{C}_s\not\cong\mathrm{Id}$.
Therefore the action of  $\mathcal{L}\mathrm{C}_s$
of $\mathcal{D}^b(\mathcal{O}_{\lambda})$ cannot be used to categorify
$\mathbb{S}_n$. However, we have the following:

\begin{proposition}\label{prop510}
Let $s,t\in W$ be simple reflections. 
\begin{enumerate}[$($a$)$]
\item\label{prop510.1} If $st=ts$, then we have both
$\mathrm{C}_s\circ\mathrm{C}_t\cong \mathrm{C}_t\circ\mathrm{C}_s$
and $\mathcal{L}\mathrm{C}_s\circ\mathcal{L}\mathrm{C}_t
\cong \mathcal{L}\mathrm{C}_t\circ\mathcal{L}\mathrm{C}_s$.
\item\label{prop510.2} If $sts=sts$, then we have both
$\mathrm{C}_s\circ\mathrm{C}_t\circ\mathrm{C}_s
\cong \mathrm{C}_t\circ\mathrm{C}_s\circ\mathrm{C}_t$
and $\mathcal{L}\mathrm{C}_s\circ\mathcal{L}\mathrm{C}_t
\circ\mathcal{L}\mathrm{C}_s
\cong \mathcal{L}\mathrm{C}_t\circ\mathcal{L}\mathrm{C}_s
\circ\mathcal{L}\mathrm{C}_t$.
\end{enumerate}
\end{proposition}

Proposition~\ref{prop510} says that shuffling functors satisfy braid
relations and hence define a weak action of the braid group $\mathbb{B}_n$ 
on  $\mathcal{D}^b(\mathcal{O}_{\lambda})$. From Corollary~\ref{cor509}
it follows that the decategorification of this action gives a
representation of $\mathbb{S}_n$ (which is isomorphic to the right 
regular representation of $\mathbb{S}_n$ by Proposition~\ref{prop521}).
This action can be understood in terms of $2$-categories. 
Rouquier defines in \cite{Ro0} a $2$-category $\mathfrak{B}_{n}$
whose decategorification is isomorphic to $\mathbb{B}_n$
and shows that the above action of shuffling functors on 
$\mathcal{D}^b(\mathcal{O}_{\lambda})$ can be considered as a 
$2$-representation of $\mathfrak{B}_{n}$.

\subsection{Singular braid monoid}\label{s5.5}

To this end we have both projective and (co)shuffling functors
acting on $\mathcal{D}^b(\mathcal{O}_{\lambda})$. The former
categorify an action of $\mathbb{Z}[\mathbb{S}_n]$ and the latter
categorify an action of $\mathbb{B}_n$. A natural question is:
What kind of object is categorified if we consider all these functors
together? It turns out that the answer is: the singular braid
monoid $\mathcal{S}\mathbb{B}_n$. This monoid is defined in
\cite{Ba,Bi} as the monoid generated by the following elements:
\begin{displaymath}
\xymatrix@!=0.1pc{
&\text{\tiny 1}\ar@{-}[dd]& & \text{\tiny i-1}\ar@{-}[dd] & \text{\tiny i}\ar@{-}[ddrr]&&
\text{\tiny i+1}\ar@{-}'[dl][ddll] &  \text{\tiny i+2}\ar@{-}[dd] & & 
\text{\tiny n}\ar@{-}[dd]\\
\sigma_{\text{\tiny i}}=&&\dots&&&&&&\dots&\\
&\text{\tiny 1}&&\text{\tiny i-1}&\text{\tiny i}&&\text{\tiny i+1}
&\text{\tiny i+2}&&\text{\tiny n}
}
\end{displaymath}
\begin{displaymath}
\xymatrix@!=0.1pc{
&\text{\tiny 1}\ar@{-}[dd]& & \text{\tiny i-1}\ar@{-}[dd] & \text{\tiny i}\ar@{-}'[dr][ddrr]&&
\text{\tiny i+1}\ar@{-}[ddll] &  \text{\tiny i+2}\ar@{-}[dd] & & 
\text{\tiny n}\ar@{-}[dd]\\
\sigma_{\text{\tiny i}}^{-1}=&&\dots&&&&&&\dots&\\
&\text{\tiny 1}&&\text{\tiny i-1}&\text{\tiny i}&&\text{\tiny i+1}
&\text{\tiny i+2}&&\text{\tiny n}
}
\end{displaymath}
\begin{displaymath}
\xymatrix@!=0.1pc{
&\text{\tiny 1}\ar@{-}[dd]& & \text{\tiny i-1}\ar@{-}[dd] & \text{\tiny i}\ar@{-}[ddrr]&&
\text{\tiny i+1}\ar@{-}[ddll] &  \text{\tiny i+2}\ar@{-}[dd] & & 
\text{\tiny n}\ar@{-}[dd]\\
\tau_{\text{\tiny i}}=&&\dots&&&\bullet&&&\dots&\\
&\text{\tiny 1}&&\text{\tiny i-1}&\text{\tiny i}&&\text{\tiny i+1}
&\text{\tiny i+2}&&\text{\tiny n}
}
\end{displaymath}
These generators satisfy the following relations:
\begin{displaymath} 
\begin{array}{cc}
\sigma_i\sigma_i^{-1}=\sigma_i^{-1}\sigma_i=1, &\mbox{ for all } i;\\
\sigma_i\sigma_{i+1}\sigma_i=\sigma_{i+1}\sigma_i\sigma_{i+1},&
\mbox{ for all } i;\\
\sigma_i\sigma_j=\sigma_j\sigma_i,&\mbox{  if } |i-j|>1;\\        
\tau_{i}\sigma_j\sigma_{i}=\sigma_j\sigma_{i}\tau_j,&
\mbox{  if   } |i-j|=1;\\     
\sigma_i\tau_j=\tau_j\sigma_i,&\mbox{  if }|i-j|\not=1;\\      
\tau_i\tau_j=\tau_j\tau_i,&\mbox{  if }|i-j|>1.
\end{array}
\end{displaymath}
The following theorem is proved in \cite{MS01}:

\begin{theorem}[Weak categorification of $\mathcal{S}\mathbb{B}_n$]
\label{thm542}
Let $\lambda$ be regular, integral and dominant. Then, mapping 
$\sigma_i\mapsto \mathcal{L}\mathrm{C}_{s_i}$,
$\sigma_i^{-1}\mapsto \mathcal{R}\mathrm{K}_{s_i}$ and
$\tau_i\mapsto\theta_{s_i}$, for $i=1,2,\dots,n-1$, defines a weak action 
of $\mathcal{S}\mathbb{B}_n$ on $\mathcal{D}^b(\mathcal{O}_{\lambda})$.
\end{theorem}

\subsection{$\mathfrak{gl}_2$-example}\label{s5.6}

The action of projective, (co)shuffling and the corresponding
homology functors on the principal block of $\mathfrak{gl}_2$ is 
given by the following table (notation as in Subsection~\ref{s4.7}):
\begin{displaymath}
\begin{array}{|c||c|c|c|c|c|}
\hline
M&
\begin{array}{c}e\\\\\end{array}&\begin{array}{c}s\\\\\end{array}&
\begin{array}{c}e\\s\\\end{array}&\begin{array}{c}s\\e\\\end{array}&
\begin{array}{c}s\\e\\s\end{array}\\
\hline\hline
\theta_s\, M&
\begin{array}{c}0\end{array}&\begin{array}{c}s\\e\\s\end{array}&
\begin{array}{c}s\\e\\s\end{array}&\begin{array}{c}s\\e\\s\end{array}&
\begin{array}{ccc}s&&s\\e&\oplus&e\\s&&s\end{array}\\
\hline
\mathrm{C}_s\, M&
\begin{array}{c}0\end{array}&\begin{array}{c}s\\e\\\end{array}&
\begin{array}{c}s\\\\\end{array}&\begin{array}{c}s\\e\\\end{array}&
\begin{array}{ccc}s\\e\\s\end{array}\\
\hline
\mathrm{K}_s\, M&
\begin{array}{c}0\end{array}&\begin{array}{c}e\\s\\\end{array}&
\begin{array}{c}e\\s\\\end{array}&\begin{array}{c}s\\\\\end{array}&
\begin{array}{ccc}s\\e\\s\end{array}\\
\hline
\mathcal{L}_1\mathrm{C}_s\, M&
\begin{array}{c}e\end{array}&\begin{array}{c}0\end{array}&
\begin{array}{c}0\end{array}&\begin{array}{c}e\end{array}&
\begin{array}{ccc}0\end{array}\\
\hline
\mathcal{R}_1\mathrm{K}_s\, M&
\begin{array}{c}e\end{array}&\begin{array}{c}0\end{array}&
\begin{array}{c}e\end{array}&\begin{array}{c}0\end{array}&
\begin{array}{ccc}0\end{array}\\
\hline
\end{array}
\end{displaymath}

\section{Category $\mathcal{O}$: twisting and completion}\label{s6}

\subsection{Zuckerman functors}\label{s6.1}

As mentioned in Subsection~\ref{s4.6}, for every parabolic 
$\mathfrak{p}\subset\mathfrak{g}$ we have the corresponding Zuckerman functor
$\mathrm{Z}_{\mathfrak{p}}:\mathcal{O}\to\mathcal{O}^{\mathfrak{p}}$
and dual Zuckerman functor
$\hat{\mathrm{Z}}_{\mathfrak{p}}:\mathcal{O}\to\mathcal{O}^{\mathfrak{p}}$,
which are the left and the right adjoints of the natural inclusion
$\mathcal{O}^{\mathfrak{p}}\hookrightarrow \mathcal{O}$. It is easy to
see that $\hat{\mathrm{Z}}_{\mathfrak{p}}\cong
\star\circ\mathrm{Z}_{\mathfrak{p}}\circ\star$. As
every $\mathcal{O}^{\mathfrak{p}}$ is stable under the action of
projective functors, we have the following:

\begin{proposition}\label{prop671}
Both $\mathrm{Z}_{\mathfrak{p}}$, $\hat{\mathrm{Z}}_{\mathfrak{p}}$ and
the corresponding derived functors commute with all projective
functors.
\end{proposition}

If the semi-simple part of the Levi factor of $\mathfrak{p}$ is isomorphic
to $\mathfrak{sl}_2$ and $s$ is the corresponding simple reflection,
we will denote the corresponding Zuckerman functors simply
by $\mathrm{Z}_s$ and $\hat{\mathrm{Z}}_s$, respectively.
For the derived versions of Zuckerman functors we have:

\begin{proposition}[\cite{EW}]\label{prop672}
Let $s$ be a simple reflection.
\begin{enumerate}[$($a$)$]
\item\label{prop672.1} We have $\mathcal{L}_i\mathrm{Z}_s=0$, $i>2$.
\item\label{prop672.2} We have $\mathcal{L}_2\mathrm{Z}_s=
\hat{\mathrm{Z}}_s$.
\item\label{prop672.3} For the functor $\mathrm{Q}_s:=
\mathcal{L}_1\mathrm{Z}_s$ we have $\mathrm{Q}_s\cong\star\circ
\mathrm{Q}_s\circ\star$.
\item\label{prop672.4} We have $\mathcal{L}\mathrm{Z}_s[-1]\cong
\mathcal{R}\hat{\mathrm{Z}}_s[1]$. 
\end{enumerate}
\end{proposition}

\begin{example}\label{exm673}
{\rm The action of $\mathrm{Z}_s$, $\hat{\mathrm{Z}}_s$ and
$\mathrm{Q}_s$ on the principal block of the category $\mathcal{O}$ 
for the algebra $\mathfrak{gl}_2$ is given in the following table:
\begin{displaymath}
\begin{array}{|c||c|c|c|c|c|}
\hline
M&
\begin{array}{c}e\\\\\end{array}&\begin{array}{c}s\\\\\end{array}&
\begin{array}{c}e\\s\\\end{array}&\begin{array}{c}s\\e\\\end{array}&
\begin{array}{c}s\\e\\s\end{array}\\
\hline\hline
\mathrm{Z}_s\, M&
\begin{array}{c}e\end{array}&\begin{array}{c}0\end{array}&
\begin{array}{c}e\end{array}&\begin{array}{c}0\end{array}&
\begin{array}{c}0\end{array}\\
\hline
\hat{\mathrm{Z}}_s\, M&
\begin{array}{c}e\end{array}&\begin{array}{c}0\end{array}&
\begin{array}{c}0\end{array}&\begin{array}{c}e\end{array}&
\begin{array}{ccc}0\end{array}\\
\hline
\mathrm{Q}_s\, M&
\begin{array}{c}0\end{array}&\begin{array}{c}e\end{array}&
\begin{array}{c}0\end{array}&\begin{array}{c}0\end{array}&
\begin{array}{ccc}0\end{array}\\
\hline
\end{array}
\end{displaymath}
}
\end{example}

\subsection{Twisting functors}\label{s6.2}

For $i=1,2,\dots,n-1$ denote by $U_i$ the localization of the algebra
$U(\mathfrak{g})$ with respect to the multiplicative set
$\{e_{i+1i}^k:k\in\mathbb{N}\}$. If $\{x_1=e_{i+1i},x_2,\dots,x_k\}$
is a basis of $\mathfrak{g}$, then $U_i$ has a {\em PBW-type}
\index{PBW-type basis}
basis consisting of all monomials of the form $x_1^{m_1}x_2^{m_2}\cdots
x_k^{m_k}$, where $m_1\in\mathbb{Z}$ and $m_2,\dots,m_k\in
\mathbb{N}_0:=\mathbb{N}\cup
\{0\}$. In particular, $U(\mathfrak{g})$ is a subalgebra of $U_i$ is 
the natural way. Denote by $\mathtt{B}_i$ the 
$U(\mathfrak{g})\text{-}U(\mathfrak{g})$ bimodule
$U_i/U(\mathfrak{g})$. Let $\alpha_i$ be an automorphism of 
$\mathfrak{g}$ corresponding to the simple reflection $s_i$.
Consider the bimodule ${}^{\alpha_i}\mathtt{B}_i$ in which the left
action is twisted by $\alpha_i$ and define the endofunctor
$\mathrm{T}_i$ of $\mathfrak{g}\text{-}\mathrm{Mod}$ as tensoring with 
${}^{\alpha_i}\mathtt{B}_i$. The functor $\mathrm{T}_i$ is called
{\em twisting functor}. Twisting functors appeared first in \cite{Ar}, 
\index{twisting functor}
were used in \cite{So,Ar2}, in particular, to prove Ringel self-duality 
of $\mathcal{O}$, and were studied in \cite{AS} and \cite{KhMa} in detail.

The action of $\mathrm{T}_i$ on a module $M$ can be understood as 
the composition  of the following three operations:
\begin{itemize}
\item invert the action of the element $e_{i+1i}$, obtaining the
localized module $M_1$;
\item factor out the canonical image of $M$ in $M_1$, obtaining the
quotient $M_2$;
\item twist the action of $U(\mathfrak{g})$ on $M_2$ by 
$\alpha_i$, obtaining $\mathrm{T}_i\, M$.
\end{itemize}
Twisting functors have the following basic properties (see \cite{AS}):

\begin{proposition}\label{prop601}
Let $i\in\{1,2,\dots,n-1\}$.
\begin{enumerate}[$($a$)$]
\item\label{prop601.1} The functor $\mathrm{T}_i$ preserves
both $\mathcal{O}$ and $\mathcal{O}_{\lambda}$ for any 
$\lambda\in\mathfrak{h}^*_{\mathrm{dom}}$.
\item\label{prop601.2} The functor $\mathrm{T}_i$ is right exact.
\item\label{prop601.3} The functor $\mathrm{T}_i$ is left adjoint
to the functor $\mathrm{G}_i:=\star\circ \mathrm{T}_i\circ\star$.
\item\label{prop601.4} For any finite-dimensional $\mathfrak{g}$-module
$V$ the functors $V\otimes_{\mathbb{C}}{}_-$ and $\mathrm{T}_i$ commute.
In particular, $\mathrm{T}_i$ commutes will all projective functors.
\end{enumerate}
\end{proposition}

It turns out that twisting functors satisfy braid relations
(see e.g. \cite{KhMa}):

\begin{proposition}\label{prop602}
\begin{enumerate}[$($a$)$]
\item\label{prop602.1} For $i,j\in\{1,2,\dots,n-1\}$, $i\neq j\pm 1$, 
we have $\mathrm{T}_i\circ\mathrm{T}_j\cong \mathrm{T}_j\circ\mathrm{T}_i$.
\item\label{prop602.2} For $i\in\{1,2,\dots,n-2\}$ we have
$\mathrm{T}_i\circ\mathrm{T}_{i+1}\circ\mathrm{T}_i
\cong \mathrm{T}_{i+1}\circ\mathrm{T}_i\circ\mathrm{T}_{i+1}$.
\end{enumerate}
\end{proposition}

For $w\in W$ choose a reduced decomposition 
$w=s_{i_1}s_{i_2}\cdots s_{i_k}$ of $w$ and define the
{\em twisting functor} $\mathrm{T}_w$ as follows:
\index{twisting functor}
$\mathrm{T}_w:=\mathrm{T}_{i_1}\circ\mathrm{T}_{i_2}\circ
\cdots\circ\mathrm{T}_{i_k}$. From Proposition~\ref{prop602} 
it follows that $\mathrm{T}_w$ does not depend (up to isomorphism)
on the choice of the reduced decomposition. In the following
claim we collected some combinatorial properties of twisting functors
(see \cite{AS}):

\begin{proposition}\label{prop603}
\begin{enumerate}[$($a$)$]
\item\label{prop603.1} For any $\mu\in\mathfrak{h}^*$ and any
$s\in S$ we have
\begin{displaymath}
\mathrm{T}_s\,\nabla(\mu)=
\begin{cases}
\nabla(s\cdot \mu), & \mu< s\cdot\mu;\\
\nabla(\mu), &   \text{otherwise}.
\end{cases}
\end{displaymath}
\item\label{prop603.2} For all $\lambda\in\mathfrak{h}^*_{\mathrm{dom}}$
and $w\in W$ we have 
$[\mathrm{T}_w\,\Delta(\lambda)]=[\Delta(w\cdot\lambda)]$.
\end{enumerate}
\end{proposition}

The most important properties of twisting functors are the
ones connected to their extensions to the derived category
(see \cite{AS,MS01}):

\begin{proposition}\label{prop604}
Let $s\in S$  and $w\in W$.
\begin{enumerate}[$($a$)$]
\item\label{prop604.1} We have $\mathcal{L}_i\mathrm{T}_w\Delta(\mu)=0$ 
for all $\mu\in\mathfrak{h}^*$ and $i>0$.
\item\label{prop604.2} We have $\mathcal{L}_i\mathrm{T}_s=0$ for all
$i>1$.
\item\label{prop604.3} We have $\mathcal{L}_1\mathrm{T}_s=
\hat{\mathrm{Z}}_s$.
\item\label{prop604.4} If $x,y\in W$ are such that $\mathfrak{l}(xy)=
\mathfrak{l}(x)+\mathfrak{l}(y)$, then 
$\mathcal{L}\mathrm{T}_{xy}\cong\mathcal{L}\mathrm{T}_x\circ
\mathcal{L}\mathrm{T}_y$.
\item\label{prop604.5} The functor $\mathcal{L}\mathrm{T}_s$
is a self-equivalence of $\mathcal{D}^b(\mathcal{O})$ with inverse
$\mathcal{R}\mathrm{G}_s$.
\end{enumerate}
\end{proposition}

Another important property of twisting functors is the following
(see \cite{AS,KhMa,MS01}):

\begin{proposition}\label{prop605}
For any simple reflection $s$ there is an exact sequence of functors 
as follows: $\mathrm{Q}_s\hookrightarrow \mathrm{T}_s\to\mathrm{Id}$.
\end{proposition}

Combining Proposition~\ref{prop602} and 
Proposition~\ref{prop604}\eqref{prop604.4}, we obtain that 
the functors $\mathcal{L}\mathrm{T}_i$, $i=1,2,\dots,n-1$, satisfy
braid relations and hence define a weak action of the braid group on
$\mathcal{D}^b(\mathcal{O})$. Decategorifying this action we obtain:

\begin{proposition}[Na{\"\i}ve categorification of the left regular
$\mathbb{S}_n$-module]\label{prop606}
Let $\lambda\in\mathfrak{h}^*$ be dominant, integral and regular.
Then there is a unique isomorphism $\overline{\varphi}:
\mathbb{Z}[\mathbb{S}_n]\to [\mathcal{D}^b(\mathcal{O}_{\lambda})]$
such that $\overline{\varphi}(w)=[\Delta(w\cdot\lambda)]$
for any $w\in \mathbb{S}_n$. Furthermore,
for any $w\in \mathbb{S}_n$ the following 
diagram commutes:
\begin{displaymath}
\xymatrix{ 
\mathbb{Z}[\mathbb{S}_n]\ar[rr]^{w\cdot}\ar[d]_{\overline{\varphi}}
&&\mathbb{Z}[\mathbb{S}_n]\ar[d]_{\overline{\varphi}} \\
[\mathcal{D}^b(\mathcal{O}_{\lambda})]\ar[rr]^{[\mathrm{T}_w]\cdot}
&&[\mathcal{D}^b(\mathcal{O}_{\lambda})]
},
\end{displaymath}
where $\varphi$ is as in Proposition~\ref{prop521}.
\end{proposition}

The above categorification is only a na{\"\i}ve one since
$\mathcal{L}\mathrm{T}_i\circ \mathcal{L}\mathrm{T}_i\neq \mathrm{Id}$
(for example because of Propositions~\ref{prop604}\eqref{prop604.3}).
Combining Propositions~\ref{prop521} and \ref{prop606} with
Proposition~\ref{prop601}\eqref{prop601.4} we obtain a na{\"\i}ve
categorification of the regular $\mathbb{Z}[\mathbb{S}_n]$-bimodule.

\begin{example}\label{exm607}
{\rm The action of $\mathrm{T}_s$ on the principal block of 
the category $\mathcal{O}$  for the algebra $\mathfrak{gl}_2$ 
is given by the following table:
\begin{displaymath}
\begin{array}{|c||c|c|c|c|c|}
\hline
M&
\begin{array}{c}e\\\\\end{array}&\begin{array}{c}s\\\\\end{array}&
\begin{array}{c}e\\s\\\end{array}&\begin{array}{c}s\\e\\\end{array}&
\begin{array}{c}s\\e\\s\end{array}\\
\hline\hline
\mathrm{T}_s\, M&
\begin{array}{c}0\end{array}&\begin{array}{c}s\\e\\\end{array}&
\begin{array}{c}s\\\\\end{array}&\begin{array}{c}s\\e\\\end{array}&
\begin{array}{c}s\\e\\s\end{array}\\
\hline
\end{array}
\end{displaymath}
}
\end{example}

\subsection{Completion functors}\label{s6.3}

Let $\lambda'$ be integral and dominant, and $\lambda$ be integral, 
dominant and regular. According to Theorem~\ref{thm507}, the category
$\mathcal{O}_{\lambda'}$ is equivalent to the category
${}_{\,\lambda'}^{\infty}\mathcal{H}_{\lambda}^1$ of 
Harish-Chandra bimodules. The latter is a subcategory of the category
${}_{\,\lambda'}^{\infty}\mathcal{H}_{\lambda}^{\infty}$. The
category ${}_{\,\lambda'}^{\infty}\mathcal{H}_{\lambda}^{\infty}$
is very symmetric. In particular, it has both the usual (left)
projective functors $\theta_{\lambda',\mu}^l$, as defined in 
Subsection~\ref{s5.1}, and the right projective functors 
$\theta_{\lambda,\mu}^r$, defined similarly using tensoring of
Harish-Chandra bimodules with finite dimensional 
$\mathfrak{g}$-modules on the right. Since $\lambda$ is assumed to 
be regular, we can index indecomposable right projective functors
on ${}_{\,\lambda'}^{\infty}\mathcal{H}_{\lambda}^{\infty}$ by
elements from $W$ and denote them by $\theta_w^r$, $w\in W$.
Unfortunately, the functors $\theta_w^r$ do not preserve
the subcategory ${}_{\,\lambda'}^{\infty}\mathcal{H}_{\lambda}^1$
in general. However, similarly to the left projective functors, 
for every simple reflection $s$ there are adjunction morphisms
\begin{displaymath}
\mathrm{adj}_s^r:\theta_e^r\to\theta_s^r;\qquad 
\overline{\mathrm{adj}}_s^r:\theta_s^r\to\theta_e^r.
\end{displaymath}
Let $\mathrm{C}_s^r$ and $\mathrm{K}_s^r$ denote the cokernel of the
first morphism and the kernel of the second morphism, respectively.
Similarly to Proposition~\ref{prop508}\eqref{prop508.1}, we have
the pair $(\mathrm{C}_s^r,\mathrm{K}_s^r)$ of adjoint functors on
${}_{\,\lambda'}^{\infty}\mathcal{H}_{\lambda}^{\infty}$.

It is easy to show that both $\mathrm{C}_s^r$ and $\mathrm{K}_s^r$
preserve ${}_{\,\lambda'}^{\infty}\mathcal{H}_{\lambda}^1$ and hence
we can restrict the adjoint pair $(\mathrm{C}_s^r,\mathrm{K}_s^r)$
to ${}_{\,\lambda'}^{\infty}\mathcal{H}_{\lambda}^1$. Via the
equivalence from Theorem~\ref{thm507} this defines a pair of
adjoint functors on $\mathcal{O}_{\lambda'}$. Let us denote by
$\mathrm{G}_s$ the endofunctor of $\mathcal{O}_{\lambda'}$, which 
corresponds to $\mathrm{K}_s^r$ under the equivalence of 
Theorem~\ref{thm507}. The functor $\mathrm{G}_s$ is called
(Joseph's version of) {\em completion functor} and was defined in
\index{completion functor}
the above form in \cite{Jo}. The definition was inspired by another 
(Enright's version of) completion functor, defined
earlier  in \cite{En} in a more restrictive situation.
The name {\em completion functor} is motivated by the property, dual 
to Proposition~\ref{prop603}\eqref{prop603.1}, which can be understood
in the way that all Verma modules can be obtained from the antidominant
Verma module by a certain completion process with respect to
the action of various $\mathfrak{sl}_2$-subalgebras.
Connection between completion functors and twisting functors is
clarified by the following:

\begin{theorem}[\cite{KhMa}]\label{thm608}
The functor $\mathrm{G}_s$ is right adjoint to the functor 
$\mathrm{T}_s$.
\end{theorem}

In particular, it follows that $\mathrm{T}_s$ corresponds, 
under the equivalence of Theorem~\ref{thm507}, to $\mathrm{C}_s^r$.
From the previous subsection we obtain that $\mathrm{G}_s\cong
\star\circ\mathrm{T}_s\circ\star$ and that completion functors
satisfy braid relations. This allows us to define the completion
functor $\mathrm{G}_w$ for every $w\in W$. The action of
derived completion functors on the bounded derived category of
$\mathcal{O}_{\lambda}$ gives a na{\"\i}ve categorification of the
left regular $\mathbb{Z}[\mathbb{S}_n]$-module.

\begin{example}\label{exm609}
{\rm The action of $\mathrm{G}_s$ on the principal block of 
the category $\mathcal{O}$  for the algebra $\mathfrak{gl}_2$ 
is given in the following table:
\begin{displaymath}
\begin{array}{|c||c|c|c|c|c|}
\hline
M&
\begin{array}{c}e\\\\\end{array}&\begin{array}{c}s\\\\\end{array}&
\begin{array}{c}e\\s\\\end{array}&\begin{array}{c}s\\e\\\end{array}&
\begin{array}{c}s\\e\\s\end{array}\\
\hline\hline
\mathrm{G}_s\, M&
\begin{array}{c}0\end{array}&\begin{array}{c}e\\s\\\end{array}&
\begin{array}{c}e\\s\\\end{array}&\begin{array}{c}s\\\\\end{array}&
\begin{array}{c}s\\e\\s\end{array}\\
\hline
\end{array}
\end{displaymath}
}
\end{example}

\subsection{Alternative description via (co)approximations}\label{s6.4}

Let $A$ be a finite dimensional associative $\Bbbk$-algebra and
$e\in A$ an idempotent. Then we have the following natural
pair of adjoint functors:
\begin{equation}\label{eq630}
\xymatrix{ 
A\text{-}\mathrm{mod}\ar@/^/[rrrr]^{\mathrm{Hom}_A(Ae,{}_-)}&&&& 
eAe\text{-}\mathrm{mod}\ar@/^/[llll]^{Ae\otimes_{eAe}{}_-}
}
\end{equation}
Denote by $\mathcal{X}_e$ the full subcategory of $A\text{-}\mathrm{mod}$,
which consists of all modules $M$ admitting a presentation
$X_1\to X_0\tto M$ with $X_0,X_1\in\mathrm{add}(Ae)$. Then \eqref{eq630}
restricts to the following equivalence of categories, see 
\cite[Section~5]{Au}:
\begin{displaymath}
\xymatrix{ 
\mathcal{X}_e\ar@/^/[rrrr]^{\mathrm{Hom}_A(Ae,{}_-)}&&&& 
eAe\text{-}\mathrm{mod}\ar@/^/[llll]^{Ae\otimes_{eAe}{}_-}.
}
\end{displaymath}
As mentioned in Subsection~\ref{s1.6}, the category $\mathcal{X}_e$
is equivalent to the quotient of $A\text{-}\mathrm{mod}$ by the Serre
subcategory consisting of all modules $M$ such that $eM=0$. Set
\begin{displaymath}
\mathrm{P}_e:=Ae\otimes_{eAe}\mathrm{Hom}_A(Ae,{}_-):
A\text{-}\mathrm{mod}\to A\text{-}\mathrm{mod}.
\end{displaymath}
This functor is right exact (since $\mathrm{Hom}_A(Ae,{}_-)$ is
exact because of the projectivity of $Ae$) and hence is isomorphic
to the functor $Ae\otimes_{eAe}eA\otimes_{A}{}_-$. We also
have the right exact functor $\mathrm{R}_e:=AeA\otimes_{A}{}_-$.
The functor $\mathrm{P}_e$ is called {\em coapproximation}
\index{coapproximation}
with respect to $Ae$ and the functor $\mathrm{R}_e$ is called 
{\em partial coapproximation} with respect to $Ae$. We also have
the corresponding right adjoint functors $\overline{\mathrm{P}}_e$
and $\overline{\mathrm{R}}_e$ of approximation and partial
approximation with respect to the injective module
$\mathrm{Hom}_A(eA,\Bbbk)$, respectively.

The functor $\mathrm{R}_e$ can be understood as the composition of
the following two steps (see \cite{KhMa}):
\begin{itemize}
\item Given an $A$-module $M$ consider some projective cover
$P_M\tto M$ of $M$, and denote by $M'$ the quotient of $P_M$
modulo the trace of $Ae$ in the kernel of this projective cover.
\item Consider the trace $M''$ of $Ae$ in $M'$. We have
$M''\cong\mathrm{R}_e\, M$.
\end{itemize}
Though none of these two steps is functorial, their composition
turns out to be functorial, in particular, independent of the
choice of $P_M$. The functor $\overline{\mathrm{R}}_e$ can be described
dually using injective modules.

Let now $\lambda$ be integral and dominant. For $\mu\in W\cdot \lambda$
let $e_{\mu}$ denote the primitive idempotent of $B_{\lambda}$
corresponding to $P(\mu)$. For a simple reflection $s$ let
$e_s$ denote the sum of all $e_{\mu}$ such that 
$\mu\geq s\cdot\mu$.

\begin{theorem}[\cite{KhMa,MS01}]\label{thm612}
Let $s$ be a simple reflection. Then we have:
\begin{displaymath}
\mathrm{T}_s\cong \mathrm{R}_{e_s},\quad 
\mathrm{T}_s^2\cong \mathrm{P}_{e_s},\quad 
\mathrm{G}_s\cong \overline{\mathrm{R}}_{e_s},\quad 
\mathrm{G}_s^2\cong \overline{\mathrm{P}}_{e_s}.
\end{displaymath}
\end{theorem}

Theorem~\ref{thm612} can be extended to the following situation:
let $\mathfrak{p}\subset \mathfrak{g}$ be a parabolic subalgebra
and $W_{\mathfrak{p}}$ the corresponding
parabolic subgroup of $W$. Denote by $e_{\mathfrak{p}}$ 
the sum of all $e_{\mu}$ such that 
$\mu\geq s\cdot\mu$ for any simple reflection $s\in W_{\mathfrak{p}}$.
Let $w_o^{\mathfrak{p}}$ denote the longest element of
$W_{\mathfrak{p}}$.

\begin{theorem}[\cite{KhMa,MS01}]\label{thm614}
We have:
\begin{displaymath}
\mathrm{T}_{w_o^{\mathfrak{p}}}\cong \mathrm{R}_{e_{\mathfrak{p}}},\quad 
\mathrm{T}_{w_o^{\mathfrak{p}}}^2\cong \mathrm{P}_{e_{\mathfrak{p}}},\quad 
\mathrm{G}_{w_o^{\mathfrak{p}}}\cong 
\overline{\mathrm{R}}_{e_{\mathfrak{p}}},\quad 
\mathrm{G}_{w_o^{\mathfrak{p}}}^2\cong 
\overline{\mathrm{P}}_{e_{\mathfrak{p}}}.
\end{displaymath}
\end{theorem}

\subsection{Serre functor}\label{s6.5}

Let $\cC$ be a $\Bbbk$-linear additive category with finite 
dimensional morphism spaces. A {\em Serre functor} on $\cC$ is an 
\index{Serre functor}
additive auto-equivalence $\mathrm{F}$ of $\cC$ together with isomorphisms
\begin{displaymath}
\Psi_{x,y}:\cC(x,\mathrm{F}\,y)\cong\cC(y,x)^*,
\end{displaymath}
for all $x,y\in\cC$, natural in $x$ and $y$ (see \cite{BK}). If a 
Serre functor exists, it is unique (up to isomorphism) and commutes 
with all  auto-equivalences of $\cC$. 

\begin{proposition}[\cite{Ha}]\label{prop625}
If $A$ is a finite dimensional $\Bbbk$-algebra of finite global dimension,
then the left derived of the {\em Nakayama functor}
\index{Nakayama functor}
\begin{displaymath}
\mathrm{N}:=\mathrm{Hom}_A({}_-,A)^{*}\cong A^*\otimes_A{}_- 
\end{displaymath}
is a Serre functor on $\mathcal{D}^b(A)$.
\end{proposition}

Proposition~\ref{prop625} implies existence of a Serre functor on
$\mathcal{D}^b(\mathcal{O}_{\lambda})$. It turns out that this
functor can be described in terms of both shuffling and twisting functors.

\begin{theorem}[\cite{MS15}]\label{thm627}
Let $\lambda\in\mathfrak{h}^*_{\mathrm{dom}}$ be integral and regular.
Then the Nakayama functor on $\mathcal{O}_{\lambda}$ is
isomorphic to both $\mathrm{T}^2_{w_o}$ and $\mathrm{C}^2_{w_o}$,
in particular, the latter two functors are isomorphic.
\end{theorem}

The Serre functor on $\mathcal{O}_{\lambda}$ is described geometrically
in \cite{BBM}. A description of the Serre functor in terms of 
Harish-Chandra bimodules can be found in \cite{MM1}.
As we have seen, the braid group acts on 
$\mathcal{D}^b(\mathcal{O}_{\lambda})$ both via derived twisting
and shuffling functors. The Serre functor commutes with all
auto-equivalences, hence with all functors from these actions. Note
that the Serre functor is given by $\mathrm{T}^2_{w_o}$
(or $\mathrm{C}^2_{w_o}$) and the element $w_o^2$ generates the
center of the braid group.

\section{Category $\mathcal{O}$: grading and combinatorics}\label{s7}

\subsection{Double centralizer property}\label{s7.1}

We start with the following observation:

\begin{proposition}\label{prop701}
For any integral $\mu\in\mathfrak{h}^*$ the module $P(\mu)$ is
injective, in particular, tilting, if and only if 
$\mu$ is antidominant.
\end{proposition}

\begin{proof}[Idea of the proof.]
The integral weight $-\rho$ satisfies $W_{-\rho}=W$. Hence the block
$\mathcal{O}_{-\rho}$ contains exactly one simple module $L(-\rho)$,
which is a dominant Verma module, thus projective. This implies
that $\mathcal{O}_{-\rho}$ is semi-simple and hence $L(-\rho)$ is
also injective. From Proposition~\ref{prop503}\eqref{prop503.6}
it follows that the module $\theta\,\Delta(-\rho)$ is both projective
and injective for any projective functor $\theta$, which shows that 
$P(\mu)$ is injective for all antidominant $\mu$.

On the other hand, any projective module has a Verma flag and
every Verma has simple socle, which is given by an antidominant weight.
\end{proof}

Now let $\lambda\in \mathfrak{h}^*_{\mathrm{dom}}$ be integral
and dominant. Then, by the BGG reciprocity, we 
have $\Delta(\lambda)\hookrightarrow P(w_o\cdot\lambda)$ and,
since $P(w_o\cdot\lambda)$ is a tilting module,
the cokernel of this embedding has a Verma flag. In particular,
the injective envelope of this cokernel is in 
$\mathrm{add}(P(w_o\cdot\lambda))$. This gives us an
exact sequence of the form 
\begin{displaymath}
0\to  \Delta(\lambda)\to 
P(w_o\cdot\lambda)\to X 
\end{displaymath}
with $X\in \mathrm{add}(P(w_o\cdot\lambda))$. 
Applying $\theta_w$, $w\in W$,
and summing up over all $w$ we obtain the exact sequence
\begin{equation}\label{eq791}
0\to P_{\lambda}\to X_0\to X_1,
\end{equation}
with $X_0,X_1\in \mathrm{add}(P(w_o\cdot\lambda))$. Denote by $C_{\lambda}$ 
the endomorphism algebra of $P(w_o\cdot\lambda)$. Using the standard
results on quasi-Frobenius ring (see \cite{Ta,KSX}), we obtain
the following theorem, known as Soergel's {\em Struktursatz}:
\index{Soergel's Struktursatz} 

\begin{theorem}[\cite{So0}]\label{thm702}
Let $\lambda\in \mathfrak{h}^*_{\mathrm{dom}}$ be integral and dominant.
\begin{enumerate}[$($a$)$]
\item\label{thm702.1} The functor 
\begin{displaymath}
\mathbb{V}=\mathbb{V}_{\lambda}:=
\mathrm{Hom}_{\mathfrak{g}}(P(w_o\cdot\lambda),{}_-):
\mathcal{O}_{\lambda}\to C_{\lambda}^{\mathrm{op}}\text{-}\mathrm{mod}
\end{displaymath}
is full and faithful on projective modules.
\item\label{thm702.2} The $B_{\lambda}$-module 
$\mathrm{Hom}_{\mathfrak{g}}(P_{\lambda},P(w_o\cdot\lambda))$
has the {\em double centralizer property}, that is the action of 
\index{double centralizer property}
$B_{\lambda}$ on  
$\mathrm{Hom}_{\mathfrak{g}}(P_{\lambda},P(w_o\cdot\lambda))$
coincides with the centralizer of the action of the endomorphism
algebra of this module.
\end{enumerate}
\end{theorem}

The functor $\mathbb{V}$ is called {\em Soergel's combinatorial}
functor.
\index{Soergel's combinatorial functor}

\subsection{Endomorphisms of the antidominant projective}\label{s7.2}

Theorem~\ref{thm702} reduces projective modules in the category 
$\mathcal{O}_{\lambda}$ to certain (not projective) modules over
the endomorphism algebra $C_{\lambda}$ of the unique indecomposable 
projective-injective module $P(w_o\cdot\lambda)$ in $\mathcal{O}_{\lambda}$.
It is very natural to ask what the algebra $C_{\lambda}$ is. 
The answer is given by
the following theorem. Consider the polynomial algebra 
$\mathbb{C}[\mathfrak{h}]$. Recall that $W\cong\mathbb{S}_n$ acts on 
$\mathfrak{h}$, which also induces an action of $W$ on 
$\mathbb{C}[\mathfrak{h}]$. We regard $\mathbb{C}[\mathfrak{h}]$ 
as a (positively) {\em graded} algebra 
\index{positively graded algebra}
by putting $\mathfrak{h}$ in degree $2$. We have 
the (graded) subalgebra  $\mathbb{C}[\mathfrak{h}]^{W}$ of $W$-invariants 
in $\mathbb{C}[\mathfrak{h}]$. Denote by $(\mathbb{C}[\mathfrak{h}]^{W}_+)$
the ideal of $\mathbb{C}[\mathfrak{h}]$ generated by all homogeneous
invariants of positive degree. Then we have the corresponding quotient
$\mathbb{C}[\mathfrak{h}]/(\mathbb{C}[\mathfrak{h}]^{W}_+)$,
called the {\em coinvariant} algebra. The coinvariant algebra 
\index{coinvariant algebra}
is a finite dimensional associative and commutative graded algebra 
of dimension $|W|$. The ``strange'' grading on $\mathbb{C}[\mathfrak{h}]$
(i.e. $\mathfrak{h}$ being of degree two) is motivated by the realization 
of the coinvariant algebra as the cohomology algebra of a flag variety, 
see \cite{Hi}. If $\lambda$ is singular and $W_{\lambda}$ is the 
dot-stabilizer of $\lambda$, then we have the (graded) subalgebra 
$(\mathbb{C}[\mathfrak{h}]/(\mathbb{C}[\mathfrak{h}]^{W}_+))^{W_{\lambda}}$ 
of $W_{\lambda}$-invariants in 
$\mathbb{C}[\mathfrak{h}]/(\mathbb{C}[\mathfrak{h}]^{W}_+)$. 
We have the following theorem, know as Soergel's {\em Endomorphismensatz}:
\index{Soergel's Endomorphismensatz}

\begin{theorem}[\cite{So0}]\label{thm703}
Let $\lambda\in \mathfrak{h}^*_{\mathrm{dom}}$ be integral and dominant.
\begin{enumerate}[$($a$)$]
\item\label{thm703.1} The center $Z(\mathfrak{g})$ of $U(\mathfrak{g})$
surjects onto $C_{\lambda}$, in particular, $C_{\lambda}$ is commutative.
\item\label{thm703.2} The algebras $C_{\lambda}$ and 
$(\mathbb{C}[\mathfrak{h}]/(\mathbb{C}[\mathfrak{h}]^{W}_+))^{W_{\lambda}}$
are isomorphic.
\end{enumerate}
\end{theorem}

\begin{corollary}\label{cor704}
Let $\lambda\in \mathfrak{h}^*_{\mathrm{dom}}$ be integral and dominant
and $\mathrm{Id}_{\lambda}$ the identity functor on  $\mathcal{O}_{\lambda}$.
Then the evaluation map
\begin{displaymath}
\begin{array}{rccc}
\mathrm{ev}:&\mathrm{End}(\mathrm{Id}_{\lambda})&\to &
\mathrm{End}_{\mathfrak{g}}(P(w_o\cdot\lambda))\\ 
&\varphi&\mapsto &\varphi_{P(w_o\cdot\lambda)}\\ 
\end{array}
\end{displaymath}
is an isomorphism of algebras. In particular, the center of $B_{\lambda}$
is isomorphic to $C_{\lambda}$.
\end{corollary}

\subsection{Grading on $B_{\lambda}$}\label{s7.3}

The functor $\mathbb{V}$ induces an equivalence between the additive
category of projective modules in $\mathcal{O}_{\lambda}$ and the
additive subcategory $\mathrm{add}(\mathbb{V}P_{\lambda})$
of $C_{\lambda}\text{-}\mathrm{mod}$. All projective functors are
endofunctors of the category of projective modules in $\mathcal{O}_{\lambda}$.
Hence it is natural to ask what are the images of projective functors
under $\mathbb{V}$. It turns out that these images have a very nice
description. Let $\lambda\in \mathfrak{h}^*_{\mathrm{dom}}$ be 
integral and regular and $\mu\in \mathfrak{h}^*_{\mathrm{dom}}$ 
be integral and such that $W_{\mu}=\{e,s\}$ for some 
simple reflection $s\in S$. Then $C_{\lambda}$ is the whole coinvariant
algebra for $W$ and $C_{\mu}$ is the subalgebra of $s$-invariants 
in $C_{\lambda}$. Recall from \eqref{eq591} that the projective 
endofunctor $\theta_{s}$ of $\mathcal{O}_{\lambda}$ is the composition of
$\theta_{\lambda,\mu}$ and $\theta_{\mu,s\cdot\lambda}$

\begin{theorem}[\cite{So0}]\label{thm705}
There are isomorphisms of functors as follows:
\begin{enumerate}[$($a$)$]
\item\label{thm705.1} 
$\mathbb{V}_{\mu}\circ \theta_{\lambda,\mu}\cong 
\mathrm{Res}^{C_{\lambda}}_{C_{\mu}}\circ \mathbb{V}_{\lambda}$;
\item\label{thm705.2} 
$\mathbb{V}_{\lambda}\circ \theta_{\mu,s\cdot\lambda}\cong \mathrm{Ind}^{C_{\lambda}}_{C_{\mu}}\circ \mathbb{V}_{\mu}$;
\item\label{thm705.3} 
$\mathbb{V}\circ \theta_{s}\cong C_{\lambda}\otimes_{C_{\mu}}\mathbb{V}{}_-$.
\end{enumerate}
\end{theorem}

Consider the category $C_{\lambda}\text{-}\mathrm{gmod}$ of
{\em finite dimensional graded} $C_{\lambda}$-modules
(here morphisms are homogeneous maps of degree zero).
Note that $\mathbb{V}P(\lambda)\cong\mathbb{C}$ can be considered
as a graded $C_{\lambda}$-module concentrated in degree zero. 
The subalgebra $C_{\mu}$ of $C_{\lambda}$ is graded and
$C_{\lambda}$ is free as a $C_{\mu}$-module with generators
being of degrees zero and two. It follows that 
the functor $C_{\lambda}\otimes_{C_{\mu}}{}_-$ is an exact endofunctor 
of $C_{\lambda}\text{-}\mathrm{gmod}$. Since every indecomposable
projective in $\mathcal{O}_{\lambda}$ is a direct summand
of some module of the form $\theta_{s_1}\theta_{s_2}\cdots
\theta_{s_k}\,P(\lambda)$, by standard arguments (see e.g. \cite{St2})
we obtain that every $\mathbb{V}P(w\cdot \lambda)$, $w\in W$,
can be considered as an object in $C_{\lambda}\text{-}\mathrm{gmod}$
(unique up to isomorphism and shift of grading).

We fix a grading of $\mathbb{V}P(w\cdot \lambda)$ such that its unique
simple quotient is concentrated in degree zero. This fixes a
grading on the module $\mathbb{V}P_{\lambda}$ and thus induces
a grading on its endomorphism algebra $B_{\lambda}$. We denote by
$\mathcal{O}_{\lambda}^{\mathbb{Z}}:=B_{\lambda}\text{-}\mathrm{gmod}$ 
the category of {\em finite dimensional graded} $B_{\lambda}$-modules
and  consider this category as a graded
version of the category $\mathcal{O}_{\lambda}$. Note that not
all modules in $\mathcal{O}_{\lambda}$ are gradable, see \cite{St3}.

Theorem~\ref{thm705}\eqref{thm705.1} provides a graded interpretation
of translation to the wall. This induces a canonical grading 
on $B_{\lambda}$ for any integral 
$\lambda\in \mathfrak{h}^*_{\mathrm{dom}}$. For any integral
$\lambda,\mu\in \mathfrak{h}^*_{\mathrm{dom}}$ and any
indecomposable projective functor 
$\theta_{\lambda,\nu}:\mathcal{O}_{\lambda}\to
\mathcal{O}_{\mu}$ there is a unique (up to isomorphism) 
{\em graded lift} of $\theta_{\lambda,\nu}$, that is a functor from
\index{graded lift}
$B_{\lambda}\text{-}\mathrm{gmod}\to B_{\mu}\text{-}\mathrm{gmod}$
which maps the graded module $P(\lambda)$ to the graded module 
$P(\nu)$ (note that there are no graded shifts here!) and is isomorphic 
to $\theta_{\lambda,\nu}$ after forgetting the grading.
Abusing notation, we denoted this graded lift of $\theta_{\lambda,\nu}$
also by $\theta_{\lambda,\nu}$.
All this was worked out in \cite{St3} in details.

\begin{example}\label{exm711}
{\rm  
Let $n=2$. Then the algebra $B_0$ is given by \eqref{eq457} and
it is easy to check that the above procedure results in
the following grading:
\begin{displaymath}
\begin{array}{|c||c|}
\hline
\text{degree }&\text{component } \\
i&(B_0)_i\\
\hline\hline
0& \mathtt{1}_e, \mathtt{1}_s\\
\hline
1& a, b\\
\hline
2& ba\\
\hline\end{array}
\end{displaymath}
One observes that this grading is {\em positive}, that is 
\index{positive grading}
all nonzero components have non-negative degrees, and the zero
component is semisimple.
} 
\end{example}

Abusing notation we will denote {\em standard graded lifts} in
\index{standard graded lifts}
$\mathcal{O}_{\lambda}^{\mathbb{Z}}$ of structural modules 
from $\mathcal{O}_{\lambda}$ in the same way. Thus $P(\lambda)$
is the standard graded lift described above. It has simple
top $L(\lambda)$ concentrated in degree zero. Further,
$\Delta(\lambda)$ is a graded quotient of $P(\lambda)$.
The duality $\star$ lifts to $\mathcal{O}_{\lambda}^{\mathbb{Z}}$
in the standard way such that  for 
$M\in \mathcal{O}_{\lambda}^{\mathbb{Z}}$ we have
$(M^{\star})_i=(M_{-i})^*$, $i\in\mathbb{Z}$. Using the graded version
of $\star$ we obtain standard graded lifts of $I(\lambda)$
and $\nabla(\lambda)$. The standard graded lift of $T(\lambda)$
is defined so that the unique up to scalar inclusion 
$\Delta(\lambda)\hookrightarrow T(\lambda)$ 
is homogeneous of degree zero. Note that all canonical maps between
structural modules are homogeneous of degree zero.

\begin{example}\label{exm712}
{\rm  
Let $n=2$. Here are graded filtration of all standard graded 
lifts of structural modules in $B_0\text{-}\mathrm{gmod}$
(as usual, we abbreviate $L(e):=L(0)$ by $e$ and 
$L(s):=L(s\cdot 0)$ by $s$). 
\begin{displaymath}
\begin{array}{|c||c|c|c|c|c|c|c|}
\hline
\text{degree }&\text{\tiny$L(e)$}&\text{\tiny$L(s)=T(s)=\Delta(s)=\nabla(s)$}
&\text{\tiny$\Delta(e)=P(e)$}&\text{\tiny$P(s)$}&
\text{\tiny$I(e)=\nabla(e)$}&\text{\tiny$I(s)$}&\text{\tiny$T(e)$}\\
\hline\hline
-2&&&&&&s&\\\hline
-1&&&&&s&e&s\\\hline
0&e&s&e&s&e&s&e\\\hline
1&&&s&e&&&s\\\hline
2&&&&s&&&\\
\hline\end{array}
\end{displaymath}
} 
\end{example}

After Example~\ref{exm711} it is natural to ask whether the grading on $B_{\lambda}$ will always be positive. It turns out that the
answer is ``yes'', but to motivate and explain it we would need
to ``upgrade'' the Weyl group to the Hecke algebra.

\subsection{Hecke algebra}\label{s7.4}

Denote by $\mathbb{H}=\mathbb{H}_n=\mathbb{H}(W,S)$ the 
{\it Hecke algebra} of $W$. It is defined as a free 
$\mathbb{Z}[v,v^{-1}]$-module with  the {\em standard} 
\index{standard basis}
basis $\{H_x: x\in W\}$  and multiplication given by
\begin{equation} \label{eqhecke}
H_xH_y=H_{xy}\, \text{if $\mathfrak{l}(x)+
\mathfrak{l}(y)=\mathfrak{l}(xy),\,\,$ and}\,\,
H_s^2=H_e+(v^{-1}-v)H_s\, \text{for $s\in S$}.
\end{equation}
The algebra $\mathbb{H}$ is a deformation of the group  algebra
$\mathbb{Z}[W]$. As a $\mathbb{Z}[v,v^{-1}]$-algebra, it is generated 
by $\{H_s: s\in S\}$, or (which will turn out to be more convenient) 
by the set $\{\underline{H}_s=H_s+vH_e:s\in S\}$. 

There is a unique involution ${}^-$ on $\mathbb{H}$ which maps 
$v\mapsto v^{-1}$ and  $H_s\mapsto (H_s)^{-1}$. Note that this
involution fixes all $\underline{H}_s$. More general, $\mathbb{H}$
has a unique basis, called {\em Kazhdan-Lusztig basis}, which 
\index{Kazhdan-Lusztig basis}
consists of fixed under ${}^-$ elements $\underline{H}_x$, $x\in W$,  
such that $\underline{H}_x=H_x+\sum_{y\in W,\, y\neq x}h_{y,x}H_y$,
where $h_{y,x}\in v\mathbb{Z}[v]$ (here we use the 
normalization of \cite{So3}). Set $h_{x,x}:=1$.
The polynomials $h_{y,x}$ are called
{\em Kazhdan-Lusztig polynomials} and were defined in \cite{KaLu}.
\index{Kazhdan-Lusztig polynomial}

With respect to the generators $\underline{H}_s$, $s\in S$, we have 
the following set of defining relations for $\mathbb{H}$:
\begin{eqnarray*} 
\underline{H}_s^2&=&(v+v^{-1})\underline{H}_s;\\ \nonumber
\underline{H}_s\underline{H}_t&=&\underline{H}_t\underline{H}_s,
\quad\quad\quad\quad\, \text{ if } ts=st;\\ \nonumber
\underline{H}_s\underline{H}_t\underline{H}_s+
\underline{H}_t&=&\underline{H}_t\underline{H}_s\underline{H}_t
+\underline{H}_s, \,\, \text{ if }\, tst= sts.
\end{eqnarray*}

The algebra $\mathbb{H}$ has a symmetrizing trace form
$\tau:\mathbb{H}\to \mathbb{Z}[v,v^{-1}]$, which sends 
$H_e$ to $1$ and $H_w$ to $0$ for $w\neq e$. With respect to $\tau$
we have the {\em dual Kazhdan-Lusztig basis} 
\index{dual Kazhdan-Lusztig basis}
$\{\hat{\underline{H}}_x:x\in W\}$, defined via 
$\tau(\hat{\underline{H}}_x\underline{H}_{y^{-1}})=\delta_{x,y}$.

\begin{example}\label{exm777}
{\rm 
In the case $n=3$ we have two simple reflections $s$ and $t$,
$W=\{e,s,t,st,ts,sts=tst\}$ and the elements of the
Kazhdan-Lusztig basis are given by the following table:
\begin{displaymath}
\begin{array}{cclcccl}
\underline{H}_e&=&H_e,&&\underline{H}_{st}&=&H_{st}+vH_s+vH_t+v^2H_e,\\ 
\underline{H}_s&=&H_s+vH_e,&&
\underline{H}_{ts}&=&H_{ts}+vH_s+vH_t+v^2H_e,\\ 
\underline{H}_t&=&H_t+vH_e,&&
\underline{H}_{sts}&=&H_{sts}+vH_{st}+vH_{ts}+v^2H_s+v^2H_t+v^3H_e. 
\end{array}
\end{displaymath}
The  elements of the dual Kazhdan-Lusztig basis
are given by the following table:
\begin{displaymath}
\begin{array}{cclcccl}
\hat{\underline{H}}_e&=&H_e-vH_s-vH_t+v^2H_{st}+v^2H_{ts}-v^3H_{sts},&&
\hat{\underline{H}}_{st}&=&H_{st}-vH_{sts},\\ 
\hat{\underline{H}}_s&=&H_s-vH_{st}-vH_{ts}+v^2H_{sts},&&
\hat{\underline{H}}_{ts}&=&H_{ts}-vH_{sts},\\ 
\hat{\underline{H}}_t&=&H_t-vH_{st}-vH_{ts}+v^2H_{sts},&&
\hat{\underline{H}}_{sts}&=&H_{sts}. 
\end{array}
\end{displaymath}
}
\end{example}

Let $\mathbb{F}$ be any commutative ring and 
$\iota:\mathbb{Z}[v,v^{-1}]\to\mathbb{F}$ be a homomorphism of 
unitary rings. Then we have the {\em specialized} Hecke algebra
\index{specialized Hecke algebra}
$\mathbb{H}^{(\mathbb{F},\iota)}=\mathbb{F}\otimes_{\mathbb{Z}[v,v^{-1}]}
\mathbb{H}$. Again, if $\iota$ is clear from the context (for
instance if $\iota$ is the natural inclusion), we will omit it in
the notation.

\subsection{Categorification of the right regular 
$\mathbb{H}$-module}\label{s7.5}

For any regular and integral 
$\lambda\in\mathfrak{h}^*_{\mathrm{dom}}$ we have the following:

\begin{lemma}\label{lem731}
Let $s\in S$ be a simple reflection. 
\begin{enumerate}[$($a$)$]
\item\label{lem731.1} For any $w\in W$ with $ws>w$ there is an 
exact sequence in $\mathcal{O}_{\lambda}^{\mathbb{Z}}$ as follows:
\begin{displaymath}
0\to\Delta(w\cdot\lambda)\langle -1\rangle\to \theta_s\Delta(w\cdot\lambda)
\to \Delta(ws\cdot\lambda)\to 0.
\end{displaymath}
\item\label{lem731.2} For any $w\in W$ with $ws<w$ there is an 
exact sequence in $\mathcal{O}_{\lambda}^{\mathbb{Z}}$ as follows:
\begin{displaymath}
0\to\Delta(ws\cdot\lambda)\to \theta_s\Delta(w\cdot\lambda)\to 
\Delta(w\cdot\lambda)\langle 1\rangle\to 0.
\end{displaymath}
\end{enumerate}
\end{lemma}

Lemma~\ref{lem731} implies the following ``upgrade'' of 
Proposition~\ref{prop521} to $\mathcal{O}_{\lambda}^{\mathbb{Z}}$:

\begin{proposition}[Categorification of the right regular
$\mathbb{H}$-module]\label{prop732}
Let $\lambda$ be dominant, regular and integral.
\begin{enumerate}[$($a$)$]
\item\label{prop732.1} There is a unique isomorphism of 
$\mathbb{Z}[v,v^{-1}]$-modules
$\varphi:\mathbb{H}\to [\mathcal{O}_{\lambda}^{\mathbb{Z}}]$ such that
$\varphi(H_w)=[\Delta(w\cdot\lambda)]$ for all $w\in W$.
\item\label{prop732.2} For any $s\in S$ the following 
diagram commutes:
\begin{displaymath}
\xymatrix{ 
\mathbb{H}\ar[rr]^{\cdot \underline{H}_s}\ar[d]_{\varphi}
&&\mathbb{H}\ar[d]_{\varphi} \\
[\mathcal{O}_{\lambda}^{\mathbb{Z}}]\ar[rr]^{[\theta_s]\cdot}
&&[\mathcal{O}_{\lambda}^{\mathbb{Z}}]
}
\end{displaymath}
\end{enumerate}
\end{proposition}

However, now we can say much more, namely:

\begin{theorem}[Graded reformulation of the Kazhdan-Lusztig conjecture]
\label{thm733}
For any $w\in W$ the following diagram commutes:
\begin{displaymath}
\xymatrix{ 
\mathbb{H}\ar[rr]^{\cdot \underline{H}_w}\ar[d]_{\varphi}
&&\mathbb{H}\ar[d]_{\varphi} \\
[\mathcal{O}_{\lambda}^{\mathbb{Z}}]\ar[rr]^{[\theta_w]\cdot}
&&[\mathcal{O}_{\lambda}^{\mathbb{Z}}]
}
\end{displaymath} 
\end{theorem}

Kazhdan-Lusztig conjecture was formulated in a different (equivalent)
way in \cite{KaLu} and proved in \cite{BB,BrKa}. It has many important
combinatorial consequences. The most basic ones are described in the
next subsection.

\subsection{Combinatorics of $\mathcal{O}_{\lambda}$}\label{s7.6}

Recall first that for $x\in W$ the indecomposable projective 
functor $\theta_x$ was uniquely defined by the property $\theta_x\,
\Delta(\lambda)=P(x\cdot \lambda)$. As every projective in $\mathcal{O}_{\lambda}$ has a standard filtration, going to the 
Grothendieck group, Theorem~\ref{thm733} implies that the
(graded) multiplicities of standard modules in a standard
filtration of an indecomposable projective module in 
$\mathcal{O}_{\lambda}$ are given by Kazhdan-Lusztig polynomials
as follows:

\begin{corollary}\label{cor734}
For $x\in W$ we have 
\begin{displaymath}
[P(x\cdot \lambda)]=\varphi(\underline{H}_x)=
\sum_{y\in W}h_{y,x}[\Delta(y\cdot \lambda)].
\end{displaymath}
\end{corollary}

Using the BGG-reciprocity we obtain that the (graded) composition 
multiplicities of Verma modules are also given by Kazhdan-Lusztig 
polynomials as follows:

\begin{corollary}\label{cor735}
For $y\in W$ we have 
\begin{displaymath}
[\Delta(y\cdot \lambda)]=
\sum_{x\in W}h_{y,x}[L(x\cdot \lambda)].
\end{displaymath}
\end{corollary}

We would like to categorify (i.e. present a categorical analogue for) 
the form $\tau$. Define the linear map 
$\Phi:[\mathcal{O}_{\lambda}^{\mathbb{Z}}]\to\mathbb{Z}[v,v^{-1}]$
as follows. For $M\in \mathcal{O}_{\lambda}^{\mathbb{Z}}$ set
\begin{displaymath}
\Phi([M]):=\sum_{i\in\mathbb{Z}}
\dim\mathrm{Hom}(\Delta(\lambda)\langle i\rangle,M) v^{-i}.
\end{displaymath}
This is well-defined as $\Delta(\lambda)$ is projective.
Then, for any $M\in \mathcal{O}_{\lambda}^{\mathbb{Z}}$ 
we have $\Phi([M])=\tau(\varphi^{-1}([M]))$ (this is enough to check
for $M=\Delta(x\cdot\lambda)$, $x\in W$, in which case it is clear).
The form $\Phi$ implies the following:

\begin{corollary}\label{cor736}
For $y\in W$ we have $[L(y\cdot \lambda)]=
\varphi(\hat{\underline{H}}_y)$.
\end{corollary}

\begin{proof}[Idea of the proof.]
For $x,y\in W$, using adjunction and definitions, we have 
\begin{displaymath}
\begin{array}{ccl}
\Phi([\theta_xL(y^{-1}\cdot\lambda)])&=& \displaystyle
\sum_{i\in\mathbb{Z}}
\dim\mathrm{Hom}(\Delta(\lambda)\langle i\rangle,
\theta_xL(y^{-1}\cdot\lambda)) v^{-i}\\
&=&\displaystyle\sum_{i\in\mathbb{Z}}
\dim\mathrm{Hom}(\theta_{x^{-1}}\Delta(\lambda)\langle i\rangle,
L(y^{-1}\cdot\lambda)) v^{-i}\\
&=&\displaystyle\sum_{i\in\mathbb{Z}}
\dim\mathrm{Hom}(P(x^{-1}\cdot\lambda)\langle i\rangle,
L(y^{-1}\cdot\lambda)) v^{-i}\\
&=&\delta_{x,y}
\end{array}
\end{displaymath}
and the claim follows from Corollary~\ref{cor734}.
\end{proof}

Corollary~\ref{cor736} says that, categorically, the dual Kazhdan-Lusztig
basis is the ``most natural'' basis of $\mathbb{H}$ as it corresponds
to the ``most natural'' basis of the Grothendieck group of 
$\mathcal{O}_{\lambda}^{\mathbb{Z}}$ consisting of the classes of
simple modules. Later on we will see that all ``nice'' categorifications
of $\mathbb{H}$-modules have a dual Kazhdan-Lusztig basis.
From the above we have that all composition subquotients of all
standard lifts of indecomposable projective modules 
in $\mathcal{O}_{\lambda}^{\mathbb{Z}}$ live in non-negative
degrees with only simple top being in degree zero. Hence we get:

\begin{corollary}\label{cor737}
The algebra $B_{\lambda}$ is positively graded.
\end{corollary}

\section{$\mathbb{S}_n$-categorification: Soergel bimodules, cells
and Specht modules}\label{s8}

\subsection{Soergel bimodules}\label{s8.1}

Let $\lambda\in\mathfrak{h}^*_{\mathrm{dom}}$ be integral and regular.
Then we have the (strict) monoidal category $\cS_{\lambda}$
of projective endofunctors on $\mathcal{O}_{\lambda}$. As was mentioned
in the previous section, both $\mathcal{O}_{\lambda}$ and
$\cS_{\lambda}$ admit graded lifts, which we can formalize
as follows. Denote by $\cS=\cS_n$ the $2$-category with unique object
$\mathtt{i}$ and $\cS(\mathtt{i},\mathtt{i})=\cS_{\lambda}$
(see \cite{Bac2} for a description of homomorphisms between projective
functors).
Denote also by  $\cS^{\mathbb{Z}}$ the $2$-category with unique object
$\mathtt{i}$, whose $1$-morphisms are all endofunctors of 
$\mathcal{O}_{\lambda}^{\mathbb{Z}}$ isomorphic to finite direct sums
of graded shifts of standard graded lifts of $\theta_w$, $w\in W$;
and $2$-morphisms are all natural transformations of functors which
are homogeneous of degree zero. The category $\cS$ is a fiat category
and the category $\cS^{\mathbb{Z}}$ is a $\mathbb{Z}$-cover of $\cS$
(in the sense that there is a free action of $\mathbb{Z}$ on 
$\cS^{\mathbb{Z}}$, by grading shifts, such that the quotient is
isomorphic to $\cS$ and thus endows $\cS$ with the structure of a 
graded category).

Using Soergel's combinatorial description of $\mathcal{O}_{\lambda}$,
indecomposable $1$-morphisms of $\cS$ and $\cS^{\mathbb{Z}}$ can be described,
up to isomorphism and graded shift, in the following way: Let 
$w\in W$ and $w=s_1s_2\cdots s_k$ be a reduced decomposition of $w$. 
Consider the graded $C_{\lambda}\text{-}C_{\lambda}$-bimodule
\begin{displaymath}
D_w:= 
C_{\lambda}\otimes_{C^{s_k}_{\lambda}}C_{\lambda}
\otimes_{C^{s_{k-1}}_{\lambda}}\dots
\otimes_{C^{s_2}_{\lambda}}
C_{\lambda}\otimes_{C^{s_1}_{\lambda}}C_{\lambda}
\langle\mathfrak{l}(w)\rangle. 
\end{displaymath}
Define the $C_{\lambda}\text{-}C_{\lambda}$-bimodules $\underline{D}_w$
recursively as follows: $\underline{D}_e=D_e=C_{\lambda}$; for
$\mathfrak{l}(w)>0$ let  $\underline{D}_w$ be the unique indecomposable
direct summand of $D_w$ which is not isomorphic (up to shift of grading)
to $\underline{D}_x$
for some $x$ such that $\mathfrak{l}(x)<\mathfrak{l}(w)$. 
From Subsection~\ref{s7.3} we derive that the bimodule 
$\underline{D}_w$ realizes the action of $\theta_w$ on the
level of $C_{\lambda}\text{-}\mathrm{mod}$.

\begin{proposition}\label{prop801}
For every $w\in W$ there is an isomorphism of graded functors 
as follows: $\mathbb{V}\circ\theta_w({}_-)
\cong \underline{D}_w\otimes_{C_{\lambda}}\circ\mathbb{V}({}_-)$.
\end{proposition}

Bimodules $\underline{D}_w$, $w\in W$, are called {\em Soergel
bimodules} and were introduced in \cite{So0,So4}. Usually they are defined
\index{Soergel bimodule}
in the deformed version as bimodules over the polynomial ring
$\mathbb{C}[\mathfrak{h}]$. The disadvantage of the latter definition
(which corresponds to the action of projective functors on the
category ${}_{\,\,\,0}^{\infty}\mathcal{H}_{0}^{\infty}$) is that it does 
not produce a fiat category. Therefore in this paper we will restrict
ourselves to the case of Soergel bimodules over the coinvariant
algebra. Soergel's combinatorial functor $\mathbb{V}$ provides a
biequivalence between the $2$-category $\cS^{\mathbb{Z}}$ and the
$2$-category with one object whose endomorphism category is the minimal
full fully additive subcategory of the category of 
$C_{\lambda}\text{-}C_{\lambda}$-bimodules containing all Soergel bimodules
and closed under isomorphism and shifts of grading.

\begin{theorem}[Categorification of the Hecke algebra]\label{thm802}
The map 
\begin{displaymath}
\begin{array}{ccc}
\left[\cS^{\mathbb{Z}}(\mathtt{i},\mathtt{i})\right]&\longrightarrow&\mathbb{H}\\
\left[\theta_w\right]&\mapsto&\underline{H}_w
\end{array}
\end{displaymath}
induces an anti-isomorphism of unital $\mathbb{Z}[v,v^{-1}]$-rings.
\end{theorem}

Now we would like to study $2$-representation of the 
$2$-categories $\cS$ and $\cS^{\mathbb{Z}}$. 
Note that both categories come together with the canonical
{\em natural} $2$-representation, namely the action of
$\cS$ on $\mathcal{O}_{\lambda}$ and of $\cS^{\mathbb{Z}}$
on $\mathcal{O}_{\lambda}^{\mathbb{Z}}$. The idea now is
to use these natural representations to construct other
$2$-representation.
\index{natural $2$-representation}

\subsection{Kazhdan-Lusztig cells}\label{s8.2}

The $2$-category $\cS$ is a fiat category and hence we have the
corresponding notions of left, right and two-sided cells as defined
in Subsection~\ref{s3.4}. Indecomposable objects of $\cS$ correspond
to elements of the Kazhdan-Lusztig basis in $\mathbb{H}$. To be
able to describe cells for $\cS$ we need to know structure
constants of $\mathbb{H}$ with respect to the Kazhdan-Lusztig basis.

For different $x,y\in W$ denote by $\mu(y,x)$ the coefficient of $v$ in 
the Kazhdan-Lusztig  polynomial $h_{y,x}$ from Subsection~\ref{s7.4}.
The function $\mu$ is called {\em Kazhdan-Lusztig $\mu$-function}. 
\index{Kazhdan-Lusztig $\mu$-function}
Its importance is motivated by the following (see \cite{So3}
for our normalization):

\begin{proposition}[\cite{KaLu}]\label{prop803}
For any $x\in W$ and any simple reflection $s$ we have
\begin{displaymath}
\underline{H}_x\underline{H}_s=
\begin{cases}
\underline{H}_{xs}+\sum_{y<x,ys<y}\mu(y,x)\underline{H}_y, 
& xs>x;\\
(v+v^{-1})\underline{H}_x, & xs<x;
\end{cases}
\end{displaymath}
and
\begin{displaymath}
\underline{H}_s\underline{H}_x=
\begin{cases}
\underline{H}_{sx}+\sum_{y<x,sy<y}\mu(y,x)\underline{H}_y, 
& sx>x;\\
(v+v^{-1})\underline{H}_x, & sx<x.
\end{cases}
\end{displaymath}
\end{proposition}

\begin{example}\label{exm804}
{\rm  
Here is the essential part of the multiplication table for 
$\mathbb{H}$ in the Kazhdan-Lusztig basis for $n=3$ 
(see Example~\ref{exm777}):
{\small
\begin{displaymath}
\begin{array}{|c||c|c|c|c|c|c|}
\hline
*&\underline{H}_e&\underline{H}_s&\underline{H}_t&
\underline{H}_{st}&\underline{H}_{ts}&\underline{H}_{sts}\\
\hline\hline
\underline{H}_e&\underline{H}_e&\underline{H}_s&\underline{H}_t&
\underline{H}_{st}&\underline{H}_{ts}&\underline{H}_{sts}\\
\hline
\underline{H}_s&\underline{H}_s&(v+v^{-1})\underline{H}_s&
\underline{H}_{st}&
(v+v^{-1})\underline{H}_{st}&
\underline{H}_{sts}+\underline{H}_{s}&(v+v^{-1})\underline{H}_{sts}\\
\hline
\underline{H}_t&\underline{H}_t&\underline{H}_{ts}&
(v+v^{-1})\underline{H}_t&
\underline{H}_{sts}+\underline{H}_{t}&
(v+v^{-1})\underline{H}_{ts}&(v+v^{-1})\underline{H}_{sts}\\
\hline
\end{array}
\end{displaymath}
}
}
\end{example}

For $x,y\in W$ write $x\leq_L y$ provided that there is $z\in W$
such that $\underline{H}_y$ occurs with a nonzero coefficient in
the decomposition of $\underline{H}_z\underline{H}_x$ in the
Kazhdan-Lusztig basis. Then $\leq_L$ is a partial pre-order on
$W$. Define $\leq_R$ and $\leq_{LR}$ similarly for the right and
the two-sided multiplications, respectively. Equivalence classes
with respect to $\leq_L$, $\leq_R$ and $\leq_{LR}$ are called
{\em Kazhdan-Lusztig left, right and two-sided cells}, respectively.
\index{Kazhdan-Lusztig cell}
We denote the corresponding equivalence relations by
$\sim_L$, $\sim_R$ and $\sim_{LR}$, respectively.
It turns out that the latter
can be given a nice combinatorial description. Recall that
the  Robinson-Schensted correspondence associates to every 
$w\in W$ a pair $(p(w),q(w))$ of standard Young tableaux
of the same shape, see \cite[Section~3.1]{Sa}.

\begin{proposition}[\cite{KL}]\label{prop805}
For $x,y\in W$ we have the following:
\begin{enumerate}[$($a$)$]
\item\label{prop805.1} $x\sim_R y$
if and only if $p(x)=p(y)$.
\item\label{prop805.2} $x\sim_L y$
if and only if $q(x)=q(y)$.
\item\label{prop805.3} $x\sim_{LR} y$
if and only if $p(x)$ and $p(y)$ have the same shape.
\end{enumerate}
\end{proposition}

\begin{example}\label{exm806}
{\rm  
In the case $n=3$ we have the following
\begin{displaymath}
\begin{array}{rl}
\text{right cells:}\quad&
\{e\} , \quad \{s,{st}\}, 
\quad \{t,{ts}\}, \quad \{{sts}\};\\
\text{left cells:}\quad&
\{e\} , \quad \{s,{ts}\}, 
\quad \{t,{st}\}, \quad \{{sts}\};\\
\text{two-sided cells:}\quad&
\{e\} , \quad \{s,{st},t,{ts}\}, 
\quad \{{sts}\}.
\end{array}
\end{displaymath}
}
\end{example}

\subsection{Cell modules}\label{s8.3}

Fix now a left cell $\mathcal{L}$ of $W$. Then the 
$\mathbb{Z}[v,v^{-1}]$-linear span 
$\mathbb{Z}[v,v^{-1}]\mathcal{L}$
of $\underline{H}_w$, $w\in \mathcal{L}$,
has the natural structure of an $\mathbb{H}$-module, given by the
left multiplication with the elements in the Kazhdan-Lusztig basis
(and treating all vectors which do not belong to 
$\mathbb{Z}[v,v^{-1}]\mathcal{L}$ as zero). This module
is called the {\em cell module} corresponding to $\mathcal{L}$.
\index{cell module}
Similarly one defines a right cell $\mathbb{H}$-module for every
right cell $\mathcal{R}$. Note that the cell module 
$\mathbb{Z}[v,v^{-1}]\mathcal{L}$ comes along with
a distinguished $\mathbb{Z}[v,v^{-1}]$-bases, namely the 
Kazhdan-Lusztig basis $\{\underline{H}_w:w\in \mathcal{L}\}$.

\begin{proposition}[\cite{KaLu,Na}]\label{prop807}
Let $\mathcal{L}\subset W$ be a left cell and $\lambda\vdash n$
such that $p(w)$ has shape $\lambda$ for all $w\in \mathcal{L}$.
Then the $W$-module obtained by complexification and generic
specialization of $\mathbb{Z}[v,v^{-1}]\mathcal{L}$ is isomorphic 
to the Specht $W$-module corresponding to $\lambda$, in particular,
it is simple.
\end{proposition}

Using the abstract approach of Subsection~\ref{s3.5} we get 
the $2$-representations of $\cS$ and $\cS^{\mathbb{Z}}$ corresponding
to right cells. By construction, these $2$-representations of $\cS$
are (genuine) categorifications of the corresponding cell modules.
Note that our terminology from Subsection~\ref{s3.5} is opposite
to the one above. This is done to compensate for
the anti-nature of our categorification of $\mathbb{S}_n$ and
$\mathbb{H}$ (see Corollary~\ref{cor506} and 
Theorem~\ref{thm802}). Later on we will construct these 
$2$-representations intrinsically in terms of the category $\mathcal{O}$.

\subsection{Categorification of the induced sign module}\label{s8.4}

Let $\mathfrak{p}\subset\mathfrak{g}$ be a parabolic subalgebra. 
Then the category 
$\mathcal{O}_{\lambda}^{\mathfrak{p}}$ is stable with respect to
the action of projective functors and hence defines, by restriction,
a $2$-representation of $\cS$. The algebra $B_{\lambda}^{\mathfrak{p}}$
inherits from $B_{\lambda}$ a positive grading. Hence we can
consider the graded lift 
${}^{\mathbb{Z}}\mathcal{O}_{\lambda}^{\mathfrak{p}}$ of
$\mathcal{O}_{\lambda}^{\mathfrak{p}}$, which is stable with respect to
the action of graded lifts of projective functors. In this way we obtain,
by restriction, a $2$-representation of $\cS^{\mathbb{Z}}$.
A natural question to ask is: What do these $2$-representations 
categorify?

\begin{example}\label{exm808}
{\rm  
If $\mathfrak{p}=\mathfrak{g}$, the category 
$\mathcal{O}_{\lambda}^{\mathfrak{p}}$ is semi-simple and
$L(\lambda)$ is the unique up to isomorphism simple module in this
category. We have $\theta_s L(\lambda)=0$ for every simple
reflection $s$. Thus $[\mathcal{O}_{\lambda}^{\mathfrak{p}}]^{\mathbb{C}}$
is one-dimensional and  we have $(e+s)[L(\lambda)]=0$, that is
$s[L(\lambda)]=-[L(\lambda)]$. This means that our 
$2$-representation of $\cS$ on $\mathcal{O}_{\lambda}^{\mathfrak{p}}$
is a categorification of the {\em sign} $W$-module.
\index{sign module}
}
\end{example}

The above example suggests the answer in the general case. 
Let $\mathbb{H}^{\mathfrak{p}}$ denote the Hecke algebra of
$W^{\mathfrak{p}}$. For $u\in\{-v,v^{-1}\}$ define on 
$\mathbb{Z}[v,v^{-1}]$ the structure of an 
$\mathbb{H}^{\mathfrak{p}}$-bimodule, which we denote by $\mathcal{V}_u$,
via the surjection $H_s\mapsto u$ for every simple reflection 
$s\in W^{\mathfrak{p}}$. Define the {\em parabolic} $\mathbb{H}$-module
\index{parabolic module}
$\mathcal{M}_u$ as follows:
\begin{displaymath}
\mathcal{M}_u:=\mathcal{V}_u\otimes_{\mathbb{H}^{\mathfrak{p}}}
\mathbb{H}. 
\end{displaymath}
Let $W(\mathfrak{p})$ denote the set of shortest coset representatives
in $W^{\mathfrak{p}}\setminus W$. The elements $M_x=1\otimes H_x$,
$x\in W(\mathfrak{p})$, form a basis of $\mathcal{M}_u$. The action
of $\underline{H}_s$, $s\in S$, in this basis is given by:
\begin{equation}\label{eq821}
M_x \underline{H}_s=
\begin{cases}
M_{xs}+vM_x,& xs\in W(\mathfrak{p}), xs>x;\\
M_{xs}+v^{-1}M_x,& xs\in W(\mathfrak{p}), xs<x;\\
(v+v^{-1})M_x,& xs\not\in W(\mathfrak{p}), u=v^{-1};\\
0,& xs\not\in W(\mathfrak{p}), u=-v.
\end{cases}
\end{equation}
Under the specialization $v\mapsto 1$, the module 
$\mathcal{M}_{v^{-1}}$ specializes to the permutation module
${}_{\mathbb{Z}[W^{\mathfrak{p}}]}\setminus\mathbb{Z}[W]$, while the module
$\mathcal{M}_{-v}$ specializes to the {\em induced sign} module
\index{induced sign module}
\begin{displaymath}
\mathrm{sign}\bigotimes_{\mathbb{Z}[W^{\mathfrak{p}}]} \mathbb{Z}[W],
\end{displaymath}
where the $\mathbb{Z}[W^{\mathfrak{p}}]$-module $\mathrm{sign}$ is 
given via the surjection $\mathbb{Z}[W^{\mathfrak{p}}]\to\mathbb{Z}$ 
which sends a simple reflection $s$ to $-1$.

Applying $\mathrm{Z}_{\mathfrak{p}}$ to Lemma~\ref{lem731}, we obtain:

\begin{lemma}\label{lem809}
For every $x\in W(\mathfrak{p})$ and any simple reflection $s$ we have:
\begin{enumerate}[$($a$)$]
\item\label{lem809.1} If $xs\not\in W(\mathfrak{p})$, then
$\theta_s \Delta^{\mathfrak{p}}(x\cdot\lambda)=0$.
\item\label{lem809.2} If $xs\in W(\mathfrak{p})$ and $xs>x$, then
there is a short exact sequence as follows:
\begin{displaymath}
0\to \Delta^{\mathfrak{p}}(x\cdot\lambda)\langle -1\rangle\to
\theta_s \Delta^{\mathfrak{p}}(x\cdot\lambda)\to
\Delta^{\mathfrak{p}}(xs\cdot\lambda)\to 0.
\end{displaymath}
\item\label{lem809.3} If $xs\in W(\mathfrak{p})$ and $xs<x$, then
there is a short exact sequence as follows:
\begin{displaymath}
0\to \Delta^{\mathfrak{p}}(xs\cdot\lambda)\to
\theta_s \Delta^{\mathfrak{p}}(x\cdot\lambda)\to
\Delta^{\mathfrak{p}}(x\cdot\lambda)\langle 1\rangle\to 0.
\end{displaymath}
\end{enumerate}
\end{lemma}

There is a unique $\mathbb{Z}[v,v^{-1}]$-linear isomorphism 
from $\mathcal{M}_{-v}$ to 
$[{}^{\mathbb{Z}}\mathcal{O}_{\lambda}^{\mathfrak{p}}]$
sending $M_x$ to $[\Delta^{\mathfrak{p}}(x\cdot\lambda)]$
for all $x\in W(\mathfrak{p})$. Comparing \eqref{eq821} with
Lemma~\ref{lem809}, we obtain:

\begin{proposition}[\cite{So3}]\label{prop811}
The action of $\cS^{\mathbb{Z}}$ (resp. $\cS$) on 
${}^{\mathbb{Z}}\mathcal{O}_{\lambda}^{\mathfrak{p}}$
(resp. $\mathcal{O}_{\lambda}^{\mathfrak{p}}$) categorifies
the parabolic module $\mathcal{M}_{-v}$ (resp. the induced sign module).
\end{proposition}

Since $\mathcal{O}_{\lambda}^{\mathfrak{p}}$, being quasi-hereditary,
has finite projective dimension, the $\mathbb{H}$-module 
$\mathcal{M}_{-v}\cong [{}^{\mathbb{Z}}\mathcal{O}_{\lambda}^{\mathfrak{p}}]$
has the following natural bases:
\begin{itemize}
\item The {\em standard} basis $M_x$ given by classes of parabolic
Verma modules.
\index{standard basis}
\item The {\em Kazhdan-Lusztig} basis given by classes of indecomposable
projective modules.
\index{Kazhdan-Lusztig basis}
\item The {\em dual Kazhdan-Lusztig} basis given by classes of 
simple modules.
\index{dual Kazhdan-Lusztig basis}
\item The {\em twisted Kazhdan-Lusztig} basis given by classes of 
indecomposable tilting  modules.
\index{twisted  Kazhdan-Lusztig basis}
\end{itemize}

\subsection{Categorification of Specht modules}\label{s8.5}

Similarly to the previous subsection, let 
$\mathfrak{p}\subset\mathfrak{g}$ be a parabolic subalgebra 
and $\lambda\in\mathfrak{h}^*_{\mathrm{dom}}$
be integral and regular. Denote by $w_o^{\mathfrak{p}}$ the
longest element of $W_{\mathfrak{p}}$. Then the element 
$w_o^{\mathfrak{p}}w_o$ is the unique longest element
in $W(\mathfrak{p})$. In \cite{IS} it was shown  (by explicit calculation)
that there exists an integral $\lambda'\in\mathfrak{h}^*_{\mathrm{dom}}$
such that the category $\mathcal{O}_{\lambda'}^{\mathfrak{p}}$ has
a unique simple module and hence is semi-simple. In particular,
this unique simple module is both projective and injective. Using
translation functors it follows that there is at least one
projective injective module in $\mathcal{O}_{\lambda}^{\mathfrak{p}}$.
In fact, we have the following:

\begin{theorem}[\cite{Ir0}]\label{thm821}
For $w\in W(\mathfrak{p})$ the following conditions are equivalent:
\begin{enumerate}[$($a$)$]
\item\label{thm821.1} The module $P^{\mathfrak{p}}(w\cdot\lambda)$
is injective.
\item\label{thm821.2} The module $L^{\mathfrak{p}}(w\cdot\lambda)$
occurs in the socle of some $\Delta^{\mathfrak{p}}(\mu)$.
\item\label{thm821.3} The module $L^{\mathfrak{p}}(w\cdot\lambda)$
has maximal Gelfand-Kirillov dimension in
$\mathcal{O}_{\lambda}^{\mathfrak{p}}$.
\item\label{thm821.4} $w\sim_R 
w_o^{\mathfrak{p}}w_o$.
\end{enumerate}
\end{theorem}

As a direct corollary we have that indecomposable projective injective
modules in $\mathcal{O}_{\lambda}^{\mathfrak{p}}$ are indexed by 
elements of the right cell of $w_o^{\mathfrak{p}}w_o$. Since projective
functors preserve both projective and injective modules, the latter suggests
that their action on the category of projective-injective modules
in $\mathcal{O}_{\lambda}^{\mathfrak{p}}$ should categorify the
cell module corresponding to the right cell of $w_o^{\mathfrak{p}}w_o$.

Denote by $\mathcal{C}$ the Serre subcategory of 
$\mathcal{O}_{\lambda}^{\mathfrak{p}}$ consisting of all modules
which do not have maximal Gelfand-Kirillov dimension. 
Let $\mathcal{R}_{\mathfrak{p}}$ 
denote the right cell of $w_o^{\mathfrak{p}}w_o$. By 
Theorem~\ref{thm821}, $\mathcal{C}$ is generated by
$L^{\mathfrak{p}}(w\cdot\lambda)$ such that 
$w\not\in\mathcal{R}_{\mathfrak{p}}$. As projective functors are
essentially tensoring with finite-dimensional modules, they preserve
$\mathcal{C}$. Therefore we can restrict the action of 
$\cS$ to $\mathcal{C}$ and also get the induced action of $\cS$ on the 
quotient $\mathcal{O}_{\lambda}^{\mathfrak{p}}/\mathcal{C}$.
If we denote by $e_{\mathfrak{p}}$ the sum of all primitive 
idempotents in $B_{\lambda}^{\mathfrak{p}}$, corresponding to 
$L(w\cdot\lambda)$, $w\in\mathcal{R}_{\mathfrak{p}}$, we get
$\mathcal{O}_{\lambda}^{\mathfrak{p}}/\mathcal{C}\cong
e_{\mathfrak{p}}B_{\lambda}^{\mathfrak{p}}e_{\mathfrak{p}}
\text{-}\mathrm{mod}$. Let $\mathcal{C}^{\mathbb{Z}}$ denote the graded
lift of $\mathcal{C}$. Denote by $\xi$ the transpose of the 
partition of $n$ corresponding to $\mathfrak{p}$.

\begin{theorem}[\cite{KMS}]\label{thm812}
\begin{enumerate}[$($a$)$]
\item\label{thm812.1}
The action of $\cS^{\mathbb{Z}}$ (resp. $\cS$) on 
the additive category of projective-injective modules 
in ${}^{\mathbb{Z}}\mathcal{O}_{\lambda}^{\mathfrak{p}}$ 
(resp. $\mathcal{O}_{\lambda}^{\mathfrak{p}}$) categorifies
the cell module $\mathbb{Z}[v,v^{-1}]\mathcal{R}_{\mathfrak{p}}$ 
(resp. the Specht module corresponding to $\xi$).
\item\label{thm812.2}
The action of $\cS^{\mathbb{Z}}$ (resp. $\cS$) on 
${}^{\mathbb{Z}}\mathcal{O}_{\lambda}^{\mathfrak{p}}/
\mathcal{C}^{\mathbb{Z}}$ (resp. 
$\mathcal{O}_{\lambda}^{\mathfrak{p}}/\mathcal{C}$) categorifies  
the cell module $\mathbb{Z}[v,v^{-1}]\mathcal{R}_{\mathfrak{p}}$ 
(resp. the Specht module corresponding to $\xi$) after extending
scalars to $\mathbb{Q}$.
\end{enumerate}
\end{theorem}

\begin{proof}[Idea of the proof.]
Compare the action of $\theta_s$ in the basis of indecomposable
projective modules modules with the action of $\underline{H}_s$. 
\end{proof}

We have the usual {\em Kazhdan-Lusztig} basis in the split Grothendieck
group of the category of projective-injective modules in
${}^{\mathbb{Z}}\mathcal{O}_{\lambda}^{\mathfrak{p}}$ given by
classes of indecomposable projective modules. The only obvious natural 
basis of $[\mathcal{O}_{\lambda}^{\mathfrak{p}}/\mathcal{C}]$
is the one given by classes of simple modules
(the {\em dual Kazhdan-Lusztig} basis). The algebra 
$e_{\mathfrak{p}}B_{\lambda}^{\mathfrak{p}}e_{\mathfrak{p}}$
is known to be symmetric (see \cite{MS15}), in particular, it
has infinite global dimension in general. As a consequences, the classes
of projective modules do not usually form a $\mathbb{Z}$-basis
of $[\mathcal{O}_{\lambda}^{\mathfrak{p}}/\mathcal{C}]$.

\begin{example}\label{exm815}
{\em  
To categorify the trivial $\mathbb{S}_n$-module we have to take
$\mathfrak{p}=\mathfrak{b}$. In this case $W_{\mathfrak{p}}=\{e\}$,
$\mathcal{R}_{\mathfrak{p}}=\{w_o\}$ and 
$e_{\mathfrak{p}}B_{\lambda}^{\mathfrak{p}}e_{\mathfrak{p}}\cong
C_{\lambda}$ is the coinvariant algebra. For $n>1$ the class 
$[P(w_o\cdot\lambda)]$ does not form a $\mathbb{Z}$-basis in the
Grothendieck group, but gives a basis if one extends scalars
to $\mathbb{Q}$ or $\mathbb{C}$.
}
\end{example}

A partition $\xi$ of $n$ does not define the parabolic subalgebra 
$\mathfrak{p}$ uniquely, so one could naturally ask whether
our categorification of the Specht module corresponding to $\xi$
depends on the choice of $\mathfrak{p}$. The answer is given
by the following:

\begin{theorem}[\cite{MS15}]\label{thm817}
If $\mathfrak{p}$ and $\mathfrak{q}$ are two parabolic subalgebras
which give rise to the same partition $\xi$, then the corresponding
categorifications of the Specht modules are equivalent
(in the sense of representations of $2$-categories).
\end{theorem}

The equivalence of Theorem~\ref{thm817} is provided by a certain
composition of derived Zuckerman functors (which naturally
commute with projective functors).

\begin{remark}\label{rem819}
{\rm 
Let $W=\{e,s,t,st,ts,sts,tst,stst\}$ be of type $B_2$
and $V=\mathbb{C}$ be the one dimensional module on which $s$ and $t$
acts via $1$ and $-1$, respectively. Then $V$ cannot be categorified
via the action of the $B_2$-analogue of $\cS$ on some module
category. Indeed, for any such action $\theta_t$ must act by the
zero functor and $\theta_s$ must act by a nonzero functor, which is
not possible as $s$ and $t$ belong to the same two-sided cell.
On the other hand, $V$ admits a categorification using a 
triangulated category.
} 
\end{remark}

\section{$\mathbb{S}_n$-categorification: (induced) cell modules}\label{s9}

\subsection{Categories $\mathcal{O}^{\hat{\mathcal{R}}}$}\label{s9.1}

Let $\mathcal{R}$ be a right cell of $W$ and 
$\lambda\in\mathfrak{h}^*_{\mathrm{dom}}$ be integral and regular. Set 
\begin{displaymath}
\hat{\mathcal{R}}=\{w\in W:w\leq_R x \text{ for any } x\in \mathcal{R}\}
\end{displaymath}
and denote by $\mathcal{O}_{\lambda}^{\hat{\mathcal{R}}}$ the 
Serre subcategory of $\mathcal{O}_{\lambda}$ generated by 
$L(w\cdot\lambda)$, $w\in \hat{\mathcal{R}}$.

\begin{example}\label{exm901}
{\rm
If $\mathcal{R}$ contains $w_o^{\mathfrak{p}}w_o$ for some
parabolic subalgebra $\mathfrak{p}$ of $\mathfrak{g}$, 
then $\mathcal{O}_{\lambda}^{\hat{\mathcal{R}}}=
\mathcal{O}_{\lambda}^{\mathfrak{p}}$. 
}
\end{example}

In the general case the associative algebra describing 
$\mathcal{O}_{\lambda}^{\hat{\mathcal{R}}}$ might have infinite
global dimension (in particular, it does not have to be quasi-hereditary).
Denote by $\mathrm{Z}_{\mathcal{R}}:\mathcal{O}_{\lambda}\to
\mathcal{O}_{\lambda}^{\hat{\mathcal{R}}}$ the right exact
functor of taking the maximal quotient contained in 
$\mathcal{O}_{\lambda}^{\hat{\mathcal{R}}}$. This is a natural
generalization of Zuckerman's functor. Similarly to the case of
the parabolic category we will denote by 
$P^{\mathcal{R}}(w\cdot\lambda)$ etc. all structural modules
in $\mathcal{O}_{\lambda}^{\hat{\mathcal{R}}}$.

From the Kazhdan-Lusztig
combinatorics it follows that $\mathcal{O}_{\lambda}^{\hat{\mathcal{R}}}$
is closed with respect to the action of projective functors. 
Hence the restriction of the action of projective functors to
$\mathcal{O}_{\lambda}^{\hat{\mathcal{R}}}$ gives a $2$-representation
of $\cS$. Let ${}^{\mathbb{Z}}\mathcal{O}_{\lambda}^{\hat{\mathcal{R}}}$
denote the graded lift of $\mathcal{O}_{\lambda}^{\hat{\mathcal{R}}}$,
which, obviously, exists. Then 
${}^{\mathbb{Z}}\mathcal{O}_{\lambda}^{\hat{\mathcal{R}}}$ carries
the natural structure of a $2$-representation of $\cS^{\mathbb{Z}}$.

\begin{proposition}[\cite{MS2}]\label{prop902}
For $w\in \hat{\mathcal{R}}$ the following conditions are equivalent:
\begin{enumerate}[$($a$)$]
\item\label{prop902.1} The module $P^{\mathcal{R}}(w\cdot\lambda)$
is injective.
\item\label{prop902.2} The module $L^{\mathcal{R}}(w\cdot\lambda)$
has maximal Gelfand-Kirillov dimension in
$\mathcal{O}_{\lambda}^{\hat{\mathcal{R}}}$.
\item\label{prop902.3} $w\in\mathcal{R}$.
\end{enumerate}
\end{proposition}

\subsection{Categorification of cell modules}\label{s9.2}

Denote by $\mathcal{C}$ the Serre subcategory of 
$\mathcal{O}_{\lambda}^{\hat{\mathcal{R}}}$ consisting of all modules
which do not have maximal Gelfand-Kirillov dimension. 
By Proposition~\ref{prop902}, $\mathcal{C}$ is generated by
$L^{\mathcal{R}}(w\cdot\lambda)$ such that 
$w\not\in\mathcal{R}$ and is stable under the action of $\cS$. 
Therefore we have the induced action of $\cS$ on the 
quotient $\mathcal{O}_{\lambda}^{\hat{\mathcal{R}}}/\mathcal{C}$.
Let $\mathcal{C}^{\mathbb{Z}}$ denote the graded
lift of $\mathcal{C}$. 

\begin{theorem}[\cite{MS2}]\label{thm903}
\begin{enumerate}[$($a$)$]
\item\label{thm903.1}
The action of $\cS^{\mathbb{Z}}$ (resp. $\cS$) on 
the additive category of projective-injective modules 
in ${}^{\mathbb{Z}}\mathcal{O}_{\lambda}^{\hat{\mathcal{R}}}$ 
(resp. $\mathcal{O}_{\lambda}^{\hat{\mathcal{R}}}$) categorifies
the cell module $\mathbb{Z}[v,v^{-1}]\mathcal{R}$ 
(resp. its specialization for $v=1$).
\item\label{thm903.2}
The action of $\cS^{\mathbb{Z}}$ (resp. $\cS$) on 
${}^{\mathbb{Z}}\mathcal{O}_{\lambda}^{\hat{\mathcal{R}}}/
\mathcal{C}^{\mathbb{Z}}$ 
(resp. $\mathcal{O}_{\lambda}^{\hat{\mathcal{R}}}/
\mathcal{C}^{\mathbb{Z}}$) categorifies 
the cell module $\mathbb{Z}[v,v^{-1}]\mathcal{R}$ 
(resp. its specialization for $v=1$) after extending
scalars to $\mathbb{Q}$.
\end{enumerate}
\end{theorem}

We have the usual {\em Kazhdan-Lusztig} basis in the split Grothendieck
group of the category of projective-injective modules in
$\mathcal{O}_{\lambda}^{\hat{\mathcal{R}}}$ given by
classes of indecomposable projective modules. We also have the natural 
basis of $[\mathcal{O}_{\lambda}^{\hat{\mathcal{R}}}/\mathcal{C}]$
given by classes of simple modules (the {\em dual Kazhdan-Lusztig} basis). 

\begin{theorem}[\cite{MS2}]\label{thm904}
If $\mathcal{R}$ and $\mathcal{R}'$ are two right cells of $W$ inside
the same two-sided cell, then the corresponding $2$-representations
of $\cS^{\mathbb{Z}}$ and $\cS$ categorifying cell modules are
equivalent.
\end{theorem}

\begin{proof}[Idea of the proof.]
The equivalence is provided by composing functors $\mathrm{Q}_s$, $s\in S$. 
\end{proof}

\begin{corollary}[\cite{MS2}]\label{cor905}
The endomorphism algebra of a basic projective injective module in 
$\mathcal{O}_{\lambda}^{\hat{\mathcal{R}}}$ does not depend on the choice of
$\mathcal{R}$ inside a fixed two-sided cell. In particular, this
endomorphism algebra is symmetric.
\end{corollary}

\begin{remark}\label{rem973}
{\rm  
The categorification of the cell module constructed in 
Theorem~\ref{thm903} is isomorphic to the cell $2$-representation
$\mathbf{C}_{\mathcal{R}}$ of $\cS$ in the sense of Subsection~\ref{s3.5}. 
}
\end{remark}

\subsection{Categorification of permutation modules}\label{s9.3}

Let $\lambda\in\mathfrak{h}^*_{\mathrm{dom}}$ be integral and 
regular and $\nu\in\mathfrak{h}^*_{\mathrm{dom}}$ be integral
(but, possibly, singular). Let $\mathfrak{p}$ be a parabolic subalgebra
of $\mathfrak{g}$ such that 
$W_{\mathfrak{p}}=W_{\nu}$. Denote by $W[\mathfrak{p}]$ the
set of {\em longest} representatives in cosets from 
\index{longest representative}
$W_{\mathfrak{p}}\setminus W$. Let $\mathcal{C}$ denote the 
Serre subcategory of $\mathcal{O}_{\lambda}$ generated by 
$L(w\cdot\lambda)$, $w\in W\setminus W[\mathfrak{p}]$.
We identify the quotient $\mathcal{O}_{\lambda}/\mathcal{C}$ 
with the full subcategory of $\mathcal{O}_{\lambda}$ consisting of
all $M$ having a two step presentation $P_1\to P_0\tto M$,
where both $P_1$ and $P_0$ are isomorphic to direct sums of 
projectives of the form $P(w\cdot\lambda)$, $w\in W[\mathfrak{p}]$.
Then $\mathcal{O}_{\lambda}/\mathcal{C}$ can be described
in terms of Harish-Chandra bimodules as follows:

\begin{theorem}[\cite{BG}]\label{thm906}
Tensoring with $\Delta(\nu)$ induces an equivalence between
${}^{\infty}_{\,\,\lambda}\mathcal{H}_{\nu}^1$ and
$\mathcal{O}_{\lambda}/\mathcal{C}$.
\end{theorem}

The equivalence from Theorem~\ref{thm906} obviously commutes with
the left action of projective functors. Therefore we may
regard ${}^{\infty}_{\,\,\lambda}\mathcal{H}_{\nu}^1$
as a $2$-representation of $\cS$. Note that the category 
$\mathcal{C}$ clearly admits a graded lift $\mathcal{C}^{\mathbb{Z}}$,
which allows us to consider a graded version 
$({}^{\infty}_{\,\,\lambda}\mathcal{H}_{\nu}^1)^{\mathbb{Z}}$ of
${}^{\infty}_{\,\,\lambda}\mathcal{H}_{\nu}^1$. The category 
$({}^{\infty}_{\,\,\lambda}\mathcal{H}_{\nu}^1)^{\mathbb{Z}}$ is a 
$2$-representation of $\cS^{\mathbb{Z}}$. If $\nu$ is singular, the
category ${}^{\infty}_{\,\,\lambda}\mathcal{H}_{\nu}^1$ has infinite
global dimension and hence is not described by a quasi-hereditary 
algebra.

For $w\in W[\mathfrak{p}]$ denote by $\Delta_{\mathfrak{p}}(w)$
the preimage in ${}^{\infty}_{\,\,\lambda}\mathcal{H}_{\nu}^1$ of
the quotient of $P(w\cdot \lambda)$ by the trace of all 
$P(x\cdot \lambda)$ such that $x\in W[\mathfrak{p}]$ and $x<w$.
We also let $P(w)$ denote the preimage of $P(w\cdot\lambda)$
and define $L(w)$ as the simple top of $P(w)$. Finally, denote by
$\overline{\Delta}_{\mathfrak{p}}(w)$ the quotient of
$\Delta_{\mathfrak{p}}(w)$ modulo the trace of 
$\Delta_{\mathfrak{p}}(w)$ in the radical of 
$\Delta_{\mathfrak{p}}(w)$.

\begin{theorem}[\cite{FKM,KM}]\label{thm907}
Let $w\in W[\mathfrak{p}]$.
\begin{enumerate}[$($a$)$]
\item\label{thm907.1} The kernel of the natural projection
$P(w)\tto \Delta_{\mathfrak{p}}(w)$ has a filtration
whose subquotients are isomorphic to $\Delta_{\mathfrak{p}}(x)$
for $x<w$.
\item\label{thm907.2} The kernel of the natural projection
$\overline{\Delta}_{\mathfrak{p}}(w)\tto L(w)$ has a filtration
whose subquotients are isomorphic to $L(x)$
for $x>w$.
\item\label{thm907.3} The module $\Delta_{\mathfrak{p}}(w)$
has a filtration whose subquotients are isomorphic to 
$\overline{\Delta}_{\mathfrak{p}}(w)$.
\end{enumerate}
\end{theorem}

Theorem~\ref{thm907} says that the associative algebra describing
the category ${}^{\infty}_{\,\,\lambda}\mathcal{H}_{\nu}^1$ is
{\em properly stratified} in the sense of  \cite{Dl}. Modules
\index{properly stratified algebra}
$\Delta_{\mathfrak{p}}(w)$ are called {\em standard modules}
\index{standard module}
and modules $\overline{\Delta}_{\mathfrak{p}}(w)$ are called
{\em proper standard modules} with respect to this structure.
\index{proper standard module}
Note that standard modules have finite projective dimension,
while proper standard modules usually have infinite projective 
dimension. There is an appropriate titling theory developed
for properly stratified algebras in \cite{AHLU}.
In \cite{FKM} it is shown that 
${}^{\infty}_{\,\,\lambda}\mathcal{H}_{\nu}^1$ is Ringel self-dual.

\begin{lemma}\label{lem909}
For every $x\in W[\mathfrak{p}]$ and any simple reflection $s$ we have:
\begin{enumerate}[$($a$)$]
\item\label{lem909.1} If $xs\not\in W(\mathfrak{p})$, then
there is a short exact sequence as follows:
\begin{displaymath}
0\to \Delta_{\mathfrak{p}}(x)\langle -1\rangle\to
\theta_s \Delta_{\mathfrak{p}}(x)\to
\Delta_{\mathfrak{p}}(x)\langle 1\rangle\to 0.
\end{displaymath}
\item\label{lem909.2} If $xs\in W[\mathfrak{p}]$ and $xs>x$, then
there is a short exact sequence as follows:
\begin{displaymath}
0\to \Delta_{\mathfrak{p}}(x)\langle -1\rangle\to
\theta_s \Delta_{\mathfrak{p}}(x)\to
\Delta_{\mathfrak{p}}(xs)\to 0.
\end{displaymath}
\item\label{lem909.3} If $xs\in W[\mathfrak{p}]$ and $xs<x$, then
there is a short exact sequence as follows:
\begin{displaymath}
0\to \Delta_{\mathfrak{p}}(xs)\to
\theta_s \Delta_{\mathfrak{p}}(x)\to
\Delta_{\mathfrak{p}}(x)\langle 1\rangle\to 0.
\end{displaymath}
\end{enumerate}
\end{lemma}

There is a unique $\mathbb{Z}[v,v^{-1}]$-linear homomorphism 
from $\mathcal{M}_{v^{-1}}$ to 
$[({}^{\infty}_{\,\,\lambda}\mathcal{H}_{\nu}^1)^{\mathbb{Z}}]$
sending $M_x$ to $[\Delta_{\mathfrak{p}}(x)]$
for all $x\in W(\mathfrak{p})$. This becomes an isomorphism if we
extend our scalars to $\mathbb{Q}$. Comparing \eqref{eq821} with
Lemma~\ref{lem909}, we obtain:

\begin{proposition}[\cite{MS0}]\label{prop910}
The action of $\cS^{\mathbb{Z}}$ (resp. $\cS$) on 
$({}^{\infty}_{\,\,\lambda}\mathcal{H}_{\nu}^1)^{\mathbb{Z}}$
(resp. ${}^{\infty}_{\,\,\lambda}\mathcal{H}_{\nu}^1$) categorifies
the parabolic module $\mathbb{Q}\otimes_{\mathbb{Z}}\mathcal{M}_{v^{-1}}$ 
(resp. the corresponding permutation module).
The integral version $\mathcal{M}_{v^{-1}}$ is categorified by the 
action of $\cS^{\mathbb{Z}}$ (resp. $\cS$) on the 
additive category of projective modules in 
$({}^{\infty}_{\,\,\lambda}\mathcal{H}_{\nu}^1)^{\mathbb{Z}}$
(resp. ${}^{\infty}_{\,\,\lambda}\mathcal{H}_{\nu}^1$).
\end{proposition}

The Grothendieck group  of 
$({}^{\infty}_{\,\,\lambda}\mathcal{H}_{\nu}^1)^{\mathbb{Z}}$
has two natural bases:
\begin{itemize}
\item The {\em proper standard} basis given by classes of
proper standard modules.
\index{proper standard basis}
\item The {\em dual Kazhdan-Lusztig} basis given by classes of
simple modules.
\end{itemize}
Extending the scalar to $\mathbb{Q}$ we get three other bases:
\begin{itemize}
\item The {\em standard} basis given by classes of
standard modules.
\index{standard basis}
\item The {\em Kazhdan-Lusztig} basis given by classes of
indecomposable projective modules.
\item The {\em twisted Kazhdan-Lusztig} basis given by classes of
indecomposable tilting modules.
\end{itemize}

\subsection{Parabolic analogues of $\mathcal{O}$}\label{s9.4}

Comparing our categorifications of the induced sign module
(Subsection~\ref{s8.4}) and the permutation module
(Subsection~\ref{s9.3}) one could observe certain similarities 
in constructions. A natural question is: could they be two special
cases of some more general construction. The answer turns out
to be ``yes'' and the corresponding general construction is 
the general approach to parabolic generalizations of $\mathcal{O}$
proposed in \cite{FKM0}.

Let $\lambda$ and $\mathfrak{p}$ be as in the previous section.
Denote by $\mathfrak{n}$ and $\mathfrak{a}$ the nilpotent radical 
and the Levi quotient of $\mathfrak{p}$, respectively. Then
$\mathfrak{a}$ is a reductive Lie algebra, isomorphic to the direct
sum of some $\mathfrak{gl}_{k_i}$. For any category $\mathcal{C}$ 
of $\mathfrak{a}$-modules one can consider the full subcategory 
$\mathcal{O}(\mathfrak{p},\mathcal{C})$ of the category of all
$\mathfrak{g}$-modules, which consists of all modules $M$ satisfying
the following conditions:
\begin{itemize}
\item $M$ is finitely generated;
\item the action of $U(\mathfrak{n})$ on $M$ is locally finite;
\item after restriction to $\mathfrak{a}$, the module $M$ decomposes
into a direct sum of modules from $\mathcal{C}$.
\end{itemize}
Under some rather general assumptions on $\mathcal{C}$ one can show
that the category $\mathcal{O}(\mathfrak{p},\mathcal{C})$ admits a
block decomposition such that each block is equivalent to the module
category over a finite dimensional associative algebra (see 
\cite{FKM0}). Here we will deal with some very special categories
$\mathcal{C}$ and will formulate all results for these categories.

Fix some right cell $\mathcal{R}$ in $W_{\mathfrak{p}}$ and 
consider the Serre subcategory $\mathcal{C}_{\mathcal{R}}$ of the 
category $\mathcal{O}$ for the algebra $\mathfrak{a}$, which is 
generated by simple modules $L(w\cdot\lambda_{\mathfrak{a}})$ and all 
their translations  to all possible walls, where $\lambda_{\mathfrak{a}}$ 
is integral, regular and dominant and $w\in \hat{\mathcal{R}}$. Denote by
$\mathcal{C}'_{\mathcal{R}}$ the Serre subcategory 
of $\mathcal{C}_{\mathcal{R}}$ generated by simple modules 
$L(w\cdot\lambda_{\mathfrak{a}})$ and all their translations  
to all possible walls, 
where $\lambda_{\mathfrak{a}}$ is integral, regular and dominant and 
$w\in \hat{\mathcal{R}}\setminus\mathcal{R}$. By
Subsection~\ref{s9.2}, the action of projective functors on
a regular block of $\mathcal{C}_{\mathcal{R}}/\mathcal{C}'_{\mathcal{R}}$
categorifies the cell $W_{\mathfrak{p}}$-module corresponding
to $\mathcal{R}$. As usual, we identify the quotient
$\mathcal{C}:=\mathcal{C}_{\mathcal{R}}/\mathcal{C}'_{\mathcal{R}}$
with the full subcategory of $\mathcal{C}_{\mathcal{R}}$
consisting of modules which have a presentation by projective-injective
modules (see Proposition~\ref{prop902}).

Now we can consider the category $\mathcal{O}(\mathfrak{p},\mathcal{C})$,
as was proposed in \cite{MS2}. To obtain our categorification of the
induced sign module we have to take $\mathcal{R}=\{e\}$, which 
implies that $\mathcal{C}_{\mathcal{R}}$ is the category of all
finite dimensional semi-simple integral $\mathfrak{a}$-modules and
$\mathcal{C}'_{\mathcal{R}}=0$. To obtain our categorification of the
permutation module we have to take $\mathcal{R}=\{w_o^{\mathfrak{p}}\}$, 
which implies that each regular block of $\mathcal{C}_{\mathcal{R}}/
\mathcal{C}'_{\mathcal{R}}$ is the category of modules over the
coinvariant algebra of $W_{\mathfrak{p}}$, realized via modules
in $\mathcal{O}$ (for $\mathfrak{a}$) admitting a presentation by
projective-injective modules.  By construction, 
$\mathcal{O}(\mathfrak{p},\mathcal{C})$ is a Serre subquotient of
$\mathcal{O}$, in particular, it inherits from $\mathcal{O}$
a decomposition into blocks, indexed by dominant $\lambda$.

\begin{lemma}[\cite{MS2}]\label{lem911}
Let $\lambda\in\mathfrak{h}^*_{\mathrm{dom}}$ be integral and regular.
Denote by $\mathcal{K}$ and $\mathcal{K}'$ the Serre subcategories of 
$\mathcal{O}_{\lambda}$ generated by $L(xy\cdot\lambda)$, where
$y\in W(\mathfrak{p})$ and $x\in \hat{\mathcal{R}}$ or
$x\in \hat{\mathcal{R}}\setminus\mathcal{R}$, respectively. Then
$\mathcal{O}(\mathfrak{p},\mathcal{C})_{\lambda}$ is equivalent to
$\mathcal{K}/\mathcal{K}'$.
\end{lemma}

For $w\in \mathcal{R}W(\mathfrak{p})\subset W$ denote by 
$P^{\mathfrak{p},\mathcal{R}}(w)$ the projective cover of
$L(w\cdot\lambda)$ in $\mathcal{K}/\mathcal{K}'$.
For $x,x'\in \mathcal{R}$ and $y,y'\in W(\mathfrak{p})$
we write $xy\preceq x'y'$ if and only if $y\leq y'$. Denote by
$\Delta^{\mathfrak{p},\mathcal{R}}(w)$ the quotient of 
$P^{\mathfrak{p},\mathcal{R}}(w)$ modulo the trace of all
projectives $P^{\mathfrak{p},\mathcal{R}}(w')$ such that
$w'\prec w$. Denote by $\overline{\Delta}^{\mathfrak{p},\mathcal{R}}(w)$ 
the quotient of  $\Delta^{\mathfrak{p},\mathcal{R}}(w)$ modulo the 
trace of $\Delta^{\mathfrak{p},\mathcal{R}}(w')$ such that
$w'\preceq w$ in the radical of $\Delta^{\mathfrak{p},\mathcal{R}}(w)$.
Denote by $L^{\mathfrak{p},\mathcal{R}}(w)$ the unique simple top of
$P^{\mathfrak{p},\mathcal{R}}(w)$. The $\mathfrak{g}$-module
$L(w\cdot\lambda)$ is a representative for the simple object
$L^{\mathfrak{p},\mathcal{R}}(w)$ of $\mathcal{K}/\mathcal{K}'$.
We can now formulate the following principal structural properties 
of the category $\mathcal{K}/\mathcal{K}'$.

\begin{theorem}[\cite{MS2}]\label{thm915}
Let $w\in \mathcal{R}W(\mathfrak{p})$.
\begin{enumerate}[$($a$)$]
\item\label{thm915.1} The kernel of the natural projection
$P^{\mathfrak{p},\mathcal{R}}(w)\tto \Delta^{\mathfrak{p},\mathcal{R}}(w)$ 
has a filtration whose subquotients are isomorphic to 
$\Delta^{\mathfrak{p},\mathcal{R}}(x)$ for $x\prec w$.
\item\label{thm915.2} The kernel of the natural projection
$\overline{\Delta}^{\mathfrak{p},\mathcal{R}}(w)\tto 
L^{\mathfrak{p},\mathcal{R}}(w)$ 
has a filtration whose subquotients are isomorphic to 
$L^{\mathfrak{p},\mathcal{R}}(x)$ for $w\prec x$.
\item\label{thm915.3} The module $\Delta^{\mathfrak{p},\mathcal{R}}(w)$
has a filtration whose subquotients are isomorphic to 
$\overline{\Delta}^{\mathfrak{p},\mathcal{R}}(x)$ such that
$w\preceq x$ and $x\preceq w$.
\end{enumerate}
\end{theorem}

Theorem~\ref{thm915} says that the algebra describing 
$\mathcal{O}(\mathfrak{p},\mathcal{C})_{\lambda}$ is {\em standardly
stratified} in the sense of \cite{CPS2}. Modules
\index{standardly stratified algebra}
$\Delta^{\mathfrak{p},\mathcal{R}}(w)$ and
$\overline{\Delta}^{\mathfrak{p},\mathcal{R}}(w)$ are called
{\em standard} and {\em proper standard} modules with respect
to this structure, respectively. There is an appropriate titling 
theory developed for standardly stratified algebras in \cite{Fr}.
In \cite{MS2} it is shown that 
$\mathcal{O}(\mathfrak{p},\mathcal{C})_{\lambda}$ 
is Ringel self-dual and that for this category one has an
analogue of Theorem~\ref{thm821}.

Clearly, the category 
$\mathcal{O}(\mathfrak{p},\mathcal{C})_{\lambda}$
admits a graded analogue, 
$\mathcal{O}(\mathfrak{p},\mathcal{C})_{\lambda}^{\mathbb{Z}}$.

\subsection{Categorification of induced cell modules}\label{s9.5}

Let $\mathcal{R}$ be as in the previous subsection. By definition,
the induced cell module $\mathbb{Z}[v,v^{-1}]\mathcal{R}
\otimes_{\mathbb{H}^{\mathfrak{p}}}\mathbb{H}$ has a basis, which
consists of all elements of the form $\underline{H}_x\otimes H_y$, where
$x\in \mathcal{R}$ and $y\in W(\mathfrak{p})$ (see \cite{HY} for 
detailed combinatorics of induced cell modules). Define a
$\mathbb{Z}[v,v^{-1}]$-linear map $\Psi$ from 
$\mathbb{Z}[v,v^{-1}]\mathcal{R}
\otimes_{\mathbb{H}^{\mathfrak{p}}}\mathbb{H}$
to $[\mathcal{O}(\mathfrak{p},\mathcal{C})^{\mathbb{Z}}_{\lambda}]$ by
sending $\underline{H}_x\otimes H_y$ to the class of
$\Delta^{\mathfrak{p},\mathcal{R}}(xy)$.
As both $\mathcal{K}$ and $\mathcal{K}'$, as defined in the
previous subsection, are closed with respect to the action of
projective functors, we have a natural action of 
$\cS^{\mathbb{Z}}$  (resp. $\cS$) on
the additive category of projective objects in 
$\mathcal{O}(\mathfrak{p},\mathcal{C})^{\mathbb{Z}}_{\lambda}$
(resp. $\mathcal{O}(\mathfrak{p},\mathcal{C})_{\lambda}$).

\begin{theorem}[\cite{MS2}]\label{thm917}
\begin{enumerate}[$($a$)$]
\item\label{thm917.1} After extending the scalars to 
$\mathbb{Q}$, the map $\Psi$ becomes an isomorphism of
$\mathbb{H}$-modules.
\item\label{thm917.2} The action of $\cS^{\mathbb{Z}}$ 
(resp. $\cS$) on
the additive category of projective objects in 
$\mathcal{O}(\mathfrak{p},\mathcal{C})^{\mathbb{Z}}_{\lambda}$
(resp. $\mathcal{O}(\mathfrak{p},\mathcal{C})_{\lambda}$)
is a categorification of $\mathbb{Z}[v,v^{-1}]\mathcal{R}
\otimes_{\mathbb{H}^{\mathfrak{p}}}\mathbb{H}$
(resp. the corresponding specialized induced cell module).
\item\label{thm917.3} The action of $\cS^{\mathbb{Z}}$ 
(resp. $\cS$) on 
$\mathcal{O}(\mathfrak{p},\mathcal{C})^{\mathbb{Z}}_{\lambda}$
(resp. $\mathcal{O}(\mathfrak{p},\mathcal{C})_{\lambda}$)
is a categorification of $\mathbb{Z}[v,v^{-1}]\mathcal{R}
\otimes_{\mathbb{H}^{\mathfrak{p}}}\mathbb{H}$
(resp. the corresponding specialized induced cell module)
after extending scalars to $\mathbb{Q}$.
\end{enumerate}
\end{theorem}

\begin{proof}[Idea of the proof.]
One should check that the action of projective functors on 
standard modules is compatible with the action of 
$\underline{H}_s$ on the module 
$\mathbb{Z}[v,v^{-1}]\mathcal{R}
\otimes_{\mathbb{H}^{\mathfrak{p}}}\mathbb{H}$ in the
standard basis $\underline{H}_x\otimes H_y$. This reduces either
to short exact sequences similar to the ones in
Lemma~\ref{lem809}\eqref{lem809.2}-\eqref{lem809.3}, or to
the action of the $\mathfrak{a}$-analogue of $\theta_s$
on projective modules in $\mathcal{C}$. That the latter
action is the right one follows from Theorem~\ref{thm903}.
\end{proof}

The $\mathbb{H}$-module 
$[\mathcal{O}(\mathfrak{p},\mathcal{C})^{\mathbb{Z}}_{\lambda}]$ has
two natural bases:
\begin{itemize}
\item The {\em proper standard} basis given by classes of proper
standard modules.
\item The {\em dual Kazhdan-Lusztig} basis given by classes of 
simple modules.
\end{itemize}
After extending the scalars to 
$\mathbb{Q}$, we get three other natural bases:
\begin{itemize}
\item The {\em standard} basis given by classes of 
standard modules.
\item The {\em Kazhdan-Lusztig} basis given by classes of 
indecomposable projective modules.
\item The {\em twisted Kazhdan-Lusztig} basis given by classes of 
indecomposable tilting modules.
\end{itemize}

Similarly to the case of cell modules, we have the following
uniqueness result:

\begin{proposition}[\cite{MS2}]\label{prop937}
The categorification of the module 
$\mathbb{Z}[v,v^{-1}]\mathcal{R}
\otimes_{\mathbb{H}^{\mathfrak{p}}}\mathbb{H}$
described in Theorem~\ref{thm917} does not depend, up to
equivalence, on the choice of $\mathcal{R}$ inside a fixed
two-sided cell.
\end{proposition}

\section{Category $\mathcal{O}$: Koszul duality}\label{s10}

\subsection{Quadratic dual of a positively graded algebra}\label{s10.1}

Let $A=\oplus_{i\in\mathbb{Z}}A_i$ be a {\em positively graded} 
\index{positively graded algebra}
$\mathbb{C}$-algebra, that is:
\begin{itemize}
\item for all $i<0$ we have  $A_i=0$;
\item for all $i\in\mathbb{Z}$ we have $\dim A_i<\infty$;
\item for some $k\in\mathbb{N}$ we have $A_0\cong \mathbb{C}^k$.
\end{itemize}
Denote by $A\text{-}\mathrm{gfmod}$ the category of
locally finite dimensional graded $A$-modules.
Let $\overline{A}$ denote the subalgebra of $A$, generated by
$A_0$ and $A_1$. Clearly, $\overline{A}$ inherits from $A$ a
positive grading. The multiplication in $A$ defines on $A_1$
the structure of an $A_0\text{-}A_0$--bimodule. Consider the {\em free tensor 
algebra} $A_0[A_1]$ of this bimodule. It is defined as follows:
\index{free tensor algebra}
\begin{displaymath}
A_0[A_1]:=\bigoplus_{i\geq 0} A_1^{\otimes i},
\end{displaymath}
where $A_1^{\otimes 0}:=A_0$, for $i>0$ we have
$A_1^{\otimes i}=A_1\otimes_{A_0}A_1\otimes_{A_0}\cdots\otimes_{A_0}A_1$
(with $i$ factors $A_1$), and multiplication is given, as usual, 
by tensoring over $A_0$. The identity maps on $A_0$ and $A_1$
induce an algebra homomorphism $A_0[A_1]\tto A$, whose image 
coincides with $\overline{A}$. Let $I$ be the kernel of this
homomorphism. The algebra $A_0[A_1]$ is graded in the natural way
($A_1^{\otimes i}$ has degree $i$) and $I$ is a homogeneous ideal 
of $A_0[A_1]$. Hence $I=\oplus_{i\in\mathbb{Z}}I_i$.

\begin{definition}\label{def1001}
{\rm  
The algebra $A$ is called {\em quadratic} if $A=\overline{A}$
and $I$ is generated in degree $2$.
} 
\end{definition}
\index{quadratic algebra}

\begin{example}\label{exm1002}
{\em  
Any path algebra of a quiver without relations is graded in the
natural way (each arrow has degree one) and is obviously quadratic. 
Another example of a quadratic algebra is the algebra $D$ of dual 
numbers.
}
\end{example}

Consider the dual $A_0\text{-}A_0$--bimodule 
$A_1^*=\mathrm{Hom}_{\mathbb{C}}(A_1,\mathbb{C})$.
Since $\mathbb{C}$ is symmetric, we may identify 
$(A_1^*)^{\otimes i}$ with $(A_1^{\otimes i})^*$.

\begin{definition}\label{def1003}
{\rm  
The {\em quadratic dual} of $A$ is defined as the algebra
\begin{displaymath}
A^!:=A_0[A_1^*]/(I_2^*),
\end{displaymath}
where $I_2^*:=\{f:A_0[A_1]_2\to\mathbb{C}\text{ such that } f|_{I_2}=0\}$
(under the above identification).
} 
\end{definition}
\index{quadratic dual}

Directly from the definition we have that $(A^!)^!\cong A$
in case $A$ is quadratic. 

\begin{example}\label{exm1004}
{\em  
The quadratic dual of $\mathbb{C}[x]/(x^2)$ is $\mathbb{C}[x]$
and the quadratic dual of $\mathbb{C}[x]$ is $\mathbb{C}[x]/(x^2)$.
On the other hand, for any $k>2$ the quadratic dual of
$\mathbb{C}[x]/(x^k)$ is $\mathbb{C}[x]/(x^2)$ (note that in this
case $\mathbb{C}[x]/(x^k)$ is not quadratic). The quadratic dual
of the path algebra of a quiver without relations is the path
algebra of the opposite quiver with radical square zero.
}
\end{example}

\subsection{Linear complexes of projectives}\label{s10.2}

Denote by $\mathscr{LC}(A)$ the category of {\em linear complexes}
\index{linear complex}
of projective $A$-modules defined as follows: objects of
$\mathscr{LC}(A)$ are complexes
\begin{displaymath}
\mathcal{X}^{\bullet}:\qquad
\dots\to X^{i-1}\to X^{i}\to X^{i+1}\to \dots  
\end{displaymath}
such that for any $i\in\mathbb{Z}$ we have 
$X^i\in\mathrm{add}(A\langle i\rangle)$; morphisms in 
$\mathscr{LC}(A)$ are just usual morphisms between complexes
of graded modules. 

\begin{example}\label{ex1051}
{\rm
The positively graded algebra $A=\mathbb{C}[x]$, where $x$ has degree one,
has a unique graded indecomposable projective module (up to isomorphism
and shift of grading), namely $A$ itself. Then the following complexes
(in which the middle term $A$ is in position zero) are indecomposable objects
of $\mathscr{LC}(A)$:
\begin{displaymath}
\begin{array}{ccccccccccccc}
\dots&\to&0&\to&0&\to&A&\to&0&\to&0&\to&\dots\\ 
\dots&\to&0&\to&A\langle -1\rangle&
\overset{\cdot x}{\to}&A&\to&0&\to&0&\to&\dots,\\ 
\end{array}
\end{displaymath}
moreover, one can show that every indecomposable object in 
$\mathscr{LC}(A)$ is isomorphic, up to a shift of the form 
$\langle s\rangle[-s]$, $s\in\mathbb{Z}$, to one of 
these complexes. The second complex
is both projective and injective in $\mathscr{LC}(A)$.
}
\end{example}

\begin{example}\label{ex1052}
{\rm
The positively graded algebra $D$ of dual numbers (where $x$ has degree one),
has a unique graded indecomposable projective module (up to isomorphism
and shift of grading), namely $D$ itself. Then for any $n$ the 
following complex 
(in which the rightmost occurrence of $D$ is in position zero) 
is an indecomposable objects of $\mathscr{LC}(A)$:
{\small
\begin{displaymath}
\begin{array}{ccccccccccccccccc}
\dots&\to&\hspace{-1mm}0&\to&\hspace{-1mm}D\langle -n\rangle&
\overset{\cdot x}{\to}
&\dots&\overset{\cdot x}{\to}&\hspace{-1mm}D\langle -2\rangle&
\overset{\cdot x}{\to}&\hspace{-1mm}D\langle -1\rangle&
\overset{\cdot x}{\to}&\hspace{-1mm}D&\to&\hspace{-1mm}0
&\to&\dots
\end{array}
\end{displaymath}
}\noindent
The following complex 
(in which the rightmost occurrence of $D$ is in position zero) 
is an indecomposable objects of $\mathscr{LC}(A)$ as well:
{\small
\begin{displaymath}
\begin{array}{ccccccccccccccccc}
\dots&\overset{\cdot x}{\to}&\hspace{-1.5mm}D\langle -n-1\rangle\hspace{-1.5mm}
&\overset{\cdot x}{\to}&\hspace{-1.5mm}D\langle -n\rangle\hspace{-1.5mm}&
\overset{\cdot x}{\to}
&\dots&\overset{\cdot x}{\to}&\hspace{-1.5mm}D\langle -2\rangle\hspace{-1.5mm}&\overset{\cdot x}{\to}&
\hspace{-1.5mm}D\langle -1\rangle\hspace{-1.5mm}
&\overset{\cdot x}{\to}&\hspace{-1.5mm}D\hspace{-1.5mm}&\to&\hspace{-1.5mm}0\hspace{-1.5mm}
&\to&\dots
\end{array}
\end{displaymath}
}\noindent
The last complex is injective in  $\mathscr{LC}(A)$.
The following complex 
(in which the leftmost occurrence of $D$ is in position zero) 
is an indecomposable objects of $\mathscr{LC}(A)$:
{\small
\begin{displaymath}
\begin{array}{cccccccccccccccccc}
\dots&\to&0&\to&D&\overset{\cdot x}{\to}&
\hspace{-1.5mm}D\langle 1\rangle\hspace{-1.5mm}&\overset{\cdot x}{\to}
&\hspace{-1.5mm}D\langle 2\rangle\hspace{-1.5mm}&
\overset{\cdot x}{\to}& &\dots&\overset{\cdot x}{\to}&
\hspace{-1.5mm}D\langle n\rangle\hspace{-1.5mm}&
\overset{\cdot x}{\to}&\hspace{-1.5mm}D\langle n+1\rangle\hspace{-1.5mm}
&\overset{\cdot x}{\to}&\dots
\end{array}
\end{displaymath}
}\noindent
This complex is projective in  $\mathscr{LC}(A)$.
}
\end{example}

\begin{theorem}[\cite{MVS,MO,MOS}]\label{thm1005}
The category $\mathscr{LC}(A)$ is equivalent to the category 
$A^!\text{-}\mathrm{gfmod}$ of
locally finite dimensional graded $A^!$-modules.
\end{theorem}

Let $e_{\mathtt{i}}$, $\mathtt{i}\in\Lambda$, be a complete and 
irredundant list of primitive idempotents for $A$. 
Then $A_0=\oplus_{\mathtt{i}\in\Lambda}\mathbb{C}
\langle e_{\mathtt{i}}\rangle$. For $\mathtt{i}\in\Lambda$
denote by $\mathcal{P}^{\bullet}(\mathtt{i})$ a minimal projective 
resolution of the simple graded $A$-module $e_{\mathtt{i}}A_0$. 
Since $A$ is positively graded, every indecomposable 
direct summand of $\mathcal{P}^{i}(\mathtt{i})$,
$i\in \mathbb{Z}$, is isomorphic to $Ae_{\mathtt{j}}\langle j\rangle$
for some $\mathtt{j}\in\Lambda$ and $j\leq i$. Taking all summands
of the form $Ae_{\mathtt{j}}\langle i\rangle$ produces
a subcomplex $\overline{\mathcal{P}}^{\bullet}(\mathtt{i})$ of 
$\mathcal{P}^{\bullet}(\mathtt{i})$, called the {\em linear part}
\index{linear part}
of $\mathcal{P}^{\bullet}(\mathtt{i})$. Similarly one defines 
injective (co)resolutions $\mathcal{I}^{\bullet}(\mathtt{i})$
and their linear parts $\overline{\mathcal{I}}^{\bullet}(\mathtt{i})$.
Denote by $\circledast$ the graded duality. Then we have the Nakayama 
functor $\mathrm{N}:=\mathrm{Hom}_A({}_-,A)^{\circledast}$,
which induces an equivalence between the additive categories of
graded projective and graded injective $A$-modules.

\begin{proposition}[\cite{MO,MOS}]\label{prop1006}
\begin{enumerate}[$($a$)$]
\item\label{prop1006.1} Every simple object of $\mathscr{LC}(A)$ 
is isomorphic, up to a shift of the form $\langle s\rangle[-s]$,
$s\in\mathbb{Z}$, to $Ae_{\mathtt{i}}$, considered
as a complex concentrated in position zero.
\item\label{prop1006.2} Every indecomposable injective object of 
$\mathscr{LC}(A)$ is isomorphic, up to a similar shift, to 
$\overline{\mathcal{P}}^{\bullet}(\mathtt{i})$ for some
$\mathtt{i}\in\Lambda$.
\item\label{prop1006.3} Every indecomposable projective object of 
$\mathscr{LC}(A)$ is isomorphic, up to a similar shift, to 
$\mathrm{N}^{-1}\,\overline{\mathcal{I}}^{\bullet}(\mathtt{i})$ 
for some $\mathtt{i}\in\Lambda$.
\end{enumerate}
\end{proposition}

Denote by $\mathfrak{Q}$ the full subcategory of 
$\mathscr{LC}(A)$ with objects being
all $\mathrm{N}^{-1}\,\overline{\mathcal{I}}^{\bullet}(\mathtt{i})$, 
$\mathtt{i}\in\Lambda$, and their shifts as above. The group
$\mathbb{Z}$ acts on $\mathfrak{Q}$ by  shifts 
of the form $\langle s\rangle[-s]$. This equips the quotient of the
endomorphism algebra of $\mathfrak{Q}$ modulo this action with the
structure of a $\mathbb{Z}$-graded algebra. From 
Theorem~\ref{thm1005} and Proposition~\ref{prop1006} it follows that
this $\mathbb{Z}$-graded algebra is isomorphic to $A^!$. Hence 
$\mathfrak{Q}$ can be regarded as a complex of graded
$A\text{-}A^!$-bimodules which is projective on both sides. 
For any category $\mathcal{A}$,
whose objects are complexes of graded modules, denote by 
$\mathcal{A}^{\downarrow}$ and $\mathcal{A}^{\uparrow}$ the full
subcategories of $\mathcal{A}$ whose nonzero components are concentrated
inside the corresponding regions as shown on Figure~\ref{fig1} (see \cite{MOS}
for explicit formulae).
\begin{figure}[tb]
\special{em:linewidth 0.4pt}
\unitlength 0.80mm
\linethickness{0.4pt}
\begin{center}
\begin{picture}(120.00,40.00)
\drawline(50.00,40.00)(50.00,10.00)
\drawline(51.00,13.00)(50.00,10.00)
\drawline(49.00,13.00)(50.00,10.00)
\drawline(00.00,25.00)(100.00,25.00)
\drawline(97.00,26.00)(100.00,25.00)
\drawline(97.00,24.00)(100.00,25.00)
\drawline(40.00,30.00)(40.00,10.00)
\drawline(40.00,30.00)(10.00,10.00)
\drawline(60.00,20.00)(60.00,40.00)
\drawline(60.00,20.00)(90.00,40.00)
\put(50.00,05.00){\makebox(0,0)[cc]{(grading) degree}}
\put(110.00,25.00){\makebox(0,0)[cc]{position}}
\put(30.00,15.00){\makebox(0,0)[cc]{$\mathscr{A}^{\downarrow}$}}
\put(70.00,35.00){\makebox(0,0)[cc]{$\mathscr{A}^{\uparrow}$}}
\end{picture}
\end{center}
\caption{The supports of objects from the categories $\mathscr{A}^{\downarrow}$ and $\mathscr{A}^{\uparrow}$.}\label{fig1}
\end{figure}
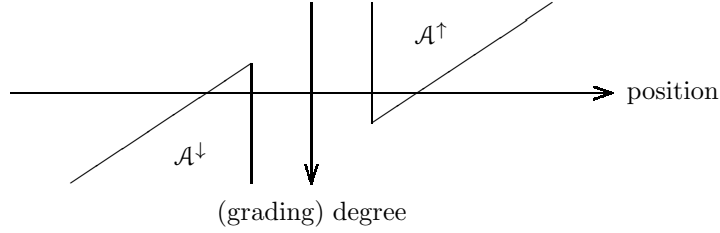
With this notation we have the following result proved in \cite{MOS}:

\begin{theorem}[Quadratic duality]\label{thm1007}
\begin{enumerate}[$($a$)$]
\item\label{thm1007.1} There is a pair of adjoint functors as
follows:
\begin{displaymath}
\xymatrix{ 
\mathcal{D}^{\downarrow}(A\text{-}\mathrm{gfmod})
\ar@/^1pc/[rrrr]^{\mathrm{K}:=\mathcal{R}\mathrm{Hom}_A(\mathfrak{Q},{}_-)}
&&&&
\mathcal{D}^{\uparrow}(A^!\text{-}\mathrm{gfmod})
\ar@/^1pc/[llll]^{\mathrm{K}':=\mathfrak{Q}
\overset{\mathcal{L}}{\otimes}_{A^!}{}_-}
}
\end{displaymath}
\item\label{thm1007.2} $\mathrm{K}$ maps simple $A$-modules to
injective $A^!$-modules and $\mathrm{K}'$ maps simple $A^!$-modules
to projective $A$-modules. 
\item\label{thm1007.3} For all $i,j\in\mathbb{Z}$ we have
$\mathrm{K}\circ\langle j\rangle\circ[i]=
\langle -j\rangle\circ[i+j]\circ\mathrm{K}$.
\end{enumerate}
\end{theorem}

The functors $\mathrm{K}$ and $\mathrm{K}'$ are called {\em quadratic
duality} functors.
\index{quadratic duality functor}

\subsection{Koszul duality}\label{s10.3}

In a special case Theorem~\ref{thm1007} can be improved, see \cite{MOS}
for details.

\begin{definition}\label{def1008}
{\rm
A positively graded algebra $A$ is called Koszul provided that
$\overline{\mathcal{P}}^{\bullet}(\mathtt{i})=
\mathcal{P}^{\bullet}(\mathtt{i})$ for every $\mathtt{i}\in\Lambda$.
}
\end{definition}

Typical examples of Koszul algebras are $\mathbb{C}[x]$, 
$D$, quiver algebras without relations and quiver algebras 
with radical square zero. 

\begin{theorem}[Koszul duality]\label{thm1009}
\begin{enumerate}[$($a$)$]
\item\label{thm1009.1} The functors $\mathrm{K}$ and $\mathrm{K}'$
from Theorem~\ref{thm1007} are mutually inverse equivalences of
categories if and only if $A$ is Koszul.
\item\label{thm1009.2} If $A$ is a finite dimensional Koszul algebra
of finite global dimension, then the functors $\mathrm{K}$ and $\mathrm{K}'$
restrict to an equivalence between $\mathcal{D}^b(A\text{-}\mathrm{gmod})$
and $\mathcal{D}^b(A^!\text{-}\mathrm{gmod})$.
\end{enumerate}
\end{theorem}

In the case of a Koszul algebra $A$ the functors $\mathrm{K}$ and 
$\mathrm{K}'$ are called {\em Koszul duality} functors.
\index{Koszul duality functor}

\begin{corollary}\label{cor10091}
If $A$ is Koszul, then $(A^!)^{\mathrm{op}}$ is isomorphic to
the Yoneda algebra  $\mathrm{Ext}_A^*(A_0,A_0)$.
\end{corollary}

The algebra $(A^!)^{\mathrm{op}}$ is called the {\em Koszul dual} of 
\index{Koszul dual}
$A$. Koszul algebras and duals were introduced in \cite{Pr}. 
The Koszul duality theorem as an equivalence of derived categories 
was  established in \cite{BGS}. For the category $\mathcal{O}$
we have the following:

\begin{theorem}[\cite{So0,BGS}]\label{thm1011}
Let $\mathfrak{p}$ be a parabolic subalgebra of $\mathfrak{g}$ containing
$\mathfrak{b}$, $\lambda\in\mathfrak{h}^*_{\mathrm{dom}}$ be 
integral and regular, and $\mu\in\mathfrak{h}^*_{\mathrm{dom}}$ be 
integral such that $W_{\mu}=W_{\mathfrak{p}}$. Then both
$B_{\lambda}^{\mathfrak{p}}$ and $B_{\mu}$ are Koszul, moreover,
they are Koszul dual to each other. In particular, 
$B_{\lambda}$ is Koszul and Koszul self-dual.
\end{theorem}

\begin{example}\label{ex1053}
{\rm 
The associative algebra $B_0$ describing the principal block of
category $\mathcal{O}$ for $\mathfrak{gl}_2$ is given by \eqref{eq457}. 
Using the notation from Subsection~\ref{s4.7} 
(and identifying $B_0\text{-}\mathrm{mod}$ with $\mathcal{O}_0$), 
an indecomposable
projective object in $\mathscr{LC}(B_0)$ is isomorphic, up to a
shift of the form $\langle s\rangle[-s]$, for some $s\in\mathbb{Z}$, 
to one of the following two complexes (here the rightmost nonzero
component is in position zero):
\begin{displaymath}
\begin{array}{ccccccccc}
0&\to&0&\to&P(0,0)&\to& P(-1,1)&\to&0;\\
0&\to&P(0,0)&\to& P(-1,1)&\to&P(0,0)&\to&0.\\
\end{array}
\end{displaymath}
As objects of $\mathscr{LC}(B_0)$, both these complexes are uniserial,
moreover, the Loewy structure of the first one is similar to that
of $P(0,0)$, while the Loewy structure of the second one is
similar to that of $P(-1,1)$. In this easy example such comparison 
of the Loewy structures for indecomposable projectives in
$\mathscr{LC}(B_0)$ and $\mathcal{O}_0$ implies Koszul 
self-duality of $B_0$.
}
\end{example}

As the above example suggests, the Koszul self-duality of 
$B_{\lambda}$ does not preserve
primitive idempotents but rather acts on their indexing set $W$
as the map $w\mapsto w^{-1}w_o$.
Theorem~\ref{thm1011} is generalized to arbitrary categories 
${}^{\infty}_{\,\,\lambda}\mathcal{H}^{1}_{\mu}$ in \cite{Bac}.
An approach to the Koszul dualities on $\mathcal{O}$ via
linear complexes was proposed in \cite{Ma4} and was further explored
in \cite{MOS,Ma2,Ma3}. There is a well-developed theory of 
Koszul duality for abstract quasi-hereditary algebras, see
\cite{ADL,MO,Ma5,Ma6}.

Instead of the category $\mathscr{LC}(A)$ one could alternatively
consider the category of linear complexes of injective 
$A$-modules. In the case of quasi-hereditary algebra $A$ one can
also consider the category of linear complexes of tilting 
$A$-modules. In the case of $\mathcal{O}_{\lambda}$ all these
categories are equivalent via $\mathrm{T}_{w_o}$ (or
$\mathrm{C}_{w_o}$).

Many structural modules in $\mathcal{O}_{\lambda}$ admit linear
projective or linear injective or linear tilting ``resolutions''.
For example (see \cite{Ma2}):
\begin{itemize}
\item Simple modules admit both linear injective, projective and
tilting resolutions.
\item Standard (Verma) modules admit linear projective and
tilting resolutions.
\item Costandard modules admit linear injective and
tilting resolutions.
\item Projective modules in $\mathcal{O}_{\lambda}^{\mathfrak{p}}$ 
admit linear projective and tilting resolutions (in $\mathcal{O}_{\lambda}$).
\item Tilting modules in $\mathcal{O}_{\lambda}^{\mathfrak{p}}$ 
admit linear tilting resolutions (in $\mathcal{O}_{\lambda}$).
\end{itemize}
As one could expect after Proposition~\ref{prop1006}, all these resolutions
play some role in the structure of the corresponding category of
linear complexes. For example, linear tilting resolutions of simple
modules are tilting objects in the category of linear complexes of
tilting modules. In fact, the category of linear complexes of tilting
modules seems to be the most ``symmetric'' one.

\subsection{Koszul dual functors}\label{s10.4}

Let $A$ be a finite dimensional Koszul algebra of finite global dimension
and  $\mathrm{K}:\mathcal{D}^b(A\text{-}\mathrm{gmod})\to
\mathcal{D}^b(A^!\text{-}\mathrm{gmod})$ the corresponding Koszul
duality functor. Let $\mathrm{F}$ be an endofunctor of
$\mathcal{D}^b(A\text{-}\mathrm{gmod})$ and $\mathrm{G}$ an 
endofunctor of $\mathcal{D}^b(A^!\text{-}\mathrm{gmod})$.

\begin{definition}\label{def1012}
{\em  
The functors $\mathrm{F}$ and $\mathrm{G}$ are called 
{\em Koszul dual} provided that the following diagram commutes
\index{Koszul dual functor}
up to an isomorphism of functors:
\begin{displaymath}
\xymatrix{ 
\mathcal{D}^b(A\text{-}\mathrm{gmod})\ar[rrrr]^{\mathrm{K}}
\ar[d]_{\mathrm{F}}&&&&
\mathcal{D}^b(A^!\text{-}\mathrm{gmod})\ar[d]^{\mathrm{G}}\\
\mathcal{D}^b(A\text{-}\mathrm{gmod})\ar[rrrr]^{\mathrm{K}}&&&&
\mathcal{D}^b(A^!\text{-}\mathrm{gmod})\\
}
\end{displaymath}
}
\end{definition}

For the category $\mathcal{O}$ we have the following pairs
of Koszul dual functors:

\begin{theorem}[\cite{RH,MOS}]\label{thm1015}
Under the identification of $B_{\lambda}^{\mathfrak{p}}$
and $B_{\mu}^!$, given by Theorem~\ref{thm1011},
the following functors (all corresponding to
a fixed simple reflection $s$) are Koszul dual to each other
(up to an appropriate shift in grading and position):
\begin{enumerate}[$($a$)$]
\item\label{thm1015.1} Translation to the wall and 
derived Zuckerman functor.
\item\label{thm1015.2} Translation out of the wall and 
canonical inclusion.
\item\label{thm1015.3} Derived shuffling and derived twisting.
\item\label{thm1015.4} Derived coshuffling and derived completion.
\end{enumerate}
\end{theorem}

The statements of Theorem~\ref{thm1015}\eqref{thm1015.1}-\eqref{thm1015.2}
were conjectured in \cite{BGS} and first proved in \cite{RH} using
dg-categories. In \cite{MOS} one finds a unifying approach to
all statements based on the study of the 
category $\mathscr{LC}(B_{\lambda})$.

\subsection{Alternative categorification of the 
permutation module}\label{s10.5}

As we know from Proposition~\ref{prop811}, the action of 
$\cS^{\mathbb{Z}}$ (resp. $\cS$) on 
${}^{\mathbb{Z}}\mathcal{O}_{\lambda}^{\mathfrak{p}}$
(resp. $\mathcal{O}_{\lambda}^{\mathfrak{p}}$) categorifies
the parabolic module $\mathcal{M}_{-v}$ (resp. the induced sign module).
Since all our functors are exact, we can derive all involved
categories and get the same result. 

The point in deriving the picture is that we can do more
with the derived picture. Namely, we may now apply
the Koszul duality functor $\mathrm{K}$ to obtain
an action of $\cS^{\mathbb{Z}}$ on the bounded derived
category of graded modules over the Koszul dual of
$B_{\lambda}^{\mathfrak{p}}$. By Theorem~\ref{thm1011}, the 
Koszul dual of  $B_{\lambda}^{\mathfrak{p}}$ is
$B_{\mu}$, where $\mu\in\mathfrak{h}^*_{\mathrm{dom}}$
is integral and such that $W_{\mu}=W_{\mathfrak{p}}$. 

Now we have to be careful, since $\mathrm{K}$ does not
obviously commute with shifts (both in position and grading).
Instead, we have the formula of Theorem~\ref{thm1007}\eqref{thm1007.3}.
Because of our conventions, this formula implies that 
$\mathrm{K}$ induces an automorphism $\psi$ of $\mathbb{H}$
given by  $v\mapsto -v^{-1}$. This means that the image,
under the Koszul duality functor $\mathrm{K}$, of the
$\mathbb{H}$-module $\mathcal{M}_{-v}$ becomes the 
$\mathbb{H}$-module $\mathcal{M}_{v^{-1}}$. Taking 
Theorem~\ref{thm1015} into account, we obtain:

\begin{proposition}\label{prop1017}
Let $\mu$ be as above. The action of derived Zuckerman  
functors on $\mathcal{D}^b(B_{\mu}\text{-}\mathrm{gmod})$
categorifies (in the image of the Kazhdan-Lusztig basis under
$\psi$) the parabolic module $\mathcal{M}_{v^{-1}}$.
\end{proposition}

Taking into account the relation between twisting and Zuckerman 
functors from Proposition~\ref{prop604}, the na{\"\i}ve picture
of the corresponding categorification of the permutation
module can be described as follows:

\begin{corollary}\label{cor1019}
Let $\mu$ be as above. The action of derived twisting functors 
on $\mathcal{D}^b(B_{\mu}\text{-}\mathrm{gmod})$
gives a na{\"\i}ve categorification of the 
permutation module $W_{\mu}$ in the standard basis of $W$.
\end{corollary}

In many cases, studying some categorification picture, it might
be useful to switch to the Koszul dual (if it exists), where
computation might turn out to be much easier (for example, some
derived functors might become exact after taking the Koszul dual).

\section{$\mathfrak{sl}_2$-categorification: simple finite
dimensional modules}\label{s11}

\subsection{The algebra $U_v(\mathfrak{sl}_2)$}\label{s11.1}

Recall that $\mathbb{N}_0$ denotes the set of non-negative integers.
Denote by $\mathbb{C}(v)$ the field of rational functions with
complex coefficients in an indeterminate $v$. Recall
the following definition (see e.g. \cite{Ja3}):

\begin{definition}\label{def1101}
{\rm  
The quantum group $U_v(\mathfrak{sl}_2)$ is the associative 
$\mathbb{C}(v)$-algebra with generators $E$, $F$, $K$, $K^{-1}$
and relations
\begin{displaymath}
KE=v^2EK,\, KF=v^{-2}FK,\, KK^{-1}=K^{-1}K=1,\,
EF-FE=\frac{K-K^{-1}}{v-v^{-1}}.
\end{displaymath}
}
\end{definition}

The algebra $U_v(\mathfrak{sl}_2)$ is a Hopf algebra with 
comultiplication $\Delta$ given by 
\begin{displaymath}
\Delta(E)=1\otimes E+E\otimes K,\quad 
\Delta(F)=K^{-1}\otimes F+F\otimes 1,\quad 
\Delta(K^{\pm 1})=K^{\pm 1}\otimes K^{\pm 1}.
\end{displaymath}

For $a\in\mathbb{Z}$ set
\begin{displaymath}
[a]:=\frac{v^a-v^{-a}}{v-v^{-1}}
\end{displaymath}
and for all $n\in\mathbb{N}$ put
\begin{displaymath}
[n]!:=[1][2]\cdots [n].
\end{displaymath}
Set $[0]!=1$ and define
\begin{displaymath}
\left[\begin{array}{c}a\\n\end{array}\right]:=
\frac{[a][a-1]\cdots [a-n+1]}{[n]!}.
\end{displaymath}
Define also 
\begin{displaymath}
E^{(a)}:=\frac{E^a}{\left[a\right]!},\qquad
F^{(a)}:=\frac{F^a}{\left[a\right]!}.
\end{displaymath}

\begin{proposition}[\cite{Ja3}]\label{prop1102}
\begin{enumerate}[$($a$)$]
\item\label{prop1102.1} The algebra $U_v(\mathfrak{sl}_2)$ has 
no zero divisors.
\item\label{prop1102.2} The algebra $U_v(\mathfrak{sl}_2)$ has a 
PBW basis consisting of monomials $F^sK^nE^r$, where $r,s\in\mathbb{N}_0$ 
and $n\in\mathbb{Z}$.
\item\label{prop1102.3} The center of $U_v(\mathfrak{sl}_2)$ is 
generated by the {\em quantum Casimir element}
\index{quantum Casimir element}
\begin{displaymath}
C:=FE+\frac{Kv+K^{-1}v^{-1}}{(v-v^{-1})^2}. 
\end{displaymath}
\end{enumerate}
\end{proposition}

\subsection{Finite dimensional representations of  
$U_v(\mathfrak{sl}_2)$}\label{s11.2}

For $n\in\mathbb{N}_0$ let $\mathcal{V}_n$ and $\hat{\mathcal{V}}_n$ 
denote the $\mathbb{C}(v)$-vector space with basis 
$\{w_0,w_1,\dots,w_n\}$. Define linear operators
$E$, $F$, $K^{\pm 1}$ on $\mathcal{V}_n$ using the following formulae:
\begin{equation}\label{eq1191}
Ew_k=[k+1]w_{k+1},\quad Fw_k=[n-k+1]w_{k-1},\quad
K^{\pm 1}w_k=v^{\pm(2k-n)}w_k
\end{equation}
(under the convention $w_{-1}=w_{n+1}=0$). Define linear operators
$E$, $F$, $K^{\pm 1}$ on $\hat{\mathcal{V}}_n$ using the following formulae:
\begin{equation}\label{eq1197}
Ew_k=[k+1]w_{k+1},\quad Fw_k=-[n-k+1]w_{k-1},\quad
K^{\pm 1}w_k=-v^{\pm(2k-n)}w_k.
\end{equation}
By a direct calculation we get the following:

\begin{lemma}\label{lem1103}
Formulae \eqref{eq1191} and \eqref{eq1197}  define on both
$\mathcal{V}_n$ and $\hat{\mathcal{V}}_n$ the structure of a
simple $U_v(\mathfrak{sl}_2)$-module of dimension $n+1$.
\end{lemma}

For $n=1$ the module $\mathcal{V}_1$ is called the {\em natural}
$U_v(\mathfrak{sl}_2)$-module.
\index{natural module}

\begin{theorem}[\cite{Ja3}]\label{thm1104}
\begin{enumerate}[$($a$)$]
\item\label{thm1104.1} For every $n\in\mathbb{N}_0$ there is a unique
simple $U_v(\mathfrak{sl}_2)$-module of dimension $n+1$, namely
$\mathcal{V}_n$, on which 
$K$ acts semi-simply with powers of $v$ as eigenvalues.
\item\label{thm1104.15} For every $n\in\mathbb{N}_0$ there is a unique
simple $U_v(\mathfrak{sl}_2)$-module of dimension $n+1$, namely
$\hat{\mathcal{V}}_n$,  on which 
$K$ acts semi-simply with minus powers of $v$ as eigenvalues.
\item\label{thm1104.16} Every simple 
$U_v(\mathfrak{sl}_2)$-module of dimension $n+1$ is isomorphic
to either $\mathcal{V}_n$ or $\hat{\mathcal{V}}_n$ and these two
modules are not isomorphic.
\item\label{thm1104.2} Every finite-dimensional  
$U_v(\mathfrak{sl}_2)$-module on which $K$ acts semi-simply is
completely reducible.
\end{enumerate}
\end{theorem}

Let $X,Y$ be two $U_v(\mathfrak{sl}_2)$-modules. Using the 
comultiplication $\Delta$ we can turn $X\otimes Y$ into a
$U_v(\mathfrak{sl}_2)$-module. As usual, for $n\in\mathbb{N}_0$
we denote by $X^{\otimes n}$ the product 
$\underbrace{X\otimes X\otimes\cdots\otimes X}_{n\text{ factors}}$.

\begin{proposition}\label{prop1105}
\begin{enumerate}[$($a$)$]
\item\label{prop1105.1}
For $n\in\mathbb{N}_0$ we have:
\begin{displaymath}
\mathcal{V}_1\otimes \mathcal{V}_n\cong
\begin{cases}
\mathcal{V}_1, & n=0;\\
\mathcal{V}_{n-1}\oplus \mathcal{V}_{n+1}, & \text{otherwise}. 
\end{cases}
\end{displaymath}
\item\label{prop1105.2}
For every $n\in\mathbb{N}_0$ the module $\mathcal{V}_n$
appears in $\mathcal{V}_1^{\otimes n}$ with multiplicity one
and all other irreducible direct summands of $\mathcal{V}_1^{\otimes n}$
are isomorphic to  $\mathcal{V}_k$ for some $k<n$.
\end{enumerate}
\end{proposition}

\subsection{Categorification of $\mathcal{V}_1^{\otimes n}$}\label{s11.3}

Let $n\in\mathbb{N}_0$. Set $\mathfrak{gl}_0:=0$ (then
$U(\mathfrak{gl}_0)=\mathbb{C}$). Consider the algebra
$\mathfrak{gl}_n$. For $i=0,1,2,\dots,n$ denote by 
$\lambda_i$ the dominant integral weight such that $\lambda_i+\rho=
(1,1,\dots,1,0,0,\dots,0)$, where the number of $1$'s equals $i$. 
The weight $\lambda_i$ is singular and $W_{\lambda_i}\cong
\mathbb{S}_{i}\oplus \mathbb{S}_{n-i}$ is the parabolic subgroup of
$\mathbb{S}_{n}$ corresponding to permutations of the first $i$
and the last $n-i$ elements. Consider the categories
\begin{displaymath}
\mathfrak{V}_n:=\bigoplus_{i=0}^n\mathcal{O}_{\lambda_i}\quad\text{ and }
\quad\mathfrak{V}_n^{\mathbb{Z}}:=
\bigoplus_{i=0}^n\mathcal{O}_{\lambda_i}^{\mathbb{Z}}.
\end{displaymath}
Directly from the definition we have:

\begin{lemma}\label{lem1106}
For $i=0,1,\dots,n$ the Grothendieck group
$[\mathcal{O}_{\lambda_i}^{\mathbb{Z}}]$ is a free
$\mathbb{Z}[v,v^{-1}]$-module of rank $\binom{n}{i}$.
In particular, $[\mathfrak{V}_n^{\mathbb{Z}}]$ is a free
$\mathbb{Z}[v,v^{-1}]$-module of rank $2^n$.
\end{lemma}

Let $V:=\mathbb{C}^n$ be the natural representation of $\mathfrak{gl}_n$.
For $i=0,1,\dots,n-1$ denote by $\mathrm{E}_i$ the projective functor
from $\mathcal{O}_{\lambda_i}$ to $\mathcal{O}_{\lambda_{i+1}}$
given by tensoring with $V$ and then projecting onto
$\mathcal{O}_{\lambda_{i+1}}$. Define the endofunctors
$\mathrm{E}$ and $\mathrm{F}$ of $\mathfrak{V}_n$ as follows:
\begin{displaymath}
\mathrm{E}:=\bigoplus_{i=0}^{n-1}\mathrm{E}_i,\qquad
\mathrm{F}:=\bigoplus_{i=0}^{n-1}\mathrm{E}_i^* 
\end{displaymath}
(our convention is that $\mathrm{E}\mathcal{O}_{\lambda_{n}}=0$,
$\mathrm{F}\mathcal{O}_{\lambda_{0}}=0$ and $\mathrm{E}_n=0$). 
Set $\mathrm{F}_{i+1}:=\mathrm{E}_i^*$. By a direct calculation on the 
level of the Grothendieck group one easily shows the following:

\begin{proposition}\label{prop1107}
The action of $\mathrm{E}$ and $\mathrm{F}$ on $\mathfrak{V}_n$ 
gives a weak categorification of the $n$-th tensor power of the
natural $\mathfrak{gl}_2$-module.
\end{proposition}

We would like, however, to upgrade this at least to the level of the
quantum algebra. For this we have to carefully define 
graded lifts for our functors $\mathrm{E}_i$. For $i=0,1,\dots,n-1$
consider the integral dominant weights 
$\mu_i=(2,2,\dots,2,1,0,0,\dots,0)-\rho$ and
$\nu_i=(2,2,\dots,2,0,0,0,\dots,0)-\rho$, 
where in both cases the number of $2$'s equals $i$. Then the functor
$\mathrm{E}_i$ can be realized as the composition of the following
projective functors:
\begin{equation}\label{eq1193}
\xymatrix{ 
\mathcal{O}_{\lambda_i}\ar[rr]^{\theta_{\lambda_i,\nu_i}}&&
\mathcal{O}_{\nu_i}\ar[rr]^{\theta_{\nu_i,s_{i+1}\cdot \mu_i}}&&
\mathcal{O}_{\mu_i}\ar[rr]^{\theta_{\mu_i,\nu_{i+1}}}&&
\mathcal{O}_{\nu_{i+1}}\ar[rr]^{\theta_{\nu_{i+1},\lambda_{i+1}}}&&
\mathcal{O}_{\lambda_{i+1}}
}
\end{equation}
Here the first and the last functors are equivalences, the
second functor is translation out of the wall and the third functor
is translation onto the wall (for $i=n-1$ the second functor
disappears and $\mathrm{E}_i$ is just translation to the wall).

Consider now usual graded lifts of all involved categories.
Let equivalences in \eqref{eq1193} be lifted as equivalences
of degree zero. Use Theorem~\ref{thm705} to lift translation
to and out of the wall such that they correspond to restriction 
and induction for the associated algebras of invariants in the
coinvariants. This determines a graded lift of $\mathrm{E}_i$,
which we will denote by the same symbol, abusing notation.
This defines graded lifts for functors $\mathrm{E}$ and $\mathrm{F}$
from the above, which we again will denote by the same symbols.
Finally, define $\mathrm{K}_i$ as the endofunctor
$\langle n-2i\rangle$ of $\mathcal{O}_{\lambda_i}^{\mathbb{Z}}$
and set 
\begin{displaymath}
\mathrm{K}:=\bigoplus_{i=0}^{n}\mathrm{K}_i. 
\end{displaymath}

\begin{theorem}[Categorification of $\mathcal{V}_1^{\otimes n}$,
\cite{BFK,St,FKS}]
\label{thm1108}
\begin{enumerate}[$($a$)$]
\item\label{thm1108.1} The functors $\mathrm{E}$, $\mathrm{F}$
and $\mathrm{K}$ satisfy the relations
\begin{displaymath}
\mathrm{K}\mathrm{E}\cong \mathrm{E}\mathrm{K}\langle -2\rangle, \quad
\mathrm{K}\mathrm{F}\cong \mathrm{F}\mathrm{K}\langle 2\rangle, \quad
\mathrm{K}\mathrm{K}^{-1}\cong\mathrm{K}^{-1}\mathrm{K}\cong\mathrm{Id}. 
\end{displaymath}
\item\label{thm1108.2} For $i=0,1,\dots,n$ there are isomorphisms
\begin{displaymath}
\mathrm{E}_{i-1}\mathrm{F}_i\oplus
\bigoplus_{r=1}^{n-i-1}\mathrm{Id}\langle 2r+2i+1-n\rangle\cong
\mathrm{F}_{i+1}\mathrm{E}_i\oplus
\bigoplus_{r=1}^{i-1}\mathrm{Id}\langle n+2r+1-2i\rangle.
\end{displaymath}
\item\label{thm1108.4} In the Grothendieck group we have the
equality
\begin{displaymath}
(v-v^{-1})([\mathrm{E}_{i-1}][\mathrm{F}_i]-
[\mathrm{F}_{i+1}][\mathrm{E}_i])=[\mathrm{K}_i]-[\mathrm{K}_i^{-1}] 
\end{displaymath}
\item\label{thm1108.3} The functor
$\mathrm{E}_{i-1}\mathrm{F}_i$ is a summand of 
$\mathrm{F}_{i+1}\mathrm{E}_i$ if $n-2i>0$, the functor
$\mathrm{F}_{i+1}\mathrm{E}_i$ is a summand of 
$\mathrm{E}_{i-1}\mathrm{F}_i$ if $n-2i<0$.
\item\label{thm1108.5} There is an isomorphism of 
$U_v(\mathfrak{sl}_2)$-modules, 
$[\mathfrak{V}_n^{\mathbb{Z}}]\cong\mathcal{V}_1^{\otimes n}$, where the
action of $U_v(\mathfrak{sl}_2)$ on the left hand side is given by
the exact functors $\mathrm{E}$, $\mathrm{F}$ and $\mathrm{K}^{\pm 1}$. 
\end{enumerate}
\end{theorem}

\begin{proof}[Idea of the proof.]
Use classification of projective functors and check all
equalities on the level of the Grothendieck group.
\end{proof}

From Theorem~\ref{thm1108} we see that the module $\mathcal{V}_1^{\otimes n}$
comes equipped with the following natural bases:
\begin{itemize}
\item the {\em standard basis} given by classes of Verma modules;
\item the {\em twisted canonical basis} given by classes of indecomposable
projective modules;
\index{twisted canonical basis}
\item the {\em dual canonical basis} given by classes of simple modules;
\index{dual canonical basis}
\item the {\em canonical basis} given by classes of indecomposable
tilting modules.
\index{canonical basis}
\end{itemize}
The canonical basis can be defined, similarly to Kazhdan-Lusztig
basis, as the set of self-dual elements with respect to some 
anti-involution. The dual of the canonical basis is the dual
in the sense of a certain bilinear form, which is categorified 
via the usual $\mathrm{Ext}$-form twisted by $\mathrm{T}_{w_o}$.

\begin{example}\label{exm1114}
{\rm 
In the case $n=0$ we have $\mathcal{O}_{\lambda_0}=
\mathbb{C}\text{-}\mathrm{mod}$ and
both $\mathrm{F}$ and $\mathrm{E}$ are zero functors.
The functor $\mathrm{K}$ is just the identity.
} 
\end{example}

\begin{example}\label{exm1111}
{\rm 
In the case $n=1$ we have $\mathcal{O}_{\lambda_0}=
\mathcal{O}_{\lambda_1}=\mathbb{C}\text{-}\mathrm{mod}$ and
both $\mathrm{F}$ and $\mathrm{E}$ are the identity functors
between the two parts of $\mathfrak{V}_1$.
The functor $\mathrm{K}$ acts as $\langle -1\rangle$ on the
part annihilated by $\mathrm{E}$ and as $\langle 1\rangle$
on the other part.
} 
\end{example}

\begin{example}\label{exm1112}
{\rm 
In the case $n=2$ we have $\mathcal{O}_{\lambda_0}=
\mathcal{O}_{\lambda_2}=\mathbb{C}\text{-}\mathrm{mod}$ and
$\mathcal{O}_{\lambda_1}$ is the category of modules over the
algebra from Subsection~\ref{s4.7}. The functors $\mathrm{E}_0$
and $\mathrm{E}_1$ are translations out of the wall and onto
the wall, respectively.
} 
\end{example}

\subsection{Categorification of $\mathcal{V}_n$}\label{s11.4}

The weight spaces of the categorified finite dimensional 
$U_v(\mathfrak{sl}_2)$-module $\mathcal{V}_1^{\otimes n}$
are certain singular blocks of the category $\mathcal{O}$.
The action of $U_v(\mathfrak{sl}_2)$ on this categorified
picture is given by projective functors. Every singular
block of $\mathcal{O}$ contains a unique (up to isomorphism)
indecomposable projective injective module. Projective functors
preserve (the additive category of) projective injective modules. 
Therefore we may restrict our $U_v(\mathfrak{sl}_2)$-action
to this category, which also induces an $U_v(\mathfrak{sl}_2)$-action
on the corresponding abelianization. This suggests the following:

Denote by $\hat{\mathfrak{V}_n^{\mathbb{Z}}}$ the full subcategory
of $\mathfrak{V}_n^{\mathbb{Z}}$ consisting of all modules which 
do not have maximal Gelfand-Kirillov dimension (alternatively, whose
Gelfand-Kirillov dimension is strictly smaller than the 
Gelfand-Kirillov dimension of a Verma module). Then 
$\hat{\mathfrak{V}_n^{\mathbb{Z}}}$ is a Serre subcategory of 
$\mathfrak{V}_n^{\mathbb{Z}}$ generated by all simples having
not maximal Gelfand-Kirillov dimension (i.e. simples which are not
Verma modules). 

\begin{theorem}[Categorification of $\mathcal{V}_n$, \cite{FKS}]
\label{thm1109}
\begin{enumerate}[$($a$)$]
\item\label{thm1109.1} The functorial action of 
$U_v(\mathfrak{sl}_2)$ on $\mathfrak{V}_n^{\mathbb{Z}}$
preserves $\hat{\mathfrak{V}_n^{\mathbb{Z}}}$ and induces
a functorial $U_v(\mathfrak{sl}_2)$-action on 
$\mathfrak{V}_n^{\mathbb{Z}}/\hat{\mathfrak{V}_n^{\mathbb{Z}}}$.
\item\label{thm1109.2} The functorial 
$U_v(\mathfrak{sl}_2)$-action on 
$\mathfrak{V}_n^{\mathbb{Z}}/\hat{\mathfrak{V}_n^{\mathbb{Z}}}$
categorifies the $U_v(\mathfrak{sl}_2)$-module $\mathcal{V}_n$.
\end{enumerate}
\end{theorem}

\begin{proof}[Idea of the proof.]
Being direct summands of tensoring with finite dimensional modules,
projective functors do not increase Gelfand-Kirillov dimension.
This implies claim \eqref{thm1109.1}. Claim \eqref{thm1109.2}
follows from claim \eqref{thm1109.1} and Theorem~\ref{thm1108}
by comparing characters.
\end{proof}

Note that the weight spaces of $\mathcal{V}_n$ are categorified as
module categories over certain algebras of invariants in
coinvariants. The latter have geometric interpretation as
cohomology algebras of certain flag varieties. Hence all the above
categorification pictures may be interpreted geometrically,
see \cite{FKS} for details.

The categorification of $\mathcal{V}_n$ constructed in 
Theorem~\ref{thm1109} equips $\mathcal{V}_n$ with two 
natural bases:
\begin{itemize}
\item the {\em canonical basis} given by classes of indecomposable
projective modules;
\index{canonical basis}
\item the {\em dual canonical basis} given by classes of simple modules.
\index{dual canonical basis}
\end{itemize}

\begin{example}\label{exm1134}
{\rm 
In the case $n=2$ we have 
$\hat{\mathfrak{V}_2^{\mathbb{Z}}}\subset \mathcal{O}_{\lambda_1}$.
From Example~\ref{exm1112} we already know that $\mathcal{O}_{\lambda_0}=
\mathcal{O}_{\lambda_2}=\mathbb{C}\text{-}\mathrm{mod}$. The category 
$\mathcal{O}_{\lambda_1}/\hat{\mathfrak{V}_n^{\mathbb{Z}}}$
is equivalent to the module category of the algebra 
$D$ of dual numbers. The functors $\mathrm{E}_0$
and $\mathrm{E}_1$ are translations out of the wall and onto
the wall, respectively (see Example~\ref{exm1112}).
} 
\end{example}

In \cite{FKS} one can find how to categorify tensor products
of arbitrary (finite) collections of simple finite-dimensional 
$U_v(\mathfrak{sl}_2)$-modules.

\subsection{Koszul dual picture}\label{s11.5}

Using the Koszul duality functor from Subsection~\ref{s10.3}
we obtain categorifications of both $\hat{\mathcal{V}}_1^{\otimes n}$ 
and $\hat{\mathcal{V}}_n$ in terms of the parabolic category 
$\mathcal{O}$. The fact that the original action of the 
$U_v(\mathfrak{sl}_2)$-functor $\mathrm{K}$, given by the shift of 
grading, transforms via the Koszul duality to a simultaneous shift in 
grading and position is the reason why the resulting modules is 
isomorphic to $\hat{\mathcal{V}}_1^{\otimes n}$ and not
$\mathcal{V}_1^{\otimes n}$. Here we briefly describe how this goes.

To categorify $\hat{\mathcal{V}}_1^{\otimes n}$ consider the direct sum 
$\mathfrak{U}_n$ of all maximal parabolic subcategories in the
regular block $\mathcal{O}_0$ for $\mathfrak{gl}_n$, that is
\begin{displaymath}
\mathfrak{U}_n:=\bigoplus_{i=0}^n  \mathcal{O}_0^{\mathfrak{p}_i},\qquad
\mathfrak{U}_n^{\mathbb{Z}}:=
\bigoplus_{i=0}^n  {}^{\mathbb{Z}}\mathcal{O}_0^{\mathfrak{p}_i}, 
\end{displaymath}
where $\mathfrak{p}_i$ is the parabolic subalgebra of 
$\mathfrak{g}$ such that the corresponding
parabolic subgroup $W_{\mathfrak{p}_i}$ is isomorphic to 
$\mathbb{S}_i\oplus \mathbb{S}_{n-1}$, where the first component
permutes the first $i$ elements and the second components permutes
the last $n-i$ elements. By Theorem~\ref{thm1011}, the category 
$\mathfrak{U}_n^{\mathbb{Z}}$ is  the Koszul dual of the 
category $\mathfrak{V}_n^{\mathbb{Z}}$.

Using Theorem~\ref{thm1015}, the functorial action of 
$U_v(\mathfrak{sl}_2)$ on $\mathfrak{V}_n^{\mathbb{Z}}$ by projective
functors transfers, via Koszul duality, to a functorial action of
$U_v(\mathfrak{sl}_2)$ on $\mathcal{D}^b(\mathfrak{U}_n^{\mathbb{Z}})$
by derived Zuckerman functors. This gives us the Koszul dual
categorification of $\hat{\mathcal{V}}_1^{\otimes n}$.

The Koszul dual categorification of the simple module 
$\hat{\mathcal{V}}_n$ is identified inside 
$\mathcal{D}^b(\mathfrak{U}_n^{\mathbb{Z}})$
as the full triangulated subcategory generated by simple
finite dimensional modules (with all their shifts).

\section{Application: categorification of the Jones polynomial}\label{s12}

\subsection{Kauffman bracket and Jones polynomial}\label{s12.1}

Let $L$ be a diagram of an oriented link. Let $n_+$ and $n_-$ denote 
the number of {\em right} and {\em left} crossing
\index{right crossing}\index{left crossing}
in $L$, respectively, as shown on the following picture:
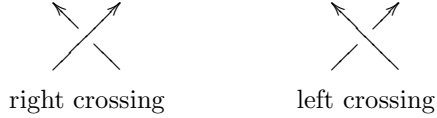
\begin{displaymath}
\begin{array}{c}
\xymatrix@-1.2pc{&&\\&\ar[lu]&\\\ar[rruu]&&\ar@{-}[lu]}\\
\text{right crossing}
\end{array}\qquad\qquad
\begin{array}{c}
\xymatrix@-1.2pc{&&\\&\ar[ru]&\\\ar@{-}[ru]&&\ar[lluu]}\\
\text{left crossing}
\end{array}
\end{displaymath}
The {\em Kauffman bracket} $\langle L \rangle\in\mathbb{Z}[v,v^{-1}]$ 
\index{Kauffman bracket}
of $L$ is defined via the following rule:
\begin{equation}\label{eq1291}
\left\langle\begin{array}{c}
\xymatrix@-1.2pc{&&\\&&\\\ar@{-}'[ru][rruu]&&\ar@{-}[lluu]}
\end{array}\right\rangle= \left\langle
\begin{array}{c}
\xymatrix{\ar@/_/@{-}[r]&\\\ar@/^/@{-}[r]&}
\end{array}\right\rangle
- v\left\langle\begin{array}{c}
\xymatrix{&\\\ar@/_/@{-}[u]&\ar@/^/@{-}[u]}
\end{array}\right\rangle
\end{equation}
together with $\langle\bigcirc L\rangle = (v+v^{-1})\langle L\rangle$
and normalized by the conditions $\langle\emptyset\rangle=1$.

The (unnormalized) {\em Jones polynomial} $\hat{\mathrm{J}}(L)$ of $L$
\index{Jones polynomial}
is defined by 
\begin{displaymath}
\hat{\mathrm{J}}(L):=(-1)^{n_-}v^{n_+-2n_-}\langle L\rangle\in
\mathbb{Z}[v,v^{-1}]
\end{displaymath}
and the usual Jones polynomial $\mathrm{J}(L)$ is defined via
$(v+v^{-1})\mathrm{J}(L)=\hat{\mathrm{J}}(L)$
(we use the normalization from \cite{BN}).

\begin{theorem}[\cite{Jn}]\label{thm1201}
The polynomial  ${\mathrm{J}}(L)$ is an invariant of an oriented 
link.
\end{theorem}

\begin{example}\label{exm1202}
{\rm
For the {\em Hopf link}
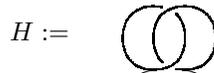
\index{Hopf link}
\begin{displaymath}
H:=\begin{array}{c}
\xygraph{
!{0;/r1.0pc/:}
!{\vunder}
!{\vunder-}
[uur]!{\hcap[2]=<}
[l]!{\hcap[-2]=>}
}
\end{array}
\end{displaymath}
we have $\hat{\mathrm{J}}=(v+v^{-1})(v+v^{5})$
and $\mathrm{J}(H)=v+v^{5}$.
} 
\end{example}

\begin{proposition}\label{prop1203}
The Jones polynomial is uniquely determined by the property
$\mathrm{J}(\bigcirc)=1$ and the {\em skein relation}
\index{skein relation}
\begin{equation}\label{eq1293}
v^2 \mathrm{J}\left(\begin{array}{c}
\xymatrix@-1.2pc{&&\\&\ar[ru]&\\\ar@{-}[ru]&&\ar[lluu]}
\end{array}\right)-
v^{-2} \mathrm{J}\left(\begin{array}{c}
\xymatrix@-1.2pc{&&\\&\ar[lu]&\\\ar[rruu]&&\ar@{-}[lu]}
\end{array}\right)=
(v-v^{-1})\mathrm{J}\left(\begin{array}{c}
\xymatrix@-1.2pc{&&\\&&\\\ar@/_/[uu]&&\ar@/^/[uu]}
\end{array}\right)
\end{equation}
\end{proposition}

\subsection{Khovanov's idea for categorification of 
${\mathrm{J}}(L)$}\label{s12.2}

Khovanov's original idea (\cite{Kv}) how to categorify ${\mathrm{J}}(L)$
(as explained in \cite{BN}) was to ``upgrade'' Kauffman's bracket
to a new bracket $\left\llbracket\cdot \right\rrbracket$, 
which takes values in complexes of graded complex vector space
(this works over $\mathbb{Z}$ as well, Khovanov originally works 
over $\mathbb{Z}[c]$ where $\deg c=2$). Denote by $V$ the
two-dimensional graded vector space such that the nonzero
homogeneous parts are of degree $1$ and $-1$. The normalization 
conditions are easy to generalize:
\begin{displaymath}
\left\llbracket\varnothing\right\rrbracket=0\to \mathbb{C}\to 0,\qquad
\left\llbracket\bigcirc L\right\rrbracket=V\otimes
\left\llbracket L\right\rrbracket. 
\end{displaymath}
To ``categorify'' the rule \eqref{eq1291} Khovanov proposes the
axiom
\begin{displaymath}
\left\llbracket\begin{array}{c}
\xymatrix@-1.2pc{&&\\&&\\\ar@{-}'[ru][rruu]&&\ar@{-}[lluu]}
\end{array}\right\rrbracket=
\mathrm{Total}\left(
0\to
\left\llbracket
\begin{array}{c}
\xymatrix{\ar@/_/@{-}[r]&\\\ar@/^/@{-}[r]&}
\end{array}
\right\rrbracket
\overset{d}{\to}
\left\llbracket
\begin{array}{c}
\xymatrix{&\\\ar@/_/@{-}[u]&\ar@/^/@{-}[u]}
\end{array}
\right\rrbracket\langle -1\rangle
\to 0
\right)
\end{displaymath}
for a certain differential $d$, the definition of which is the key 
ingredient of the construction. We refer the reader to \cite{Kv,BN}
for details on how this works.

Appearance of complexes and taking total complexes in the above 
construction suggests possibility of its functorial interpretation
(by action of functors on certain categories of complexes).
We are going to describe this functorial approach below.

\subsection{Quantum $\mathfrak{sl}_2$-link invariants}\label{s12.3}

A usual way to construct knot and link invariants is using
their connection to the braid group $\mathbb{B}_n$ given by
the following theorem:

\begin{theorem}[Alexander's Theorem]\label{thm1204}
Every link can be obtained as the closure of a braid. 
\end{theorem}

The correspondence between links and braids given by Alexander's Theorem
is not bijective, that is, different braids could give, upon
closure, isotopic links. However, there is an easy way to 
find out when two different braids give rise to same links.
For $n\in\mathbb{N}$ let $\sigma_1,\dots,\sigma_{n-1}$ be the
generators of $\mathbb{B}_n$, and $\mathfrak{i}_n:\mathbb{B}_n
\hookrightarrow \mathbb{B}_{n+1}$  the natural inclusion
given by the identity on the $\sigma_i$'s for $i=1,\dots,n-1$. 
Denote by $\mathbb{B}$ the disjoint union of all
$\mathbb{B}_n$, $n\geq 1$.

\begin{theorem}[Markov's Theorem]\label{thm1205}
Denote by $\sim$ the smallest equivalence relation on $\mathbb{B}$
which contains the conjugation relations on all 
$\mathbb{B}_n$, $n\geq 1$, and such that for any $n\in\mathbb{N}$
and any $w\in \mathbb{B}_n$ we have $w\sim \mathfrak{i}_n(w)\sigma_n^{\pm 1}$.
Then two braids $w,w'\in\mathbb{B}$ give, by closing, isotopic 
links if and only if $w\sim w'$.
\end{theorem}

The closure of a braid can be obtained as a composition of the following
``elementary'' diagrams:
\begin{equation}\label{eq1195}
\begin{array}{c}
\begin{array}{c}\xymatrix@-1.2pc{
\ar@{-}`d[ddr] `r[rr] [rr]&&\\&&\\&&}\end{array}\\
\text{the cup diagram}
\end{array}\quad
\begin{array}{c}
\begin{array}{c}\xymatrix@-1.2pc{
&&\\&&\\\ar@{-}`u[uur] `r[rr] [rr]&&}\end{array}\\
\text{the cap diagram}
\end{array}\quad
\begin{array}{c}
\begin{array}{c}
\xymatrix@-1.2pc{&&\\&\ar[lu]&\\\ar[rruu]&&\ar@{-}[lu]}\end{array}\\
\text{right crossing}
\end{array}\quad
\begin{array}{c}
\begin{array}{c}
\xymatrix@-1.2pc{&&\\&\ar[ru]&\\\ar@{-}[ru]&&\ar[lluu]}\end{array}\\
\text{left crossing}
\end{array}
\end{equation}

\begin{example}\label{exm1206}
{\rm
The Hopf link from Example~\ref{exm1202} can be obtained by the
following sequence of elementary diagrams:
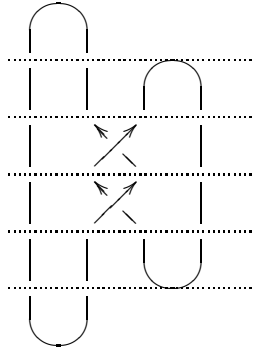
\begin{equation}\label{eq1292}
\xymatrix@-1.6pc{
&&&&&&&&&\\
&&&&&&&&&\\
\ar@{.}[rrrrrrrrr]&\ar@{-}`u[uur] `r[rr] [rr]&&&&&&&&\\
&&&&&&&&&\\
\ar@{.}[rrrrrrrrr]&\ar@{-}[uu]&&\ar@{-}[uu]&&\ar@{-}`u[uur] `r[rr] [rr]&&&&\\
&&&&&&&&&\\
\ar@{.}[rrrrrrrrr]&\ar@{-}[uu]&&\ar[rruu]&&\ar'[lu][lluu]&&\ar@{-}[uu]&&\\
&&&&&&&&&\\
\ar@{.}[rrrrrrrrr]&\ar@{-}[uu]&&\ar[rruu]&&\ar'[lu][lluu]
\ar@{-}`d[ddr] `r[rr] [rr]&&\ar@{-}[uu]&&\\
&&&&&&&&&\\
\ar@{.}[rrrrrrrrr]&\ar@{-}[uu]\ar@{-}`d[ddr] `r[rr] [rr]&&\ar@{-}[uu]&&&&&&\\
&&&&&&&&&\\
&&&&&&&&&\\
}
\end{equation}
}
\end{example}

The idea of a quantum $\mathfrak{sl}_2$-invariant (originated
in \cite{RT}) is to associate to any elementary diagram a homomorphism 
of tensor powers of the two-dimensional 
$U_v(\mathfrak{sl}_2)$-module $\hat{\mathcal{V}}_1$. The module
$\hat{\mathcal{V}}_1$ has basis $\{w_0,w_1\}$ as described in 
Subsection~\ref{s11.2}. It is convenient to denote the basis vector
$w_{i_1}\otimes w_{i_2}\otimes\cdots \otimes w_{i_k}$,
$i_1,i_2,\dots,i_k\in\{0,1\}$, of some tensor power of 
$\hat{\mathcal{V}}_1$  simply by $\mathtt{i}_1\mathtt{i}_2\cdots \mathtt{i}_k$,
the so-called {\em $0\text{-}1$ sequence}.
\index{$0\text{-}1$ sequence}
\begin{itemize}
\item To the cup diagram we associate the homomorphism
$\cup:\mathbb{C}(v)\to\hat{\mathcal{V}}_1^{\otimes 2}$ defined as follows:
\begin{displaymath}
1\mapsto \mathtt{0}\mathtt{1}+v\mathtt{1}\mathtt{0}. 
\end{displaymath}
\item To the cap diagram we associate the homomorphism
$\cap:\hat{\mathcal{V}}_1^{\otimes 2}\to\mathbb{C}(v)$ defined as follows:
\begin{displaymath}
\mathtt{0}\mathtt{0}\mapsto 0,\quad 
\mathtt{1}\mathtt{1}\mapsto 0,\quad 
\mathtt{0}\mathtt{1}\mapsto v^{-1},\quad 
\mathtt{1}\mathtt{0}\mapsto 1. 
\end{displaymath}
\item To the right crossing we associate the homomorphism
$\hat{\mathcal{V}}_1^{\otimes 2}\to
\hat{\mathcal{V}}_1^{\otimes 2}$ defined as follows:
\begin{displaymath}
\mathtt{0}\mathtt{0}\mapsto -v\mathtt{0}\mathtt{0},\quad 
\mathtt{1}\mathtt{1}\mapsto -v\mathtt{1}\mathtt{1},\quad 
\mathtt{0}\mathtt{1}\mapsto \mathtt{1}\mathtt{0}+
(v^{-1}-v)\mathtt{0}\mathtt{1},\quad 
\mathtt{1}\mathtt{0}\mapsto \mathtt{0}\mathtt{1}. 
\end{displaymath}
\item To the left crossing we associate the homomorphism
$\hat{\mathcal{V}}_1^{\otimes 2}\to\hat{\mathcal{V}}_1^{\otimes 2}$ 
defined as follows:
\begin{displaymath}
\mathtt{0}\mathtt{0}\mapsto -v^{-1}\mathtt{0}\mathtt{0},\quad 
\mathtt{1}\mathtt{1}\mapsto -v^{-1}\mathtt{1}\mathtt{1},\quad 
\mathtt{0}\mathtt{1}\mapsto \mathtt{1}\mathtt{0},\quad 
\mathtt{1}\mathtt{0}\mapsto \mathtt{0}\mathtt{1}+
(v-v^{-1})\mathtt{1}\mathtt{0}. 
\end{displaymath}
\end{itemize}
Using horizontal composition in the tensor category of finite
dimensional $U_v(\mathfrak{sl}_2)$-modules the above induces 
morphisms $\cup_{i,n}:\hat{\mathcal{V}}_1^{\otimes n}\to
\hat{\mathcal{V}}_1^{\otimes n+2}$ (inserting $\cup$ between the
$i$-th and the $i+1$-st factors) ans similarly for the cap
and crossing diagrams.

Given now an oriented link $L$, written as a composition of the 
elementary diagrams as above, we can read it (from bottom to top) 
as an endomorphism $\varphi_{L}$ of $\hat{\mathcal{V}}_0$. Applied to 
$1$, this endomorphism produces a Laurent polynomial in $v$. 

\begin{theorem}[\cite{RT}]\label{thm1206}
The polynomials $(-1)^{n_-}v^{n_+-2n_-}\varphi_{L}(1)$ and 
$\hat{\mathrm{J}}(L)$ coincide.
\end{theorem}

\begin{proof}[Idea of the proof.]
First check the skein relation \eqref{eq1293} for
$(-1)^{n_-}v^{n_+-2n_-}\varphi_{L}(1)$.
Then check that our assignments for the cup and cap diagrams 
define the correct
normalization.
\end{proof}

\begin{example}\label{exm1207}
{\rm 
For the diagram $L$ of the Hopf link as in \eqref{eq1292}, here 
is the step-by-step calculation of $\varphi_{L}(1)$
(the diagram is read from bottom to top):
\begin{displaymath}
\begin{array}{ccl}
1&\mapsto& \mathtt{0}\mathtt{1} +v\mathtt{1}\mathtt{0}\\
&\mapsto&
\mathtt{0}\mathtt{1}\mathtt{0}\mathtt{1}
+v\mathtt{0}\mathtt{1}\mathtt{1}\mathtt{0}
+v\mathtt{1}\mathtt{0}\mathtt{0}\mathtt{1}
+v^2\mathtt{1}\mathtt{0}\mathtt{1}\mathtt{0}\\
&\mapsto&
\mathtt{0}\mathtt{0}\mathtt{1}\mathtt{1}
-v^2\mathtt{0}\mathtt{1}\mathtt{1}\mathtt{0}
-v^2\mathtt{1}\mathtt{0}\mathtt{0}\mathtt{1}
+v^2\mathtt{1}\mathtt{1}\mathtt{0}\mathtt{0}
+v^2(v^{-1}-v)\mathtt{1}\mathtt{1}\mathtt{0}\mathtt{0}\\
&\mapsto&
\mathtt{0}\mathtt{1}\mathtt{0}\mathtt{1}
+(v^{-1}-v)\mathtt{0}\mathtt{0}\mathtt{1}\mathtt{1}
+v^3\mathtt{0}\mathtt{1}\mathtt{1}\mathtt{0}
+v^3\mathtt{1}\mathtt{0}\mathtt{0}\mathtt{1}
+v^2(v^{-1}-v)\mathtt{1}\mathtt{1}\mathtt{0}\mathtt{0}\\
&&+v^2((v^{-1}-v)^2+1)\mathtt{1}\mathtt{0}\mathtt{1}\mathtt{0}\\
&\mapsto&
(v^{-1}+v^3)\mathtt{0}\mathtt{1}+
v^2((v^{-1}-v)^2+2)\mathtt{1}\mathtt{0}\\
&\mapsto& v^{-2} +1 +v^2+v^4. 
\end{array}
\end{displaymath}
We thus have $\varphi_{L}(1)=v^{-2}(v+v^{-1})(v+v^{5})$,
which can be compared with $\hat{\mathrm{J}}(L)$, given by
Example~\ref{exm1202}.
} 
\end{example}

\subsection{Functorial quantum $\mathfrak{sl}_2$-link 
invariants}\label{s12.4}

From Subsection~\ref{s11.5} we know a categorical realization of 
$\hat{\mathcal{V}}_1^{\otimes n}$ via the action of derived
Zuckerman functors on the (bounded derived category of) the 
direct sum of maximal parabolic subcategories in the principal 
block $\mathcal{O}_0$ for $\mathfrak{gl}_n$. To categorify 
quantum $\mathfrak{sl}_2$-link invariants we thus only need to 
find a functorial interpretation for the cup, cap and crossing
morphisms defined in the previous section. There is a catch though,
since the cup and cap morphisms work between different tensor
powers of $\hat{\mathcal{V}}_1$, so we have also to relate
maximal parabolic subcategories in the principal
block of $\mathcal{O}_0$ for $\mathfrak{gl}_k$ with $k\leq n$
to the case of $\mathfrak{gl}_n$.

To construct functorial homomorphisms between our categorical
tensor powers of $\hat{\mathcal{V}}_1$ we have a good candidate:
projective functors. Here we see the full advantage of the Koszul
dual picture. In this picture the action of $U_v(\mathfrak{sl}_2)$
is given by derived Zuckerman functors, which commute with
projective functors (see Proposition~\ref{prop671}). Hence
projective functors would automatically induces homomorphisms
between our categorical tensor powers of $\hat{\mathcal{V}}_1$.
This observation also suggests a solution to the catch.
Among other projective functors, we have translations onto
and out of the walls. Therefore one could try to relate
maximal parabolic subcategories in certain singular blocks 
of $\mathcal{O}$ for $\mathfrak{gl}_n$ to maximal parabolic 
subcategories in regular blocks of $\mathcal{O}$ for 
$\mathfrak{gl}_k$ with $k<n$. This relation turns out to be known
under the name {\em Enright-Shelton equivalence}, \cite{ES}.
\index{Enright-Shelton equivalence}

For $i=0,1,\dots,n$ let $\mathfrak{p}_i$ be as in
Subsection~\ref{s11.5}. For an integral 
$\lambda\in\mathfrak{h}^*_{\mathrm{dom}}$ set
\begin{displaymath}
\mathcal{O}_{\lambda}^{\mathrm{max}}:=
\bigoplus_{i=0}^n \mathcal{O}_{\lambda}^{\mathfrak{p}_i},\qquad
{}^{\mathbb{Z}}\mathcal{O}_{\lambda}^{\mathrm{max}}:=
\bigoplus_{i=0}^n {}^{\mathbb{Z}}\mathcal{O}_{\lambda}^{\mathfrak{p}_i}.
\end{displaymath}
Note that $\mathcal{O}_{\lambda}^{\mathrm{max}}$ might be zero
(for example, if $n>2$ and $\lambda$ is most singular).

\begin{proposition}[\cite{ES}]\label{prop1208}
Let $\lambda\in\mathfrak{h}^*_{\mathrm{dom}}$ be integral.
\begin{enumerate}[$($a$)$]
\item\label{prop1208.1} The category
$\mathcal{O}_{\lambda}^{\mathrm{max}}\neq 0$
if and only if every direct summand of $W_{\lambda}$ is 
isomorphic to either $\mathbb{S}_1$ or $\mathbb{S}_2$.
\item\label{prop1208.2} If the category
$\mathcal{O}_{\lambda}^{\mathrm{max}}\neq 0$
and $W_{\lambda}$ contains exactly $k$ direct summands isomorphic to
$\mathbb{S}_2$, then the category $\mathcal{O}_{\lambda}^{\mathrm{max}}$
is equivalent to the category $\mathcal{O}_{0}^{\mathrm{max}}$
for the algebra  $\mathfrak{gl}_{n-2k}$.
\end{enumerate}
\end{proposition}

To proceed we observe that the closure of a braid necessarily has an
even number of strands. This reduces our picture to the case of even $n$.
For every parabolic subgroup $W_{\mathfrak{p}}\subset W$ satisfying 
the condition of Proposition~\ref{prop1208}\eqref{prop1208.1} we fix 
some integral dominant $\lambda_{W_{\mathfrak{p}}}$ such that 
$W_{\mathfrak{p}}=W_{\lambda_{W_{\mathfrak{p}}}}$. Consider the categories
\begin{displaymath}
\mathfrak{W}_{n}:=\bigoplus_{W_{\mathfrak{p}}\subset W}
\mathcal{O}_{\lambda_{W_{\mathfrak{p}}}}^{\mathrm{max}},
\qquad\mathfrak{W}_{n}^{\mathbb{Z}}:=
\bigoplus_{W_{\mathfrak{p}}\subset W} {}^{\mathbb{Z}}
\mathcal{O}_{\lambda_{W_{\mathfrak{p}}}}^{\mathrm{max}}.
\end{displaymath}
By Proposition~\ref{prop1208}\eqref{prop1208.2} and Subsection~\ref{s11.5}
the algebra $U_v(\mathfrak{sl}_2)$ acts on 
$\mathcal{D}^b(\mathfrak{W}_{n}^{\mathbb{Z}})$ via derived Zuckerman 
functors categorifying the corresponding direct sum of tensor powers of 
$\mathcal{V}_1$. To the elementary diagrams from \eqref{eq1195}
we associate the following endofunctors of 
$\mathcal{D}^b(\mathfrak{W}_{n})$
and $\mathcal{D}^b(\mathfrak{W}_{n}^{\mathbb{Z}})$:
\begin{itemize}
\item To the cup diagram we associate translation out of the
respective wall.
\item To the cap diagram we associate translation to the respective wall.
\item To the right crossing we associate the corresponding derived 
shuffling  functor.
\item To the left crossing we associate the corresponding derived
coshuffling  functor.
\end{itemize}
Note that all these functors commute with Zuckerman 
functors and hence define functorial homomorphisms between the
corresponding categorified $U_v(\mathfrak{sl}_2)$-modules.
In the graded picture we can adjust the gradings of our functors
so that they correspond to the combinatorics of the morphisms
associated to elementary diagrams as described in the previous subsection.

For $n=2k$ let $W'\subset W$ be the unique parabolic subgroup
isomorphic to the direct sum of $k$ copies of $\mathbb{S}_2$.
Then $\mathcal{O}_{\lambda_{W'}}^{\mathrm{max}}$ is semisimple an 
contains a unique simple module, which we denote by $N$. 
Reading our closed diagram $L$
of an oriented link from bottom to top defines an endofunctor
$\cF_L$ of ${}^{\mathbb{Z}}\mathcal{O}_{\lambda_{W'}}^{\mathrm{max}}$.
The following theorem was conjectured in \cite{BFK} and proved in
\cite{St}:

\begin{theorem}[Categorification of the Jones polynomial]\label{thm1209}
Let $L$ be a diagram of an oriented link.
\begin{enumerate}[$($a$)$]
\item\label{thm1209.1} The endofunctor $\cF_L\langle 3n_+\rangle[-n_+]$
is an invariant of an oriented link.
\item\label{thm1209.2} We
have $\hat{\mathrm{J}}(L)=[\cF_L\langle 2n_--n_+\rangle[n_-]\,N]$.
\end{enumerate}
\end{theorem}

\begin{proof}[Idea of the proof.]
Derived shuffling and coshuffling functors are cones of 
adjunction morphisms between the identity functor and the corresponding
translations through walls. This may be interpreted in terms
of the skein relation \eqref{eq1293}. The normalization of
$\cF_L$ is easily computed using the combinatorics of
translations through walls. There are many technical details which
are checked by a careful analysis of the general combinatorics
of translation functors.
\end{proof}

\begin{remark}\label{rem1210}
{\rm
\begin{enumerate}[$($a$)$]
\item\label{rem1210.1} Theorem~\ref{thm1209} generalizes in a straightforward
way to give invariants of oriented tangles (this, in particular, removes
the restriction on $n$ to be even), see \cite{St}.
\item\label{rem1210.2} Theorem~\ref{thm1209} gives a functorial
invariant of oriented tangles (see also \cite{Kv2} for a different 
approach), which is a stronger invariant than the
Jones polynomial (the latter is just roughly the Euler characteristics
of the complex obtained by applying the functorial invariant).
\item\label{rem1210.3} It is known (see \cite{St4}, based on \cite{Br}) 
that Khovanov's categorification of the Jones polynomial is equivalent to
the invariants constructed in Theorem~\ref{thm1209}.
\item\label{rem1210.4} Theorem~\ref{thm1209} can be extended 
to  functorial quantum $\mathfrak{sl}_k$-invariants, see \cite{Su,MS5},
and to colored Jones polynomial, see \cite{FSS}.
\end{enumerate}
}
\end{remark}

\section{$\mathfrak{sl}_2$-categorification: of Chuang and Rouquier}\label{s13}

\subsection{Genuine $\mathfrak{sl}_2$-categorification}\label{s13.1}

In this section we describe some results from \cite{CR}.
Consider the Lie algebra $\mathfrak{sl}_2$ with the standard basis
$\mathbf{e}=\left(\begin{array}{cc}0&1\\0&0\end{array}\right)$, 
$\mathbf{f}=\left(\begin{array}{cc}0&0\\1&0\end{array}\right)$ and 
$\mathbf{h}=\left(\begin{array}{cc}1&0\\0&-1\end{array}\right)$.
Let $\Bbbk$ be a field.

\begin{definition}[\cite{CR}]\label{def1301}
{\rm  
An $\mathfrak{sl}_2$-categorification (over $\Bbbk$) is a tuple
$(\mathcal{A},\mathrm{E},\mathrm{F},q,a,\chi,\tau)$ such that:
\begin{enumerate}[$($i$)$]
\item\label{def1301.1} $\mathcal{A}$ is an artinian and noetherian
$\Bbbk$-linear abelian category in which the endomorphism ring of
every simple object is $\Bbbk$;
\item\label{def1301.2} $\mathrm{E}$ and $\mathrm{F}$ are 
endofunctors of $\mathcal{A}$ which are both left and right adjoint 
to each other;
\item\label{def1301.3} mapping $\mathbf{e}\mapsto [\mathrm{E}]$ and 
$\mathbf{f}\mapsto [\mathrm{F}]$ extends to a locally finite
$\mathfrak{sl}_2$-representation on 
$\mathbb{Q}\otimes_{\mathbb{Z}}[\mathcal{A}]$;
\item\label{def1301.4} classes of simple objects in
$\mathbb{Q}\otimes_{\mathbb{Z}}[\mathcal{A}]$ are eigenvectors
for $H$;
\item\label{def1301.5} $\tau\in\mathrm{End}(E^2)$ is such that 
\begin{displaymath}
(\mathrm{id}_E\circ_0\tau)\circ_1
(\tau\circ_0\mathrm{id}_E)\circ_1
(\mathrm{id}_E\circ_0\tau)=
(\tau\circ_0\mathrm{id}_E)\circ_1
(\mathrm{id}_E\circ_0\tau)\circ_1
(\tau\circ_0\mathrm{id}_E);
\end{displaymath}
\item\label{def1301.6} $q\in\Bbbk^*$ is such that 
$(\tau+\mathrm{id}_{E^2})(\tau-q\,\mathrm{id}_{E^2})=0$;
\item\label{def1301.7} $\chi\in\mathrm{End}(E)$ is such that 
\begin{displaymath}
\tau\circ_1(\mathrm{id}_E\circ_0\chi)\circ_1\tau=
\begin{cases}
q\, \chi\circ_0\mathrm{id}_E, & q\neq 1;\\
\chi\circ_0\mathrm{id}_E-\tau, &q=1;
\end{cases}
\end{displaymath}
\item\label{def1301.8} $a\in\Bbbk$ is such that $a\neq 0$ if $q\neq 1$
and $\chi-a$ is locally nilpotent.
\end{enumerate}
}
\end{definition}

We note that $a$ can always be adjusted to either $0$ or $1$. Indeed,
if $q\neq 1$ and $a\neq 0$, replacing $\chi$ by $a\chi$ gives
a new $\mathfrak{sl}_2$-categorification with $a=1$. 
If $q=1$, replacing $\chi$ by $\chi+a \mathrm{id}_{E}$ gives a
new $\mathfrak{sl}_2$-categorification with $a=0$.

\begin{example}\label{exm1302}
{\rm 
Let $A_{-2}=A_{2}:=\Bbbk$ and $A_0=\Bbbk[x]/(x^2)$ and set
$\mathcal{A}_i:=A_i\text{-}\mathrm{mod}$ for $i=0,\pm 2$,
$\mathcal{A}:=\mathcal{A}_{-2}\oplus \mathcal{A}_{0}\oplus
\mathcal{A}_{2}$. Note that 
both $A_{-2}$ and $A_{2}$ are unital subalgebras of $A_0$
and we have the corresponding induction and restriction functors
$\mathrm{Ind}$ and $\mathrm{Res}$. Hence we may define the 
functors $\mathrm{E}$ and $\mathrm{F}$ as given by the right 
and left arrows on the following picture, respectively:
\begin{displaymath}
\xymatrix{ 
\mathcal{A}_{-2}\ar@/^/[rrr]^{\mathrm{Ind}}
&&&\mathcal{A}_{0}
\ar@/^/[rrr]^{\mathrm{Res}}\ar@/^/[lll]^{\mathrm{Res}}
&&&\mathcal{A}_{2}\ar@/^/[lll]^{\mathrm{Ind}}
}
\end{displaymath}
Let $q=1$ and $a=0$. Define $\chi$ as multiplication by $x$
on $\mathrm{Res}:\mathcal{A}_{0}\to \mathcal{A}_{2}$ and
multiplication by $-x$ on $\mathrm{Ind}:\mathcal{A}_{-2}\to 
\mathcal{A}_{0}$. Define $\tau\in\mathrm{End}_{\Bbbk}(\Bbbk[x]/(x^2))$
as the automorphism swapping $1$ and $x$. It is easy to check that 
all conditions are satisfied and thus this is an 
$\mathfrak{sl}_2$-categorification (of the simple $3$-dimensional
$\mathfrak{sl}_2$-module).
}
\end{example}

\begin{proposition}\label{prop1321}
Let $(\mathcal{A},\mathrm{E},\mathrm{F},q,a,\chi,\tau)$ 
be an $\mathfrak{sl}_2$-categorification. For $\lambda\in\mathbb{Z}$
denote by $\mathcal{A}_{\lambda}$ the full subcategory of 
$\mathcal{A}$ consisting of all objects whose
classes in the Grothendieck group have $\mathfrak{sl}_2$-weight $\lambda$.
Then $\mathcal{A}=\oplus_{\lambda\in\mathbb{Z}}\mathcal{A}_{\lambda}$.
\end{proposition}

\subsection{Affine Hecke algebras}\label{s13.2}

Let $n\in\mathbb{N}_0$ and $q\in\Bbbk^*$. 

\begin{definition}[Non-degenerate affine Hecke algebra]\label{def1303}
{\rm
If $q\neq 1$, then the affine 
Hecke algebra $\mathbb{H}_n(q)$ is defined as the $\Bbbk$-algebra
with generators $T_1,\dots, T_{n-1}$, $X_1^{\pm 1},\dots,X_n^{\pm 1}$
subject to the following relations:
\begin{displaymath}
\begin{array}{rclccrcll}
(T_i+1)(T_i-q)&=&0;&&&T_iT_j&=&T_jT_i,&|i-j|>1;\\
T_iT_{i+1}T_i&=&T_{i+1}T_iT_{i+1};&&&X_iX_j&=&X_jX_i;&\\
X_iX_i^{-1}&=&1;&&&X_i^{-1}X_i&=&1;&\\
T_iX_iT_i&=&qX_{i+1};&&&X_iT_j&=&T_jX_i,&i-j\neq 0,1.
\end{array}
\end{displaymath}
}
\end{definition}

Note that the subalgebra $\mathbb{H}^f_n(q)$ of $\mathbb{H}_n(q)$, 
generated by the $T_i$'s, is just an alternative version of the Hecke 
algebra $\mathbb{H}_n$ for $\mathbb{S}_n$ from Subsection~\ref{s7.4}.

\begin{definition}[Degenerate affine Hecke algebra]\label{def1304}
{\rm
If $q=1$, then the {\em degenerate} affine 
\index{degenerate affine Hecke algebra}
Hecke algebra $\mathbb{H}_n(1)$ is defined as the $\Bbbk$-algebra
with generators $T_1,\dots, T_{n-1}$, $X_1,\dots,X_n$
subject to the following relations:
\begin{displaymath}
\begin{array}{rclccrcll}
T_i^2&=&1;&&&T_iT_j&=&T_jT_i,&|i-j|>1;\\
T_iT_{i+1}T_i&=&T_{i+1}T_iT_{i+1};&&&X_iX_j&=&X_jX_i;&\\
X_{i+1}T_i&=&T_iX_i+1;&&&X_iT_j&=&T_jX_i,&i-j\neq 0,1.
\end{array}
\end{displaymath}
}
\end{definition}

Note that the degenerate affine Hecke algebra is not the specialization of
the affine Hecke algebra at $q=1$. The subalgebra $\mathbb{H}^f_n(1)$ 
of $\mathbb{H}_n(1)$, generated by the $T_i$'s, is isomorphic to the
group algebra of the symmetric group. Relevance of affine Hecke algebras
for $\mathrm{sl}_2$-categorifications is justified by the following
statement:

\begin{lemma}\label{lem1305}
If $(\mathcal{A},\mathrm{E},\mathrm{F},q,a,\chi,\tau)$ is
an $\mathfrak{sl}_2$-categorification, then for all
$n\in\mathbb{N}_0$ mapping
\begin{displaymath}
T_i\mapsto \mathrm{id}_{E^{n-i-1}}\circ_0 \tau\circ_0
\mathrm{id}_{E^{i-1}}\quad\text{ and }\quad
X_i\mapsto \mathrm{id}_{E^{n-i}}\circ_0 \chi\circ_0
\mathrm{id}_{E^{i-1}}
\end{displaymath}
extends to a homomorphism $\gamma_n:\mathbb{H}_n(q)\to 
\mathrm{End}(\mathrm{E}^n)$.
\end{lemma}

For $w\in W$ choose some reduces decomposition  
$w=s_{i_1}s_{i_2}\cdots s_{i_k}$ and set $T_w:=T_{i_1}T_{i_2}\cdots
T_{i_k}$. Let $\mathtt{1}$ and $\mathrm{sign}$ denote the 
one-dimensional representations of $\mathbb{H}^f_n(q)$ given by 
$T_i\mapsto q$ and $T_i\mapsto -1$, respectively. Define
\begin{displaymath}
c_n^{\mathtt{1}}:=\sum_{w\in W} T_w;\qquad\qquad
c_n^{\mathrm{sign}}:=\sum_{w\in W} (-q)^{-\mathfrak{l}(w)}T_w.
\end{displaymath}
Then both $c_n^{\mathtt{1}}$ and $c_n^{\mathrm{sign}}$ belong to the
center of $\mathbb{H}^f_n(q)$. For $\alpha\in\{\mathtt{1},\mathrm{sign}\}$
and $n\in\mathbb{N}_0$ define $\mathrm{E}^{(\alpha,n)}$ to be the
image of $c_n^{\alpha}:\mathrm{E}^{n}\to\mathrm{E}^{n}$.

\begin{corollary}\label{cor1306}
For  $\alpha\in\{\mathtt{1},\mathrm{sign}\}$
and $n\in\mathbb{N}_0$ we have 
\begin{displaymath}
\mathrm{E}^{n} \cong\underbrace{\mathrm{E}^{(\alpha,n)}\oplus
\mathrm{E}^{(\alpha,n)}\oplus\cdots\oplus\mathrm{E}^{(\alpha,n)}}_{
\text{$n!$ summands}}.
\end{displaymath}
\end{corollary}

\begin{example}\label{exm1307}
{\rm 
Let $A_{-2}=A_{2}:=\Bbbk[x]/(x^2)$ and $A_0=\Bbbk$ and set
$\mathcal{A}_i:=A_i\text{-}\mathrm{mod}$ for $i=0,\pm 2$,
$\mathcal{A}:=\mathcal{A}_{-2}\oplus \mathcal{A}_{0}\oplus
\mathcal{A}_{2}$. Then $A_0$ is a unital subalgebra of
both $A_{-2}$ and $A_{2}$ and we have the corresponding induction 
and restriction functors $\mathrm{Ind}$ and $\mathrm{Res}$
as in Example~\ref{exm1302}. Hence we may define the 
functors $\mathrm{E}$ and $\mathrm{F}$ as given by the right 
and left arrows on the following picture, respectively:
\begin{displaymath}
\xymatrix{ 
\mathcal{A}_{-2}\ar@/^/[rrr]^{\mathrm{Res}}
&&&\mathcal{A}_{0}
\ar@/^/[rrr]^{\mathrm{Ind}}\ar@/^/[lll]^{\mathrm{Ind}}
&&&\mathcal{A}_{2}\ar@/^/[lll]^{\mathrm{Res}}
}
\end{displaymath}
Note that this is a weak categorification of the simple
$3$-dimensional $\mathfrak{sl}_2$-module.
The functor $\mathrm{E}^2:\mathcal{A}_{-2}\to \mathcal{A}_{2}$
is isomorphic to tensoring with the bimodule
$\Bbbk[x]/(x^2)\otimes_{\Bbbk}\Bbbk[x]/(x^2)$, which is easily
seen to be indecomposable. This contradicts Corollary~\ref{cor1306}
and hence this is not an $\mathfrak{sl}_2$-categorification.
}
\end{example}

\subsection{Morphisms of $\mathfrak{sl}_2$-categorifications}\label{s13.3}

Let $\mathfrak{A}:=(\mathcal{A},\mathrm{E},\mathrm{F},q,a,\chi,\tau)$ and
$\mathfrak{A}':=(\mathcal{A}',\mathrm{E}',\mathrm{F}',q',a',\chi',\tau')$ be
two $\mathfrak{sl}_2$-categorifications. Denote by
$\varepsilon:\mathrm{E}\circ \mathrm{F}\to\mathrm{Id}$,
$\varepsilon':\mathrm{E}'\circ \mathrm{F}'\to\mathrm{Id}$,
$\eta:\mathrm{Id}\to\mathrm{E}\circ \mathrm{F}$ and
$\eta':\mathrm{Id}\to\mathrm{E}'\circ \mathrm{F}'$
some fixed counits and units, respectively.

\begin{definition}[\cite{CR}]\label{def1308}
{\rm 
A morphism of $\mathfrak{sl}_2$-categorifications 
from $\mathfrak{A}'$ to $\mathfrak{A}$ is a tuple
$(\mathrm{R},\xi,\zeta)$ such that:
\begin{enumerate}[$($i$)$]
\item\label{def1308.1} $\mathrm{R}:\mathcal{A}'\to \mathcal{A}$ is a functor;
\item\label{def1308.2} $\xi:\mathrm{R}\circ\mathrm{E}'\cong \mathrm{E}\circ\mathrm{R}$ is an isomorphism of functors;
\item\label{def1308.3} $\zeta:\mathrm{R}\circ\mathrm{F}'\cong \mathrm{F}\circ\mathrm{R}$ is an isomorphism of functors;
\item\label{def1308.4} the following diagram commutes:
\begin{displaymath}
\xymatrix{
\mathrm{R}\circ\mathrm{F}'\ar[rrr]^{\zeta}
\ar[d]_{\eta_{\mathrm{R}\circ\mathrm{F}'}}
&&&\mathrm{F}\circ\mathrm{R}
\ar[d]^{\mathrm{F}\circ\mathrm{R}(\varepsilon')}\\
\mathrm{F}\circ\mathrm{E}\circ\mathrm{R}\circ\mathrm{F}'
\ar[rrr]^{\mathrm{F}(\xi^{-1}_{\mathrm{F}'})}
&&&\mathrm{F}\circ\mathrm{R}\circ\mathrm{E}'\circ\mathrm{F}';
}
\end{displaymath}
\item\label{def1308.5} $a=a'$, $q=q'$ and the following diagrams commute:
\begin{displaymath}
\xymatrix{
\mathrm{R}\circ\mathrm{E}'\ar[rr]^{\xi}
\ar[d]_{\mathrm{R}(\chi')}&&\mathrm{E}\circ\mathrm{R}
\ar[d]^{\chi_{\mathrm{R}}}\\ 
\mathrm{R}\circ\mathrm{E}'\ar[rr]^{\xi}&&\mathrm{E}\circ\mathrm{R};\\ 
}
\end{displaymath} 
\begin{displaymath}
\xymatrix{
\mathrm{R}\circ\mathrm{E}'\circ\mathrm{E}'
\ar[rr]^{\xi_{\mathrm{E}'}}\ar[d]_{\mathrm{R}(\tau')}&&
\mathrm{E}\circ\mathrm{R}\circ\mathrm{E}'
\ar[rr]^{\mathrm{E}(\xi)}&&
\mathrm{E}\circ\mathrm{E}\circ\mathrm{R}
\ar[d]_{\tau_{\mathrm{R}}}\\
\mathrm{R}\circ\mathrm{E}'\circ\mathrm{E}'
\ar[rr]^{\xi_{\mathrm{E}'}}&&
\mathrm{E}\circ\mathrm{R}\circ\mathrm{E}'
\ar[rr]^{\mathrm{E}(\xi)}&&
\mathrm{E}\circ\mathrm{E}\circ\mathrm{R}.
}
\end{displaymath} 
\end{enumerate}
} 
\end{definition}

In the terminology from \cite{Kh}, a morphism of two
$\mathfrak{sl}_2$-categorifications is a functor
from $\mathcal{A}'$ to $\mathcal{A}$ which {\em naturally} 
intertwines the functorial action of $\mathfrak{sl}_2$
on these two categories (see Subsection~\ref{s3.8}).

\subsection{Minimal $\mathfrak{sl}_2$-categorification of simple
finite dimensional modules}\label{s13.4}

Let $a$ and $q$ be as above. Let $\mathbb{P}_n$ denote the
subalgebra of $\mathbb{H}_n(q)$, generated by 
$X_1,\dots,X_n$ (and $X_1^{-1},\dots,X_n^{-1}$ if $q\neq 1$).
Denote by $\mathfrak{m}_n$ the ideal of $\mathbb{P}_n$ generated by 
$X_i-a$ and set $\mathfrak{n}_n=(\mathfrak{m}_n)^{\mathbb{S}_n}$
(were $\mathbb{S}_n$ acts on $\mathbb{P}_n$, as usual, by permuting 
the variables). Set $\overline{\mathbb{H}}_n(q):=
\mathbb{H}_n(q)/(\mathbb{H}_n(q)\mathfrak{n}_n)$.

For $k\leq n$ there is the obvious natural map 
$\mathbb{H}_k(q)\mapsto \mathbb{H}_n(q)$ (it is the identity on 
generators of $\mathbb{H}_k(q)$). We denote by 
$\overline{\mathbb{H}}_{k,n}(q)$ the image of $\mathbb{H}_k$
in $\overline{\mathbb{H}}_n(q)$. One can show that the algebra
$\overline{\mathbb{H}}_{k,n}(q)$ is local, symmetric and
independent of $a$ and $q$ up to isomorphism.

\begin{lemma}\label{lem1309}
For $i\leq j$ the algebra $\overline{\mathbb{H}}_{j,n}(q)$ 
is a free $\overline{\mathbb{H}}_{i,n}(q)$-module of rank
$\frac{(n-i)!j!}{(n-j)!i!}$.
\end{lemma}

For a fixed $n\in\mathbb{N}_0$ set 
$B_i:=\overline{\mathbb{H}}_{i,n}(q)$, $i=0,1,2\dots,n$.
Define the category $\mathcal{A}(n)$ as follows:
\begin{displaymath}
\mathcal{A}(n):=\bigoplus_{i=0}^{n}
\mathcal{A}(n)_{-n+2i}, \qquad\text{ where }\qquad
\mathcal{A}(n)_{-n+2i}:=
B_{i}\text{-}\mathrm{mod}.
\end{displaymath}  
Set further
\begin{displaymath}
\mathrm{E}:=\bigoplus_{i=0}^{n-1}\mathrm{Ind}_{B_i}^{B_{i+1}},\qquad
\mathrm{F}:=\bigoplus_{i=1}^{n}\mathrm{Res}_{B_{i-1}}^{B_{i}}.
\end{displaymath}
The image of $X_{i+1}$ in $B_{i+1}$ gives an endomorphism of 
$\mathrm{Ind}_{B_i}^{B_{i+1}}$ by the right multiplication.
Taking the sum over all $i$, we get an endomorphism of $\mathrm{E}$,
which we denote by $\chi$. Similarly, the image of $T_{i+1}$
in $B_{i+2}$ gives an endomorphism of 
$\mathrm{Ind}_{B_i}^{B_{i+2}}$ and, taking the sum over all $i$, we 
get an endomorphism of $\mathrm{E}^2$, which we denote by $\tau$.
The second part of the following statement from \cite{CR} is 
truly remarkable:

\begin{theorem}[Minimal categorification and its uniqueness]
\label{them1310}
\begin{enumerate}[$($a$)$]
\item\label{them1310.1} The tuple 
$\mathfrak{A}:=(\mathcal{A}(n),\mathrm{E},\mathrm{F},q,a,\chi,\tau)$ from the
above is an $\mathfrak{sl}_2$-categorification of the simple
$n+1$-dimensional $\mathfrak{sl}_2$-module.
\item\label{them1310.2} 
If $\mathfrak{A}':=(\mathcal{A}',\mathrm{E}',\mathrm{F}',q,a,\chi',\tau')$
is an $\mathfrak{sl}_2$-categorification of the simple
$n+1$-dimensional $\mathfrak{sl}_2$-module such that the (unique)
simple object annihilated by $\mathrm{F}'$ is projective, then
the $\mathfrak{sl}_2$-categorifications
$\mathfrak{A}$ and $\mathfrak{A}'$ are equivalent.
\end{enumerate}
\end{theorem}

\subsection{$\mathfrak{sl}_2$-categorification on category
$\mathcal{O}$}\label{s13.5}

In this subsection we assume $q=1$. Let $a\in\mathbb{C}$.
For $\lambda,\mu\in\mathfrak{h}^*$ write 
$\lambda\rightarrow_a\mu$ provided that there exists $j\in\{1,2,\dots,n\}$
such that $\lambda_j-j+1=a-1$, $\mu_j-j+1=a$ and $\lambda_i=\mu_i$
for all $i\neq j$. For $\lambda,\mu\in\mathfrak{h}^*_{\mathrm{dom}}$
we write $\chi_{\lambda}\to_a\chi_{\mu}$ provided that there exist
$\lambda'\in W\cdot \lambda$ and $\mu'\in W\cdot \mu$ such that 
$\lambda'\to_a\mu'$. 

Let $V$ be the natural $n$-dimensional representation of 
$\mathfrak{g}$. Then the projective endofunctor $V\otimes_{\mathbb{C}}{}_-$
of $\mathcal{O}$ decomposes as follows:
\begin{displaymath}
V\otimes_{\mathbb{C}}{}_-=\bigoplus_{a\in\mathbb{C}}\mathrm{E}_a,\qquad
\text{ where }\qquad
\mathrm{E}_a=\bigoplus_{\lambda,\mu\in \mathfrak{h}^*_{\mathrm{dom}};
\chi_{\lambda}\to_a\chi_{\mu}}
\mathrm{Pr}_{\mathcal{O}_{\mu}}\circ(V\otimes_{\mathbb{C}}{}_-)\circ
\mathrm{Pr}_{\mathcal{O}_{\lambda}},
\end{displaymath}
here the projection $\mathrm{Pr}$ is defined with respect to 
the decomposition from Theorem~\ref{thm403}. Denote by 
$\mathrm{F}_a$ the (both left and right) adjoint of $\mathrm{E}_a$.

Define $\tau\in\mathrm{End}((V\otimes_{\mathbb{C}}{}_- )^2)$
as follows: for $M\in\mathcal{O}$, $m\in M$ and $v,v'\in V$ set 
$\tau_M(v\otimes v'\otimes m):=v'\otimes v\otimes m$.
Consider the element $\Omega:=\sum_{i,j=1}^ne_{ij}\otimes e_{ji}\in
\mathfrak{g}\otimes \mathfrak{g}$.
Define $\chi\in \mathrm{End}(V\otimes_{\mathbb{C}}{}_- )$ as follows:
for $M\in\mathcal{O}$, $m\in M$ and $v\in V$ set 
$\chi_M(v\otimes m):=\Omega(v\otimes m)$.

\begin{lemma}\label{lem1311}
For every $a\in\mathbb{C}$ the natural transformations
$\chi$ and $\tau$ restrict to $\mathrm{E}_a$ and $\mathrm{E}_a^2$,
respectively.
\end{lemma}

We denote by $\chi_a$ and $\tau_a$ the restrictions of 
$\chi$ and $\tau$ to $\mathrm{E}_a$ and $\mathrm{E}_a^2$, respectively.

\begin{proposition}[\cite{CR}]\label{prop1312}
For every $a\in\mathbb{C}$ the tuple 
$(\mathcal{O},\mathrm{E}_a,\mathrm{F}_a,1,a,\chi_a,\tau_a)$ 
is an $\mathfrak{sl}_2$-categorification.
\end{proposition}

\begin{corollary}\label{cor1314}
The categorification of the simple $n+1$-dimensional
$\mathfrak{sl}_2$-module from Subsection~\ref{s11.4} is
an $\mathfrak{sl}_2$-categorification.
\end{corollary}

\subsection{Categorification of the simple reflection}\label{s13.6}

Consider the element
\begin{displaymath}
S=\left(\begin{array}{cc}0&1\\-1&0\end{array}\right)=
\mathrm{exp}(-\mathbf{f})\mathrm{exp}(\mathbf{e})\mathrm{exp}(-\mathbf{f})\in\mathrm{SL}_2.
\end{displaymath}

\begin{lemma}\label{lem1315}
Let $V$ be an integrable $\mathfrak{sl}_2$-module, that is a direct 
sum of finite dimensional $\mathfrak{sl}_2$-modules. Let
$\lambda\in\mathbb{Z}$ and $v\in V_{\lambda}$.
\begin{enumerate}[$($a$)$]
\item\label{lem1315.1} The action of $S$ induces an
isomorphism between $V_{\lambda}$ and $V_{-\lambda}$.
\item\label{lem1315.2} 
\begin{displaymath}
S(v)=\sum_{i=\mathrm{max}(0,-\lambda)}^{\mathrm{max}\{j:\mathbf{f}^j(v)\neq 0\}} 
\frac{(-1)^i}{i!(\lambda+i)!}\mathbf{e}^{\lambda+i}\mathbf{f}^i(v). 
\end{displaymath}
\end{enumerate}
\end{lemma}

Consider some $\mathfrak{sl}_2$-categorification
$(\mathcal{A},\mathrm{E},\mathrm{F},q,a,\chi,\tau)$. 
Let $\lambda\in\mathbb{Z}$. 
For $i\in\mathbb{Z}$ such that $i,\lambda+i\geq 0$, 
denote by $(\Theta_{\lambda})^{-i}$ the restriction of 
$\mathrm{E}^{(\mathrm{sign},\lambda+i)}\circ \mathrm{F}^{(\mathrm{1},i)}$
to $\mathcal{A}_{\lambda}$. Set $(\Theta)^{-i}=0$ otherwise.
The map
\begin{displaymath}
\xymatrix{
f:\mathrm{E}^{\lambda+i}\circ\mathrm{F}^{i}=
\mathrm{E}^{\lambda+i-1}\circ\mathrm{E}\circ\mathrm{F}\circ\mathrm{F}^{i-1}
\ar[rrrr]^>>>>>>>>>>>>>>>>>>>{\mathrm{id}_{\mathrm{E}^{\lambda+i-1}}\circ_0
\varepsilon\circ_0\mathrm{id}_{\mathrm{F}^{i-1}}}
&&&&\mathrm{E}^{\lambda+i-1}\circ \mathrm{F}^{i-1}
}
\end{displaymath}
restricts to a map
\begin{displaymath}
d^{-i}:
\mathrm{E}^{(\mathrm{sign},\lambda+i)}\circ\mathrm{F}^{(\mathrm{1},i)}\to
\mathrm{E}^{(\mathrm{sign},\lambda+i-1)}\circ\mathrm{F}^{(\mathrm{1},i-1)}.
\end{displaymath}
Set
\begin{displaymath}
\Theta_{\lambda}:=
\xymatrix{
\dots\ar[rr]^{d^{-i-1}} &&(\Theta_{\lambda})^{-i}\ar[rr]^{d^{-i}}
&&(\Theta_{\lambda})^{-i+1}\ar[rr]^{d^{-i+1}}&&\dots
}
\end{displaymath}
and put $\Theta:=\oplus_{\lambda}\Theta_{\lambda}$.
The following is described by the authors of \cite{CR}
as their main result:

\begin{theorem}[\cite{CR}]\label{thm1316}
We have the following:
\begin{enumerate}[$($a$)$]
\item\label{thm1316.1} $\Theta_{\lambda}$ is a complex.
\item\label{thm1316.2} The action of 
$[\Theta_{\lambda}]:[\mathcal{A}_{\lambda}]\to [\mathcal{A}_{-\lambda}]$
coincides with the action of $S$.
\item\label{thm1316.3} The complex of functors $\Theta$ induces
a self-equivalence of $\mathcal{K}^{b}(\mathcal{A})$ and of
$\mathcal{D}^{b}(\mathcal{A})$ and restricts to equivalences
$\mathcal{K}^{b}(\mathcal{A}_{\lambda})\overset{\sim}{\to}
\mathcal{K}^{b}(\mathcal{A}_{-\lambda})$ and
$\mathcal{D}^{b}(\mathcal{A}_{\lambda})\overset{\sim}{\to}
\mathcal{D}^{b}(\mathcal{A}_{-\lambda})$.
\end{enumerate}
\end{theorem}

\begin{proof}[Idea of the proof.]
First, prove the claim for a minimal categorification.
Use the first step to prove the claim for any categorification
of an isotypic module. Finally, there is a filtration of 
$\mathcal{A}$ by Serre subcategories such that 
subquotients of this filtration descend exactly to 
isotypic $\mathfrak{sl}_2$-components in $[\mathcal{A}]$.
Deduce now the claim in the general case using this filtration
and step two. 
\end{proof}

For a construction of a $2$-category categorifying 
$U(\mathfrak{sl}_2)$ see \cite{La,Ro}.

\section{Application: blocks of $\mathbb{F}[\mathbb{S}_n]$ and 
Brou{\'e}'s conjecture}\label{s14}

\subsection{Jucys-Murphy elements and formal characters}\label{s14.1}

In the first part of this section we closely follow the exposition 
from \cite{BrKl}.
Let $\mathbb{F}$ denote a field of arbitrary characteristic $p$
and $\mathbb{F}_p=\mathbb{Z}/p\mathbb{Z}$. For $n\in\mathbb{N}_0$ 
consider the symmetric group $\mathbb{S}_n$. 

\begin{definition}\label{def1401}
{\rm  
For $k=1,\dots,n$ define the {\em Jucys-Murphy} element
\index{Jucys-Murphy element}
\begin{displaymath}
\mathtt{x}_k:=(1,k)+(2,k)+\dots+(k-1,k)\in\mathbb{F}[\mathbb{S}_n]. 
\end{displaymath}
} 
\end{definition}

Importance of Jucys-Murphy elements is explained by the following:

\begin{proposition}[\cite{Ju,Mu}]\label{prop1402}
The elements $\mathtt{x}_1$, $\mathtt{x}_2$,\dots, $\mathtt{x}_n$
commute with one other and the center of the group 
algebra $\mathbb{F}[\mathbb{S}_n]$ is precisely the set of symmetric 
polynomials in these elements.
\end{proposition}

Let $M$ be a finite dimensional $\mathbb{F}[\mathbb{S}_n]$-module. For
$\mathbf{i}=(i_1,i_2,\dots,i_n)\in \mathbb{F}_p^n$ set
\begin{displaymath}
M_{\mathbf{i}}:=\{
v\in M:(\mathtt{x}_r-i_r)^N v=0 \text{ for all } N\gg 0
\text{ and } r=1,2,\dots,n
\}.
\end{displaymath}
Amazingly enough, we have the following property:

\begin{proposition}[\cite{BrKl}]\label{prop1403}
We have  $\displaystyle
M=\bigoplus_{\mathbf{i}\in \mathbb{F}_p^n}M_{\mathbf{i}}$.
\end{proposition}

Consider the free $\mathbb{Z}$-module $\mathrm{Char}$ with basis
$\{\exp(\mathbf{i}):\mathbf{i}\in \mathbb{F}_p^n\}$.

\begin{definition}\label{def1404}
{\rm  
Let $M$ be an $\mathbb{F}[\mathbb{S}_n]$-module. The
{\em formal character} of $M$ is
\index{formal character}
\begin{displaymath}
\mathrm{ch}\,M:=\sum_{\mathbf{i}\in \mathbb{F}_p^n} 
 \dim M_{\mathbf{i}}\exp(\mathbf{i})\in \mathrm{Char}.
\end{displaymath}
} 
\end{definition}

\begin{proposition}[\cite{Va}]\label{prop1405}
\begin{enumerate}[$($a$)$]
\item\label{prop1405.1} If $M\hookrightarrow N\tto K$ is a short
exact sequence of $\mathbb{F}[\mathbb{S}_n]$-modules, then
$\mathrm{ch}\,N=\mathrm{ch}\,M+\mathrm{ch}\,K$.
\item\label{prop1405.2} The formal characters of the inequivalent
irreducible $\mathbb{F}[\mathbb{S}_n]$-modules are linearly
independent.
\end{enumerate}
\end{proposition}

Define the set
\begin{displaymath}
\Gamma_n:=\{\gamma=(\gamma_r)_{r\in \mathbb{F}_p}:
\gamma_r\in\mathbb{N}_0,\sum_{r\in \mathbb{F}_p}\gamma_r=n\}. 
\end{displaymath}
For $\mathbf{i}\in \mathbb{F}_p^n$ set
\begin{displaymath}
\mathrm{wt}(\mathbf{i}):=(\gamma_r)_{r\in \mathbb{F}_p}\in
\Gamma_n,\quad\text{ where }
\gamma_r=|\{j\in\{1,2,\dots,n\}:i_j=r\}|.
\end{displaymath}
For $\gamma\in\Gamma_n$ set
\begin{displaymath}
M(\gamma):=\bigoplus_{\mathbf{i}\in \mathbb{F}_p^n\,:\,
\mathrm{wt}(\mathbf{i})=\gamma}M_{\mathbf{i}}\subset M.
\end{displaymath}

\begin{corollary}\label{cor1406}
For every  $\gamma\in\Gamma_n$ the subspace 
$M(\gamma)$ is a submodule of $M$.
\end{corollary}

As a consequence, we have the following decomposition
of $\mathbb{F}[\mathbb{S}_n]\text{-}\mathrm{mod}$ into {\em blocks}:
\index{block}
\begin{displaymath}
\mathbb{F}[\mathbb{S}_n]\text{-}\mathrm{mod}\cong
\bigoplus_{\gamma\in\Gamma_n} \mathcal{B}_{\gamma},
\end{displaymath}
where $\mathcal{B}_{\gamma}$ denotes the full subcategory of 
$\mathbb{F}[\mathbb{S}_n]\text{-}\mathrm{mod}$ consisting of all
modules $M$ satisfying $M=M(\gamma)$.

\subsection{Induction and restriction}\label{s14.2}

For every $i\in\mathbb{F}_p$ define the insertion operation
$\mathrm{ins}_i:\Gamma_n\to\Gamma_{n+1}$ as follows:
for $\gamma=(\gamma_r)_{r\in \mathbb{F}_p}$ set
$\mathrm{ins}_i(\gamma)$ to be $(\gamma'_r)_{r\in \mathbb{F}_p}$
such that $\gamma'_r=\gamma_r$ for all $r\neq i$
and $\gamma'_i=\gamma_i+1$.

For every $i\in\mathbb{F}_p$ define the removal operation 
$\mathrm{rem}_i:\Gamma_n\to\Gamma_{n-1}$ (this will be a
partially defined map) as follows: for 
$\gamma=(\gamma_r)_{r\in \mathbb{F}_p}$ the image of
$\gamma$ under $\mathrm{rem}_i$ is defined if and only if
$\gamma_i> 0$, and, if this condition is satisfied, we set
$\mathrm{rem}_i(\gamma)$ to be $(\gamma'_r)_{r\in \mathbb{F}_p}$
such that $\gamma'_r=\gamma_r$ for all $r\neq i$
and $\gamma'_i=\gamma_i-1$.

For every $i\in\mathbb{F}_p$ we now define functors
\begin{displaymath}
\mathrm{E}_i:\mathbb{F}[\mathbb{S}_n]\text{-}\mathrm{mod}\to
\mathbb{F}[\mathbb{S}_{n-1}]\text{-}\mathrm{mod}\quad\text{ and }\quad
\mathrm{F}_i:\mathbb{F}[\mathbb{S}_n]\text{-}\mathrm{mod}\to
\mathbb{F}[\mathbb{S}_{n+1}]\text{-}\mathrm{mod}
\end{displaymath}
as follows: For $\gamma\in\Gamma_n$ and $M\in \mathcal{B}_{\gamma}$ set
\begin{displaymath}
\mathrm{E}_i\, M:=
\begin{cases}
(\mathrm{Res}_{\mathbb{S}_{n-1}}^{\mathbb{S}_n}\, M)
(\mathrm{rem}_i(\gamma)), &  
\mathrm{rem}_i(\gamma) \text{ is defined};\\
0,& \text{otherwise},
\end{cases}
\end{displaymath}
and
\begin{displaymath}
\mathrm{F}_i\, M:=
(\mathrm{Ind}^{\mathbb{S}_{n+1}}_{\mathbb{S}_n}\, M)
(\mathrm{ins}_i(\gamma)).
\end{displaymath}
Extending additively, this defines $\mathrm{E}_i$ and
$\mathrm{F}_i$ on the whole of 
$\mathbb{F}[\mathbb{S}_n]\text{-}\mathrm{mod}$. Directly from the
definition we obtain the following:
for every $\mathbb{F}[\mathbb{S}_n]$-module $M$ we have
\begin{displaymath}
\mathrm{Res}_{\mathbb{S}_{n-1}}^{\mathbb{S}_n}\, M\cong
\bigoplus_{i\in\mathbb{F}_p} \mathrm{E}_i\, M,\qquad
\mathrm{Ind}^{\mathbb{S}_{n+1}}_{\mathbb{S}_n}\, M\cong
\bigoplus_{i\in\mathbb{F}_p} \mathrm{F}_i\, M.
\end{displaymath}
Using the Frobenius reciprocity, it follows that $\mathrm{E}_i$ is 
both left and right adjoint of $\mathrm{F}_i$. In particular, both 
$\mathrm{E}_i$ and $\mathrm{F}_i$ are exact functors.

\subsection{Categorification of the basic representation of 
an affine Kac-Moody algebra}\label{s14.3}

Denote by $\mathtt{R}_n$ the free $\mathbb{Z}$-module spanned by the
formal characters of irreducible $\mathbb{F}[\mathbb{S}_n]$-modules,
the so-called {\em character ring} of $\mathbb{F}[\mathbb{S}_n]$.
\index{character ring}
By Proposition~\ref{prop1405}\eqref{prop1405.2}, the map
$\mathrm{ch}$ induces an isomorphism between $\mathtt{R}_n$
and the Grothendieck group of the category of all finite dimensional
$\mathbb{F}[\mathbb{S}_n]$-modules. Set
\begin{displaymath}
\mathtt{R}:=\bigoplus_{n\in\mathbb{N}_0}\mathtt{R}_n.
\end{displaymath}
The exact functors $\mathrm{E}_i$ and $\mathrm{F}_i$ induce
$\mathbb{Z}$-linear endomorphisms of $\mathtt{R}$. Extending the
scalars to $\mathbb{C}$ we get the vector space
\begin{displaymath}
\mathtt{R}_{\mathbb{C}}=
\mathbb{C}\otimes_{\mathbb{Z}}\mathtt{R}
\end{displaymath}
and linear operators $[\mathrm{E}_i]$ and $[\mathrm{F}_i]$ on it.

\begin{theorem}[\cite{LLT,Ari,Gr}]\label{thm1407}
\begin{enumerate}[$($a$)$]
\item\label{thm1407.1} The linear operators 
$[\mathrm{E}_i]$ and $[\mathrm{F}_i]$, $i\in \mathbb{F}_p$,
satisfy the defining relations for the Chevalley generators of an
affine Kac-Moody Lie algebra $\mathfrak{a}$ of type
$A_{p-1}^{(1)}$ (resp. $A_{\infty}$ in the case $p=0$). 
\item\label{thm1407.2} The  $\mathfrak{a}$-module 
$\mathtt{R}_{\mathbb{C}}$, given by \eqref{thm1407.1}, is isomorphic
to the basic highest weight representation $V(\Lambda_0)$ of
this Kac-Moody algebra,
where $\Lambda_0=(1,0,0,\dots)$, and is generated by
the highest weight vector, represented by the character  of the
irreducible $\mathbb{F}[\mathbb{S}_0]$-module.
\end{enumerate}
\end{theorem}

\subsection{Brou{\'e}'s conjecture}\label{s14.4}

In this overview we closely follow \cite{Ric} and refer the reader
to this nice survey for more details.
Let $G$ be a finite group and $\mathbb{F}$ be a field of characteristic
$p> 0$. The finite dimensional group $\mathbb{F}$-algebra $\mathbb{F}[G]$
decomposes into a direct sum of indecomposable subalgebras,
called {\em blocks}. Let $A$ be a block of $\mathbb{F}[G]$. 
A {\em defect group} of $A$ is a minimal subgroup $D$ of $G$ such that 
\index{defect group}
every  $A$-module is a direct summand of a module induced from 
$\mathbb{F}[D]$. A defect group is always a $p$-subgroup of $G$ and
is determined uniquely up to conjugacy in $G$. Denote by 
$N_{G}(D)$ the normalizer of $D$ in $G$.

\begin{theorem}[Brauer's First Main Theorem]\label{thm1408}
If $D$ is a $p$-subgroup of $G$, then there is a natural bijection
between the blocks of $\mathbb{F}[G]$ with defect group $D$
and the blocks of $\mathbb{F}[N_{G}(D)]$ with defect group $D$.
\end{theorem}

If $A$ is a block of $G$ with defect group $D$, then the block $B$
of $N_{G}(D)$ corresponding to $A$ via the correspondence described by
Theorem~\ref{thm1408} is called the {\em Brauer correspondent} of $A$.
\index{Brauer correspondent}
For example, if $A\text{-}\mathrm{mod}$ contains the trivial
$G$-module, then the defect group $D$ of $A$ is a Sylow $p$-subgroup 
of $G$ and the Brauer correspondent of $A$ is the block of 
$\mathbb{F}[N_{G}(D)]$ containing the trivial $N_{G}(D)$-module.

\begin{conjecture}[Alperin's Weight Conjecture
for abelian defect groups]\label{conj1409}
{\rm 
Let $G$ be a finite group, $A$ a block of $G$ with an abelian
defect group $D$ and  $B$ the Brauer correspondent of $A$. 
Then $A\text{-}\mathrm{mod}$ and $B\text{-}\mathrm{mod}$ have 
the same number of isomorphism classes  of simple modules.
}
\end{conjecture}

The number of isomorphism classes of simple modules can be
interpreted as an isomorphism between the Grothendieck groups
of $A\text{-}\mathrm{mod}$ and $B\text{-}\mathrm{mod}$.
Conceptually, it would be nice if this isomorphism would
come from some ``higher level'' isomorphism, for example
from an isomorphism (or Morita equivalence) of $A$ and $B$.
Unfortunately, small examples show that this cannot be expected.
However, there is still a conjectural higher level
isomorphism as described in:

\begin{conjecture}[Brou{\'e}'s abelian defect group 
Conjecture]\label{conj1419}
{\rm 
Let $G$ be a finite group, $A$ a block of $G$ with an abelian
defect group $D$ and  $B$ the Brauer correspondent of $A$. 
Then $\mathcal{D}^b(A)$ and $\mathcal{D}^b(B)$ are equivalent
as triangulated categories.
}
\end{conjecture}

\subsection{Brou{\'e}'s conjecture for $\mathbb{S}_n$}\label{s14.5}

Let us go back to the character ring $\mathtt{R}_n$ 
of $\mathbb{F}[\mathbb{S}_n]$, whose complexification
$\mathtt{R}_{\mathbb{C}}$ was realized  in Subsection~\ref{s14.3}
as the basic highest weight representation $V(\Lambda_0)$ over some 
affine Kac-Moody Lie algebra $\mathfrak{a}$.
It turns out that the theory of blocks has a natural description
in terms of this categorification picture:

\begin{theorem}[\cite{LLT,CR}]\label{thm1411}
\begin{enumerate}[$($a$)$]
\item\label{thm1411.2} The decomposition of $\mathtt{R}$ into blocks
gives exactly the weight space decomposition of $V(\Lambda_0)$ with 
respect to the standard Cartan subalgebra of $\mathfrak{g}$.
\item\label{thm1411.1} Two blocks of symmetric groups have 
isomorphic defect groups if and only if the corresponding weight 
spaces are in the same orbit under the action of the affine Weyl group.
\end{enumerate}
\end{theorem}

Now we can state the principal application of 
$\mathfrak{sl}_2$-categorification from \cite{CR}:

\begin{theorem}[\cite{CR}]\label{thm1412}
Let $p>0$ and $\mathbb{F}$ be either $\mathbb{Z}_{(p)}$ or an
algebraically closed  field of characteristic $p$.
\begin{enumerate}[$($a$)$]
\item\label{thm1412.1} For every $i\in\mathbb{F}_p$ the 
action of $\mathrm{E}_i$ and $\mathrm{F}_i$ on 
\begin{displaymath}
\mathfrak{S}:=\bigoplus_{n\in\mathbb{N}_0} 
\mathbb{F}[\mathbb{S}_n]\text{-}\mathrm{mod}
\end{displaymath}
extends to an $\mathfrak{sl}_2$-categorification over $\mathbb{F}$.
\item\label{thm1412.2} Let $A$ and $B$ be two blocks of symmetric
groups over $\mathbb{F}$ with isomorphic defect groups. Then
$\mathcal{D}^b(A)$ and $\mathcal{D}^b(B)$ are equivalent
as triangulated categories.
\end{enumerate}
\end{theorem}

\begin{proof}[Idea of the proof of \eqref{thm1412.2}.]
Because of Theorem~\ref{thm1411}\eqref{thm1411.2}, it is enough
to show that two blocks which are connected by a simple reflection
of the affine Weyl group are derived equivalent. To prove this,
use \eqref{thm1412.1} and Theorem~\ref{thm1316}. 
\end{proof}

\begin{corollary}\label{thm1467}
Brou{\'e}'s abelian defect group Conjecture holds for $G=\mathbb{S}_n$.
\end{corollary}

\subsection{Divided powers}\label{s14.6}

Consider again the functorial action of the affine Kac-Moody Lie
algebra $\mathfrak{a}$ on the direct sum of module categories of
all symmetric groups as constructed in Subsection~\ref{s14.3}.
There is an important set of elements of $\mathfrak{a}$ called
{\em divided powers} and defined (in terms of the Chevalley 
\index{divided powers}
generators $e_i$'s and $f_i$'s of $\mathfrak{a}$) as follows:
\begin{displaymath}
e_i^{(r)}:=\frac{e_i^r}{r!}\qquad\text{ and } \qquad
f_i^{(r)}:=\frac{f_i^r}{r!}. 
\end{displaymath}
It turns out that these elements also admit a natural functorial
description. As in Subsection~\ref{s14.2}, we let $i\in\mathbb{F}_p$.

Let $M$ be an $\mathbb{F}[\mathbb{S}_n]$-module and assume
that there exists $\gamma\in\Gamma_n$
such that $M=M(\gamma)$. Fix $r\in\mathbb{N}$. View $M$ as an
$\mathbb{F}(\mathbb{S}_n\oplus \mathbb{S}_r)$-module by letting 
$\mathbb{S}_r$ act trivially. Embed $\mathbb{S}_n\oplus \mathbb{S}_r$
into $\mathbb{S}_{n+r}$ in the obvious way and define
\begin{displaymath}
\mathrm{F}_i^{(r)}\, M:=
(\mathrm{Ind}_{\mathbb{S}_n\oplus \mathbb{S}_r}^{\mathbb{S}_{n+r}}\,
M)(\mathrm{ins}_i^r(\gamma)).
\end{displaymath}
To define $\mathrm{E}_i^{(r)}$ write $n=k+r$. If $k<0$, then
the definition below gives $\mathrm{E}_i^{(r)}\, M=0$. If
$k\geq 0$, consider the obvious embedding of
$\mathbb{S}_k\oplus \mathbb{S}_r$ into $\mathbb{S}_n$. Then the
set $M^{\mathbb{S}_r}$ of $\mathbb{S}_r$-fixed points in $M$
is naturally an $\mathbb{S}_k$-module. Set
\begin{displaymath}
\mathrm{E}_i^{(r)}\, M :=
\begin{cases}
(M^{\mathbb{S}_r})(\mathrm{rem}_i^r(\gamma)),
&\text{$\mathrm{rem}_i^r(\gamma)$ is defined};\\
0, & \text{otherwise}.
\end{cases}
\end{displaymath}
Extending additively, defines functors $\mathrm{E}_i^{(r)}$
and $\mathrm{F}_i^{(r)}$ on all $\mathbb{F}[\mathbb{S}_n]$-modules
and thus gives the corresponding endofunctors of $\mathfrak{S}$.
Clearly, $\mathrm{E}_i^{(r)}$ is both left and right adjoint
to $\mathrm{F}_i^{(r)}$. Denote by $\mathtt{R}^*_n$ the
$\mathbb{Z}$-submodule of $\mathtt{R}_n$ spanned by the formal
characters of indecomposable projective $\mathbb{F}[\mathbb{S}_n]$-modules
and set
\begin{displaymath}
\mathtt{R}^*:=\bigoplus_{n\in\mathbb{N}_0}\mathtt{R}^*_n .
\end{displaymath}

\begin{theorem}[\cite{Gr}]\label{thm1414}
\begin{enumerate}[$($a$)$]
\item\label{thm1414.1} There are isomorphisms of functors as follows:
\begin{displaymath}
\mathrm{E}_i^{r}\cong (\mathrm{E}_i^{(r)})^{\oplus r!}\qquad
\text{ and }\qquad
\mathrm{F}_i^{r}\cong (\mathrm{F}_i^{(r)})^{\oplus r!}.
\end{displaymath}
\item\label{thm1414.2} The lattice $\mathtt{R}^*\subset 
\mathtt{R}_{\mathbb{C}}$ is the $\mathbb{Z}$-submodule of 
$\mathtt{R}_{\mathbb{C}}$ generated by the character of the
irreducible $\mathbb{F}[\mathbb{S}_0]$-module under the action of
$\mathrm{F}_i^{(r)}$, $i\in\mathbb{F}_p$, $r\in\mathbb{N}_0$.
\item\label{thm1414.3} The lattice  $\mathtt{R}\subset 
\mathtt{R}_{\mathbb{C}}$ is dual to $\mathtt{R}^*$ with respect
to the Shapovalov form.
\end{enumerate}
\end{theorem}

Claim \eqref{thm1414.1} of Theorem~\ref{thm1414} can be derived from
Corollary~\ref{cor1306} and Theorem~\ref{thm1412}\eqref{thm1412.1}.

\section{Applications: of $\mathbb{S}_n$-categorifications}\label{s15}

\subsection{Wedderburn basis for $\mathbb{C}[\mathbb{S}_n]$}\label{s15.1}

Artin-Wedderburn theorem says that every semi-simple complex finite
dimensional algebra $A$ is isomorphic to a direct sum of matrix algebras
of the form $\mathrm{Mat}_{k\times k}(\mathbb{C})$. 
If such a decomposition of $A$ is fixed, then the basis 
of $A$, consisting of matrix units in all components, is called 
a {\em Wedderburn basis} for $A$. A nice property of this basis is
\index{Wedderburn basis}
that if we take the linear span of basis elements along a fixed
row in some matrix component $\mathrm{Mat}_{k\times k}(\mathbb{C})$
of $A$, then this linear span is stable with respect to the right
multiplication with elements from $A$ and thus gives a simple submodule
of the right regular module $A_A$. In other words, each 
Wedderburn basis for $A$ automatically gives a decomposition of 
the right regular $A$-module into a direct sum of simple modules.
For the left regular module one should just consider columns in 
matrix components instead of rows.

Consider $A=\mathbb{C}[\mathbb{S}_n]$, which is a semi-simple algebra by
Maschke's Theorem. The right regular representation of $A$
is categorified in Proposition~\ref{prop521} via the action of
projective functors on a regular integral block $\mathcal{O}_{\lambda}$ 
for the algebra $\mathfrak{gl}_n$. In this categorification picture 
the standard basis of the module $\mathbb{C}[\mathbb{S}_n]$ is given by 
the classes $[\Delta(w\cdot \lambda)]$ of Verma modules in 
$\mathcal{O}_{\lambda}$. For a simple reflection $s$, the element $e+s$ 
of $\mathbb{C}[\mathbb{S}_n]$ acts on $\mathcal{O}_{\lambda}$ via the 
projective functor $\theta_s$.

For $w\in \mathbb{S}_n$ denote by $\overline{w}$ the unique involution
in the right cell of $w$, that is the element of $\mathbb{S}_n$
such that $p(\overline{w})=q(\overline{w})=p(w)$, where 
$(p(w),q(w))$ is the pair of standard Young tableaux associated to
$w$ by the Robinson-Schensted correspondence (see \cite[Section~3.1]{Sa}).
Consider the {\em evaluation map}
\index{evaluation map}
\begin{displaymath}
\mathrm{ev}:\mathbb{C}\otimes_{\mathbb{Z}}\mathbb{H}
\overset{\mathrm{proj}}{\longrightarrow}
(\mathbb{C}\otimes_{\mathbb{Z}}\mathbb{H})/(v-1)
\overset{\sim}{\longrightarrow}\mathbb{C}[\mathbb{S}_n],
\end{displaymath}
which sends $1\otimes H_w$ to $w$. For $w\in \mathbb{S}_n$ define
\begin{displaymath}
f_w:=\mathrm{ev}(\hat{\underline{H}}_{\overline{w}}\underline{H}_w).
\end{displaymath}

\begin{example}\label{exm1502}
{\rm 
In the case $n=3$ denote by $s$ and $t$ the simple reflections in
$\mathbb{S}_3$. Then we have:
\begin{gather*}
f_e=(e-s-t+st+ts-sts)e=e-s-t+st+ts-sts;\\ 
f_s=(s-st-ts+sts)(e+s)=e+s-t-ts;\\ 
f_s=(t-st-ts+sts)(e+t)=e+t-s-st;\\ 
f_{st}=(s-st-ts+sts)(e+s+t+st)=s+st-ts-sts;\\ 
f_{ts}=(t-st-ts+sts)(e+s+t+ts)=t+ts-st-sts;\\ 
f_{sts}=sts(e+s+t+ts+st+sts)=e+s+t+ts+st+sts. 
\end{gather*}
}
\end{example}

Now observe that, by Proposition~\ref{prop521}, the element $f_w$
can be interpreted as $[\theta_w L(\overline{w}\cdot\lambda)]$.

\begin{proposition}[\cite{MS6}]\label{prop1501}
Let $w\in \mathbb{S}_n$ and $\mathcal{R}$ be the right cell of 
$w$. Then the module $\theta_w L(\overline{w}\cdot\lambda)$
is indecomposable and both projective and injective in 
$\mathcal{O}_{\lambda}^{\hat{\mathcal{R}}}$.
\end{proposition}

Combining Propositions~\ref{prop1501} and \ref{prop902} we see that 
$\{\theta_x L(\overline{w}\cdot\lambda):x\in\mathcal{R}\}$ is a
complete and irredundant list of indecomposable projective-injective
modules in $\mathcal{O}_{\lambda}^{\hat{\mathcal{R}}}$. By 
Theorem~\ref{thm903}, the action of projective functors preserves
the additive category of projective-injective modules in
$\mathcal{O}_{\lambda}^{\hat{\mathcal{R}}}$ and, going to the
Grothendieck group, categorifies the Kazhdan-Lusztig cell module
of $\mathcal{R}$. For $\mathbb{S}_n$ we know that all cell modules
are irreducible. This implies the following statement:

\begin{theorem}[\cite{MS6}]\label{thm1503}
\begin{enumerate}[$($a$)$]
\item\label{thm1503.1} The elements $\{f_w:w\in \mathbb{S}_n\}$ form
a basis of $\mathbb{C}[\mathbb{S}_n]$.
\item\label{thm1503.2} For $w\in \mathbb{S}_n$ let
$\mathbf{S}(w)$ denote the linear span of $f_x$, where $x$ is in
the right cell of $w$. Then $\mathbf{S}(w)$ is an irreducible 
submodule of the right regular module $\mathbb{C}[\mathbb{S}_n]$.
\end{enumerate}
\end{theorem}

We also have the following statement:

\begin{proposition}[\cite{Ne}]\label{prop1504}
The basis $f_w$ is, up to normalization, a Wedderburn basis 
for $\mathbb{C}[\mathbb{S}_n]$. 
\end{proposition}

The normalization factors which appear in Proposition~\ref{prop1504}
can be formulated in terms of the dimension of the endomorphism algebra
of the modules $\theta_x L(\overline{w}\cdot\lambda)$. This gives
a very clear categorical interpretation of this Wedderburn basis 
for $\mathbb{C}[\mathbb{S}_n]$. 

\subsection{Kostant's problem}\label{s15.2}

For every two $\mathfrak{g}$-modules $M$ and $N$ the space
$\mathrm{Hom}_{\mathbb{C}}(M,N)$ carries the natural structure of
a $U(\mathfrak{g})\text{-}U(\mathfrak{g})$-bimodule. Denote by
$\mathcal{L}(M,N)$ the subspace of $\mathrm{Hom}_{\mathbb{C}}(M,N)$
consisting of all elements, the adjoint action of $\mathfrak{g}$
on which is locally finite. The space $\mathcal{L}(M,N)$ turns out
to be a subbimodule of $\mathrm{Hom}_{\mathbb{C}}(M,N)$,
see \cite[Kapitel~6]{Ja}. 

Since the adjoint action of $U(\mathfrak{g})$ on $U(\mathfrak{g})$ 
is locally finite, for any $\mathfrak{g}$-module $M$ the natural
image of $U(\mathfrak{g})$ in $\mathrm{Hom}_{\mathbb{C}}(M,M)$
belongs to $\mathcal{L}(M,M)$. The kernel of this map is, by definition,
the annihilator $\mathrm{Ann}_{U(\mathfrak{g})}(M)$ of $M$, which 
gives us the following injective map of 
$U(\mathfrak{g})\text{-}U(\mathfrak{g})$-bimodules:
\begin{equation}\label{eq1591}
U(\mathfrak{g})/\mathrm{Ann}_{U(\mathfrak{g})}(M)\hookrightarrow
\mathcal{L}(M,M).
\end{equation}
The classical problem of Kostant, as described in \cite{Jo2},
can be formulated in the following way:

\begin{problem}[Kostant's problem]\label{prob1505}
For which $\mathfrak{g}$-modules $M$ the map \eqref{eq1591}
is surjective (and hence an isomorphism)?
\end{problem}

Unfortunately, the complete answer to Kostant's problem in not
even known for simple highest weight modules. However, many
special cases are settled. Here is a list of known results:
\begin{itemize}
\item Kostant's problem has positive answer for all Verma modules, see 
\cite{Jo2,Ja}.
\item Kostant's problem has positive answer for all quotients of
$\Delta(\lambda)$ if $\lambda\in\mathfrak{h}^*_{\mathrm{dom}}$, see
\cite[6.9]{Ja}.
\item If $\lambda\in\mathfrak{h}^*_{\mathrm{dom}}$ is integral
and regular and $\mathfrak{p}\subset\mathfrak{g}$ is a parabolic 
subalgebra, then Kostant's problem has 
positive answer for $L(w_o^{\mathfrak{p}}w_o\cdot\lambda)$, see \cite{GJ}.
\item If $\lambda\in\mathfrak{h}^*_{\mathrm{dom}}$ is integral
and regular and $\mathfrak{p}\subset\mathfrak{g}$ is a parabolic 
subalgebra, then Kostant's problem 
has positive answer for all quotients of
$\Delta(w_o^{\mathfrak{p}}w_o\cdot\lambda)$, see \cite{Kaa}.
\item If $\lambda\in\mathfrak{h}^*_{\mathrm{dom}}$ is integral
and regular, $\mathfrak{p}\subset\mathfrak{g}$ is a parabolic 
subalgebra and
$s\in W_{\mathfrak{p}}$ is a simple reflection, then Kostant's problem has 
positive answer for $L(sw_o^{\mathfrak{p}}w_o\cdot\lambda)$, see \cite{Ma7}.
\item If $\lambda\in\mathfrak{h}^*_{\mathrm{dom}}$ is integral
and regular, $\mathfrak{q}\subset\mathfrak{p}\subset\mathfrak{g}$ are  parabolic subalgebras, then Kostant's problem 
has positive answer for 
$L(w_o^{\mathfrak{q}}w_o^{\mathfrak{p}}w_o\cdot\lambda)$, see \cite{Kaa}.
\item If $\mathfrak{g}$ is of type $B_2$ with two simple
reflections $s$ and $t$, and $\lambda\in\mathfrak{h}^*_{\mathrm{dom}}$ 
is integral and regular, then Kostant's problem has {\em negative} answer
for $L(st\cdot \lambda)$, see \cite{Jo2}.
\item If $\mathfrak{g}$ is of type $A_3$ with three simple
reflections $r,s$ and $t$ (such that $rt=tr$), 
and $\lambda\in\mathfrak{h}^*_{\mathrm{dom}}$  is integral and regular, 
then Kostant's problem has {\em negative} answer
for $L(rt\cdot \lambda)$, see \cite{MS6}.
\end{itemize}
Some further results on Kostant's problem can be found in 
\cite{KaMa,Ma3,MS2,MM1} (see also references therein). In particular, 
in \cite{Ma3} it is shown that the positive answer to Kostant's problem 
for certain simple highest weight modules
can be equivalently reformulated in terms of the double centralizer
property with respect to projective injective modules in the category 
$\mathcal{O}_{\lambda}^{\hat{\mathcal{R}}}$.

In this subsection we describe one of the main results from \cite{MS2}. 
It is related to Kostant's problem and its proof is based
on the idea of categorification. We have the following classical statement:

\begin{theorem}[\cite{Vo,Jo3}]\label{thm1506}
Let $\lambda\in\mathfrak{h}^*_{\mathrm{dom}}$ be dominant and regular.
Then for any $x,y\in\mathbb{S}_n$ the equality 
$\mathrm{Ann}_{U(\mathfrak{g})}(L(x\cdot\lambda))=
\mathrm{Ann}_{U(\mathfrak{g})}(L(y\cdot\lambda))$
is equivalent to the condition $x\sim_L y$. 
\end{theorem}

Theorem~\ref{thm1506} described the left hand side of \eqref{eq1591}.
In the case $M\in\mathcal{O}$, the 
$U(\mathfrak{g})\text{-}U(\mathfrak{g})$-bimodule on the right hand
side of \eqref{eq1591} is a Harish-Chandra bimodule (see 
Subsection~\ref{s5.3}). In particular, with respect to the adjoint
action of $\mathfrak{g}$, the bimodule $\mathcal{L}(M,M)$ decomposes
into a direct sum of simple finite dimensional $\mathfrak{g}$-modules.
Moreover, if $V$ is a simple finite dimensional $\mathfrak{g}$-module, 
then, by \cite[6.8]{Ja}, we have:
\begin{equation}\label{eq1592}
[\mathcal{L}(M,M):V]=\dim\mathrm{Hom}_{\mathfrak{g}} (V\otimes M,M).
\end{equation}

Now let us recall Proposition~\ref{prop606}, which says that the action
of (derived) twisting functors on (the bounded derived category of)
$\mathcal{O}_{\lambda}$ gives a na{\"\i}ve categorification of the left
regular $\mathbb{Z}[\mathbb{S}_n]$-module. Twisting functors commute
with projective functors $V\otimes{}_-$, see Proposition~\ref{prop601}.
Further, twisting functors are Koszul dual of shuffling functors, 
see Theorem~\ref{thm1015}. Since projective functors combinatorially
preserve right cell modules, shuffling functors, being linear
combinations of projective functors, do the same. Taking the
Koszul dual switches left and right, which implies that 
twisting functors combinatorial preserve left cells and hence
give a na{\"\i}ve categorification of left cell modules. These are
again irreducible as we work with $\mathbb{S}_n$. Applying 
twisting functors (or, more accurately, the subfunctors
$\mathrm{Q}_s$ from Proposition~\ref{prop605}) to \eqref{eq1592}, 
one obtains the following:

\begin{proposition}[\cite{MS2}]\label{prop1508}
Let $\lambda\in\mathfrak{h}^*_{\mathrm{dom}}$ be dominant and regular.
Let $x,y\in \mathbb{S}_n$ be in the same left cell. Then 
\begin{displaymath}
\mathcal{L}(L(x\cdot \lambda),L(x\cdot \lambda))\cong
\mathcal{L}(L(y\cdot \lambda),L(y\cdot \lambda)).
\end{displaymath}
\end{proposition}

Combining Proposition~\ref{prop1508} with Theorem~\ref{thm1506} we obtain:

\begin{theorem}[\cite{MS2}]\label{thm1507}
Let $\lambda\in\mathfrak{h}^*_{\mathrm{dom}}$ be dominant and regular
and $x\in \mathbb{S}_n$. The answer to Kostant's problem for
$L(x\cdot \lambda)$ is an invariant of a left cell.
\end{theorem}

Theorem~\ref{thm1507} does not generalize to other types. Thus in
type $B_2$ with simple reflections $s$ and $t$ we have $s\sim_L st$,
however, Kostant's problem has positive answer for
$L(s\cdot\lambda)$, see \cite{Ma7}, and negative answer for
$L(st\cdot\lambda)$, see \cite{Jo2}. 

\subsection{Structure of induced modules}\label{s15.3}

We would like to finish with the following classical question for
which the idea of categorification gives a new interesting insight.
Let $\mathfrak{p}\subset\mathfrak{g}$ be a parabolic subalgebra, 
$\mathfrak{n}$ be the nilpotent radical
of $\mathfrak{p}$, and $\mathfrak{a}$ be the Levi factor
of $\mathfrak{p}$. Then any (simple) $\mathfrak{a}$-module $V$ can be
trivially extended to a $\mathfrak{p}$-module via
$\mathfrak{n}V:=0$. The induced module
\begin{displaymath}
M(\mathfrak{p},V):=U(\mathfrak{g})\bigotimes_{U(\mathfrak{p})}V
\end{displaymath}
is called the {\em generalized Verma module} associated to
\index{generalized Verma module}
$\mathfrak{p}$ and $V$. Taking $\mathfrak{p}=\mathfrak{b}$ gives
usual Verma modules. Taking $V$ finite dimensional, produces 
parabolic Verma modules. In fact, starting from a module in
the category $\mathcal{O}$ (for the algebra $\mathfrak{a}$)
produces a module in the category $\mathcal{O}$ (now, for the algebra 
$\mathfrak{g}$). The structure of usual Verma modules is combinatorially
given by Corollary~\ref{cor735} (that is, by Kazhdan-Lusztig
polynomials). A natural problem is: 

\begin{problem}\label{prop1611}
{\rm
Describe the structure of  $M(\mathfrak{p},V)$ for arbitrary $V$.
}
\end{problem}

The general case of this problem is still open, however, some 
important special cases are settled. The first difficulty of the
problem lies in the fact that simple modules over simple Lie algebras
are not classified (apart from the algebra $\mathfrak{sl}_2$,
which was done by R.~Block, see \cite[Section~6]{Ma}). So, at the first
stage several authors studied special cases of Problem~\ref{prop1611}
for some known classes of simple modules $V$, see e.g. \cite{MiSo,FM,MO0}
and references therein (see also the paper \cite{MS2} for an overview of 
the problem).

The first really general result came in \cite{KhMa0}, where it was shown 
that in the case when the semi-simple part of $\mathfrak{a}$
is isomorphic to $\mathfrak{sl}_2$, an essential part of the
structure of $M(\mathfrak{p},V)$, called the {\em rough structure},
\index{rough structure}
does not really depend on the module $V$ but rather on its
annihilator, that is a primitive ideal of $U(\mathfrak{g})$. 
In \cite{MS2} this idea was further developed to obtain the strongest, 
so far, result on Problem~\ref{prop1611}. Here is a very rough description
how it goes:

Instead of looking at $M(\mathfrak{p},V)$ we consider this module 
as an object of a certain category $\mathcal{A}$, which is a kind of
``smallest reasonable category'' containing $M(\mathfrak{p},V)$ and
closed under projective functors. The category $\mathcal{A}$ has the
form $\mathcal{O}(\mathfrak{p},\mathcal{C})$ for some $\mathfrak{p}$ 
and $\mathcal{C}$ as in Subsection~\ref{s9.4}. One shows that the category 
$\mathcal{A}$ decomposes into blocks and each block is equivalent
to a module category over a standardly stratified algebra.
The original module $M(\mathfrak{p},V)$ is related to a proper standard 
object with respect to this structure. Unfortunately, this relation is
not an equality. In general, the corresponding proper standard object 
only surjects onto $M(\mathfrak{p},V)$. But the kernel, which may be 
nonzero, is  always ``smaller'' in the sense
that it has a filtration by generalized Verma modules induced
from simple modules having {\em strictly bigger annihilators} than $V$.

Now, the main observation is that, using the action of projective
functors on $\mathcal{A}$, one can prove the following: 

\begin{theorem}[\cite{MS2}]\label{thm1521}
Let $\mathfrak{g}=\mathfrak{gl}_n$.
\begin{enumerate}[$($a$)$]
\item\label{thm1521.1}
The action of projective functors on a (regular) block of 
$\mathcal{A}$ categorifies an induced cell module, where the cell in 
question is determined (up to a choice inside the corresponding 
two-sided cell) by the annihilator of $V$.
\item\label{thm1521.2}
Every (regular) block of  $\mathcal{A}$ is equivalent to a
subcategory of $\mathcal{O}$ appearing in Theorem~\ref{thm917}, where
induced cell modules were categorified.
\end{enumerate}
\end{theorem}

Since $M(\mathfrak{p},V)$ has categorical description as a
proper standard module, any equivalence from the above reduces
(a part of) the question about the structure of $M(\mathfrak{p},V)$ to 
the corresponding question for the image of $M(\mathfrak{p},V)$
in $\mathcal{O}$. The latter can be solved using Kazhdan-Lusztig's 
combinatorics (similarly to how it is done for usual Verma modules). 
However, the categorical picture completely disregards 
simple subquotients of $M(\mathfrak{p},V)$ which are induced from
modules with strictly bigger annihilators, so this part of the
structure of $M(\mathfrak{p},V)$ will be lost. What we really can
see using the above picture is simple subquotients of 
$M(\mathfrak{p},V)$ induced from modules with ``comparable''
annihilators. This is what is called the {\em rough structure}
\index{rough structure}
of $M(\mathfrak{p},V)$. However, in many cases, for example, 
in those studied in \cite{MiSo,MO0},  it is known that the rough 
structure coincides with the full structure of the module.

\section{Exercises}\label{s16}

\begin{exercise}\label{exs1}
\begin{enumerate}[$($a$)$]
\item\label{exs1.1} Prove Lemma~\ref{lem1}. 
\item\label{exs1.2} Construct an example of an algebra
$A$ such that the groups $\varphi([\mathcal{P}(A)])$ 
and  $[A\text{-}\mathrm{mod}]$ (notation from Subsection~~\ref{s1.2}) 
have different ranks.
\item\label{exs1.3} Let $\cA$ be an abelian category and
$\cC$ a Serre subcategory of $\cA$. Show that 
$[\cA/\cC]\cong[\cA]/[\cC]$.
\item\label{exs1.4} Consider the algebra $U(\mathfrak{sl}_2)$ with
the generating system
$\mathbf{e}=\left(\begin{array}{cc}0&1\\0&0\end{array}\right)$,
$\mathbf{f}=\left(\begin{array}{cc}0&0\\1&0\end{array}\right)$. Let 
$(B\text{-}\mathrm{mod},\varphi,\mathrm{E},\mathrm{F})$ be a na{\"\i}ve
categorification of a simple finite dimensional $A$-module. Show
that for every simple $B$-module $L$ the element $\varphi^{-1}([L])$
is a weight vector.
\item\label{exs1.5} Describe all weak categorifications of the
algebra $\mathbb{C}[a]/(a^2-a)$ with involution $a^*=a$.
\item\label{exs1.6} Construct a categorification $\cC$ for the 
semigroup algebra of a finite semigroup $S$ with involution $*$ such
that $*$ is categorified by an anti-autoequivalence $\circledast$ of 
$\cC$ and in every $2$-representation of $\cC$ the anti-autoequivalence
$\circledast$ would correspond to taking the biadjoint functor.
\item\label{exs1.7} Prove Proposition~\ref{prop306}. 
\item\label{exs1.8} Let $\cC$ be a fiat category and
$\mathcal{R}_1$ and $\mathcal{R}_2$  be two right cells of
$\cC$ which do not belong to the same two-sided cell.
Show that the cell modules $\mathbf{C}_{\mathcal{R}_1}$
and $\mathbf{C}_{\mathcal{R}_2}$ are not isomorphic.
\item\label{exs1.9} Construct explicitly a ``minimal''
fiat-category $\cC$ categorifying the algebra 
$\mathbb{C}[a]/(a^2-a)$ with involution $a^*=a$. Construct also
cell modules for $\cC$.
\end{enumerate}
\end{exercise}

\begin{exercise}\label{exs2}
\begin{enumerate}[$($a$)$]
\item\label{exs2.1} Prove that $M(\lambda)$ is the maximal quotient 
of $P(\lambda)$ having the property Theorem~\ref{thm402}\eqref{thm402.2}.
\item\label{exs2.2} Prove Corollary~\ref{cor406}. 
\item\label{exs2.3} Let $A$ be a quasi-hereditary algebra. Prove that 
both the images of indecomposable projective modules and the images 
of indecomposable tilting modules in $[A\text{-}\mathrm{mod}]$ form
a basis there.
\item\label{exs2.4} Prove that projective functors commute with $\star$.
\item\label{exs2.5} Let $w\in W$ and $s$ be a simple reflection.
Show that $ws<w$ implies that
$\theta_s\circ\theta_w\cong \theta_w\oplus\theta_w$.
Show further that $ws>w$ implies that $\theta_s\circ\theta_w$ has a unique
direct summand isomorphic to $\theta_{ws}$.
\item\label{exs2.6} For integral, regular and dominant $\lambda$ show
that $\mathrm{C}_{w_o}\,P_{\lambda}\cong T_{\lambda}$ and deduce 
Ringel self-duality of $\mathcal{O}_{\lambda}$.
\item\label{exs2.7} Prove that $\mathrm{Z}_s\cong\star\circ
\hat{\mathrm{Z}}_s\circ\star$, where $s\in W$ is
a simple reflection.
\item\label{exs2.8} In the case $\mathfrak{g}=\mathfrak{gl}_2$ show
that $\mathrm{T}_s\neq\mathrm{C}_s$ despite of the fact that 
$\mathrm{T}_s\, M=\mathrm{C}_s\, M$ for any $M\in\mathcal{O}_0$.
\item\label{exs2.9} In the notation of Subsection~\ref{s6.4} show
that $\mathrm{R}_s^2\cong \mathrm{P}_s$.
\end{enumerate}
\end{exercise}

\begin{exercise}\label{exs3}
\begin{enumerate}[$($a$)$]
\item\label{exs3.1} Prove Theorem~\ref{thm702} using \eqref{eq791}.
\item\label{exs3.2} Deduce Corollary~\ref{cor704} from
the double centralizer property and Theorem~\ref{thm703}.
\item\label{exs3.3} Prove Lemma~\ref{lem731}.
\item\label{exs3.4} For $n=3$ compute the complete multiplication
table of $\mathbb{H}$ in the Kazhdan-Lusztig basis.
\item\label{exs3.5} Check that two different categorifications
of the two-dimensional Specht $\mathbb{S}_3$-module
(using different parabolic subalgebras of $\mathfrak{gl}_3$)
are equivalent.
\item\label{exs3.6} Compute explicitly the algebra categorifying
the two-dimensional simple $\mathbb{S}_3$-module.
\item\label{exs3.7} In the notation of Subsection~\ref{s9.1} 
let $\mathcal{R}$ be a right cell and $s$ a simple reflection such 
that $sw<w$ for any $w\in \mathcal{R}$. Show that the functor 
$\mathrm{Q}_s$ is exact on 
$\mathcal{O}_{\lambda}^{\hat{\mathcal{R}}}/\mathcal{C}$.
\item\label{exs3.8} Prove Lemma~\ref{lem909}.
\item\label{exs3.9} Check that both $\mathcal{K}$ and
$\mathcal{K}'$, as defined in Subsection~\ref{s9.4}, are closed
under the action of projective functors.
\end{enumerate}
\end{exercise}

\begin{exercise}\label{exs4}
\begin{enumerate}[$($a$)$]
\item\label{exs4.1} Show that there are no nontrivial homotopies 
between linear complexes of projective modules.
\item\label{exs4.2} Determine explicitly all indecomposable
objects in the category $\mathscr{LC}(B_0)$ in the case $n=2$
and verify Koszul self-duality of $B_0$ in this case.
\item\label{exs4.3} Verify Theorem~\ref{thm1015} in the case 
$n=2$ using an explicit calculation. In particular, determine 
explicitly the ``appropriate shift in grading and position''.
\item\label{exs4.4} Prove that $[-a]=-[a]$; 
$[a]\in\mathbb{Z}[v,v^{-1}]$;
$\left[\begin{array}{c}a\\n\end{array}\right]=(-1)^n
\left[\begin{array}{c}-a+n-1\\n\end{array}\right]$;
$\left[\begin{array}{c}a\\n\end{array}\right]=
\frac{[a]!}{[n]![a-n]!}$; 
$\left[\begin{array}{c}a\\n\end{array}\right]\in\mathbb{Z}[v,v^{-1}]$.
\item\label{exs4.5} Compute explicitly the categorifications of
the modules $\mathcal{V}_1^3$ and $\mathcal{V}_3$.
\item\label{exs4.6} Show that Serre subcategories defined via
increasing Gelfand-Kirillov dimension define a filtration of
$\mathcal{V}_1^{\otimes n}$ whose subquotients are isotypic
components of $\mathcal{V}_1^{\otimes n}$.
\item\label{exs4.7} Compute Jones polynomial for the
{\em trefoil knot}
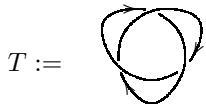
\index{trefoil knot}
\begin{displaymath}
T:=
\begin{array}{c}
\xygraph{
!{0;/r-1.0pc/:}
!P3"a"{~>{}}
!P9"b"{~:{(1.3288,0):}~>{}}
!P3"c"{~:{(2.5,0):}~>{}}
!{\vunder~{"b2"}{"b1"}{"a1"}{"a3"}}
!{\vcap~{"c1"}{"c1"}{"b4"}{"b2"}=<}
!{\vunder~{"b5"}{"b4"}{"a2"}{"a1"}}
!{\vcap~{"c2"}{"c2"}{"b7"}{"b5"}=<}
!{\vunder~{"b8"}{"b7"}{"a3"}{"a2"}}
!{\vcap~{"c3"}{"c3"}{"b1"}{"b8"}=<}
} 
\end{array}
\end{displaymath}
\item\label{exs4.8} Prove Proposition~\ref{prop1203}.
\item\label{exs4.9} Check that the endomorphism of
$\hat{\mathcal{V}}_1^2$, associated to the
right crossing, equals $-v\mathrm{Id}+\cup\circ\cap$, where
$\cup$ and $\cap$ denote morphisms associate to the cup and
cap diagrams, respectively. Find a similar expression for the 
morphism associated to the left crossing. 
\item\label{exs4.10} Show that the ranks of 
$[\mathcal{O}_{\lambda}^{\mathrm{max}}]$ and
$[\mathcal{O}_{0}^{\mathrm{max}}]$ from 
Proposition~\ref{prop1208} coincide.
\item\label{exs4.11} Check that in the case $n=2$ the combinatorics
of translations to and from the walls and that of shuffling and
coshuffling functors corresponds to the combinatorics of morphisms
between tensor powers of $\mathcal{V}_1$ as described in
Subsection~\ref{s12.3}.
\end{enumerate}
\end{exercise}

\begin{exercise}\label{exs5}
\begin{enumerate}[$($a$)$]
\item\label{exs5.2} Classify all $\mathfrak{sl}_2$-categorifications
of the simple $1$- and $2$-dimensional $\mathfrak{sl}_2$-modules.
\item\label{exs5.3} Prove Lemma~\ref{lem1315}. 
\item\label{exs5.4} Given an $\mathfrak{sl}_2$-categorification,
show that there is a filtration of 
$\mathcal{A}$ by Serre subcategories such that 
subquotients of this filtration descend exactly to 
isotypic $\mathfrak{sl}_2$-components in $[\mathcal{A}]$.
\item\label{exs5.5} Prove Proposition~\ref{prop1402} in the
case $\mathbb{F}=\mathbb{C}$.
\item\label{exs5.6} Let $M$ be an $\mathbb{F}[\mathbb{S}_n]$-module.
Show that 
\begin{displaymath}
\mathrm{ch}\, M=
\sum_{\mathbf{i}\in\mathbb{F}_p^n}a_{\mathbf{i}}\exp(\mathbf{i})
\quad\text{ implies }\quad
\mathrm{ch}\, (\mathrm{E}_i\,M)=
\sum_{\mathbf{i}\in\mathbb{F}_p^{n-1}}
a_{(\mathbf{i},i)}\exp(\mathbf{i}).
\end{displaymath}
\item\label{exs5.7} Compute explicitly all blocks and the 
corresponding defect groups for $\mathbb{S}_3$ over an algebraically
closed field of characteristic $2$ and $3$.
\item\label{exs5.8} Prove Theorem~\ref{thm1414}\eqref{thm1414.1}
on the level of the Grothendieck group.
\item\label{exs5.9} Show that $\{f_w:w\in \mathbb{S}_n\}$
is not a basis of $\mathbb{Z}\mathbb{S}_n$ for $n>1$.
\item\label{exs5.10} Prove Proposition~\ref{prop1508} in the case
$n=3$ by a direct computation.
\item\label{exs5.11} Prove that the generalized Verma module
induced from a Verma module is again a Verma module.
\end{enumerate}
\end{exercise}

\vspace{1cm}

\noindent
Volodymyr Mazorchuk, Department of Mathematics, Uppsala University,
Box 480, 751 06, Uppsala, SWEDEN, {\tt mazor\symbol{64}math.uu.se};
http://www.math.uu.se/$\tilde{\hspace{1mm}}$mazor/.
\newpage

\printindex

\end{document}